\documentclass[a4paper, 11pt,reqno]{amsart}
\usepackage[top = 1in, bottom = 1in, left=1.1in, right=1.1in]{geometry}

\usepackage{amssymb,amsthm,amsmath,graphicx,xcolor,mathtools,enumerate,mathrsfs}
\usepackage{tabularx}
\usepackage[shortlabels]{enumitem} 
\usepackage[british]{babel}
\usepackage[utf8]{inputenc}
\usepackage{tikz}
\usetikzlibrary{patterns}

\allowdisplaybreaks

\usepackage[mathlines]{lineno}
\usepackage{etoolbox} 

\newcommand*\linenomathpatch[1]{%
	\expandafter\pretocmd\csname #1\endcsname {\linenomath}{}{}%
	\expandafter\pretocmd\csname #1*\endcsname{\linenomath}{}{}%
	\expandafter\apptocmd\csname end#1\endcsname {\endlinenomath}{}{}%
	\expandafter\apptocmd\csname end#1*\endcsname{\endlinenomath}{}{}%
}
\newcommand*\linenomathpatchAMS[1]{%
	\expandafter\pretocmd\csname #1\endcsname {\linenomathAMS}{}{}%
	\expandafter\pretocmd\csname #1*\endcsname{\linenomathAMS}{}{}%
	\expandafter\apptocmd\csname end#1\endcsname {\endlinenomath}{}{}%
	\expandafter\apptocmd\csname end#1*\endcsname{\endlinenomath}{}{}%
}

\expandafter\ifx\linenomath\linenomathWithnumbers
\let\linenomathAMS\linenomathWithnumbers
\patchcmd\linenomathAMS{\advance\postdisplaypenalty\linenopenalty}{}{}{}
\else
\let\linenomathAMS\linenomathNonumbers
\fi

\linenomathpatchAMS{gather}
\linenomathpatchAMS{multline}
\linenomathpatchAMS{align}
\linenomathpatchAMS{alignat}
\linenomathpatchAMS{flalign}
\linenomathpatch{equation}

\usepackage{hyperref}
\hypersetup{colorlinks, linkcolor={red!50!black}, citecolor={green!50!black}, urlcolor={blue!50!black}}

\usepackage{caption}
\usepackage{subcaption}

\usepackage{cleveref}
\theoremstyle{plain}
\newtheorem{theorem}{Theorem}[section]
\crefname{theorem}{Theorem}{Theorems}

\newtheorem{proposition}[theorem]{Proposition}
\crefname{proposition}{Proposition}{Propositions}

\newtheorem{corollary}[theorem]{Corollary}
\crefname{corollary}{Corollary}{Corollaries}

\newtheorem{lemma}[theorem]{Lemma}
\crefname{lemma}{Lemma}{Lemmas}

\newtheorem{conjecture}[theorem]{Conjecture}
\crefname{conjecture}{Conjecture}{Conjectures}

\crefname{problem}{Problem}{Problem}

\newtheorem{claim}[theorem]{Claim}
\crefname{claim}{Claim}{Claims}

\newtheorem{observation}[theorem]{Observation}
\crefname{observation}{Observation}{Observations}

\crefname{setup}{Setup}{Setups}

\crefname{fact}{Fact}{Facts}

\newtheorem{algorithm}[theorem]{Algorithm}
\crefname{algorithm}{Algorithm}{Algorithms}

\crefname{remark}{Remark}{Remarks}

\crefname{example}{Example}{Examples}

\theoremstyle{definition}
\newtheorem{definition}[theorem]{Definition}
\crefname{definition}{Definition}{Definitions}

\newtheorem{construction}[theorem]{Construction}
\crefname{construction}{Construction}{Constructions}

\crefname{question}{Question}{Questions}

\numberwithin{equation}{section}

\crefname{section}{Section}{Sections}
\crefname{appendix}{Appendix}{Appendices}

\crefname{figure}{Figure}{Figures}

\newcommand{\rf}[1]{\cref{#1} (\nameref*{#1})}
\newcommand{\fr}[1]{\nameref*{#1} (\cref{#1})}

\definecolor{DarkDesaturatedBlue}{HTML}{3A3556}
\definecolor{VividOrange}{HTML}{F15918}
\definecolor{PureOrange}{HTML}{FFBA00}
\definecolor{LightGrayishPink}{HTML}{EEC5D5}
\definecolor{VerySoftBlue}{HTML}{B5AFDB}

\renewcommand{\vec}{\mathbf}
\DeclareMathOperator{\probability}{Pr}
\DeclareMathOperator{\expectation}{\mathbf{E}}
\newcommand{\es}{\emptyset}
\newcommand{\eps}{\varepsilon}
\renewcommand{\rho}{\varrho}
\newcommand{\sm}{\setminus}
\renewcommand{\subset}{\subseteq}

\newcommand{\NATS}{\mathbb{N}}
\newcommand{\REALS}{\mathbb{R}}
\newcommand{\INTS}{\mathbb{Z}}
\newcommand{\maxnorm}[1]{\lVert #1 \rVert_{\infty}}
\newcommand{\Lonenorm}[1]{\lVert #1 \rVert_{1}}

\def\le{\leqslant}
\def\leq{\leqslant}
\def\ge{\geqslant}
\def\geq{\geqslant}

\newcommand{\cA}{\mathcal{A}}
\newcommand{\cB}{\mathcal{B}}
\newcommand{\cC}{\mathcal{C}}

\newcommand{\cF}{\mathcal{F}}
\newcommand{\cG}{\mathcal{G}}
\newcommand{\cH}{\mathcal{H}}

\newcommand{\cJ}{\mathcal{J}}

\newcommand{\cM}{\mathcal{M}}

\newcommand{\cQ}{\mathcal{Q}}

\newcommand{\cV}{\mathcal{V}}

\newcommand{\cX}{\mathcal{X}}
\newcommand{\cY}{\mathcal{Y}}

\usepackage[mathscr]{eucal}

\newcommand{\fU}{\mathscr{U}}
\newcommand{\fV}{\mathscr{V}}
\newcommand{\fW}{\mathscr{W}}
\newcommand{\fF}{\mathscr{F}}
\newcommand{\fJ}{\mathscr{J}}
\newcommand{\fI}{\mathscr{I}}
\newcommand{\fZ}{\mathscr{Z}}
\newcommand{\fX}{\mathscr{X}}
\newcommand{\tfX}{\mathscr{\tilde X}}

\newenvironment{proofclaim}[1][Proof of the claim]{\begin{proof}[#1]}{\end{proof}}

\title[]{On sufficient conditions for  \\ spanning structures in dense graphs}
\date{\today}

\author[R.~Lang]{Richard Lang}
\address[R.~Lang]{Universität Hamburg,
	Fachbereich Mathematik,
	Bundesstraße 55,
	20146 Hamburg, Germany}
\email{richard.lang@uni-hamburg.de}

\author[N.~Sanhueza-Matamala]{Nicolás Sanhueza-Matamala}
\address[N.~Sanhueza-Matamala]{Departamento de Ingeniería Matemática, Facultad de Ciencias Físicas y Matemáticas, Universidad de Concepción, Chile}
\email{nicolas@sanhueza.net}

\thanks{The research leading to these results was supported by the Czech Science Foundation, grant number GA19-08740S with institutional support RVO: 67985807, also by ANID-Chile through the FONDECYT Iniciación Nº11220269 grant (N.~Sanhueza-Matamala) and also by the Deutsche Forschungsgemeinschaft (DFG, German Research Foundation) -- 450397222 (R. Lang)}


\begin{document}
\begin{abstract}
	We study structural conditions in dense graphs that guarantee the existence of vertex-spanning substructures such as Hamilton cycles.
	It is easy to see that every Hamiltonian graph is connected, has a perfect fractional matching and, excluding the bipartite case, contains an odd cycle.
	A simple consequence of the Robust Expander Theorem of K\"uhn, Osthus and Treglown tells us that any large enough graph that robustly satisfies these properties must already be Hamiltonian.
	Our main result generalises this phenomenon to powers of cycles and graphs of sublinear bandwidth subject to natural generalisations of connectivity, matchings and odd cycles.
	
	This answers a question of Ebsen, Maesaka, Reiher, Schacht and Sch\"ulke and solves the embedding problem that underlies multiple lines of research on sufficient conditions for spanning structures in dense graphs.
	As applications, we recover and establish Bandwidth Theorems in a variety of settings including Ore-type degree conditions, Pósa-type degree conditions, deficiency-type conditions, locally dense and inseparable graphs, multipartite graphs as well as robust expanders.
\end{abstract}

\maketitle
\thispagestyle{empty}
\vspace{-0.4cm}
 
\section{Introduction}
An old question in discrete mathematics is whether a given graph contains certain vertex-spanning substructures, such as a Hamilton cycle.
Since the corresponding decision problems are often computationally intractable, we do not expect to find `simple' characterisations of the graphs that contain a particular spanning structure.
The extremal approach to these questions has therefore focused on easily-verifiable sufficient conditions.
A classic example in this direction is Dirac's theorem~\cite{Dir52}, which states that every graph on $n \geq 3$ vertices and minimum degree at least $n/2$ contains a Hamilton cycle.
Since its inception, Dirac's theorem has been extended in numerous ways by replacing the assumptions on the host graph and the conditions on the guest graph~\cite{Gou14,KO14}.
Here, we propose a general framework to approach these problems in dense graphs, and apply it to prove several new results.

One natural way to extend Dirac's theorem is to weaken the assumptions on the host graph.
Simple constructions show that a minimum degree of $n/2$ is best-possible.
However, as it turns out, not all of the vertices need to have this degree in order to ensure that the graph is Hamiltonian.
A well-known theorem of Ore~\cite{Ore60} states that a graph on $n$ vertices contains a Hamilton cycle if $\deg(u) + \deg(v) \geq n$ for all pairs of non-adjacent vertices $u$ and~$v$.
Pósa~\cite{Pos62} extended this by showing that a graph on $n$ vertices contains a Hamilton cycle provided that its degree sequence $d_1 \leq \dots \leq d_n$ satisfies $d_i \geq i+1$ for all $i < n/2$.\footnote{Indeed, a quick exercise shows that every graph satisfying Ore's conditions also satisfies Pósa's conditions.}
A decade later, Chvátal~\cite{Chv72} gave a complete characterisation of the integer sequences that guarantee Hamiltonicity for graphs of pointwise greater degree sequence.
Similar results have been obtained in the setting of bipartite host graphs~\cite{MM63}.
More recently, structural assumptions such as local density together with inseparability, deficiency conditions and robust expansion have been investigated, which we will discuss in more detail below.

Another way to generalise Dirac's theorem consists in embedding structurally more complex guest graphs, such as clique factors and powers of cycles.
A \emph{$k$-clique factor} in a graph $G$ is a collection of pairwise disjoint $k$-cliques (complete graphs on $k$ vertices) that cover all vertices of $G$.
Note that this concept generalises perfect matchings as the latter are simply $2$-clique factors. 
It was conjectured by Erd\H{o}s~\cite{Erd64a} and proved by Hajnal and Szemer\'edi~\cite{HS70} that graphs on $n$ vertices and minimum degree at least $(k-1)n/k$ have a $k$-clique factor,
assuming the obvious necessary condition that $k$ divides $n$.
Similarly, the notion of cycles can be generalised in terms of their powers.
The \emph{$k$th power} (or \emph{square} when $k=2$) of a graph $G$ is obtained from $G$ by joining any two vertices at distance at most $k$.\footnote{The \emph{distance} between two vertices in a common component is the number of edges in a shortest path connecting them.
Vertices of distinct components have infinite distance.}
Pósa (see~\cite{Erd64a}) (for $k=3$) and Seymour~\cite{Sey73} (for $k\geq3$) conjectured that any graph on $n$ vertices with minimum degree at least $(k-1)n/k$ contains the $(k-1)$st power of a Hamilton cycle.
Note that this presents a generalisation of the Hajnal--Szemerédi Theorem.
The conjecture was confirmed by Koml\'os, S\'ark\"ozy and Szemer\'edi~\cite{KSS98b} for sufficiently large  $n$.
While powers of cycles might appear to be a somewhat particular class of graphs, their embedding turned out to be an important milestone with regards to embedding the much richer class {of $k$-colourable graphs with} bounded degree and sublinear bandwidth (as defined below) under essentially the same minimum degree conditions necessary for $k$-clique factors.
This last result is known as the Bandwidth Theorem and was proved by Böttcher, Schacht and Taraz~\cite{BST09}.

In recent years there has been a surge of activity in the study of whether and when the above assumptions on the host graph allow us to embed increasingly complex guest graphs (as surveyed below).
Many of these results are proved using similar embedding techniques, but differ in their structural analysis.
One might therefore wonder whether there is a common structural base that `sits between' all of these assumptions and (variations of) Hamiltonicity.
For (ordinary) Hamilton cycles, this was done by Kühn, Osthus and Treglown~\cite{KOT2010} by means of the notion of {robust expanders}.
This was further extended to a wider class of $2$-colourable graphs by Knox and Treglown~\cite{KT13}.
Ebsen, Maesaka, Reiher, Schacht and Schülke~\cite{EMR+20} raised the question of generalising the concept of robust expanders to handle powers of cycles and other $k$-colourable graphs.

In this paper we resolve this question.
In particular, we introduce the notion of \emph{Hamilton frameworks} (\cref{def:Hamilton-framework}), which comprises the characteristic properties of powers of Hamilton cycles, while remaining computationally tractable.
Our main results (\cref{thm:main-powham,thm:main-bandwidth}) state that graphs that have a robust Hamilton framework are (in a strong sense) Hamiltonian and admit the corresponding powers of cycles.
As an application we can easily recover the above mentioned contributions and also prove several new results, including multiple conjectures.

\newcommand{\bw}{\mathrm{bw}}
\section{Applications}\label{sec:applications}
To formulate extensions of the Bandwidth Theorem, we introduce some further notation.
We denote the vertex and edge set of a graph $G$ by $V(G)$ and $E(G)$, respectively.
We write $v(G)=|V(G)|$ for its \emph{order} and $e(G)=|E(G)|$ for its number of edges.
A graph $H$ admits an ordering with \emph{bandwidth at most $b$} if the vertices of $H$ can be labelled with $1, \dotsc, n$ such that {$|i - j| \leq b$ for all edges $ij$}. We abbreviate this by $\bw(H) \leq b$.
The chromatic number and maximum degree of $H$ are denoted by $\chi(H)$ and $\Delta(H)$, respectively.
\begin{definition}[$(\beta, \Delta, k)$-Hamiltonian]
	A graph $G$ on $n$ vertices is \emph{$(\beta,\Delta,k)$-Hamiltonian} if $G$ contains every graph $H$ on $n$ vertices with $\chi(H) \leq k$, $\Delta(H)\leq \Delta$ and $\bw(H) \leq  \beta n$.
\end{definition}
Note that, for $\beta >0$ and $n$ large enough, the embedded guest graphs include $(k-1)$st powers of cycles $H$ on $n$ vertices when $n$ is divisible by $k$, since $H$ is $k$-colourable in this case.
Following Allen, Böttcher, Ehrenmüller and Taraz~\cite{ABET2020}, we also consider $(k+1)$-coloured graphs where one colour is used only sparingly, as follows.
A \emph{$\beta$-block} of $\{1, \dotsc, n\}$ is an interval of the type $\{ (i-1) \lceil \beta n \rceil + 1, \dotsc, i \lceil \beta n \rceil \}$ for some $1 \leq i \leq \beta^{-1}$.
A $(k+1)$-colouring $\chi\colon \{1, \dotsc, n\} \rightarrow \{0, 1, \dotsc, k\}$ is said to be \emph{$(z,\beta)$-zero-free} if, among every $z$ consecutive $\beta$-blocks at most one of them uses colour~$0$.
\begin{definition}[$(z,\beta, \Delta, k)$-Hamiltonian]
	A graph $G$ on $n$ vertices is \emph{$(z,\beta,\Delta,k)$-Hamiltonian} if $G$ contains all graphs $H$ on $n$ vertices with $\Delta(H)\leq \Delta$ which admit an ordering that certifies simultaneously a bandwidth of at most $\beta n$ and a $(z,\beta)$-zero-free $(k+1)$-colouring.
\end{definition}
We remark that the embedded guest graphs in this definition cover $(k-1)$st powers of cycles $H$ (of any order), when $\beta>0$, $z \leq 1/\beta$ and $n$ is large enough.
Moreover, a $(\beta,\Delta,k+1)$-Hamiltonian graph is $(z,\beta,\Delta,k)$-Hamiltonian for any $k,\Delta,z,\beta$.
Given this terminology, we can now formally state the aforementioned result of Böttcher, Schacht and Taraz~\cite{BST09}.
\begin{theorem}[Bandwidth Theorem]
	For any $\mu > 0$, $k \geq 2$, $\Delta > 0$, there exists $z, \beta > 0$ such that every sufficiently large graph $G$ on $n$ vertices with $\delta(G) \geq \frac{k-1}{k} n + \mu n$ is $(z, \beta, \Delta, k)$-Hamiltonian.
\end{theorem}

In the following, we present a series of consequences of our main results (\cref{thm:main-powham,thm:main-bandwidth}).
We remark that each of these applications benefits from already established structural insights (e.g. clique factors, connectivity), which are used as black boxes.
The proofs of these results can be found in \cref{sec:proof-applications}.

\subsection{Robust expansion} \label{sec:robustexpansion}
\newcommand{\RN}{\textrm{\upshape{RN}}}
Kühn, Osthus and Treglown~\cite{KOT2010} introduced the concept of \emph{robust expanders} to study Hamilton cycles in dense graphs and digraphs.\footnote{Similar concepts have been studied by Hefetz, Krivelevich and Szabó~\cite{HKS09} and Brandt, Broersma, Diestel and Kriesell~\cite{BBD+06}.}
Given $0 < \nu \leq \tau < 1$, a graph $G$ on $n$ vertices and $S \subseteq V(G)$, the \emph{$\nu$-robust neighbourhood $\RN_{\nu, G}(S)$ of $S$} is the set of vertices $v \in V(G)$ with $\deg_G(v, S) \geq \nu n$.
We say $G$ is a \emph{robust $(\nu, \tau)$-expander} if for every $S \subseteq V(G)$ with $\tau n \leq |S| \leq (1 - \tau)n$, we have $|\RN_{\nu, G}(S)| \geq |S| + \nu n$.

Robust expanders of sufficiently large minimum degree contain Hamilton cycles~\cite{KOT2010}.
Knox and Treglown~\cite{KT13} extended this to the bandwidth setting.
\begin{theorem}[Hamiltonicity under robust expansion] \label{thm:knoxtreglown-robustexpanders}
	For all $\eta,\Delta>0$, there are $\tau,\nu,\beta >0$ with $\nu \leq \tau$ and $n_0 \in \NATS$ such that the following holds.
	Let $G$ be a graph on $n \geq n_0$ vertices with $\delta(G) \geq \eta n$ which is a robust $(\nu, \tau)$-expander.
	Then $G$ is $(\beta, \Delta, 2)$-Hamiltonian.
\end{theorem}

We recover~\cref{thm:knoxtreglown-robustexpanders} by reducing it to our notion of Hamilton frameworks (\cref{pro:expander-to-framework}).
In fact, it can be observed that the notions of robust expansion (plus minimum degree) and robust Hamilton frameworks are essentially equivalent (\cref{pro:framework-to-expander}).

In the next section, we introduce \emph{Hamilton frameworks} which generalise the combination of robust expanders with linear minimum degree.
Our main result then extends \cref{thm:knoxtreglown-robustexpanders} to $(\beta, \Delta, k)$-Hamiltonicity for all $k\geq 2$.

\subsection{Ore-type conditions}
Motivated by the Hajnal--Szemer\'edi Theorem, Kierstead and Kostochka~\cite{KK08} investigated optimal Ore-type conditions which ensure the existence of clique factors and proved the following result.
\begin{theorem}[Clique factors under Ore-type conditions]\label{thm:ore-clique-factor}
	For $n$ divisible by $k$, let $G$ be a graph on $n$ vertices with $\deg(x) + \deg (y	) \geq 2 \frac{k-1}{k} n-1$ for all $xy \notin E(G)$.
	Then $G$ contains a $k$-clique factor.
\end{theorem}
Note that the result is tight, as witnessed, for instance, by slightly imbalanced complete $k$-partite graphs.
As an extension of this, Kühn, Osthus and Treglown~\cite{KOT09} showed an Ore-type result for general $F$-factors.\footnote{Given a graph $F$, an \emph{$F$-factor} of a graph $G$ is a collection of pairwise disjoint copies of $F$ that together cover all vertices of $G$. When $F$ is a clique, this gives the usual clique factor.}
For sufficiently large graphs, Châu~\cite{Cha13} proved a generalisation of Ore's theorem for squares of Hamilton cycles,
and also conjectured generalisations of this for all $k \geq 3$.
For $k = 2$, Knox and Treglown~\cite{KT13} were able to strengthen Ore's theorem to the bandwidth setting.
(This is a corollary to the more general \cref{thm:knoxtreglown-robustexpanders}).
They also conjectured corresponding extensions for all $k \geq 3$.
This was confirmed for $k=3$ by Böttcher and Müller~\cite{BM09}.
The following result proves these conjectures in a strong sense for all $k \geq 2$.

\begin{theorem}[Bandwidth Theorem under Ore-type conditions]\label{thm:bandwidth-ore}
	For $k,\Delta \in \NATS$ and $\mu > 0$, there are $z, \beta>0$ and $n_0 \in \NATS $ with the following property.
	Let $G$ be a graph on $n\geq n_0$ vertices with $\deg(x) + \deg (y) \geq 2 \frac{k-1}{k} n  + \mu n$ for all $xy \notin E(G)$.
	Then $G$ is $(z,\beta,\Delta,k)$-Hamiltonian.
\end{theorem}

\subsection{Pósa-type conditions}
Balogh, Kostochka and Treglown~\cite{BKT11,BKT13} studied degree conditions that guarantee the existence of clique factors and powers of Hamilton cycles.
Treglown~\cite{Tre16} proved the following Pósa-type result for clique factors.
\begin{theorem}[Clique factors under Pósa-type conditions]\label{thm:posa-clique-factor}
	For $k \in \NATS$ and $\mu > 0$, there is an $n_0 \in \NATS $ with the following property.
	Let $G$ be a graph with degree sequence $d_1\leq \dots \leq d_n$ where $n \geq n_0$ is divisible by $k$.
	Suppose that
	$d_i \geq  \frac{k-2}{k} n  + i + \mu n$
	for every $i \leq n/k$.
	Then $G$ contains a $k$-clique factor.
\end{theorem}
We remark that the degree conditions in \cref{thm:posa-clique-factor} are best possible apart from the term $\mu n$.
This result was further extended by Hyde, Liu and Treglown~\cite{HLT19} and Hyde and Treglown~\cite{HT20} to degree conditions which ensure factors and partial factors of arbitrary graphs (not just cliques).

Balogh, Kostochka and Treglown~\cite{BKT11} asked whether one could improve the degree sequences which ensure the existence of the $(k-1)$st power of a Hamilton cycle by allowing a non-negligible number of vertices to have degree less than $(k-1)n/k$.
Staden and Treglown~\cite{ST17} answered this question for $k=3$, showing that the conditions of \cref{thm:posa-clique-factor} also imply the existence of a squared Hamilton cycle, and conjectured extensions of this to all $k \geq 4$.

We prove this conjecture in the (more general) bandwidth setting.
{Here, Knox and Treglown~\cite{KT13} had previously} proved a Bandwidth Theorem for degree sequences for $k=2$.
(Again, a corollary to \cref{thm:knoxtreglown-robustexpanders}).
Staden and Treglown~\cite{ST17} conjectured such a result could be true for $k = 3$,
and Treglown~\cite{Tre20} extended the conjecture to all $k\geq 4$, which we hence confirm as well.

\begin{theorem}[Bandwidth Theorem under Pósa-type conditions]\label{thm:bandwidth-posa}
	For $k \geq 2$, $\Delta \in \NATS$ and $\mu > 0$, there are $z,\beta>0$ and $n_0 \in \NATS $ with the following property.
	Let $G$ be a graph with degree sequence $d_1 \leq \dots \leq d_n$ such that
	$d_i \geq  \frac{k-2}{k} n  + i + \mu n$
	for every $i \leq n/k$.
	Then $G$ is $(z,\beta,\Delta,k)$-Hamiltonian.
\end{theorem}

We remark that the degree conditions are essentially tight: by adapting examples of Balogh, Kostochka and Treglown~\cite{BKT11}, we show \cref{thm:bandwidth-posa} becomes false if $\mu n$ is replaced by $o(n^{1/2})$ (see \cref{sec:posa-lower-bound}).
Moreover, unlike the case $k=2$, the conditions of \cref{thm:bandwidth-posa} and of \cref{thm:bandwidth-ore} do not imply each other whenever $k \ge 3$, as witnessed by further constructions (see \cref{sec:posa-ore-different}).

\subsection{Locally dense and inseparable graphs}
Another type of sufficient condition for Hamiltonicity involves the notion of locally dense graphs.
For a set of vertices $U$ in a graph $G$, we denote by $e(U)$ the {number of edges of $G$ which are contained in $U$}.
For $\rho,d >0$, we say that a graph $G$ on $n$ vertices is $(\rho,d)$-\emph{dense} if $e(U) \geq d|U|^2/2-\rho n^2$ for every $U \subset V(G)$.
It was shown by Staden and Treglown~\cite{ST20} that for all $k, \Delta \in \NATS$ and $\mu,d >0$ there exists $\rho,\beta >0$ and $n_0 \in \NATS$ such that all locally dense graphs on $n \geq n_0$ vertices with minimum degree at least $(1/2+\mu)n$ are $(\beta,\Delta,k)$-Hamiltonian.
Note that in particular, the degree condition does no longer depend on $k$.

This result was further generalised by Ebsen, Maesaka, Reiher, Schacht and Schülke~\cite{EMR+20} by replacing the minimum degree with a condition of inseparability.
For $\mu > 0$, we say a graph $G$ is $\mu$-\emph{inseparable} if $e(X,Y) \geq \mu |X||Y|$ for every partition $\{X,Y\}$ of the vertex set of $G$.
Ebsen et al. proved Hamiltonicity for powers of cycles in locally dense and inseparable graphs,
and extended this to the bandwidth setting.

\begin{theorem}[Bandwidth Theorem for uniformly dense and inseparable graphs,{~\cite{EMR+20}}]\label{thm:dense-seperable-framework}
	For $k,\Delta \in \NATS$ and $\mu,d > 0$, there are $\beta,\rho>0$ and $n_0 \in \NATS $ with the following property.
	Let $G$ be a $(\rho,d)$-dense and $\mu$-inseparable graph on $n \geq n_0$ vertices.
	Then $G$ is {$(\beta,\Delta,k)$-Hamiltonian}.
\end{theorem}

Note that since a {$(\beta,\Delta,k+1)$-Hamiltonian} graph is also $(z, \beta, \Delta, k)$-Hamiltonian for all $k,\Delta,z,\beta$, this result already includes the zero-free setting.
Our framework also allows us to reprove this result.

\subsection{Deficiency-type problems}
Nenadov, Sudakov and Wagner~\cite{NSW20} proposed the study of `deficiency' problems for global spanning properties.
The \emph{join} $G \ast H$ of graphs $G$ and $H$	is the graph obtained from taking vertex-disjoint copies of $G$ and $H$ and adding every edge between the copies.
Now suppose that $G \ast K_t$ does not satisfy a particular graph property (such as being Hamiltonian).
How many edges can $G$ have in terms of its order and $t$?
Nenadov, Sudakov and Wagner~\cite{NSW20} gave a complete answer whenever the property in question is Hamiltonicity, and gave a partial answer for the property of containing a triangle factor.
Recently, Freschi, Hyde and Treglown~\cite{FHT21} resolved the problem completely for the property of containing a $k$-clique factor, for any fixed $k \ge 3$.

\begin{theorem}[Clique factors under deficiency conditions] \label{theorem:deficiency-cliquefactors}
	Let $n, t, k \in \NATS$ such that $n \ge 2$, $t \ge 0$ and $k \ge 3$ such that $t < (k-1)n$ and $k$ divides $n+t$.
	Let $r = \lceil (t+1)/(k-1) \rceil$ and $q$ be the integer remainder when $t$ is divided by $k-1$.
	Let $G$ be a graph on $n$ vertices such that $G \ast K_t$ does not have a $K_k$-factor.
	Then $e(G) \leq g(n,t,k)$ where
	\begin{equation}
		g(n,t,k) := \max \left\lbrace \binom{n}{2} - \binom{\frac{n+t}{k}+1}{2}, \binom{n}{2}-\binom{r}{2}-r(n-r-(k-2-q)) \right\rbrace. \label{equation:optimal-cliquefactor-deficiency}
	\end{equation}
\end{theorem}
These bounds are best possible when $k$ divides $n+t$, as shown by examples in the work of Freschi, Hyde and Treglown~\cite{FHT21}.
The same authors also proved an asymptotically optimal deficiency version of the Bandwidth Theorem for bipartite graphs, and asked for generalisations.
The next result extends this to $k$-colourable graphs, for any $k \geq 3$.

\begin{theorem}[Bandwidth Theorem under deficiency conditions] \label{theorem:deficiency-bandwidth}
	For $k,\Delta \in \NATS$ with $k\geq 2$ and $\mu > 0$ there are $\beta, z >0$ and $n_0 \in \NATS $ with the following property.
	Given $t \ge 0$, let $G$ be a graph on $n \ge n_0$ vertices with $e(G) \ge g(n,k,t) + \mu n^2$, as defined in \cref{equation:optimal-cliquefactor-deficiency}.
	Then $G \ast K_t$ is $(z, \beta, \Delta, k)$-Hamiltonian.
\end{theorem}

\subsection{Multipartite graphs}\label{sec:multipartite}
It is well known that Dirac's theorem also holds in bipartite graphs, where each part has size $n$ and each vertex has degree at least $n/2$~\cite{MM63}.
To formulate extensions of this result, we introduce some further notation.
Consider a \emph{balanced} family $\fU$ of $r$ disjoint sets  each of size $n$.
Let $G$ be a $\fU$-partite graph, meaning that $G$ has no edges with both endpoints in the same part when $r \geq 2$.
(For $r=1$, the definition is vacuous.)
The \emph{$\fU$-partite minimum degree} of $G$, written $\delta(G;\mathscr{U})$, is the least minimum degree of the bipartite subgraphs of $G$ induced by pairs of parts of $\fU$.
When $\fU$ is clear from the context, we simply write $\delta_r(G)$.

Multipartite extensions of Dirac's theorem were first investigated in terms of clique factors.
Fischer~\cite{Fischer1999} conjectured a multipartite version of the Hajnal--Szemerédi theorem in balanced $r$-partite graphs.
According to this, $\delta_r(G) \ge (k-1)n/k$ should be enough to find a $k$-clique factor in an $r$-partite graph $G$, of course assuming that the necessary condition that $k$ divides $rn$ holds.
Examples by Catlin~\cite{Catlin1980}, later generalised by Keevash and Mycroft~\cite{KM15b},
show that if $rn/k$ is odd and $k$ divides $n$, then $\delta_r(G) \ge (k-1)n/k+1$ is necessary.
The modified Fischer conjecture~\cite{KO09b} posits that, nevertheless, those examples are the only cases where Fischer's original conjecture does not hold.
For large $n$, this conjecture was confirmed by Magyar and Martin~\cite{MM02} for $k = r = 3$ and by Martin and Szemerédi~\cite{MS08} for $k = r = 4$.
An approximate version was shown by Lo and Markström~\cite{LM13b} and, independently, Keevash and Mycroft~\cite{KM15}.
Finally, Keevash and Mycroft~\cite{KM15b} proved the full conjecture for all large $n$.

\begin{theorem}[Multipartite Hajnal--Szemerédi Theorem]\label{theorem:multipartite-keevashmycroft}
	For $r,k \in \NATS$ with $r \ge k \geq 2$ and $\mu > 0$,  there exists $n_0 \in \NATS$ such that for all $n \ge n_0$ where $k$ divides $rn$,
	the following holds.
	Let $G$ be a balanced $r$-partite graph on $rn$ vertices with $\delta_r(G) > \frac{k-1}{k}n$.
	Then $G$ contains a $k$-clique factor.
\end{theorem}

Our next result extends this to powers of cycles and suitable $k$-colourable graphs.
We add an extra piece of terminology to describe our results.
A graph $H$ on $nk$ vertices admits an \emph{equitable $k$-colouring} if it admits a $k$-colouring where each colour is used exactly $n$ times.
A graph $G$ on $nk$ vertices is \emph{partite $(\beta, \Delta, k)$-Hamiltonian} if $G$ contains every graph $H$ on $nk$ vertices with $\Delta(H) \leq \Delta$, bandwidth at most $\beta kn$ and an equitable $k$-colouring.

\begin{theorem}[Multipartite Bandwidth Theorem] \label{thm:bandwidth-multipartite}
	For $r, k, \Delta \in \mathbb{N}$ with $r \ge k \geq 2$ and $\mu > 0$, there exists $\beta, z > 0$ and $n_0 \in \mathbb{N}$ such that the following holds.
	Let $G$ be a balanced $r$-partite graph on $rn\geq n_0$ vertices with $\delta_r(G) \geq \left( \frac{k-1}{k} + \mu \right)n$.
	It follows that
	\begin{enumerate}[{\upshape (i)}]
		\item \label{item:multipartite-r=k} if $r = k$, then $G$ is partite $(\beta, \Delta, k)$-Hamiltonian, and
		\item \label{item:multipartite-r>k} if $r > k$, then $G$ is $(z, \beta, \Delta, k)$-Hamiltonian.
	\end{enumerate}
\end{theorem}

It is worth noting that the second part is not restricted to partite embeddings, as one might have expected.
Moreover, observe that $(k-1)$st powers of tight cycles on $kn$ vertices admit equitable $k$-colourings,
thus in particular \cref{thm:bandwidth-multipartite} shows the existence of $(k-1)$st powers of Hamilton cycles in $G$, in all cases.
In the particular case of $(k-1)$st powers of cycles,
\cref{thm:bandwidth-multipartite} was shown to be true by DeBiasio, Martin and Molla~\cite{DMM21}, whose setting also allowed for suitably imbalanced partite graphs.

\section{Hamilton frameworks and main results}\label{sec:frameworks}
In the following, we introduce and motivate the definition of our central notion, \emph{robust Hamilton frameworks}, and then state our main results (\cref{thm:main-powham,thm:main-bandwidth}).

\subsection{Hamilton frameworks}
Recall that our goal is to find a suitable replacement for the notion of robust expanders in order to embed non-bipartite spanning structures.
To motivate the following definitions, consider a Hamiltonian graph $G$ that is (for sake of simplicity) non-bipartite.
Then $G$ contains no isolated vertices, is connected, has a perfect fractional matching (as defined below) and contains an odd cycle.
We call a graph that satisfies these properties a \emph{Hamilton framework}.
This property is not  equivalent to Hamiltonicity as witnessed by two vertex-disjoint odd cycles connected by a single edge.
Observe however that the framework properties of this construction are somewhat fragile.
Deleting few vertices or few edges (at every vertex) may easily remove the properties of connectivity, having a perfect fractional perfect matching or an odd cycle.

To exclude examples like this, we could restrict our attention to Hamilton frameworks that are robust against such operations and ask again whether this already guarantees the existence of a Hamilton cycle.
As it turns out, this is the case.
In fact, robust Hamilton frameworks $G$ contain any bipartite graph $H$ of sublinear bandwidth and maximum degree.
This is because robust Hamilton frameworks are equivalent to robust expansion (\cref{pro:expander-to-framework,pro:framework-to-expander}), and hence the work of K\"uhn, Osthus and Treglown~\cite{KOT2010} and Knox and Treglown~\cite{KT13} applies.
If in addition $G$ cannot be made triangle-free by deleting few vertices or few edges (at every vertex),
this can be extended to graphs $H$ with suitable zero-free $3$-colourings.
(Note that this condition is necessary, since such a $H$ might contain triangles.)
This discussion corresponds to the case $k=2$.
In the following we introduce the terminology to extend these ideas to powers of cycles and sublinear bandwidth graphs $H$ of higher chromatic number.

The general structural properties of Hamilton frameworks are formulated in terms of hypergraphs.
In a $k$-\emph{uniform hypergraph} (or \emph{$k$-graph} for short) $\cH$ every edge has exactly $k$ vertices.
A \emph{tight (Hamilton) cycle} $C \subset \cH$ is a (spanning) subgraph whose vertices can be cyclically ordered such that its edges consist precisely of all sets of $k$ consecutive vertices in this ordering.
For an (ordinary) graph $G$, we denote by $K_k(G)$ the \emph{$k$-clique (hyper)graph} of $G$ which has vertex set $V(G)$ and a $k$-edge $X$ whenever $X$ induces a $k$-clique in $G$.
We remark that $G$ contains the $(k-1)$st power of a Hamilton cycle if and only if $K_k(G)$ contains a tight Hamilton cycle.

A \emph{(closed) tight walk} in a $k$-graph $\cH$ is a (cyclically) ordered sequence of vertices such that every set of $k$ consecutive vertices forms an edge of $\cH$.
Note that vertices and edges are allowed to be visited more than once in a closed tight walk.
Moreover, a closed tight walk is a homomorphism of a tight cycle.
The \emph{length} of a tight walk is the number of vertices in the underlying sequence (counting repetitions).
We say that $\cH$ is \emph{tightly connected} if there is a closed tight walk that contains all its edges.
Finally, a \emph{tight component} of $\cH$ is an edge-maximal tightly connected subgraph in $\cH$.

A \emph{matching} $M$ in a $k$-graph $G$ is a subgraph of vertex-disjoint edges.
We use the following linear programming relaxation of matchings.
A \emph{fractional matching} is an edge weighting $\vec w \colon E(G) \to [0,1]$ such that  $\sum_{e \ni v} \vec w(e)\leq1$ for every vertex $v \in V(H)$.
The \emph{size} of $\vec w$ is the sum of its weights $\sum_{e \in E(G) } \vec w (e)$.
A {fractional matching} is \emph{perfect} if it attains the maximum possible size $v(G)/k$.

Given these preliminaries, we can now define general Hamilton frameworks.

\begin{definition}[Hamilton framework]\label{def:Hamilton-framework}
	For $k \geq 2$, let $G$ be a graph and $\cH \subset K_k(G)$ be a vertex-spanning subgraph.
	We say that $(G,\cH)$ is a $k$-uniform \emph{Hamilton framework} if
	\begin{enumerate}[(i)]
		\item $\cH$ is contained in a tight component of $K_k(G)$ and
		\item $\cH$ has a perfect fractional matching.
	\end{enumerate}
	Moreover, the Hamilton framework $(G,\cH)$ is
	\begin{enumerate}[(i), resume]
		\item \emph{aperiodic} if $\cH$ contains a closed tight walk of length coprime to $k$, and
		\item \emph{zero-free} if $\cH$ contains a $(k+1)$-clique.
	\end{enumerate}
\end{definition}
Note that zero-freeness implies aperiodicity, since a $(k+1)$-clique contains a tight cycle of length $k+1$.
For convenience, we use the convention of writing `(aperiodic/zero-free) Hamilton framework' in the assumptions and outcomes of our results to mean that the Hamilton framework in question is guaranteed to be aperiodic, guaranteed to be zero-free or not guaranteed to be either of the two.

\subsection{Multipartite graphs}
Our results concern graphs $G$, which may or may not be $k$-partite.
To keep the presentation compact, we generally assume that $G$ is $\fU$-partite with respect to some partition $\fU$ consisting of $r$ parts where either $r=1$ or $r=k$, with the convention that $\fU$-partiteness is vacuous when $r=1$.

More precisely, consider a hypergraph $\cH$ with an $(n,r)$-\emph{sized} vertex partition $\fU$, meaning that $\fU$ has $r$ parts of size $n$ each.
An edge $e \in E(\cH)$ is \emph{$\fU$-partite} if $e$ has at most one vertex in each part of $\fU$, or if $r=1$.
Similarly, $G$ is \emph{$\fU$-partite} if all edges of $G$ are $\fU$-partite.
We sometimes say that a Hamilton framework $(G,\cH)$ is $\fU$-partite meaning that $G$ is $\fU$-partite,
noting that thus $\cH$ must be $\fU$-partite as well.

\subsection{Robustness}
As discussed before, we require our frameworks to be robust under small modifications of the original graph.
We formalise this in terms of \emph{approximations}.

\begin{definition}[Approximation]\label{def:approximation}
	Let $\eps,d > 0$ and $G$ be a $\fU$-partite graph.
	A subgraph $G' \subseteq G$ is a \emph{$\fU$-partite $(\eps,d)$-approximation of $G$} if for all $U,U' \in \fU$ and $v \in V(G')$, we have
	\begin{enumerate}[(i)]
		\item $\deg_{G'}(v;U) \geq \deg_{G}(v;U) - d |U|$,
		\item $|V(G') \cap U| \geq (1 - \eps)|U|$, and
		\item $|V(G') \cap U| = |V(G') \cap U'|$.
	\end{enumerate}
	When $\eps=d$, we simply speak of an \emph{$\eps$-approximation}.
\end{definition}
We also require that every vertex is appropriately connected to the hypergraph structure of frameworks.
Consider an edge $e$ and a vertex $v \notin e$ in a $k$-graph.
We say that $v$ is \emph{linked} to $e$ if there is an edge $f$ with $v \in f$ and  $|e \cap f| = k-1$.
For two $k$-graphs $G_1=(V_1,E_1)$  and $G_2=(V_2,E_2)$, we write $G_1 \cap G_2 = (V_1 \cap V_2, E_1 \cap E_2)$.

With these definitions at hand, we can describe robust $k$-uniform Hamilton frameworks $(G,\cH)$ using subgraphs $\cH \subseteq K_k(G)$ which preserve their properties after passing to {any} approximation; and which in addition also have many linked edges at every vertex.

\begin{definition}[Robustness]\label{def:robustness}
	Let $G$ be a $\fU$-partite graph, with an $(n,r)$-sized partition $\fU$,
	and let $\cH \subset K_k(G)$.
	For $\mu >0$, we say that $(G,\cH)$ is a $k$-uniform \emph{$\mu$-proto-robust} $\fU$-partite (aperiodic/zero-free) Hamilton framework if
	\begin{enumerate}[{\upshape (R1)}]
		\item \label{item:robust-deletions}  $(G',\cH')$ is an (aperiodic/zero-free) Hamilton framework for each $\fU$-partite  $\mu$-approximation $G'$ of $G$  where $\cH'=\cH \cap K_k(G')$.
	\end{enumerate}
	Moreover, $(G,\cH)$ is called \emph{$\mu$-robust} if in addition to this
	\begin{enumerate}[{\upshape (R1)}]\addtocounter{enumi}{1}
		\item \label{item:linked-edges} each $v \in V(G)$ has at least $\mu n^k$ linked edges in $\cH$.
	\end{enumerate}
\end{definition}
As a convention, we omit the decorator `$\fU$-partite' when $r=1$ in this definition.
Before we continue, let us remark that $2$-uniform robust Hamilton frameworks are always aperiodic.
This is because, for $k=2$, the latter is equivalent to not being bipartite.
In a bipartite graph one can however eliminate all perfect fractional matchings by deleting a single vertex.

\subsection{Robust Hamilton frameworks are Hamiltonian}
Now we formulate our main results from which the outcomes in \cref{sec:applications} can be derived.
The first theorem states that robust aperiodic Hamilton frameworks contain powers of Hamilton cycles.

\begin{theorem}[Frameworks -- cycles]\label{thm:main-powham}
	For $k \in \NATS$ and $\mu > 0$ there exists $n_0 \in \NATS$ with the following property.
	Let $(G,\cH)$ be a $k$-uniform $\mu$-robust aperiodic Hamilton framework on $n\geq n_0$ vertices.
	Then $G$ contains the $(k-1)$st power of a Hamilton cycle.
\end{theorem}

The second theorem shows that graphs which admit $k$-uniform robust frameworks contain different kinds of bounded-degree sublinear-bandwidth spanning graphs,
depending on whether the framework is partite, aperiodic or zero-free.

\begin{theorem}[Frameworks -- bandwidth]\label{thm:main-bandwidth}
	For $k,\Delta \in \NATS$, $\mu > 0$, and $r=1$ or $r=k$ there are $\beta, z$ and $n_0 \in \NATS $ with the following property.
	Let $(G,\cH)$ be a $\mu$-robust $\fU$-partite $k$-uniform Hamilton framework  where $\fU$ is an $(n,r)$-sized partition of $G$ with $n\geq n_0$.

	\begin{enumerate}[{\upshape (i)}]
		\item \label{item:main-bandwidth-partite} If $r=k$, then  $G$ is $\fU$-partite $(\beta, \Delta, k)$-Hamiltonian.
		\item \label{item:main-bandwidth-aperiodic} If $r=1$ and  $(G,\cH)$ is $\mu$-robust aperiodic, then $G$ is $(\beta,\Delta,k)$-Hamiltonian.
		\item \label{item:main-bandwidth-zerofree} If $r=1$ and  $(G,\cH)$ is $\mu$-robust zero-free, then $G$ is $(z,\beta,\Delta,k)$-Hamiltonian.
	\end{enumerate}
\end{theorem}

We remark that  $\fU$ is the trivial partition in the cases~\ref{item:main-bandwidth-aperiodic} and~\ref{item:main-bandwidth-zerofree}.

\subsection{Remarks on frameworks} \label{section:remarksonframeworks}
We reflect briefly about our definition of Hamilton frameworks, robustness and our main results.

Firstly, unlike for Hamilton cycles, searching for Hamilton frameworks is computationally tractable (see \cref{sec:complexity}).
This indicates, on a theoretical level, that Hamilton frameworks are `simpler' structures.
It is moreover conceivable, that the proofs of \cref{thm:main-powham,thm:main-bandwidth} lead to efficient algorithms for constructing spanning structures given a robust framework as an input (see \cref{sec:conclusion}).

Secondly, it is not hard to see that in a $k$-uniform Hamilton framework $(G,\cH)$ of sufficiently large order, every vertex is linked to at least one edge of the tight component of $K_k(G)$ that contains $\cH$.
One could suspect that in a proto-robust framework every vertex has many linked edges.
However, we were only able to prove this when $\cH=K_k(G)$.
\begin{proposition}[Linked edges]\label{prop:linked-edges-for-free}
	For $k \in \NATS$ and $\mu > 0$ there are $0 < \mu' \leq \mu$ and $n_0 \in \NATS $ such that
	every $\fU$-partite $\mu$-proto-robust $k$-uniform (aperiodic/zero-free) Hamilton framework  $(G, K_k(G))$ on $n \ge n_0$ vertices
	is $\mu'$-robust.
\end{proposition}
\cref{prop:linked-edges-for-free} is a simple consequence of the Graph Removal Lemma. Its proof can be found in \cref{sec:linked-edges}.

Thirdly, there are interesting graphs $G$ in which $K_k(G)$ contains a tight Hamilton cycle but is not tightly connected.
It is thus very convenient for our proofs to allow that $\cH \subset K_k(G)$ must not be tightly connected itself
(see \cref{sec:inseparable}).

Next, we note that `being a robust Hamilton framework' is stronger than `robustly having a Hamilton framework', by which we mean having \emph{some} framework after the edge and vertex deletion.
This stricter assumption is necessary, as there are graphs with the latter property which do not contain the corresponding powers of cycles (see \cref{sec:fragile-frameworks}).

Finally, it is possible to relax Hamilton frameworks to almost perfect fractional matchings.
More precisely, \cref{thm:main-powham,thm:main-bandwidth} remain valid if for $\rho \geq 0$ small enough with respect to $\beta,z,k,\Delta,\mu$, we require that $\cH$ has a fractional matching of size at least $(1-\rho)v(G)/k$.
If $\rho >0$, we need as a further assumption that every vertex of $G$ lies in an edge of $\cH$, which is implicit when $\rho =0$.

\subsection{Methodology}
Let us explain how we prove our main results.
The approach rests on a combination of graph regularity and the Blow-Up Lemma.
To overcome obstacles that arise in our general setting, we require a series of new ideas, which we describe in this section.
For simplicity, we focus in the aperiodic case.
The proof of \cref{thm:main-powham} is based on two lemmas.
Given an aperiodic Hamilton framework $(G,\cH)$, we first apply \rf{lem:lemma-for-G-cycle} to find a regular partition $\fV$ of bounded size, together with a reduced graph $R$ whose vertices are the clusters of $\fV$.
(The parts of a regular partition are called \emph{clusters}.)
It turns out we can find $\cJ \subseteq K_k(R)$ so that $(R,\cJ)$ is an aperiodic Hamilton framework.
This will mean that $(R,\cJ)$ approximately captures the structure of $(G,\cH)$.
\cref{lem:lemma-for-G-cycle} also provides a spanning subgraph $R' \subset R$ of bounded maximum degree, such that $G$ has stronger regularity properties with respect to $R'$ (it consists of super-regular pairs).
We then use \rf{lem:lemma-for-C} to allocate a power of a Hamilton cycle $C$ to a $\fV$-blow-up $R^\ast$ of $R$.
Finally, we apply the Blow-Up Lemma to turn this allocation into an embedding of $C$ in $G$.

The proof of \rf{lem:lemma-for-G-cycle} begins similar as in the setting of minimum degree conditions.
However, incorporating certain exceptional vertices (into the structure described by  $R'$) turns out to be much more complicated and requires additional absorption arguments.
In the proof of \rf{lem:lemma-for-C}, a new obstacle appears.
If the partition $\fV$ is balanced after a few intermediate preparations, we can effectively turn a perfect fractional matching of $\cJ$ into the desired allocation.
Unfortunately (and unlike in the case of minimum degree conditions), \cref{lem:lemma-for-G-cycle} returns a partition $\fV$ that is too unbalanced for this argument.
We deal with this by showing that $(R,\cJ)$ inherits a stronger matching property (called `cluster-matchability'), which can handle the varying cluster sizes of~$\fV$.
This allows us to find an allocation that covers most of the vertices.
The last few vertices are then allocated using the aperiodicity of the framework.

The proof of \cref{thm:main-bandwidth} has broadly the same structure as the one of \cref{thm:main-powham}.
Instead of a powers of a cycle $C$, we are given a graph $H$ of sublinear bandwidth and bounded maximum degree, which we have to embed into $G$.
As before, the proof is based on two lemmas: \rf{lem:lemma-for-G-bandwidth} and \rf{lem:lemma-for-H}.
The difference between \cref{lem:lemma-for-G-cycle} and \cref{lem:lemma-for-G-bandwidth} is that the latter returns a subgraph $R' \subset R$, which is not only a bounded-degree graph but, more specifically, a power of a cycle; this is obtained by applying \cref{thm:main-powham} to $R$.
Unfortunately, this complicates matters when including exceptional vertices into $R'$.
We therefore have to ensure that $R'$ runs through a set of predefined edges of $\cJ$ (the hypergraph part of the reduced framework), which effectively means that we have to prove a stronger result on Hamilton cycles (\cref{thm:cycle-distributed}).

In the second part of the proof of \cref{thm:main-bandwidth}, we use \cref{lem:lemma-for-H} to allocate the given graph $H$ to the $\fV$-blow-up $R^\ast$ of the reduced graph $R$.
This allocation proceeds in several steps.
We first allocate $H$ greedily to a blow-up of a power of a path $P^\ast$, where the sizes of the blown-up clusters of $P^\ast$ may differ by a lot.
Next, we allocate $P^\ast$ to the blow-up of a power of a cycle $C^\ast$, whose blown-up clusters are almost completely balanced.
This is done by allocating the clusters of $P^\ast$ by means of a random walk on the clusters of $C^\ast$, which is ultimately possible thanks to the aperiodicity of the Hamilton framework $(G,\cH)$.
If $\fV$ was a balanced partition, we could simply let $C^\ast$ have the cluster sizes of $\fV$ (up to a tiny error) and take $R' = C$.
However, in general, we can only assume that $\fV$ is almost balanced.
We deal with this by taking the cluster sizes of $C^\ast$ to be much smaller than the cluster size of $\fV$ and then allocating $C^\ast$ into $R^\ast$ with another application of \rf{lem:lemma-for-C}.
Together, this results in an allocation that is very close to the cluster size of $\fV$.
As a final step (before applying the Blow-Up Lemma), we modify the partition $\fV$ to precisely fit this allocation.
This can be done with a simple (but novel) shifting argument, which again relies on the aperiodicity of the framework.

Finally, in the situations where the Hamilton framework is partite or zero-free, the proof takes a very similar shape.
For the case of zero-freeness, much stays the same as zero-freeness implies aperiodicity.
Only in \cref{lem:lemma-for-H} we have to be a bit more careful, when embedding the vertices of colour $0$.
More details on this can be found in \cref{sec:bandwidth}.
The partite version of \cref{thm:main-bandwidth} is actually easier to prove.
This is because most of the arguments involving aperiodicity can be replaced by simple observations on the partite assumptions on $(G,\cH)$ and $H$.

\subsection{Organisation of the paper}
The rest of the paper is organised as follows.
The next section contains some basic definitions and notation.
In \cref{sec:proof-applications}, we derive the results of \cref{sec:applications} from \cref{thm:main-powham,thm:main-bandwidth}.
The second part of the paper is dedicated to the proofs of \cref{thm:main-powham,thm:main-bandwidth}.
In \cref{sec:regularity,section:intermediate}, we introduce tools and prove a few intermediate results.
\cref{sec:cycle,sec:bandwidth} contain the proofs of \cref{thm:main-powham,thm:main-bandwidth}, respectively.
The latter depends on \rf{lem:lemma-for-H}, which is proved in \cref{sec:lemmaforH}.
We conclude the paper with a few open problems and reflections in \cref{sec:conclusion}.
Some further constructions as well as proofs of standard and borrowed results can be found in the appendix.

\section{Notation}\label{sec:notation}
In this section, we introduce some general notation that will be used throughout the paper.
For $x,y,z \in \REALS$ with $z \geq 0$, we write $x = y \pm z$ to mean $y-z \leq x \leq y + z $.
For a positive integer $n$, we write $[n]=\{1,\dots,n\}$.
A hypergraph $\cH$ consists of a set of \emph{vertices} $V(\cH)$ and \emph{edges} $E(\cH)$, where the edges are subsets of vertices.
A sub(hyper)graph $\cH' \subset \cH$ is a hypergraph with vertex set $V(\cH') \subset V(\cH)$ and edge set $E(\cH') \subset E(\cH)$.
A vertex in a hypergraph is called \emph{isolated} if it is not contained in any edge.
We say that $\cH'$ is \emph{spanning} in $\cH$ if $V(\cH') = V(\cH)$.
We say that a hypergraph is \emph{$k$-uniform} (or a \emph{$k$-graph} for short) if all edges have size $k$.
So, graphs are simply $2$-graphs.
A set  $X$ of size $k$ is a \emph{$k$-set}, or a \emph{$k$-edge} if $X$ is an edge in a hypergraph.

For a hypergraph $\cH$, a set $U \subset V(\cH)$, and $v \in V(\cH)$, the \emph{neighbours $N_\cH(v;U)$ of $v$ in $U$} is the collection of sets $X \subseteq U \setminus \{ v \}$ such that $X \cup \{ v \} \in E(\cH)$.
The \emph{degree of $v$ in $U$} is $\deg_\cH(v;U) = |N_\cH(v;U)|$.
If $U = V(\cH)$, we just write $N_\cH(v)$ and $\deg_\cH(v)$.
Finally, if $\cH$ is clear from the context, we omit its references in the index.
We denote by $\cH[U]$ the \emph{induced hypergraph of {$\cH$ on $U$}}.
A \emph{homomorphism} from a $k$-graph $\cH_1$ to a $k$-graph $\cH_2$ is a function $\phi\colon V(\cH_1) \to V(\cH_2)$, which maps edges of $\cH_1$ to edges of $\cH_2$.
An \emph{embedding} is an injective homomorphism.
The \emph{maximum degree} $\Delta(\cH)$ is defined as the maximum of $\deg_\cH(v)$ over all vertices $v \in V(\cH)$.
For $1\leq i \leq k$, we define the \emph{$i$th shadow of $\cH$}, denoted $\partial_{i} \cH $, as the $i$-graph on $V(\cH)$ which contains all $i$-edges that are a subset of an edge of $\cH$.
A \emph{tight path} in a $k$-graph is a tight walk which does not repeat vertices.
For a $k$-graph $\cH$ and $X_1, X_2 \in V(\cH)^{k-1}$,
a \emph{tight $(X_1, X_2)$-{path} in $\cH$} is a tight path in $\cH$ that starts with the vertices of $X_1$ and ends with the vertices of $X_2$.

To express the constant hierarchies in the results and definitions, we use the following standard notation.
We write $x \ll y$ to mean that for any $y \in (0, 1]$ there exists an $x_0 \in (0,1)$ such that for all $x \leq x_0$ the subsequent statements hold.
Hierarchies with more constants are defined in a similar way and are to be read from the right to the left.
Implicitly, we will assume constants appearing in a hierarchy are positive real numbers.
Moreover, if $1/x$ appears in a hierarchy we will further assume $x$ is a natural number.
Finally, for a set $X$ and a tuple $\vec v \in \REALS^X$, we let $\maxnorm{\vec v} = \max_{x \in X} |\vec v(x)|$ and $\Lonenorm{\vec v} = \sum_{x \in X} |\vec v(x)|$.

\section{Proof of the applications}\label{sec:proof-applications}
In this section, we show the applications detailed in \cref{sec:applications}.

\subsection{Observations on tight connectivity}\label{sec:connectivity}
In this subsection we make two elementary but useful observations about tight connectivity.\footnote{For a more systematic investigation of tight connectivity, see our work on tight Hamilton cycles under minimum degree conditions~\cite{LS20}.}
The following proposition implies that the tight components of a $k$-graph $\cH$ form a partition of its edge set, and hence justifies the notation of tight components.

\begin{proposition}\label{prop:tightly-connected-observation}
	Let $X$ and $Y$ be two edges in a $k$-graph which intersect in $k-1$ vertices.
	Then $X$ and $Y$ are in the same tight component.
\end{proposition}

\begin{proof}
	Consider an arbitrary closed tight walk $W$, which visits the edge $X$.
	To prove the proposition, it suffices to show that there is a tight closed walk $W'$ which covers the same edges as $W$ and in addition $Y$.
	Since $W$ contains $X$, there is a subwalk $\cX=(x_1,\dots,x_k)$ of $W$ such that $X=\{x_1,\dots,x_k\}$.
	By assumption, we have $|X \cap Y| = k-1$, so $X \setminus Y = \{x_j\}$ for some $1 \leq j \leq k$ and there exists a vertex $y$ such that $Y \setminus X = \{y\}$.
	It follows that
	\begin{align*}
		\cY:=(x_1,\dots,x_k,x_1,\dots,x_{j-1},y,x_{j+1},\dots,x_k, x_1,\dots,x_{k})
	\end{align*}
	is a tight walk which visits $X$ and $Y$.
	Hence we can replace $\cX$ with $\cY$ in $W$ to obtain a tight closed walk $W'$ which covers all edges covered by $W$ and also $Y$.
\end{proof}

Given a $k$-graph $\cH$ and a vertex $v$, we define the \emph{link graph} $L_\cH(v)$ to be the $(k-1)$-graph with vertex set $V(\cH)$ and edge set $N_\cH(v) = \{ X \colon\{v\} \cup X \in E(\cH)\}$.
\begin{proposition}\label{prop:link-graph-tightly-connected}
	Let $\cH$ be a $k$-graph and $v \in V(\cH)$.
	Suppose that $L_\cH(v)$ is tightly connected.
	Then the edges of $\cH$ containing $v$ are in the same tight component.
\end{proposition}

\begin{proof}
	By assumption, there is a tight walk $W$ in $L_\cH(v)$ that visits all of its edges.
	Let $(u_1,\dots,u_{k}) \subset W$ be a tight subwalk.
	By \cref{prop:tightly-connected-observation}, the edges $\{u_1,\dots,u_{k-1},v\}$ and $\{u_2,\dots,u_{k},v\}$ are in the same tight component of $\cH$.
	We conclude by iterating this observation along all the vertices of $W$.
\end{proof}

\subsection{Robust expansion}
In this section, we show that robust expansion together with a linear minimum degree is essentially equivalent to having a robust $2$-uniform Hamilton framework.
Recall the following result of Tutte~\cite[Corollary 30.1a]{Sch03}.

\begin{theorem}[Tutte's theorem]\label{thm:tutte}
	A graph $G$ has a perfect fractional matching if and only if $|\bigcup_{s \in S} N(s)| \geq |S|$ for every independent set $S$.
\end{theorem}

Given this, we can prove the following two simple propositions, which justify our claim that the notions of $2$-uniform Hamilton frameworks and robust expanders are essentially equivalent.

\begin{proposition}\label{pro:expander-to-framework}
	Let $1/n \ll \mu \ll \nu \leq \tau \ll \eta$.
	Let $G$ be a robust $(\nu, \tau)$-expander on $n$ vertices with $\delta(G) \geq \eta n$.
	Then $(G,G)$ is a $2$-uniform  $\mu$-robust aperiodic Hamilton framework.
\end{proposition}

\begin{proof}
	Let $G'$ be a $\mu$-approximation of $G$.
	Since $\mu \ll \nu, \tau, \eta$, it follows that $G'$ is a robust $(\nu/2, \tau/2)$-expander on $n$ vertices with $\delta(G) \geq (\eta/2) n$.
	The robust expansion and minimum-degree conditions imply easily that $G'$ is connected and not bipartite.
	The conditions also imply that $|\bigcup_{s \in S} N_{G'}(s)| \geq |S|$ for every independent set $S$ in $G'$.
	By \cref{thm:tutte}, $G'$ has a perfect fractional matching.
\end{proof}

\begin{figure}

	
	\tikzset{
		pattern size/.store in=\mcSize, 
		pattern size = 5pt,
		pattern thickness/.store in=\mcThickness, 
		pattern thickness = 0.3pt,
		pattern radius/.store in=\mcRadius, 
		pattern radius = 1pt}
	\makeatletter
	\pgfutil@ifundefined{pgf@pattern@name@_lfyd8b9q0}{
		\pgfdeclarepatternformonly[\mcThickness,\mcSize]{_lfyd8b9q0}
		{\pgfqpoint{0pt}{0pt}}
		{\pgfpoint{\mcSize+\mcThickness}{\mcSize+\mcThickness}}
		{\pgfpoint{\mcSize}{\mcSize}}
		{
			\pgfsetcolor{\tikz@pattern@color}
			\pgfsetlinewidth{\mcThickness}
			\pgfpathmoveto{\pgfqpoint{0pt}{0pt}}
			\pgfpathlineto{\pgfpoint{\mcSize+\mcThickness}{\mcSize+\mcThickness}}
			\pgfusepath{stroke}
	}}
	\makeatother
	
	
	\tikzset{
		pattern size/.store in=\mcSize, 
		pattern size = 5pt,
		pattern thickness/.store in=\mcThickness, 
		pattern thickness = 0.3pt,
		pattern radius/.store in=\mcRadius, 
		pattern radius = 1pt}
	\makeatletter
	\pgfutil@ifundefined{pgf@pattern@name@_mzwapm6zz}{
		\pgfdeclarepatternformonly[\mcThickness,\mcSize]{_mzwapm6zz}
		{\pgfqpoint{0pt}{-\mcThickness}}
		{\pgfpoint{\mcSize}{\mcSize}}
		{\pgfpoint{\mcSize}{\mcSize}}
		{
			\pgfsetcolor{\tikz@pattern@color}
			\pgfsetlinewidth{\mcThickness}
			\pgfpathmoveto{\pgfqpoint{0pt}{\mcSize}}
			\pgfpathlineto{\pgfpoint{\mcSize+\mcThickness}{-\mcThickness}}
			\pgfusepath{stroke}
	}}
	\makeatother
	
	
	\tikzset{
		pattern size/.store in=\mcSize, 
		pattern size = 5pt,
		pattern thickness/.store in=\mcThickness, 
		pattern thickness = 0.3pt,
		pattern radius/.store in=\mcRadius, 
		pattern radius = 1pt}
	\makeatletter
	\pgfutil@ifundefined{pgf@pattern@name@_tcr4zb385}{
		\makeatletter
		\pgfdeclarepatternformonly[\mcRadius,\mcThickness,\mcSize]{_tcr4zb385}
		{\pgfpoint{-0.5*\mcSize}{-0.5*\mcSize}}
		{\pgfpoint{0.5*\mcSize}{0.5*\mcSize}}
		{\pgfpoint{\mcSize}{\mcSize}}
		{
			\pgfsetcolor{\tikz@pattern@color}
			\pgfsetlinewidth{\mcThickness}
			\pgfpathcircle\pgfpointorigin{\mcRadius}
			\pgfusepath{stroke}
	}}
	\makeatother
	\tikzset{every picture/.style={line width=0.75pt}} 
	
	\begin{tikzpicture}[x=0.75pt,y=0.75pt,yscale=-1,xscale=1]
		
		\draw   (50,67) -- (290,67) -- (290,240) -- (50,240) -- cycle ;
		
		\draw  [pattern=_lfyd8b9q0,pattern size=6pt,pattern thickness=0.75pt,pattern radius=0pt, pattern color={rgb, 255:red, 0; green, 0; blue, 0}] (50,64.4) -- (291,64.4) -- (291,90) -- (50,90) -- cycle ;
		\draw  [pattern=_mzwapm6zz,pattern size=6pt,pattern thickness=0.75pt,pattern radius=0pt, pattern color={rgb, 255:red, 0; green, 0; blue, 0}] (50,181) -- (75,181) -- (75,240) -- (50,240) -- cycle ;
		
		\draw  [color={rgb, 255:red, 74; green, 144; blue, 226 }  ,draw opacity=1 ][fill={rgb, 255:red, 74; green, 144; blue, 226 }  ,fill opacity=0.5 ][line width=2]  (50,66) -- (290,66) -- (290,181) -- (50,181) -- cycle ;

		\draw  [color={rgb, 255:red, 208; green, 2; blue, 27 }  ,draw opacity=1 ][fill={rgb, 255:red, 208; green, 2; blue, 27 }  ,fill opacity=0.18 ][line width=2]  (50,64.4) -- (150,64.4) -- (150,240) -- (50,240) -- cycle ;

		\draw   (105.8,137.13) .. controls (96.8,106.13) and (158.8,91.13) .. (184.8,98.13) .. controls (210.8,105.13) and (238.8,110.13) .. (244.8,144.13) .. controls (250.8,178.13) and (240.8,214.13) .. (212.8,213.13) .. controls (184.8,212.13) and (173.8,188.13) .. (157.8,173.13) .. controls (141.8,158.13) and (114.8,168.13) .. (105.8,137.13) -- cycle ;
		
		
		\draw   (146.8,57.13) .. controls (146.85,52.46) and (144.55,50.1) .. (139.88,50.04) -- (106.53,49.64) .. controls (99.86,49.55) and (96.56,47.18) .. (96.61,42.52) .. controls (96.56,47.18) and (93.2,49.47) .. (86.53,49.39)(89.53,49.43) -- (62.08,49.09) .. controls (57.42,49.04) and (55.06,51.34) .. (55,56.01) ;
		
		\draw   (300,177) .. controls (304.67,177) and (307,174.67) .. (307,170) -- (307,127.87) .. controls (307,121.2) and (309.33,117.87) .. (314,117.87) .. controls (309.33,117.87) and (307,114.54) .. (307,107.87)(307,110.87) -- (307,77) .. controls (307,72.33) and (304.67,70) .. (300,70) ;
		\draw   (43,67) .. controls (40.12,67) and (38.68,68.44) .. (38.68,71.32) -- (38.68,71.32) .. controls (38.68,75.44) and (37.24,77.5) .. (34.35,77.5) .. controls (37.24,77.5) and (38.68,79.56) .. (38.68,83.68)(38.68,81.82) -- (38.68,83.68) .. controls (38.68,86.56) and (40.12,88) .. (43,88) ;
		
		\draw   (51.8,247.4) .. controls (51.8,250.15) and (53.17,251.52) .. (55.91,251.52) -- (55.91,251.52) .. controls (59.84,251.52) and (61.8,252.89) .. (61.8,255.64) .. controls (61.8,252.89) and (63.76,251.52) .. (67.68,251.52)(65.91,251.52) -- (67.68,251.52) .. controls (70.43,251.52) and (71.8,250.15) .. (71.8,247.4) ;
		
		\draw   (40.8,94.4) .. controls (36.13,94.4) and (33.8,96.73) .. (33.8,101.4) -- (33.8,130.26) .. controls (33.8,136.93) and (31.47,140.26) .. (26.8,140.26) .. controls (31.47,140.26) and (33.8,143.59) .. (33.8,150.26)(33.8,147.26) -- (33.8,171.4) .. controls (33.8,176.07) and (36.13,178.4) .. (40.8,178.4) ;

		\draw    (125.8,116.13) -- (164.8,109.13) ;
		\draw [shift={(164.8,109.13)}, rotate = 349.82] [color={rgb, 255:red, 0; green, 0; blue, 0 }  ][fill={rgb, 255:red, 0; green, 0; blue, 0 }  ][line width=0.75]      (0, 0) circle [x radius= 1.34, y radius= 1.34]   ;
		\draw [shift={(125.8,116.13)}, rotate = 349.82] [color={rgb, 255:red, 0; green, 0; blue, 0 }  ][fill={rgb, 255:red, 0; green, 0; blue, 0 }  ][line width=0.75]      (0, 0) circle [x radius= 1.34, y radius= 1.34]   ;
		\draw    (182.8,143.13) -- (197.8,117.13) ;
		\draw [shift={(197.8,117.13)}, rotate = 299.98] [color={rgb, 255:red, 0; green, 0; blue, 0 }  ][fill={rgb, 255:red, 0; green, 0; blue, 0 }  ][line width=0.75]      (0, 0) circle [x radius= 1.34, y radius= 1.34]   ;
		\draw [shift={(182.8,143.13)}, rotate = 299.98] [color={rgb, 255:red, 0; green, 0; blue, 0 }  ][fill={rgb, 255:red, 0; green, 0; blue, 0 }  ][line width=0.75]      (0, 0) circle [x radius= 1.34, y radius= 1.34]   ;
		\draw    (139.8,144.13) -- (207.8,191.13) ;
		\draw [shift={(207.8,191.13)}, rotate = 34.65] [color={rgb, 255:red, 0; green, 0; blue, 0 }  ][fill={rgb, 255:red, 0; green, 0; blue, 0 }  ][line width=0.75]      (0, 0) circle [x radius= 1.34, y radius= 1.34]   ;
		\draw [shift={(139.8,144.13)}, rotate = 34.65] [color={rgb, 255:red, 0; green, 0; blue, 0 }  ][fill={rgb, 255:red, 0; green, 0; blue, 0 }  ][line width=0.75]      (0, 0) circle [x radius= 1.34, y radius= 1.34]   ;
		\draw    (222.8,194.13) -- (219.8,153.13) ;
		\draw [shift={(219.8,153.13)}, rotate = 265.82] [color={rgb, 255:red, 0; green, 0; blue, 0 }  ][fill={rgb, 255:red, 0; green, 0; blue, 0 }  ][line width=0.75]      (0, 0) circle [x radius= 1.34, y radius= 1.34]   ;
		\draw [shift={(222.8,194.13)}, rotate = 265.82] [color={rgb, 255:red, 0; green, 0; blue, 0 }  ][fill={rgb, 255:red, 0; green, 0; blue, 0 }  ][line width=0.75]      (0, 0) circle [x radius= 1.34, y radius= 1.34]   ;

		\draw (88,20.4) node [anchor=north west][inner sep=0.75pt]  [font=\Large,color={rgb, 255:red, 208; green, 2; blue, 27 }  ,opacity=1 ]  {${\textstyle R}$};
		\draw (322,110.4) node [anchor=north west][inner sep=0.75pt]  [font=\Large,color={rgb, 255:red, 208; green, 2; blue, 27 }  ,opacity=1 ]  {$\textcolor[rgb]{0.29,0.56,0.89}{S}$};
		\draw (300,200.4) node [anchor=north west][inner sep=0.75pt]  [font=\Large ,opacity=1 ]  {$ {V(G)}$};
		\draw (226,216.4) node [anchor=north west][inner sep=0.75pt]  [font=\Large ,opacity=1 ]  {$F$};
		\draw (10,65.4) node [anchor=north west][inner sep=0.75pt]  [font=\large,color={rgb, 255:red, 208; green, 2; blue, 27 }  ,opacity=1 ]  {$\textcolor[rgb]{0.29,0.56,0.89}{{\displaystyle S^{\ast }}}$};
		\draw (51,260.4) node [anchor=north west][inner sep=0.75pt]  [font=\large,color={rgb, 255:red, 208; green, 2; blue, 27 }  ,opacity=1 ]  {$R^{\ast }$};
		\draw (6,130.4) node [anchor=north west][inner sep=0.75pt]  [font=\large,color={rgb, 255:red, 208; green, 2; blue, 27 }  ,opacity=1 ]  {${\displaystyle \textcolor[rgb]{0.29,0.56,0.89}{S}\textcolor[rgb]{0.29,0.56,0.89}{_{1}}}$};

	\end{tikzpicture}
	\caption{The vertex and edge sets considered in the proof of \cref{pro:framework-to-expander}.}
	\label{fig:expander}
\end{figure}
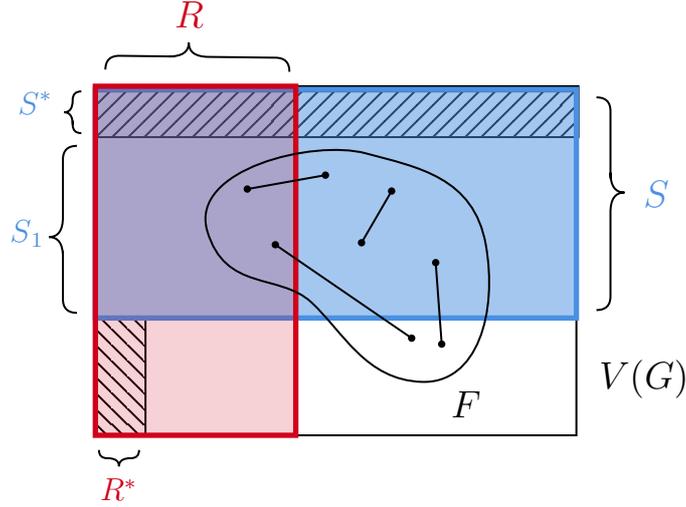

\begin{proposition} \label{pro:framework-to-expander}
	Let $\mu > 0$ and let $(G,G)$ be a $2$-uniform $\mu$-robust Hamilton framework on $n$ vertices.
	Then $G$ is a robust $(\mu^2/64, \mu/2)$-expander with $\delta(G) \geq \mu n$.
\end{proposition}

\begin{proof}
	Since $(G,G)$ is a $2$-uniform $\mu$-robust Hamilton framework, $G$ is $(\mu n)$-vertex-connected and in particular, $\delta(G) \geq \mu n$.
	Let $\nu = \mu^2/64$ and $\tau = \mu/2$.
	Set $V:=V(G)$.
	Now suppose, for sake of contradiction, that $G$ is not a robust $(\nu, \tau)$-expander.
	Then there is some set $S \subseteq V$ with $\tau n \leq |S| \leq (1 - \tau)n$ and $|R| < |S| + \nu n$,  where $R = \RN_{\nu, G}(S)$.
	This is illustrated in \cref{fig:expander}.
	Let $S^\ast = \{ v \in S \colon d(v, V \setminus R) > \mu n/8 \}$.
	We have that
	\begin{align*}
		|S^\ast|  \frac{\mu}{8} n < \sum_{v \in S} d(v, V \setminus R) = \sum_{v \in V \setminus R} d(v, S) < 2 \nu n |V \setminus R| \leq \frac{\mu^2}{32} n^2.
	\end{align*}
	From this, we deduce
	\begin{align}\label{itm:framework-to-expander-Sstar-bound}
		|S^\ast| < \frac{\mu}{4} n.
	\end{align}
	In particular, $|S \setminus S^\ast| = |S| - |S^\ast| \geq \tau n  - \mu n / 4 > 0$ by choice of $S$ and $\tau$.
	Next, let
		{$S_1 = S \setminus S^\ast$ and $F = \{ uv \in E(G)\colon u \in S_1,\, v \in V \setminus R \}$.}
	Note that $d(u,V \sm R) \leq \mu n/8$ for $u \in S_1$ and $d(v,S_1) \leq \nu n$ for $v \in V \sm R$.
	It follows that for all $v \in V$,
	\begin{align}\label{itm:framework-to-expander-degree-F}
		d_F(v) \leq  \frac{\mu n}{8} + \nu n = \frac{\mu n}{8} + \frac{\mu^2 n}{64} \leq \frac{\mu}{4}n.
	\end{align}
	Now let $G_1 = G - S^\ast - F$.
	Note that 
	\begin{align}
		\label{item:framtorobust-keyobservation}
		\text{$G_1$ has no edges with one end in $S_1$ and the other end in $V(G_1) \setminus R$.}
	\end{align}
	Consequently, 
	\begin{align}
		\label{item:framtorobust-keyobservation-otherstuff}
		\text{the neighbours of every $s \in S_1 \sm R$ in $G_1$ are in $R \sm S$.}
	\end{align}
	Next, we claim that
	\begin{align}\label{itm:framework-to-expander-RminusS}
		|R \setminus S| \geq \frac{\mu}{2} n.
	\end{align}
		To see this, suppose otherwise.
		We then have $|R \cup S| = |S| + |R \setminus S| < |S| + \mu n / 2 \leq n$, where we used that $|S| \leq (1 - \tau)n$ in the last inequality.
		So in particular, $V \setminus (S \cup R)$ is non-empty.
		Let $G_2 = G_1 - (R \setminus S)$.
		Since $|R \setminus S| + |S^\ast| \leq \mu n/2 + \mu n / 4 = 3 \mu n / 4$ by~\eqref{itm:framework-to-expander-Sstar-bound} and $\Delta(F) \leq \mu  n/4$ by \eqref{itm:framework-to-expander-degree-F}, it follows that $G_2$ is a $\mu$-approximation of $G$.
		In particular, $G_2$ is connected.
		Yet, $S_1$ and $V \setminus (S \cup R)$ are disconnected in $G_2$ by observation~\eqref{item:framtorobust-keyobservation}.
		This is a contradiction, which proves the above claim.
	
	By inequality~\eqref{itm:framework-to-expander-RminusS}, we may consider a set $R^\ast \subseteq R \setminus S$ of size $\mu n / 2$.
	Let $G_3 = G - S^\ast - R^\ast - F$.
	Since $|S^\ast|+|R^\ast| \leq 3 \mu n / 4$ by~\eqref{itm:framework-to-expander-Sstar-bound} and $\Delta(F) \leq \mu n / 4$, it follows that $G_3$ is a $\mu$-approximation of $G$.
	Note that $S_1 \setminus R$ is non-empty, since
	\begin{align*}
		|S_1 \setminus R| \geq |S \setminus R| - |S^\ast| = (|S| - |R|) + |R \setminus S| - |S^\ast| \geq - \nu n + \frac{\mu}{2} n  - \frac{\mu}{4} n > 0,
	\end{align*}
	where we used the choice of $S$ together with inequalities~\eqref{itm:framework-to-expander-Sstar-bound}~and~\eqref{itm:framework-to-expander-RminusS}.
	
	Now observation~\eqref{item:framtorobust-keyobservation} implies that $S_1 \setminus R$ is an independent set in $G_3 \subset G_1$.
	Since $G_3$ is a $\mu$-approximation of $G$ it must contain a perfect fractional matching, so together with \Cref{thm:tutte}, we have $\left| \bigcup_{s \in S_1 \setminus R} N_{G_3}(s) \right| \geq |S_1 \setminus R|$.
	On the other hand, using observation~\eqref{item:framtorobust-keyobservation} again, we see that
	\begin{align*}
		\left| \bigcup_{s \in S_1 \setminus R} N_{G_3}(s) \right|
		& = \left|  \bigcup_{s \in S_1 \setminus R} N_{G_1}(s)\setminus R^\ast   \right| \\
		& \leq |R \sm S_1 | - |R^\ast| \\
		& = |R| - |R \cap S_1| - |R^\ast| \\
		& < |S| + \nu n  - |R \cap S_1| - |R^\ast| \\
		& = |S_1| + |S^\ast| + \nu n  - |R \cap S_1|   - |R^\ast| \\
		& = |S_1 \setminus R| + |S^\ast| + \nu n - |R^\ast| \\
		& \leq |S_1 \setminus R| + \frac{\mu}{4} n  + \frac{\mu^2}{64} n - \frac{\mu}{2} n \leq |S_1 \setminus R|,
	\end{align*}
	where we used inequality~\eqref{itm:framework-to-expander-degree-F}, the definition of $\nu$ and the choice of $R^\ast$ in the last step.
	This gives a contradiction and finishes the proof.
\end{proof}

Note that \cref{pro:expander-to-framework} together with  \cref{thm:main-powham} implies that robust expanders of sufficiently large minimum degree contain Hamilton cycles,
and together with
\cref{thm:main-bandwidth}\ref{item:main-bandwidth-aperiodic}
we recover \cref{thm:knoxtreglown-robustexpanders}.
We remark that \cref{thm:main-bandwidth}\ref{item:main-bandwidth-zerofree} does not apply here, as there are triangle-free robust expanders.

\subsection{Pósa-type conditions}\label{sec:posa}
The goal of this subsection is to prove \rf{thm:bandwidth-posa}.
All we require in addition to the results of \cref{sec:applications,sec:frameworks} is the following lemma (to be proved later).
\begin{lemma}\label{lem:posa-connectivity}
	For $k\geq2$, let $G$ be a graph with vertices $v_1,\dots,v_n$.
	Suppose that $\deg(v_i)   >  (k-2)n/k  + \min\{i,n/k\}$ for every $i \leq n$.
	Then $K_k(G)$ is tightly connected, spans all vertices of $G$, and contains a $(k+1)$-clique.
\end{lemma}

Given \cref{lem:posa-connectivity}, we can easily combine \rf{thm:posa-clique-factor} and \cref{thm:main-bandwidth} to derive the main result of this section.

\begin{proof}[Proof of \cref{thm:bandwidth-posa}]
	We can assume $\mu \leq 1/2$, as otherwise the statement is vacuous.
	Let $\mu', n_1$ be given by \cref{prop:linked-edges-for-free} under input $k, \mu/2$.
	Let $n_2$ be given by \cref{thm:posa-clique-factor} under input $k, \mu/4$.
	Let $\beta, z >0$, and $n_3 \in \NATS$ be obtained from \cref{thm:main-bandwidth} applied with $k, \Delta,\mu'$,
	and let $n_0 = \max \{ n_1, (1 + \mu)(n_2+k^2), n_3, 8 \mu^{-1} k^2\}$.

	Let $n \geq n_0$ and consider a graph $G$ with degree sequence $d_1\leq \dots \leq d_n$ such that $d_i \geq \frac{k-2}{k} n  + i + \mu n$ for every $i \leq n/k$.
	Let $\cH = K_k(G)$.
	By \cref{thm:main-bandwidth}, it is enough to show that $(G, \cH)$ is a $\mu'$-robust $k$-uniform zero-free Hamilton framework. 
	By \cref{prop:linked-edges-for-free} and the choice of $\mu', n_0$, it is enough to show that $(G,\cH)$ is a $(\mu/2)$-proto-robust zero-free Hamilton framework.

	To see this, let $G' \subset G$ be a $(\mu/2)$-approximation of $G$ and let $\cH'=\cH \cap K_k(G')$.
	Let $n'$ be the number of vertices of $G'$, and note that $n' \geq (1 - \mu/2)n$.
	Moreover $G'$ satisfies the same Pósa-type degree conditions as $G$, now with $\mu/2$ instead of $\mu$.
	Consequently, $\cH'$ is tightly connected and contains a $(k+1)$-clique by \cref{lem:posa-connectivity}.
	In fact, iterating this observation we can obtain $q$ such $(k + 1)$-cliques
	where $q$ is the integer remainder when dividing $n'$ by $k$.
	Removing these cliques leaves a graph $G''$ on at least $n' - k^2 \geq n_2$ vertices.
	Since $\mu \leq 1/2$, $n \geq 8 \mu^{-1}k^2$ and $n' \geq (1 - \mu/2)n$, we have $\mu n' / 4 \geq k^2$.
	This implies that $G''$ retains its Pósa-type conditions, but now with $\mu/4$ instead of $\mu$.
	By \cref{thm:posa-clique-factor}, $G''$ has a $k$-clique factor $M$.
	Thus, we can define a perfect fractional matching in $G'$ by assigning weight $1$ to the cliques of $M$ and weight $1/k$ to the edges of the $(k+1)$-cliques.
	This shows that $(G',\cH')$ is a zero-free Hamilton framework, and finishes the proof.
\end{proof}

The rest of this subsection is dedicated to the proof of \cref{lem:posa-connectivity}.
Note that the assumptions of this lemma imply that $n \geq k+1$.
Given the statement of \cref{lem:posa-connectivity}, we say a vertex $v$ of $G$ is \emph{big} if $\deg_G(v) > (k-1)n/k$.
Note that if $k\ge 3$, then there are more than $2n/k$ big vertices in $G$.
The following observation clears the way for an induction on $k$.

\begin{proposition}\label{prop:posa-link}
	Given the assumptions of \cref{lem:posa-connectivity} and $k \ge 3$, let $v$ be a big vertex, $G'=G[N(v)]$ and $n' = | V(G')|$.
	Then $\deg_{G'}(v_j) > \frac{k-3}{k-1}n' + \min\left\lbrace j,\frac{n' }{k-1}\right\rbrace$ for each $v_j \in V(G')$.
\end{proposition}

\begin{proof}
	Let $v_j$ be a vertex in $V(G')$.
	Using that $\deg_G(v) = n' > (k-1)n/k$, we obtain
	\begin{align*}
		\deg_{G'}(v_j)
		 & \geq |N_{G}(v_j) \sm (V(G) \sm N_{G}(v))|
		\geq \deg_G(v_j) - (n- \deg_G(v))                                                                            \\
		 & > \frac{k-2}{k}n + \min\left\lbrace j,\frac{n}{k}\right\rbrace - (n - n')
		= n' - \left(\frac{2}{k} - \min\left\lbrace \frac{j}{n},\frac{1}{k}\right\rbrace \right)n                 \\
		 & > n' - \left(\frac{2}{k} - \min\left\lbrace \frac{j}{n},\frac{1}{k}\right\rbrace \right) \frac{k}{k-1} n'
		= \frac{k-3}{k-1}n' + \min\left\lbrace \frac{k}{k-1} \frac{j}{n},\frac{1}{k-1}\right\rbrace n'            \\
		 & \geq \frac{k-3}{k-1}n' + \min\left\lbrace j,\frac{{n'}}{k-1}\right\rbrace.\qedhere
	\end{align*}
\end{proof}

\begin{proof}[Proof of \cref{lem:posa-connectivity}]
	We go by induction on $k$.
	The case of $k=2$ is simple.
	By Pósa's theorem, $G$ contains a Hamilton cycle, thus $K_2(G) = G$ is spanning and tightly connected.
	A vertex $v$ of maximum degree in $G$ has $\deg_G(v) > n/2$,
	consequently any Hamilton cycle in $G$ must have an edge in $N(v)$.
	This implies $G$ contains a triangle, as needed.

	Now let us assume that $k \geq 3$ and the statement holds for all $k' < k$.
	Consider a graph $G$ with vertices $v_1,\dots,v_n$ as in the statement and abbreviate $\cH = K_k(G)$.
	Let $v$ be a big vertex.
	By \cref{prop:posa-link} and induction, the link graph $L_{\cH}(v)$ contains a $k$-clique.
	Hence $G$ contains a $k+1$ clique.

	Showing that $\cH$ is tightly connected is a bit more involved.
	We proceed in three steps.
	First, we prove that the edges in $\cH$ that contain a big vertex $v$ are in a common tight component $\cH_v$.
	Next, we show that there is a tight component $\cH'$ such that $\cH'=\cH_v$ for all big vertices $v$.
	Finally, we prove that $\cH'$ covers in fact all edges of $\cH$.
	Now come the details.

	\begin{observation}\label{cla:vi-tighlty-connected}
		If $v$ is a big vertex, then $L_{\cH}(v)$ is tightly connected and spans $N_G(v)$.
	\end{observation}

	\begin{proofclaim}
		Indeed, note that all edges of $L_{\cH}(v)$ are contained in $N_G(v)$.
		So we can restrict our analysis to $G' = G[N_G(v)]$.
		Since $k \ge 3$, \cref{prop:posa-link} tells us that $G'$ satisfies the induction hypothesis with $k' = k-1 \ge 2$,
		so we get that $K_{k-1}(G') = L_{\cH}(v)$ is tightly connected and spanning in $V(G') = N_G(v)$, as desired.
	\end{proofclaim}

	\begin{observation}\label{claim:spanningHv}
		If $v$ is a big vertex and $u \in N_G(v)$, then $u$ and $v$ belong to a common $k$-clique of $G$.
	\end{observation}

	\begin{proofclaim}
		By \cref{cla:vi-tighlty-connected}, $L_{\cH}(v)$ is tightly connected and spans all $N_G(v)$.
		Since $u \in N_G(v)$, in particular $u$ belongs to a $(k-1)$-edge of $L_{\cH}(v)$,
		which is equivalent to say that $u$ lies in a $(k-1)$-clique inside $N_G(v)$, and thus in a $k$-clique together with $v$.
	\end{proofclaim}

	Now it is easy to see that $\cH$ is spanning.

	\begin{observation} \label{claim:spanning}
		Every vertex of $G$ is contained in a $k$-clique.
	\end{observation}

	\begin{proofclaim}
		Let $v$ be a vertex of $G$.
		Using $\deg_G(v) > (k-2)n/k+1$ and the fact that there are more than $2n/k$ big vertices in $G$,
		we deduce that there exists a big vertex $w$ such that $w \in N_G(v)$.
		It follows by \cref{claim:spanningHv} that $v \in N_G(w)$ is contained in a $k$-clique.
	\end{proofclaim}

	Together with \cref{prop:link-graph-tightly-connected}, \cref{cla:vi-tighlty-connected} implies that, given a big vertex $v$, all the $k$-edges containing $v$ belong to the same tight component of $\cH$.
	We will denote that component by $\cH_v$.

	\begin{observation}\label{claim:big-component-posa}
		If $u, v$ are big vertices, then $\cH_u = \cH_v$.
	\end{observation}

	\begin{proofclaim}
		Since $u$ and $v$ are big, their common neighbourhood contains at least $(k-2)n/k$ vertices.
		Since there are more than $2n/k$ big vertices in $G$, there exists a big vertex $w$ adjacent to both $u$ and $v$.
		By \cref{claim:spanningHv}, there exists a $k$-clique $X$ containing both $u$ and $w$.
		Since $u$ and $w$ are big vertices, then $X$ belongs to $\cH_w$ as well as $\cH_u$, hence $\cH_u = \cH_w$.
		By symmetry, also $\cH_v = \cH_w$ and thus $\cH_u = \cH_v$.
	\end{proofclaim}

	By \cref{claim:big-component-posa}, there exists a tight component $\cH'$ of $\cH$ which contains $\cH_u$ for every big vertex $u$.
	By \cref{claim:spanning}, we can finish the proof by showing that $\cH'$ actually contains all edges of $\cH$.

	For the sake of contradiction, suppose that this is not true.
	Consider an edge $X=\{v_{i_1}, \dots ,v_{i_k}\}$ in $\cH$, which is not in $\cH'$ and maximises $\sum_j i_j$.
	Suppose $i_1 < \dotsb < i_k$.
	If $X$ contains a big vertex $u$, then $X \in E(\cH_u) = \cH'$, a contradiction.
	We therefore must have $i_j < n/k$ for all $1 \leq j \leq k$.
	Let $e(X)$ and $e(X, V\setminus X)$ be the number of edges of $G$ inside $X$, and the number of edges with exactly one endpoint in $X$, respectively.
	On one hand, we have
	\begin{align}
		e(X, V \setminus X)
		 & = \sum \deg_G(v_{i_j}) - 2 e(X)
		\geq k \left(\frac{k-2}{k}\right)n  + \sum  i_j - 2 \binom{k}{2} \nonumber \\
		 & = (k-2)n  +\sum  i_j - k(k-1). \label{equation:lowerbound}
	\end{align}

	Define $I_1 = \{ 1, \dotsc, i_1 - 1 \}$, $I_2 = \{ i_1,i_1+1, \dotsc, i_{k} \}$, $I_3 = \{ i_k + 1, \dotsc, n \}$, and for all $1 \leq i \leq 3$ let $V_i = \{ v_j \colon j \in I_i \} \setminus X$.
	Note that $|V_1| = i_1 - 1$, $|V_2| = |I_2| - |X| = i_k - i_1 + 1- k$ and $|V_3| = n - i_k$.
	In the following, we show that
	\begin{enumerate}[(a)]
		\item \label{itm:V1-k-neighbours} every vertex in $V_1$ has at most $k$ neighbours in $X$,
		\item \label{itm:V2-k-1-neighbours} every vertex in $V_2$ has at most $k-1$ neighbours in $X$, and
		\item \label{itm:V3-k-2-neighbours} every vertex in $V_3$ has at most $k-2$ neighbours in $X$.
	\end{enumerate}

	Part~\ref{itm:V1-k-neighbours} is trivial.
	For part~\ref{itm:V2-k-1-neighbours}, let us suppose otherwise and obtain a contradiction.
	Then there exists $\ell > i_1$ such that $v_{\ell}$ has $k$ neighbours in $X$.
	Let $X' = ( X \setminus \{ v_{i_1} \} ) \cup \{ v_{\ell} \}$.
	It follows that $X'$ is an edge in $\cH$ with a higher sum of indices as $X$.
	Moreover, $X'$ is in the same tight component as $X$ by \cref{prop:tightly-connected-observation}, a contradiction.
	Similarly, for part~\ref{itm:V3-k-2-neighbours}, suppose otherwise for sake of contradiction.
	Then there exists $\ell > i_k$ such that $v_{\ell}$ has at least $k-1$ neighbours in $X$.
	Let $X'$ be a $k$-clique obtained from $X$ by replacing a vertex not adjacent to $v_\ell$ with $v_\ell$.
	(If $v_\ell$ is adjacent to all vertices of $X$, just replace an arbitrary vertex.)
	As before, $X'$ is an edge in $\cH$ in the same component as $X$ and with a higher sum of indices as $X$, contradiction.

	Putting observations~\ref{itm:V1-k-neighbours} to~\ref{itm:V3-k-2-neighbours} together, we get
	\begin{align*}
		e(X, V \setminus X)
		 & \leq |V_1|k + |V_2|(k-1) + |V_3|(k-2)                    \\
		 & = (i_1 - 1)k + (i_k - i_1 - k + 1)(k-1) + (n - i_k)(k-2) \\
		 & = (k-2)n + i_1 + i_k -k(k-1) -1                          \\
		 & < (k-2)n + \sum i_j -k(k-1),
	\end{align*}
	which contradicts inequality~\eqref{equation:lowerbound} and finishes the proof.
\end{proof}

\subsection{Ore-type conditions}\label{sec:ore}
In this section, we show \rf{thm:bandwidth-ore}.
The general strategy follows what we have done in the last section.
Recall that a graph $G$ is $\delta$-\emph{Ore} if $\deg(x) + \deg (y) > 2\delta v(G)$ for all $xy \notin E(G)$.
In such a graph, we also say that a vertex $x$ is \emph{$\delta$-big} if $\deg(x) > \delta n$, and otherwise is \emph{$\delta$-small}.
If the context is clear, we just say \emph{big} and \emph{small} instead.
Note that the set of small vertices in a $\delta$-Ore graph forms a clique under this definition.

Unlike Dirac-type conditions, Ore-type conditions are not necessarily inherited by approximations, for instance when an edge between two small vertices is deleted.
We will therefore need stronger structural insights before applying \cref{thm:main-bandwidth}.

We show the following two lemmas.
\begin{lemma}\label{cor:ore-robust-connectivity}
	Let $1/n \ll \nu \ll \mu,1/k$ and
	let $G$ be a $(\frac{k-1}{k} + \mu)$-Ore graph on $n$ vertices.
	Let $G'$ be a $\nu$-approximation of $G$.
	Then $K_k(G')$ is tightly connected, spans all vertices, and contains a $(k+1)$-clique.
\end{lemma}
\begin{lemma}\label{lem:ore-perfect-matchin}
	Let $1/n \ll \nu \ll \mu, 1/k$.
	Let $G$ be a $(\frac{k-1}{k}+\mu)$-Ore graph on $n$ vertices.
	Let $G'$ be a $\nu$-approximation of $G$.
	Then $K_k(G')$ has a perfect fractional matching.
\end{lemma}

Given \cref{cor:ore-robust-connectivity,lem:ore-perfect-matchin}, one can derive \cref{thm:bandwidth-ore} from \cref{thm:main-bandwidth} in the same way as it was done for the proof of \rf{thm:bandwidth-posa}.
We omit the details.

The rest of the subsection is dedicated to the proof of \cref{cor:ore-robust-connectivity,lem:ore-perfect-matchin}.
We start by showing the following.

\begin{lemma}\label{lem:ore-connectivity}
	Let $G$ be a $\frac{k-1}{k}$-Ore graph on $n \geq k+1$ vertices.
	Then $K_k(G)$ is tightly connected, spans all vertices, and all big vertices are on a clique of size $k+1$.
\end{lemma}

Note that, since $n \geq k+1$ in this lemma, the set of big vertices is never empty.
The following observation clears the way for an induction on $k$.

\begin{proposition}\label{prop:ore-link}
	Let $G$ be a $\frac{k-1}{k}$-Ore graph on $n > k$ vertices with $k\geq 3$.
	Let $v$ be a big vertex, $G'=G[N(v)]$ and $n' = | V(G')|$.
	Then $G'$ is $\frac{k-2}{k-1}$-Ore with $n' > k-1$.
\end{proposition}

\begin{proof}
	Since $n > k$ and $n' = \deg_G(v) > (k-1)n/k$, we have $n' > k-1$.
	Let $x,y$ be vertices in $V(G')$ with $xy \notin E(G')$.
	Using $n' > (k-1)n/k$, we obtain
	\begin{align*}
		\deg_{G'}(x) + \deg_{G'}(y)
		 & \geq \deg_G(x) + \deg_G(y) - 2(n- \deg_G(v))
		\\ &> 2\frac{k-1}{k}n - 2\frac{k}{k}n + 2\deg_G(v)
		= 2\deg_G(v) - 2\frac{1}{k}n
		\\ &\geq 2n' - 2\frac{1}{k-1}  n'
		= 2 \frac{k-2}{k-1}n'.\qedhere
	\end{align*}
\end{proof}

\begin{proof}[Proof of \cref{lem:ore-connectivity}]
	We go by induction on $k$.
	If $k = 2$, Ore's theorem implies the existence of a Hamilton cycle, so $K_2(G)=G$ is tightly connected.
	Let $x$ be a big vertex in $G$ and note that $\deg(x) > n/2$.
	Hence any Hamilton cycle of $G$ must have an edge in $N(x)$.
	It follows that $G$ has a triangle.
	This concludes the case $k=2$.

	Now let us assume that $k \geq 3$ and the statement holds for all $k' < k$.
	Consider a graph $G$ as in the statement and abbreviate $\cH = K_k(G)$.
	Let $L$ and $S$ be the set of big and small vertices of $G$, respectively.
	Let $v$ be a big vertex.
	By \cref{prop:ore-link} and induction, the link graph $L_{\cH}(v)$ contains a $k$-clique.
	Hence $v$ is on a $(k+1)$-clique.
	It remains to show that $\cH$ is tightly connected and spans all vertices.

	The following two claims follow as in the proof of \cref{lem:posa-connectivity} with the only difference being that we apply \cref{prop:ore-link} instead of \cref{prop:posa-link}.
	\begin{observation}\label{cla:vi-tighlty-connected-ore}
		If $v$ is a big vertex, then $L_{\cH}(v)$ is tightly connected and spans $N_G(v)$.
	\end{observation}

	\begin{observation}\label{claim:spanningHv-ore}
		If $v$ is a big vertex and $u \in N_G(v)$, then $u$ and $v$ belong to a common $k$-clique of $G$.
	\end{observation}

	Together with \cref{prop:ore-link}, \cref{cla:vi-tighlty-connected-ore} implies that, given a big vertex $v$, all the $k$-edges containing $v$ belong to the same tight component of $\cH$.
	We will denote that component by $\cH_v$.

	\begin{observation}\label{cla:big-component}
		If $u, v$ are big vertices, then $\cH_u = \cH_v$.
	\end{observation}

	\begin{proofclaim}
		It suffices to show that there is a big vertex $w$ in the common neighbourhood of $u$ and $v$.
		Indeed, if this is the case, then there exists a $k$-clique $X$ containing both $u$ and $w$
		by \cref{claim:spanningHv-ore}.
		Since $u$ and $w$ are big vertices, then $X$ belongs to $\cH_w$ as well as $\cH_u$, hence $\cH_u = \cH_w$.
		By symmetry, also $\cH_v = \cH_w$ and thus $\cH_u = \cH_v$.

		Now let us show that $u$ and $v$ share a big neighbour.
		Any two big vertices $u$ and $v$ share stricly more than $(k-2)n/k \ge k-2$ neighbours (we used $n \ge k$), so $|N(u) \cap N(v)| \ge k-1$.
		If $N(u)\cap N(v)$ contains a $(k-1)$-clique $X'$, then $X' \cup \{u\}$ and $X' \cup \{v\}$ are $k$-cliques in $\cH_u, \cH_v$ respectively, which share $k-1$ vertices, this immediately shows that $\cH_u = \cH_v$.
		Therefore, we can assume that there are two non-adjacent vertices in the joint neighbourhood of $u$ and~$v$.
		But then one of these two vertices is big and we are done.
	\end{proofclaim}

	\begin{observation} \label{claim:spanning-ore}
		Every vertex of $G$ is contained in a $k$-clique.
	\end{observation}

	\begin{proof}
		We know already that big vertices in $L$ are covered by $k$-cliques, so it remains to consider small vertices.
		Let $y$ be a small vertex and let $x$ be a non-neighbour of $y$ (which exists, since $y$ is small).
		The Ore-type condition implies $\deg(x)+\deg(y) > 2(k-1)n/k$, and $\deg(x) < n$ implies $\deg(y) \ge (k-2)n/k+2 \ge k$, since $n \ge k$.
		If there exists a big vertex in $N(y)$, we are done by \cref{claim:spanningHv-ore}.
		Otherwise, $N(y) \subseteq S$, but since $S$ forms a clique, we deduce $y$ is contained in a $k$-clique, as desired.
	\end{proof}

	By \cref{cla:big-component}, there exists a tight component $\cH'$ of $\cH$ which contains $\cH_u$ for every big vertex $u$.
	By \cref{claim:spanning-ore}, we can finish the proof by showing that $\cH'$ actually contains all edges of $\cH$.
	For a contradiction, suppose $X$ is an edge in $\cH$ that is not in $\cH'$.
	Then $X$ cannot intersect $L$, so it must be completely contained in $S$,
	hence $|S| \ge k$.
	By assumption, one of any two non-connected vertices in $G$ is big, hence $G[S]$ is a clique.
	Thus, no vertex of $L$ can have $k-1$ neighbours in $S$, otherwise it would form a $k$-clique in $\cH'$ which is in the same tight component as $X$.
	Thus $\deg(x, S) \leq k-2$ for all $x \in L$, and thus $e(L, S) \leq (k-2)|L| = (k-2)(n - |S|)$.

	Let $y \in S$.
	This vertex must have a non-neighbour $x$ (otherwise it would not be small).
	Since $G[S]$ is a clique, this non-neighbour $x$ must be in $L$.
	By assumption, $x$ and $y$ have more than $(k-2)n/k$ common neighbours, and at most $k-2$ of them can be in $S$ since $\deg(x, S) \leq k-2$.
	Thus $y$ has more than $(k-2)n/k - (k-2) = (k-2)(n/k - 1)$ neighbours in $L$.
	Therefore, $\deg(y, L) > (k-2)(n/k - 1)$ holds for every $y \in S$, and therefore $e(L, S) > (k-2)(n/k - 1)|S|$.
	Combined with the upper bound we had for $e(L,S)$, we obtain
	\[ (k-2)(n/k - 1)|S| < e(L,S) \leq (k-2)(n - |S|). \]
	Since $k \ge 3$, then $k-2 > 0$, so dividing we get $|S|(n/k - 1) < n - |S|$, which after rearranging gives $|S| < k$, a contradiction.
\end{proof}

Since random (induced) subgraphs inherit Ore's condition, we can easily obtain the following strengthened version of \cref{prop:ore-link}.

\begin{proposition}\label{prop:ore-super-saturated-connectivity}
	Let $1/n   \ll 1/s \ll \mu,1/k$.
	Let $G$ be a $(\frac{k-1}{k} + \mu)$-Ore graph on $n$ vertices.
	Then every vertex of $G$ is on at least $s^{-2k}n^{k-1}$ many $k$-cliques.
	Moreover, all big vertices of $G$ are on at least $s^{-2k}n^{k}$ many $(k+1)$-cliques.
	Finally, for every two $k$-cliques $X,Y$ in $G$, there are at least $n^s/(3s^s)$ sets $S$ of size $s$ each, disjoint from $X \cup Y$, such that $K_k(G)$ contains a tight walk, which starts with $X$, ends with $Y$, and only uses vertices from $X \cup S \cup Y$.
\end{proposition}
 
\begin{proof}
	Consider a vertex $v \in V(G)$, a $(\frac{k-1}{k} + \mu)$-big vertex $w$ and any two $k$-cliques $X,Y$ in $G$.
	Select an $s$-set $S \subset V(G) \setminus ( \{ v, w \} \cup X \cup Y )$ uniformly at random, and let $S' = S \cup X \cup Y \cup \{v,w\}$.
	Let us say that $S$ is \emph{successful} if the graph $G[S']$ is $\frac{k-1}{k}$-Ore, and also $w$ is $\frac{k-1}{k}$-big in $G[S']$.
	A standard probabilistic argument shows that $S$ is successful with probability at least $1/2$.
	Indeed, (a hypergeometric version of) Chernoff's bound (\cref{lem:che}) reveals that a fixed vertex $u$ satisfies $\deg_{G[S']} (u)\geq \frac{\deg_{G}(u)}{n-1} s - (\mu/2) s$ with probability at least $1-e^{-\Omega(s)}$.
	Taking the union bound shows that all $s$ vertices of $S'$ satisfy this property simultaneously with probability at least $1/2$.
	Since $G$ is $(\frac{k-1}{k} + \mu)$-Ore and $1/n   \ll 1/s \ll \mu,1/k$, the claim follows.
	{In particular, the number of successful $s$-sets is at least $\binom{n-2(k+1))}{s}/2 \geq n^s/(3s^s)$.}
	 
	If $S$ is successful, by \cref{lem:ore-connectivity} applied to $G[S']$, it follows that $K_k(G[S'])$ is tightly connected, every vertex is on a $k$-clique in $G[S']$ and every $\frac{k-1}{k}$-big vertex is on a $(k+1)$-clique in $G[S']$.
	{Those properties in particular imply that there exists a tight walk $W$ from $X$ to $Y$ in $K_k(G[S'])$, and this proves the last part of the statement.}

	Using the previous discussion, now we can show the other desired statements.
	Let us prove first that every vertex is on at least $s^{-2k} n^{k-1}$ ordered cliques.
	For each successful $S$, we have shown $G[S']$ contains a $k$-clique which contains $v$.
	There are at least $\frac{1}{2}\binom{n-2(k+1)}{s}$ successful choices of $S$,
	and each $(k-1)$-clique (after removing $v$) is contained at most in $\binom{n-2(k+1)-(k-1)}{s-(k-1)}$ successful $S$.
	Thus the number of $k$-cliques which contain $v$ must be at least $\frac{1}{2}\binom{n-2(k+1)}{s} \binom{n-2(k+1)-(k-1)}{s-(k-1)}^{-1} \geq s^{-2k} n^{k-1}$, as required.
	The other properties follow similarly.
\end{proof}

Now \cref{prop:ore-super-saturated-connectivity} implies \cref{cor:ore-robust-connectivity} since the strengthened properties survive in every $\nu$-approximation if $\nu$ is sufficiently small.

\begin{proof}[Proof of \cref{cor:ore-robust-connectivity}]
	We focus on the tight connectivity, since the other properties follow analogously.
	Introduce $s$ with $\nu \ll 1/s \ll \mu,1/k$.
	Let $X$ and $Y$ be the vertex sets of two $k$-cliques in $G'$.
	{By \cref{prop:ore-super-saturated-connectivity},  there are at least $n^s/(3s^s)$ sets of $s$ vertices each, disjoint from $X \cup Y$, and such that $K_k(G[X \cup S \cup Y])$ contains a tight walk from $X$ to $Y$.
	Note that every edge $e \in E(G)\sm E(G')$ which is disjoint from $X \cup Y$ is contained in at most $n^{s-2}$ of these sets.
	Since there are at most $\nu n^2$ such edges, at most $\nu n^s$ sets $S$ are invalidated by this.
	On the other hand, there are at most $|X \cup Y| \nu n \leq 2 k \nu n$ edges $e \in E(G) \sm E(G')$ which intersects $X \cup Y$ in one vertex, and each of them is contained in at most $n^{s-1}$ of the sets $S$.
	Thus these edges can invalidate at most $2 k \nu n^s$ of the sets $S$.
	Thus the number of sets $S$ which are not invalidated by any edge of $E(G) \sm E(G')$ is at least $n^s/(3s^s) - (2 k + 1) \nu n^s > 0$, so there exists such a set $S$ with $G[X \cup S \cup Y] \subseteq G'[X \cup S \cup Y]$.
	Since $K_k(G[X \cup S \cup Y])$ contains a tight walk from $X$ to $Y$, we deduce that $X$ and $Y$ are also connected by a tight walk in $K_k(G')$, and thus $K_k(G')$ is tightly connected.}
\end{proof}

Next, we turn to the proof of \cref{lem:ore-perfect-matchin}.
To begin, we show the following.

\begin{lemma}\label{lem:Ore-almost-perfect-matching}
	Let $1/n \ll \nu \ll \beta, \mu, 1/k$ with $n$ divisible by $k$.
	Let $G$ be a $(\frac{k-1}{k}+\mu)$-Ore graph on $n$ vertices.
	Let $G'$ be a $\nu$-approximation of $G$ on $n$ vertices.
	Then $K_k(G')$ has a matching which covers all but at most $\beta n$ vertices.
\end{lemma}

\begin{proof}
	Introduce $s$ divisible by $k$ with $1/n \ll \nu \ll 1/s \ll \beta$.
	To begin, consider a random $s$-set $S \subset V(G')$.
	We claim that $G'[S]$ is $(\frac{k-1}{k}+\mu/2)$-Ore with probability at least $1-\beta$.
	Indeed, since $G'$ is a $\nu$-approximation of $G$, we have $|E(G) \sm E(G')| \leq \nu \binom{n}{2}$.
	So the probability that a fixed edge of $E(G) \sm E(G')$ is contained in $S$ is at most $\nu$.
	It follows by the union bound that $G[S]=G'[S]$ with probability at least $1- \nu \binom{s}{2}$.
	Moreover, by (a hypergeometric version of) Chernoff's bound (\cref{lem:che}), every fixed vertex $u \in V(G)$ satisfies $\deg_{G'[S]} (u)\geq \frac{\deg_{G'}(u)}{n-1} s - (\mu/2) s$ with probability at least $1-e^{-\Omega(s)}$.
	So by the union bound, $G'[S]$ is $(\frac{k-1}{k}+\mu/2)$-Ore with probability at least $1-\nu \binom{s}{2}-se^{-\Omega(s)} \geq 1-\beta$.
	
	{Now consider a random partition of $V(G')$ into parts of size $s$ and one residue part of size at most $s$.
	By the above, the expected number of parts $S$ in the partition such that $G'[S]$ fails to be $(\frac{k-1}{k}+\mu/2)$-Ore is at most $\beta n / s$.
	Thus, there exists such a partition where all but $\beta n/s$ parts induce a $(\frac{k-1}{k}+\mu/2)$-Ore graph in $G'$.}
	Let us fix such a partition.
	We may then finish by applying \cref{thm:ore-clique-factor} to each of the induced $(\frac{k-1}{k}+\mu/2)$-Ore graphs.
\end{proof}

We require the following fact, which is proved using a standard concentration argument.
\begin{proposition}\label{prop:absorption}
	Let $1/n \ll \eta \ll 1/k$.
	Let $\cH$ be a $k$-graph on $n$ vertices such that every vertex is contained in at least $\sqrt{\eta} n^k$ $(k+1)$-cliques.
	Then $\cH$ contains a matching $A$ with $|A| \leq \eta n$ such that every vertex of $\cH$ forms a $(k+1)$-clique with at least $\eta^2 n^k$ edges of $A$.
\end{proposition}
\begin{proof}
	Let ${A}$ be a random set containing each $k$-tuple in ${V(\cH)}^{k}$ independently with probability $p := 4 \eta^{1.5}  n^{-k+1}$.
	Since $\expectation[|{A}|] = p n^{k} =   4\eta^{1.5} n \leq (\eta/2) n$, Markov's inequality gives $\Pr \left(|{A}| > \eta n \right) \leq  {1}/{2}$.
	
	We say that two distinct $k$-tuples \emph{overlap} if there is a vertex occurring in both.
	Note that there are at most $k^2 n^{2k-1}$ ordered pairs of overlapping $k$-tuples.
	Let $P$ be the random variable which counts the number of such pairs that are both in ${A}$.
	We have $\expectation[P] \leq k^2 n^{2k-1} p^2 = k^2 16 \eta^3 n$.
	Thus another application of Markov's inequality reveals $\Pr(P > \eta^2 n) \leq  k^2 16 \eta \leq  {1}/{4}.$ 
	
	For a vertex $v \in V(\cH)$, let ${A}_v$ be the number of $k$-tuples in ${A}$ that form a $(k+1)$-clique with $v$.
	It follows that ${A}_v$ is binomially distributed. 
	By assumption, we have $\expectation[A_{v}] \geq  p \sqrt{\eta} n^{k} = 4 \eta^2 n$.
	Using {Chernoff's bound (\cref{lem:che})}, we deduce that $\Pr({A}_v \leq 3 \eta^2 n) \leq {1}/{(8 n)}.$
	
	In combination, the above bounds on the probabilities imply that there is an outcome $A$ of this random experiment such that
	\begin{itemize}[-]
		\item $A$ contains at most $\eta n$ $k$-tuples;
		\item $A$ has at most $\eta^2 n$ overlapping pairs;
		\item $A$ contains at least $3\eta^2 n$ edges that form a $(k+1)$-clique with every vertex of $\cH$.
	\end{itemize}
	To turn $A$ into a matching, we delete one tuple from every overlapping pair in $A$ and all tuples of $A$ which are not edges.
\end{proof}

\begin{proof}[Proof of \cref{lem:ore-perfect-matchin}]
	Introduce $s,\beta$ with $1/n \ll \nu \ll \beta \ll 1/s \ll \mu, 1/k$.
	By \cref{prop:ore-super-saturated-connectivity}, every vertex of $G$ is on at least $s^{-2k} n^{k-1}$ $k$-cliques.
	If $G$ has fewer than $s^{-4k} n$ small vertices, we match all small vertices with vertex-disjoint $k$-cliques, and remove these cliques from the graph.
	The proof then continues in the same way.
	Hence in the following, we will assume that $G$ has at least $s^{-4k} n$ small vertices.
	Since small vertices form  a clique in $G$, all small vertices are on at least $s^{-6k^2} n^k$ $(k+1)$-cliques.
	Moreover, by \cref{prop:ore-super-saturated-connectivity}, every big vertex of $G$ is on at least $s^{-2k} n^{k}$ $(k+1)$-cliques.
	As $1/n \ll \nu \ll \beta \ll 1/s,1/k$ and $G'$ is a $\nu$-approximation of $G'$, it follows that 
			every vertex in $G'$ is on at least $\sqrt{\beta} n^k$ $(k+1)$-cliques.
	By \cref{prop:absorption} applied with $K_k(G'),\beta$ playing the roles of $\cH,\eta$, there is a set $A$ of at most $\beta n$ disjoint $k$-cliques of $G'$ such that every $v \in V(G')$ has at least $\beta^2 n$ elements of $A$ in its neighborhood $N_{G'}(v)$.
	We define $F'=G'-V(A)$ and $F = G[V(F')]$.
	{Since $F$ is obtained from $G$ after deleting at most $\beta^2 k n \leq \mu n / (2k)$ vertices, we have that $F$ is a $((k-1)/k + \mu/2)$-Ore graph.}
	Since $G'$ is a $\nu$-approximation of $G$, it follows that $F'$ is a $(2\nu)$-approximation of $F$ on the same number of vertices as $F$.\footnote{We remark that that $F'$ is also a $(\nu+k\beta)$-approximation of $G$. But this is not compatible with the constant hierarchy of \cref{lem:Ore-almost-perfect-matching}.}
	Hence we may apply \cref{lem:Ore-almost-perfect-matching} with $F,F',2\nu,\beta^2, \mu/2$ playing the roles of $G,G',\nu,\beta, \mu$ to cover all but $\beta^2 n$ vertices of $F'$ with disjoint $k$-cliques $C$.
	To finish, we match each of the remaining vertices of $G'-V(A)-V(C)$ to a distinct $k$-clique of $A$.
	(This can be done greedily by the choice of $A$.)
	So each remaining vertex is on a private $(k+1)$-clique.
	We then give each $k$-clique contained in these $(k+1)$-cliques weight $1/k$, and all other $k$-cliques of $A \cup C$ weight $1$.
	This results in a perfect fractional matching of $K_k(G')$.
\end{proof}

\subsection{Uniformly dense and inseparable graphs}\label{sec:inseparable}
In this section, we prove \cref{thm:dense-seperable-framework}.
In light of \cref{thm:main-bandwidth}, it suffices to show that any inseparable and locally dense graph admits a robust Hamilton framework, and this is precisely the content of \cref{lem:inseparable-Hamilton}.
We introduce some further definitions to describe the statement.

For a $k$-graph $\cH$ on $n$ vertices and $\eta > 0$, let $\partial_\eta \cH = \{ X \in E(\partial_{k-1} \cH) \colon \deg_\cH(X) \ge \eta n \}$ be the set of shadow edges of degree at least $\eta n$ in $\cH$.
We define the \emph{$\eta$-adherence of $\cH$} to be the subgraph $A_\eta(\cH) \subseteq \cH$ on vertex set $V(\cH)$ that consists of the edges of $\cH$ that contain an element of $\partial_\eta \cH$.

\begin{lemma}\label{lem:inseparable-Hamilton}
	Let $1/n \ll \rho \ll \mu' \ll \eta \ll \mu,d,1/k$.
	Suppose $G$ is a $\mu$-inseparable $(\rho, d)$-dense graph on $n \geq n_0$ vertices, and let $\cH=A_\eta(K_k(G))$.
	Then $(G,\cH)$ is a  $\mu'$-robust $k$-uniform  Hamilton framework.
\end{lemma}

The connectivity part of the proof of \cref{lem:inseparable-Hamilton} relies on the following result of Ebsen, Maesaka, Reiher, Schacht and Schülke~\cite{EMR+20}, which we restate in our notation.

\begin{lemma}[{\cite[Connecting Lemma, Lemma 3.1]{EMR+20}}] \label{lemma:EMRSSconnecting}
	Let $1/n \ll  1/M, \rho, \zeta \ll d, \mu, \eta, 1/k$.
	Let $G$ be a $(\rho, d)$-dense and $\mu$-inseparable graph on $n$ vertices.
	Let $X_1, X_2 \in V(G)^{k-1}$ be disjoint tuples whose unordered sets are in $\partial_\eta (K_k(G))$.
	Then the number of tight $(X_1, X_2)$-{paths} in $K_k(G)$ of length $m$ is at least $\zeta n^{m-2(k-1)}$ for some $2(k-1)+1 \leq m \leq M$.
\end{lemma}

\begin{corollary}\label{cor:dense-inseperable-connectivity}
	Let $1/n \ll \mu' \ll \rho \ll d, \mu, \eta, 1/k$.
	Let $G$ be a $(\rho, d)$-dense and $\mu$-inseparable graph on $n$ vertices and $\cH = A_{\eta} (K_k(G))$.
	Let $G' \subset G$ be a  $\mu'$-approximation of $G$ and, let $\cH' = \cH \cap K_k(G')$.
	Then $\cH'$ is contained in a tight component of $K_k(G')$.
\end{corollary}
\begin{proof}
	Let $M, \zeta$ be such that $\mu' \ll 1/M, \zeta \ll d, \mu, 1/k$.
	Consider any two disjoint edges $Y_1,Y_2$ in $\cH'$.
	By assumption, there are $X_1 \subset Y_1$, $X_2 \subset Y_2$ with $X_1, X_2 \in \partial_\eta(\cH)$.
	It follows by \cref{lemma:EMRSSconnecting} that the number of tight $(X_1, X_2)$-{paths} in $K_k(G)$ of length $m$ is at least $\zeta n^{m-2(k-1)}$ for some $2(k-1)+1 \leq m \leq M$.
	Since $\mu' \ll \zeta,1/M$, at least one of these paths `survives' the transition from $K_k(G)$ to $K_k(G')$.
	(See also the proof of \cref{cor:ore-robust-connectivity}.)
	Hence $Y_1$ and $Y_2$ are in the same tight component of $K_k(G')$.
	The case where $Y_1,Y_2$ intersect follows by transitivity.
\end{proof}

We also need the following facts, which are a simple consequence of $(\rho, d)$-density.
In this context, an \emph{ordered clique} clique is simply a clique that comes with a vertex ordering.
(So in particular, two ordered cliques on the same vertex set are distinct if their orderings are not the same.)

\begin{lemma}[{\cite[Lemma 2.1]{EMR+20}}] \label{lemma:EMRSScliquecounting}
	Let $k \ge 1$, $d \in [0,1]$ and $\rho > 0$.
	Every $(\rho, d)$-dense graph $G$ on $n$ vertices contains at least $\left( d^{\binom{k}{2}} - (k-1)k\rho \right) n^k$ ordered $k$-cliques.
\end{lemma}

\begin{corollary}\label{cor:dense-many-well-connected-cliques}
	Let $\rho \leq \eta \ll d,k$.
	Let $G$ be a $(\rho, d)$-dense graph $G$ on $n$ vertices, and let $\cH = A_\eta(K_k(G))$.
	Then $\cH[S]$ has at least $\eta n^{k}$ edges for every $S \subset V(G)$ of size at least $\eta^{1/(2k)} n$.
\end{corollary}
\begin{proof}
	Let $F= G[S]$ and $m=|S|$.
	Observe that $F$ is $(\sqrt{\rho}, d)$-dense.
	This follows, since for any $U \subset V(F)$, we have $e(U) \geq d|U|^2/2-\rho n^2$ and, we also have $\rho n^2 \leq \sqrt{\rho} m^2$ since $m \geq \eta^{1/(2k)} n$ and $\rho \ll \eta$.

	By \cref{lemma:EMRSScliquecounting}, the number of edges in $K_{k}(F)$ is at least
	\begin{align*}
		\frac{ d^{\binom{k}{2}} - (k-1)k\sqrt{\rho} }{k!} m^{k} \geq \left(\frac{d}{2}\right)^{\binom{k}{2}} m^{k},
	\end{align*}
	where we used that $\rho \ll d,k$.
	We define a subgraph $\cJ \subseteq K_k(F)$ as follows.
	Initially, $\cJ = K_{k}(F)$.
	Next, we iteratively delete edges from $\cJ$, at each step removing those edges that contain a $(k-1)$-set that is contained in at most $\sqrt{\eta} m$ edges (in the updated hypergraph).
	So the number of removed edges is at most
	\begin{align*}
		\sqrt{\eta} m\binom{m}{k-1} \leq \frac{\sqrt{\eta}}{(k-1)!}  m^{k}.
	\end{align*}
	Since $\eta \ll d,k$ it follows that, after the removal, the number of edges in $\cJ$ is at least
	\[ \left(\frac{d}{2}\right)^{\binom{k}{2}} m^{k} - \frac{\sqrt{\eta}}{(k-1)!}  m^{k} \geq \frac{1}{2} \left(\frac{d}{2}\right)^{\binom{k}{2}} m^k \geq \frac{1}{2} \left(\frac{d}{2}\right)^{\binom{k}{2}} \eta^{1/2} n^k \geq \eta n^k, \]
	where in the second to last inequality we used $m \geq \eta^{1/(2k)}n$ and in the last inequality we used $\eta \ll d$.
	Each of these edges induces a $k$-clique $X$ in $F$,
	such that every $(k-1)$-set of $X$ is in at least $\sqrt{\eta} m$ edges of $\cJ$.
	In particular, each such $(k-1)$-set must be in at least $\sqrt{\eta} m \geq \eta n$ $k$-cliques in $K_k(F)$, and thus in $\cH=\partial_\eta(K_k(G))$.
	Thus $e(\cH[S]) \geq e(\cJ) \geq \eta n^k$, as required.
\end{proof}

\begin{proposition}\label{prop:dense-inseperable-matching}
	Let $1/n \ll \mu' \ll \rho \ll \eta \ll d, \mu, 1/k$.
	Let $G$ be a $(\rho, d)$-dense and $\mu$-inseparable graph on $n$ vertices and $\cH = A_{\eta}(K_k(G))$.
	Let $G' \subset G$ be a $\mu'$-approximation of $G$, and let $\cH' = \cH \cap K_k(G')$.
	Then $\cH'$ contains a perfect fractional matching.
\end{proposition}

\begin{proof}
	Let $\eta' > 0$ with $\eta \ll \eta' \ll d, \mu, 1/k$.
	Consider a vertex $v \in V(\cH')$.
	Since $G$ is $\mu$-inseparable, we have $|N_G(v)| \geq \mu (n-1) \geq (\eta')^{1/(2k)} n$, where the last inequality follows by our choice of $\eta'$.
	Since $A_{\eta'}(K_k(G)) \subseteq A_{\eta}(K_k(G)) = \cH$,
	an application of \cref{cor:dense-many-well-connected-cliques} with $\eta'$ in place of $\eta$ implies that there are at least $\eta' n^{k}$ edges in $\cH[N_G(v)]$.
	Since $1/n \ll \mu' \ll \eta'$, at least $(\eta'/2) n^{k}$ of these edges are still contained in $\cH'[N_{G'}(v)]$.

	By \cref{prop:absorption} applied with $\cH',\eta'/2$ playing the roles of $\cH,\eta$, there is a matching $A\subset \cH'$ with $|A| \leq (\eta'/2)^2n \leq \eta' n$ such that every $v \in V(\cH')$ has at least  $(\eta'/2)^4 n \geq \eta^{1/(2k)} n$ elements of $A$ in its neighborhood $N_{G'}(v)$.
	Next, we use \cref{cor:dense-many-well-connected-cliques} to greedily find a matching $M \subset \cH' - V(A)$ such that $M \cup A$ covers all but $\eta^{1/(2k)} n$ vertices of $\cH'$.
	To be precise, in each step we apply \cref{cor:dense-many-well-connected-cliques} (with $\eta$ as input) to find an edge in the set of uncovered vertices.
	We then join each of the remaining vertices to a (private) edge of $A$.
	This gives a collection of pairwise disjoint edges and $(k+1)$-cliques in $\cH'$ that together cover all vertices of $\cH'$.
	Thus, we can define a perfect fractional matching by assigning weight $1$ to single edges and weight $1/k$ to the edges of the $(k+1)$-cliques.
\end{proof}

\begin{proof}[Proof of \cref{lem:inseparable-Hamilton}]
	Let $G' \subset G$ be an arbitrary $\mu'$-approximation of $G$ and, let $\cH' = \cH \cap K_k(G')$.
	By \cref{cor:dense-inseperable-connectivity}, $\cH'$ is contained in a tight component of $K_k(G')$.
	By \cref{prop:dense-inseperable-matching}, $\cH'$ has a perfect fractional matching.
	Consider any vertex $v \in V(\cH')$.
	Since $G$ is $\mu$-inseparable, we have $|N_G(v)| \geq \mu n \geq \eta^{1/(2k)} n$.
	By \cref{cor:dense-many-well-connected-cliques}, there are at least $\eta n^{k}$ edges in $\cH[N_G(v)]$.
	So in particular, every vertex in $\cH$ has at least $\eta n^{k} \geq \mu' n^k$ linked edges.
	Moreover, since $\mu' \ll \eta$, one of these sets $Y$ together with $v$ induces a $(k+1)$-clique in $G'$.
	Hence $\cH'$ contains a $(k+1)$-clique.
\end{proof}

\subsection{Deficiency}

To show \cref{theorem:deficiency-bandwidth}, we will need some preparations.
Consider the function given by
\begin{align*}
	f(n,t,k) = \begin{dcases}
		           \binom{n}{2}-\binom{\frac{t}{k-1}}{2}-\frac{t}{k-1}\left(n-\frac{t}{k-1}\right) & \text{if } t < \frac{(k-1)n}{2k^2 - 2k+1},                       \\
		           \binom{n}{2} - \binom{\frac{n+t}{k}}{2}                                         & \text{if } \frac{(k-1)n}{2k^2-2k+1} \leq t < (k-1)n, \text{ and} \\
		           0                                                                               & \text{if } t \ge (k-1)n.
	           \end{dcases}.
\end{align*}
It is not difficult to check that, for fixed $k$, the function $g(n,t,k)$ appearing in \cref{theorem:deficiency-cliquefactors} satisfies $f(n,t,k) = g(n,t,k) + o(n^2)$.
We will find it more convenient to work with $f(n,t,k)$ instead.

We show next that, to show \cref{theorem:deficiency-bandwidth}, we can assume $t$ is not very close to $0$.

\begin{proposition} \label{proposition:deficiency-tnotsmall}
	Let $\mu \geq 0$ and $n, k, t$ integers with $k \ge 2$ and  {$t \leq k n$}.
	Suppose $G$ is a graph on $n$ vertices with $e(G) > f(n,t,k) + \mu n^2$.
	Then $t \ge \mu(n+t)/3$.
\end{proposition}

\begin{proof}
	We have $\mu n^2 + f(n,k,t) < e(G) \leq \binom{n}{2}$.
	Let $\alpha n = t/(k-1)$.
	For $t < (k-1)n/(2k^2 - 2k + 1)$, we obtain
	\begin{equation*}
		\mu n^2 \leq \binom{\alpha n}{2} + \alpha n (n - \alpha n) \leq \frac{1}{2} \alpha^2 n^2 + \alpha n (n - \alpha n).
	\end{equation*}
	Hence $\alpha^2 + 2\alpha(1 - \alpha) \ge 2 \mu$, which implies $\alpha  \ge \mu$.
	This gives $t \ge \mu(k-1)n$.
	We can conclude, since $t \leq k n$ and $k \geq 2$, that $\mu(k-1)n \geq \mu(n+t)/3$.
\end{proof}

Again, we have to be careful since deleting a few edges can destroy the structure of $G \ast K_t$,
so \cref{theorem:deficiency-cliquefactors} cannot be applied directly to approximations of $G\ast K_t$ to get Hamilton frameworks.
However, this is merely a technical obstacle as the next proposition shows.

\begin{proposition} \label{proposition:deficiency-robustlymatching}
	Let $1/n \ll \nu \ll \mu, 1/k$ and $k \ge 3$.
	Let $t \leq kn$.
	Suppose $G$ is a graph on $n$ vertices with $e(G) \ge f(n,t,k) + \mu n^2$.
	Let $F$ be a $\nu$-approximation of $G \ast K_t$.
	Then $\cH'=K_k(F)$ has a fractional perfect matching.
\end{proposition}
\begin{proof}
	It will be convenient to assume that $\mu \ll 1$.
	We set $\cH'=K_k(F)$, $K' = V(F) \cap V(K_t)$, and $t' = |K'|$.
	We begin by showing that there is an `absorbing' structure $A \subset \cH'[K']$.
	First, by \cref{proposition:deficiency-tnotsmall}, we can assume that $t \ge \mu(n+t)/4$.
	This implies that $\nu(n+t) \leq 4 \nu \mu^{-1} t \leq \mu t$.
	Next, fix a vertex $v \in V(F)$.
	Since $F$ is a  $\nu$-approximation and $\nu \ll \mu$, we have $\deg_F(v;K') \geq t - \nu (n+t) \geq (1 - \mu)t'$.
	Similarly, $F[N_F(v;K')]$ has minimum degree at least $(1 - 2\mu)t'$, as vertices in $K'$ have lost at most $\nu (n+t)$ neighbours each.
	It follows that $F[N_F(v;K')]$ has at least $(1 - 4k\mu) \binom{t'}{k} \geq \mu^2 \binom{n+t}{k}$ many $k$-cliques.

	Let $\cH''$ be obtained from $\cH'$ by retaining only the edges that intersect with $K'$.
	By \cref{prop:absorption} applied with $\cH'',\mu^2$ playing the roles of $\cH,\eta$, there is a matching $A \subseteq \cH'[K']$ of size at most $\mu^4 t' \leq \mu t$ with the additional property that, for every vertex $v \in V(F)$ there are at least $\mu^8 (n+t)$ edges $X \in E(A)$ such that $\{v\} \cup X$ is a $(k+1)$-clique in~$F$.
	Next, we find a large matching in $\cH' - V(A)$.
	Let $t'' = t' - |V(A)|$.
	By choice of $A$, we have $t'' \geq (1-\mu^4)t' \geq (1-\mu)t \geq \mu(n+t)/5$.
	Set $G' = F \cap G$  and $n' = v(F\cap G)$.
	Since $F$ is a $\nu$-approximation, we have  $n' \geq n - \nu (n+t) \geq (1 - 2k \nu) n$.
	Moreover, $e(G') \geq f(n,t,k) + \mu n^2 - \nu (n+t)^2$.
	Using these bounds together with $1/n \ll \nu \ll \mu, 1/k$ a tedious but straightforward calculation yields that $e(G') \geq f(n', t'', k) + (\mu/2)n'^2$.

	Now define $F' = G' \ast K_{t''}$.
	Since $e(G') \ge f(n', t'', k) + (\mu/2)n'^2$, we can apply \rf{theorem:deficiency-cliquefactors} to obtain a matching $M$ in $K_k(F')$ which misses at most $k-1$ vertices of $F'$.
	We will use $M$ to find a large matching in $K_k(F)$.

	Let $\phi\colon V(F') \rightarrow V(F)$ be defined as follows.
	Set $\phi(x) = x$ for all $x \in V(G')$,
	and consider a random bijection $\phi\colon V(K_{t''}) \to V(K'-V(A))$.
	Let $\phi(M)$ be the (random) image of $M$ in $F$ via $\phi$.
	Let $x \in V(F')$ and $y \in V(K_{t''})$.
	Since $F$ is a $\nu$-approximation, $\phi(x)$ misses at most $\nu (n+t)$ of its possible $t''$ neighbours in $F$ (or $t''-1$ if $x$ is in $K_{t''}$).
	So the probability that $\phi(x)\phi(y)$ is not an edge in $F$ is at most
	$\nu (n+t)/t'' \leq \nu/(\mu/5) \leq \nu^{3/4}$, where the last inequality follows from $\nu \ll \mu$.
	Hence, the expected number of edges of $M$ which are mapped to non-edges of $F$ via $\phi$ is at most
	$\nu^{3/4} e(M) \leq \nu^{3/4} \binom{k}{2} (n'+t'')/k \leq \nu^{2/3} (n'+t'')$.
	It follows that there exists a choice of $\phi$ where at most $\nu^{2/3} (n'+t'')$ edges of $M$ are mapped to non-edges.
	Recall that $v(M) \geq n'+t'' - k + 1$.
	Note that removing a $k$-clique which contains a non-edge uncovers $k$ vertices.
	It follows that $\cH'- V(A)$ contains a matching
	which covers all but at most $k-1 + k \nu^{2/3} (n'+t'') \leq \nu^{1/2}(n'+t'')$ vertices of $\cH'$.

	To finish, we cover the vertices of $\cH' - V(A)$ by joining each of the remaining vertices to a (private) edge of $A$.
	This is possible, since each vertex has $\mu^8 (n+t)$ edges available in $A$ and the number of uncovered vertices is certainly at most $\nu^{1/2}(n'+t'') \leq \mu^8 (n+t)$.
	This gives a collection of pairwise disjoint edges and $(k+1)$-cliques in $\cH'$ that together cover all vertices of $\cH'$.
	Thus, we can define a perfect fractional matching by assigning weight $1$ to single edges and weight $1/k$ to the edges of the $(k+1)$-cliques.
\end{proof}

\begin{proof}[Proof of \cref{theorem:deficiency-bandwidth}]
	Let $n_2, \nu$ be given by \cref{proposition:deficiency-robustlymatching} under input $k, \mu$.
	Without loss of generality we can also suppose $\nu \leq \mu/(3(k+2))$.
	Let $\mu', n_1$ be given by \cref{prop:linked-edges-for-free} under input $k, \nu$.
	Let $\beta, z>0$, and $n_3 \in \NATS$ be obtained from \cref{thm:main-bandwidth} applied with $k,\Delta,\mu'$,
	and let  $n_0 = \max \{ n_1, n_2, n_3, k^3 \mu'^{-3} \nu^{-1}\}$.

	Let $n \geq n_0$.
	Given $t \ge 0$, it suffices to consider a graph $G$ on $n$ vertices with at least $e(G) \ge g(n,k,t) + 2 \mu n^2$.
	Let $\cH = K_k(G \ast K_t)$.
	By \cref{thm:main-bandwidth}, it is enough to show that $(G, \cH)$ is a $\mu'$-robust $k$-uniform  zero-free Hamilton framework.
	By \cref{prop:linked-edges-for-free} and the choice of $\mu', n_0$, it is enough to show that $(G \ast K_t,\cH)$ is a $\nu$-proto-robust $k$-uniform  zero-free Hamilton framework.

	If $t \ge kn$, actually there is nothing to show because $G \ast K_t$ contains every $(k+1)$-colourable graph $J$ on $n+t$ vertices: the largest colour class in $J$ has at least $(n+t)/(k+1) \ge n$ vertices, so we can embed $J$ in $G \ast K_t$ making sure that the largest colour class covers all of $V(G)$ and allocate the rest arbitrarily.
	Thus we can assume $t < kn$.
	By choice of $n_0$ and the fact that $f(n,t,k) = g(n,t,k) + o(n^2)$, we have $e(G) > f(n,k,t) + \mu n^2$.
	Thus $t \ge \mu(n+t)/3$ by \cref{proposition:deficiency-tnotsmall}.
	In summary, we have $\mu(n+t)/3 \leq t < kn$.

	Now we show that $(G \ast K_t,\cH)$ is a $\nu$-proto-robust zero-free Hamilton framework.
	To do so, let $F \subseteq G \ast K_t$ be a $\nu$-approximation of $G \ast K_t$.
	Note that $\cH \cap K_k(F) = K_k(F)$.
	We will be done if we show $K_k(F)$ is a zero-free Hamilton framework, that is, we need to show $K_k(F)$ is tightly connected, contains a perfect fractional matching, and a $(k+1)$-clique.

	We show first that $K_k(F)$ is tightly connected.
	Let $L = V(G \ast K_t) \setminus V(G)$, so $|L| = t$.
	Let $L' = L \cap V(F)$.
	Since $t \ge \mu  (n+t)/3 \ge k$, we deduce that $F[L']$ has minimum degree at least $|L'|-1 - \nu (n+t) > (k-1)|L'|/k$.
	By \cref{lem:posa-connectivity}, this implies that $K_k(F[L'])$ is tightly connected.
	Given this, in order to show that $K_k(F)$ is tightly connected it is enough to show that each edge $X$ in $K_k(F)$ is in the same tight component as the edges of $K_k(F[L'])$.
	We show this latter statement by induction on $|X \cap V(G)|$.
	If $|X \cap V(G)| = 0$, then $X \subseteq L$, thus $X \in F[L']$, so there is nothing to show.
	Suppose $|X \cap V(G)| > 0$, and let $X = \{v_1, \dotsc, v_k\}$, ordered so that vertices in $L$ come first.
	In $G \ast K_t$, the common neighbourhood of $\{v_1, \dotsc, v_{k-1}\}$ in $L$ has size at least $|L|- k = t-k$,
	thus in $F$ it has size at least  $t - 2k \nu(n+t) > 0$.
	Thus there exists $v_0 \in L' \cap N_F(v_1)\cap \dots \cap N_F(v_{k-1})$, so $X' = \{v_0, v_1, \dotsc, v_{k-1}\}$ is a $k$-clique in $\cH$ such that $|X' \cap V(G)| < |X \cap V(G)|$.
	Moreover, by \cref{prop:tightly-connected-observation}, $X'$ is in the same tight component as $X$.
	By the induction hypothesis, $X'$ is tightly connected to some $Y \in K_k(F[L'])$ and so is $X$, as required.

	Next, observe that $K_k(F)$ has a $(k+1)$-clique.
	This follows from \cref{lem:posa-connectivity}, because $F[L']$ has minimum degree strictly larger than $(k-1)|L'|/k$.
	Finally, \cref{proposition:deficiency-robustlymatching} implies that $K_k(F)$ has a perfect fractional matching.
	So $K_k(F)$ is indeed a zero-free Hamilton framework, as required, and this finishes the proof.
\end{proof}

\subsection{Multipartite graphs} \label{section:applications-partite}
In this section, we show \rf{thm:bandwidth-multipartite}.
We start with the following simple observation.

\begin{observation} \label{lemma:multipartite-greedyextend}
	For $r \ge k \ge 2$, let $\fU=\{U_j\}_{j=1}^r$ be a family of $r$ disjoint sets of size $n$ each.
	Let $G$ be a $\fU$-partite graph with $\delta(G;\fU) > \frac{k-1}{k}n$.
	Let $S$ be any set of at most $k$ vertices disjoint from some part $U_j$.
	Then the vertices of $S$ have a common neighbour in $U_j$.
\end{observation}

\begin{proof}
	The common neighbourhood of $S$ in $U_j$ has size at least $|S|\delta(G;\fU) -(|S|-1)n > 0$.
\end{proof}

Next, we show a result about connectivity in multipartite graphs.

\begin{lemma} \label{lemma:multipartite-tightconn}
	Let $r \geq k \geq 2$, and let $G$ be a balanced $r$-partite graph on $rn$ vertices with $\delta_r(G) > \frac{k-1}{k}n$.
	Then $K_k(G)$ is tightly connected.
\end{lemma}

\begin{proof}
	Suppose that $G$ is $\fU$-partite with $\fU=\{U_j\}_{j=1}^r$.
	We begin by proving the case $r = k$ via induction on $k$.
	For $k = 2$, the condition says that $\delta(G;\fU) > n/2$.
	This implies that any two vertices in a given part share a neighbour, which is clearly enough to conclude tight connectivity.

	Now suppose $k > 2$ is given and the result is true for $k-1$.
	Let $v \in U_1$ be arbitrary.
	We wish to show that $k$-cliques containing $v$ belong to the same tight component in $K_k(G)$.
	For $2 \leq j \leq k$, let $U'_j = U_j \cap N_G(v)$.
	Let $n' = \min_{2 \leq j \leq k} |U'_j|$, and note that $n' > (k-1)n/k$.
	Let $G'$ be the subgraph of $G$ induced in $U'_2 \cup \dotsb \cup U'_k$,
	thus $G'$ is a $(k-1)$-partite graph with parts of size at least $n'$ each.
	If $X_1, X_2$ are two $k$-cliques containing $v$,
	then $X'_1 = X_1 \setminus \{v\}$, $X'_2 = X_2 \setminus \{v\}$ correspond to two $(k-1)$-cliques in $G'$.
	For $2 \leq j \leq k$, let $U''_j \subseteq U'_j$ have size exactly $n'$ such that $X'_1, X'_2 \subseteq \bigcup_{2 \leq j \leq k} U''_j$, and set $\fU''=\{U_j\}_{j=2}^k$.
	Let $G''$ be the subgraph of $G$ induced in $U''_2 \cup \dotsb \cup U''_k$.
	Note that $G''$ is a $\fU''$-partite graph with
	\[ \delta(G'';\fU'') \ge \delta(G;\fU) - (n - n') > n' + \frac{k-1}{k}n - n = n' - \frac{n}{k} > \frac{k-2}{k-1}n'. \]
	By the induction hypothesis, $K_{k-1}(G'')$ is tightly connected,
	and thus there exists a tight walk between $X'_1$ and $X'_2$ in $G'' \subseteq G'$.
	By~\cref{prop:tightly-connected-observation}, this  implies all $k$-edges containing $v$ are tightly connected in $K_k(G)$.

	Now let $v, v'$ be two arbitrary vertices in $U_1$.
	By~\cref{prop:link-graph-tightly-connected}, to show that $K_k(G)$ is tightly connected, it is enough to prove that $L_{K_k(G)}(v)$ and $L_{K_k(G)}(v')$ share a $(k-1)$-edge.
	This is equivalent to showing that there exists a $(k-1)$-clique $X$ in $G$, which is in the common neighbourhood of $v$ and $v'$ in $G$.
	Such a clique can be found greedily by applying~\cref{lemma:multipartite-greedyextend} repeatedly,
	first by finding $v_2 \in U_2$ which is a common neighbour of $v, v'$;
	then finding $v_3 \in U_3$ common neighbour of $v, v', v_2$, and so on.
	This concludes the proof of the inductive step, and thus the proof whenever $r=k$, for all $k \ge 2$.

	Next, we fix $k \geq 2$, and we show the $r \ge k$ case by induction in $r$.
	The base case $r=k$ is already done, so assume $r \geq k+1$.
	Let $G_1=G-U_1$ and $G_2 = G- {U_r}$.
	By induction hypothesis, $K_k(G_1)$ and $K_k(G_2)$ are tightly connected.
	We finish by showing that there is a $(k-1)$-clique in $K_{k-1}(G)$ which is contained both in an edge of $K_k(G_1)$ and in an edge of $K_k(G_2)$.
	Greedily construct a $(k-1)$-clique $X$ in the parts $U_2, \dotsc, U_{k}$, by applying~\cref{lemma:multipartite-greedyextend} iteratively as before.
	Applying~\cref{lemma:multipartite-greedyextend} again,
	we see $X$ can be extended to a clique in $K_k(G_1)$ by choosing an extra neighbour in $U_1$,
	and $X$ can be extended to a clique in $K_k(G_2)$ by choosing an extra neighbour in $U_r$.
	Thus the edges of $K_k(G_1)$ and $K_k(G_2)$ are in the same tight component by \cref{prop:tightly-connected-observation}.
	By repeating the argument with $U_r$ switched for $U_{r-1}$, we see that this component also contains the $k$-cliques with a vertex in both $U_1$ and $U_r$.
\end{proof}

The following lemma will be used to find a perfect fractional matching in slightly unbalanced partite graphs, and will be used in the case $r > k$.

\begin{lemma} \label{lemma:multipartite-matching}
	Let $r > k \geq 2$ and $1/n \ll \mu' \ll \mu, 1/r$.
	Let $G$ be a balanced $r$-partite graph on $rn$ vertices with $\delta_r(G) \geq (\frac{k-1}{k} + \mu) n$.
	Suppose $G'$ is a $\mu'$-approximation of $G$, and let $\cH' = \cH \cap K_k(G')$.
	Then $\cH'$ contains a perfect fractional matching.
\end{lemma}

\begin{proof}
	Let $\fU = \{U_1, \dotsc, U_r\}$ be the partition of $G$.
	For every $1 \leq j \leq r$, let $U'_j = U_j \cap V(G')$ and $\fU' = \{ U'_1, \dotsc, U'_r \}$.
	
	First, assume that the partition $\fU'$ is unbalanced.
	It follows that $f(\fU') := r \max_j \{ |U_j'|\} - \sum_{j = 1 }^r |U_j'| >0$.
	Suppose that, without loss of generality, $U'_{k+1}$ has maximum size.
	By \cref{lemma:multipartite-greedyextend}, there is a $(k+1)$-clique $X_1$  in $\cH$ with one vertex in each of $U_{1}',\dots, U'_{k+1}$.
	Moreover, we can find, for $2 \leq j \leq r$, pairwise disjoint (and disjoint from $X_1$) edges $X_j$ in $\cH$ such that $X_j$ has one vertex in $U_{j}',\dots, U'_{j+k-1}$ each, where the indices are computed modulo $r$.
	Let $X = \bigcup_{j=1}^r X_j$.
	It follows that $X$ has $k+1$ vertices in $U'_{k+1}$ and $k$ vertices in all the other parts.
	Moreover, it is easy to see that $X$ contains a perfect fractional matching.
	(Assign weight $1$ to the edges and weight $1/k$ to the edges in the $(k+1)$-clique.)
	Let $U_j'' =  U_j'\sm V(X)$ for $1 \leq j \leq r$, and let $\fU'' = \{U''_j\}_{j=1}^r$.
	It follows that $f(\fU'') =  f(\fU') -1$.
	We continue this process until the remaining partition is balanced.
	Note that this is possible, since each part loses at most $(k+1)f(\fU')  \leq (k+1) \mu' n \leq \mu n$ vertices in this process.
	Now the parts are balanced, say of size $n''$ each.
	Let $p \leq n''$ be divisible by $k$ and as large as possible.
	Consider a subgraph obtained by choosing a $p$-set of vertices in each part.
	By \cref{theorem:multipartite-keevashmycroft}, this subgraph contains a $K_k$-factor.
	There are $\binom{n''}{p}^r$ possible ways of selecting a subgraph of this type.
	Give each edge involved in the resulting matchings weight $\binom{n''}{p}^{-r}$ and take their summation to obtain a perfect fractional matching of the remaining graph.
	Together with the other fractional matchings, this yields a perfect fractional matching of $\cH$.
\end{proof}

The proof of \cref{thm:bandwidth-multipartite} splits in two cases, depending if $r > k$ or not.
For $\cH=K_k(G)$, we show that $(G,\cH)$ is a robust $k$-uniform zero-free Hamilton framework, when $r>k$.
When $r=k$, we show that $(G,\cH)$ is a robust $\fU$-partite $k$-uniform  Hamilton framework instead  where $\fU$ is the given partition of $V(G)$.

\begin{proof}[Proof of \cref{thm:bandwidth-multipartite}]
	Given $r, k, \Delta, \mu$, we fix $\mu' \ll \mu, 1/k, 1/r$.
	Let $\fU$ be a family of $r$ disjoint sets of size $n$ each.
	Let $G$ be a $\fU$-partite graph with $\delta(G;\fU) > \frac{k-1}{k}n$, and let $\cH = K_k(G)$.
	By \rf{thm:main-bandwidth}, it suffices to show that $(G, \cH)$ is a $k$-partite or zero-free $\mu'$-robust $k$-partite $k$-uniform Hamilton framework.
	We separate the analysis on proto-robustness into the cases when $r = k$ and $r > k$.

	\begin{claim}
		If $r=k$, then $(G, \cH)$ is a $\mu'$-proto-robust $\fU$-partite $k$-uniform Hamilton framework.
	\end{claim}

	\begin{proofclaim}
		Let $G' \subseteq G$ be a $\fU$-partite $\mu'$-approximation of $G$ and $\cH' = \cH \cap K_k(G')$.
		We have to show that $\cH'$ is tightly connected and $\cH'$ has a perfect fractional matching.
		
		Let $n' = |U_1 \cap V(G')|$.
		Since $G'$ is a $\fU$-partite $\mu'$-approximation of $G$, we have $n' = |U_j \cap V(G')|$ for all $1 \leq j \leq k$.
		Moreover, for each $1 \leq j \leq k$, each $v \in U_j$, and any $\ell \neq j$,
		we have
		$\deg_{G'}(v, U_\ell) \ge \deg_{G'}(v, U_\ell) - \mu' n \ge \left( \frac{k-1}{k} + \mu \right)n - \mu'n > \frac{k-1}{k}n'$.
		Therefore, $\delta(G';\fU') > (k-1)n'/k$  where $\fU'$ is obtained from $\fU$ by deleting the vertices not in $G'$.
		Thus $K_k(G')$ is tightly connected by~\cref{lemma:multipartite-tightconn},
		and since $n'$ is large enough, $K_k(G')$ has a perfect matching by \cref{theorem:multipartite-keevashmycroft}.
	\end{proofclaim} 

	\begin{claim}
		If $r > k$, then $(G, \cH)$ is a $\mu'$-proto-robust zero-free $k$-uniform Hamilton framework.
	\end{claim}
	\begin{proof}
		Let $G' \subseteq G$ be a $\mu'$-approximation of $G$, and let $\cH' = \cH \cap K_k(G')$.
		It follows from \cref{lemma:multipartite-matching} that $\cH'$ has a perfect fractional matching.
		It remains to show that $\cH'$ is tightly connected and contains a $(k+1)$-clique.
		
		We remark that for sake of zero-freeness we consider $G$ as a single part.
		Note that $|U_j \cap V(G')|$ may differ for different choices of $1 \leq j \leq r$.
		Let $n' = \min_{1 \leq j \leq r} |U_j \cap V(G')|$ and note that $n' \ge n - \mu' rn$.
		For each $1 \leq j \leq r$, let $U'_j \subseteq U_j \cap V(G')$ be an arbitrary set of size exactly $n'$, and let $\fU' =\{U'_j\}_{j=1}^r$.
		Let $G''$ be the subgraph of $G'$ induced on the parts of $\fU'$, and let $\cH'' = \cH' \cap K_k(G')$.
		Since $\mu' \ll \mu$, we have $\delta(G'';\fU') > (k-1)n'/k$.
		So by using \cref{lemma:multipartite-greedyextend} iteratively and \cref{lemma:multipartite-tightconn}, we have that $\cH''$ is tightly connected  and contains a $(k+1)$-clique.
		As $U'_1, \dotsc, U'_r$ were chosen arbitrarily, the same follows for $\cH'$.
		
		It still remains to show $\cH'$ has a perfect fractional matching.
		Let us first assume that the partition $\fU' = \{U'_j\}_{j=1}^r$ is unbalanced.
		It follows that $f(\fU') := r \max_j \{ |U_j'|\} - \sum_{j = 1 }^r |U_j'| >0$.
		Suppose that, without loss of generality, $U'_{k+1}$ has maximum size.
		By \cref{lemma:multipartite-greedyextend} and the additional term $\mu n$, there is a $(k+1)$-clique $X_1$  in $\cH'$ with one vertex in each of $U_{1}',\dots, U'_{k+1}$.
		Moreover, we can find, for $2 \leq j \leq r$, pairwise disjoint (and disjoint from $X_1$) edges $X_j$ in $\cH'$ such that $X_j$ has one vertex in $U_{j}',\dots, U'_{j+k-1}$ each, where the indices are computed modulo $r$.
		Let $X = \bigcup_{j=1}^r X_j$.
		It follows that $X$ has $k+1$ vertices in $X_{k+1}$ and $k$ vertices in all the other parts.
		Moreover, it is easy to see that $X$ contains a perfect fractional matching.
		(Assign weight $1$ to the edges and weight $1/k$ to the edges in the $(k+1)$-clique.)
		Let $U_j'' =  U_j'\sm V(X)$ for $1 \leq j \leq r$, and let $\fU'' = \{U''_j\}_{j=1}^r$.
		It follows that $f(\fU'') =  f(\fU') -1$.
		We continue this process until the remaining partition is balanced.
		Note that this is possible, since each part loses at most $(r+1)f(\fU')  \leq (r+1) \mu' n \leq \mu n$ vertices in this process.
		Now the parts are balanced, say of size $n''$ each.
		Let $p \leq n''$ be divisible by $k$ and as large as possible.
		Consider a subgraph obtained by choosing a $p$-set of vertices in each part.
		By \cref{theorem:multipartite-keevashmycroft}, this subgraph contains a $K_k$-factor.
		There are $\binom{n''}{p}^r$ possible ways of selecting a subgraph of this type.
		Scale all the resulting matchings by $\binom{n''}{p}^{-r}$ and take their summation to obtain a perfect fractional matching of the remaining graph.
		Together with the other fractional matchings, this yields a perfect fractional matching of $\cH'$.
	\end{proof}
	
	To finish the proof, we note that by \cref{prop:linked-edges-for-free} each $v \in V(G)$ has at least $\mu' n^k$ linked edges in $\cH$.
	Hence $(G, \cH)$ is indeed a (zero-free) $\mu'$-robust $k$-partite or zero-free Hamilton framework.
\end{proof}

\section{Regularity and the Blow-Up Lemma}\label{sec:regularity}

In the following, we introduce a few standard tools for embedding large sparse substructures in dense graphs.
\newcommand{\tX}{\tilde X}
\newcommand{\tcX}{\tilde{\mathcal{X}}}
\newcommand{\comN}{N^*}
\newcommand{\DeltaRp}{\Delta_{R'}}

\subsection{Regular and super-regular pairs}
Next, we collect a few concepts that are connected to graph regularity.
Let $G$ be a graph and~$A$ and~$B$ be non-empty, disjoint subsets of $V(G)$.
We write $e_G(A,B)$ for the number of edges in~$G$ with one vertex in~$A$ and one in~$B$ and define the \emph{density} of the pair $(A,B)$ to be $d_{G}(A,B)=e_G(A,B)/(|A||B|)$.
The pair $(A,B)$ is \emph{$\eps$-regular} in~$G$ if we have $|d_{G}(A',B') - d_{G}(A,B)| \leq \eps$ for all $A'\subseteq A$ with $|A'|\ge\eps|A|$ and $B'\subseteq B$ with $|B'|\ge\eps |B|$.
We say that $(A,B)$ is \emph{$(\eps,d)$-regular} if it is $\eps$-regular and has density at least~$d$.
A vertex $b \in B$ has \emph{typical degree} in an $(\eps,d)$-regular pair $(A,B)$ if $\deg(b,A) \geq (d-\eps)|A|$; and we say $b$ has \emph{atypical degree} otherwise.
We say a pair $(A,B)$ is \emph{$(\eps,d)$-super-regular} if it is $(\eps,d)$-regular
and contains no vertices of atypical degree (in $A$ or $B$).

The next proposition by Böttcher, Schacht and Taraz~\cite[Proposition 8]{BST08} states that regular and super-regular pairs are robust up to small alterations on their vertex sets.
We denote the  symmetric difference between sets $A$ and $B$ by $A \triangle B = (A \sm B ) \cup (B \sm A)$.

\begin{proposition}[Robust regular and super-regular pairs]\label{proposition:robust-regular}
	Let $(A,B)$ be an $(\eps, d)$-regular pair and let $(A', B')$ be a pair such that $|A \triangle A'| \leq \alpha |A'|$ and $|B \triangle B'| \leq \beta |B'|$ for some $0 \leq \alpha, \beta \leq 1$.
	Then $(A', B')$ is an $(\eps', d')$-regular pair with
	\[ \eps' = 2\eps  + 6 (\alpha^{1/2} + \beta^{1/2}) \quad \text{and} \quad d' = d - 2 (\alpha + \beta). \]
	If, moreover, $(A,B)$ is $(\eps, d)$-super-regular and each vertex in $A'$ has at least $(d-\eps')|B'|$ neighbours in $B'$ and each vertex in $B'$ has at least $(d-\eps') |A'|$ neighbours in $A'$, then $(A', B')$ is $(\eps', d')$-super-regular.
\end{proposition}

\subsection{Partitions, reduced graphs and blow-ups} \label{subsection:partitions}

Let~$G$ be a graph and $\fV=\{V_i\}_{i=1}^t$ be a partition of $V(G)$ with $t \geq 2$.
In the context of regularity, the parts of $\cV$ are called \emph{clusters}.
We say that the partition~$\fV$ is \emph{$\kappa$-balanced} (or just \emph{balanced} when $\kappa=1$) if there exists $m\in\mathbb N$ such that we have $m\le|V_i|\le\kappa m$ for all $1 \leq i\leq t$.
For a family $\fV'=\{V_i'\}_{i=1}^t$ of subsets $V_i' \subset V_i$, we denote by $G[\fV']$ the subgraph of $G$ induced on $\bigcup_{i=1}^tV_i'$.
Recall that $G$ is $\fV$-partite if every edge has at most one vertex in each cluster of~$\fV$.
Consider a graph $R$ on vertex set $[t]$.
We say $(G,\fV)$ is an \emph{$R$-partition} if $G$ is $\fV$-partite, each part of~${\fV}$ is non-empty, and an edge $ij$ in $R$ if and only if there are edges of~$G$ between~$V_i$ and~$V_j$.
We refer to $R$ as a \emph{reduced graph} of $(G, \fV)$, or, when $\fV$ is clear, the reduced graph of $G$

Suppose now $(G, \fV)$ is an $R$-partition and $\cH \subset K_k(G)$ is a $k$-graph on $V(G)$.
We define the \emph{$k$-graph induced by $(\cH,\fV)$}  to be the $k$-graph $\cJ$ on $t$ vertices, which contains an edge $K$ whenever $\cH$ \emph{has one or more edges} whose vertices are in the clusters of $K$.
If $\cJ'$ is a $k$-graph such that $\cJ' \supseteq \cJ$, we say $(\cH, \fV)$ is a $\cJ'$-partition.
Note that since $\cH \subset K_k(G)$, we have $\cJ \subset K_k(R)$.

For a $k$-graph $F$ on $t$ vertices, we define the \emph{$(F,\fV)$-blow-up} $F^\ast$ to be the $k$-graph obtained by replacing each vertex $i$ with $V_i$ and each edge of $F$ with complete $k$-partite $k$-graphs.
Note that, under this definition, $(F^\ast,\fV)$ is an $F$-partition.
When all clusters of $\fV$ have common size $m$, we sometimes omit the reference to $\fV$ and say that $F^\ast$ is the $(F,m)$-blow-up.

An $R$-partition $(G,\fV)$ is \emph{$(\eps,d)$-regular} if for each $ij\in E(R)$ the pair $(V_i,V_j)$ is \emph{$(\eps,d)$-regular}.
Let $R' \subset R$ be a spanning subgraph.
We say that $(G,\fV)$ is \emph{$(\eps,d)$-super-regular on~$R'$} if for each $ij\in E(R')$ the pair $(V_i,V_j)$ is \emph{$(\eps,d)$-super-regular}.

It is a well-known fact that every dense, regular pair contains a large super-regular subpair.
The following proposition of Kühn, Osthus and Taraz~\cite[Proposition 8]{KOT05} generalises this insight to bounded-degree subgraphs of the reduced graph.

\begin{proposition}[Super-regularising $R'$] \label{prop:super-regularizing-R'}
	Let $1/n \ll 1/t \ll \eps \ll d, 1/{\Delta_{R'}}$.
	Let $G$ be a graph on $n$ vertices and $\fV=\{V_i\}_{i=1}^t$ be a balanced partition of $V(G)$.
	Let $R$ be a graph on $t$ vertices and $R' \subset R$ be a spanning subgraph with $\Delta(R')\leq \DeltaRp$.
	Suppose that $(G,\fV)$ is an $(\eps,d)$-regular $R$-partition.
	Then there is a balanced family $\fV'=\{V_i'\}_{i=1}^t$ of subsets $V_i' \subset V_i$ of size at least $ (1 - \sqrt{\eps})n/t  $ such that
	$(G[\fV'],\fV')$ is a $(2\eps,d)$-regular $R$-partition, which is $(2\eps,d)$-super-regular on~$R'$.
\end{proposition}

We also require the following well-known facts, whose proofs can be found  in \cref{sec:appendix-regularity}.
The next proposition allows us to refine regular partitions randomly while retaining their regularity properties.

\begin{proposition}[Randomly refining regular pairs] \label{prop:random-refinement}
	Let $1/n \ll 1/t \ll \eps \ll \eps' \ll d$ and $1/n \ll 1/q$.
	Suppose that $(G,\fV)$ is a $(1+\eps)$-balanced $(\eps,d)$-regular $R$-partition, where $G$ has $n$ vertices and $\fV$ has $t$ clusters.
	Consider a refinement $\fV^\ast$ of $\fV=\{V_i\}_{i=1}^t$ obtained as follows.
	For every $1 \leq i \leq t$, we partition $V_i$ into $\{V_{i,j}\}_{j=1}^q$  by placing each vertex of $V_i$ uniformly at random in one of the parts $V_{i,1},\dots,V_{i,q}$.
	Let $R^\ast$ be the graph induced by $(R,\fV^\ast)$.
	Then the probability that $(G,\fV^\ast)$ is a $(1+\eps')$-balanced $(\eps',d)$-regular $R^\ast$-partition is at least $2/3$.
\end{proposition}

The following proposition tells us that most partite vertex sets (i.e. those with at most one vertex in a cluster) located in pairwise regular clusters have a large common neighbourhood.

\begin{proposition}[Joint degrees in regular pairs]  \label{prop:regularity-joint degree}
	Let $1/m \ll 1/t \ll \eps \ll \eps' \ll d, 1/k$.
	Let~$G$ be a graph with a balanced vertex partition $\fV=\{V_i\}_{i=1}^k$, which has parts of size $m$.
	Suppose that $(G,\fV)$ is an $(\eps,d)$-regular $K_k$-partition.
	Then all but at most $\eps' m^{k}$ sets $S=\{v_1,\dots,v_k\}$ with $v_i \in V_i$ have the property that $|\bigcap_{i \in [k]\sm \{j\}} N(v_i)\cap V_j| \geq  {(d-\eps')^{k-1} m}$ for every $1 \leq j \leq k$.
\end{proposition}

\subsection{The Regularity Lemma}

Szemer\'edi's Regularity Lemma~\cite{Sze76} allows us to partition the vertex set of a graph into clusters of vertices such that most pairs of clusters are regular.
We will use the following degree form~\cite[Theorem 1.10]{KS96} with pre-partition, which we restate here in terms of approximations and $R$-partitions.
For families of sets $\fV$ and $\fU$, we say that $\fV$ is \emph{$\fU$-refining} if every element of $\fV$ is a subset of an element of $\fU$.

\begin{lemma}[Regularity Lemma]\label{lem:regularity-classic}
	Let $1/n \ll 1/t_1 \ll 1/t_0, \eps, d, 1/r$ and $1 \leq q \leq k$.
	Let $G$ be a graph with an $(n,r)$-sized vertex partition $\fU$.
	Then there are $t_0 \leq t \leq t_1$ with $t \equiv q \bmod k$,
	a subgraph $G' \subset G$ with a $\fU$-refining balanced partition $\fV$ with $tr$ clusters in total,
	$t$ clusters in each part of $\fU$,
	and a graph $R$ such that
	\begin{enumerate}[\upshape (S1)]
		\item $G'$ is a {$\fU$-partite} $(\eps,d+\eps)$-approximation of $G$, and
		\item $(G',\fV)$ is an $(\eps,d)$-regular $R$-partition.
	\end{enumerate}
\end{lemma}

\subsection{The Blow-Up Lemma}

In this subsection, we introduce the Blow-Up Lemma, which allows us to embed spanning subgraphs into regular pairs.
The original Blow-Up Lemma was proved by Koml\'os, S\'ark\"ozy and Szemer\'edi~\cite{KSS97}.
Here we use a variant of this result due to Allen, B{\"o}ttcher, H{\`{a}}n, Kohayakawa and Person~\cite[Lemma 1.21]{ABH+16}, which is particularly suitable for our approach.
To give the statement of the lemma, we need a couple of definitions.

Let~$G$, $H$, $R$, $R'$ be graphs with $R'\subset R$.
Let $\fX=\{X_i\}_{i=1}^t$ be a partition of $V(H)$ and $\fV=\{V_i\}_{i=1}^t$ be a partition of $V(G)$, and let $\tfX=\{\tX_i\}_{i=1}^t$ be a family of subsets of $V(H)$.
The partitions~$\fX$ and~$\fV$ are \emph{size-compatible} if $|X_i|=|V_i|$ for all $1 \leq i \leq t$.
For $\alpha>0$, the family $\tfX=\{\tX_i\}_{i=1}^t$  is an \emph{$(\alpha,R')$-buffer} for $(H,\fX)$ if for each $1 \leq i \leq t$, we have $\tX_i\subset X_i$ and $|\tX_i|\ge\alpha|X_i|$, and, for each $x\in\tX_i$ and $xy,yz\in E(H)$ with $y\in X_j$ and $z\in X_k$ (possibly with $X_i = X_k$), we have $ij,jk\in E(R')$.
We require these buffer vertices for the application of the Blow-up Lemma (\cref{lem:blow-up}).
Their purpose in the proof of \cref{lem:blow-up} is to integrate the last vertices of $H$ via the edges designated by~$R'$.
In our proof, we can pick these vertices greedily, once $R'$ has been fixed.

Finally, we introduce {\emph{image restrictions}}, which will help us to pre-embed some vertices before applying the Blow-Up Lemma.
Following the exposition of Allen et al.~\cite{ABH+16}, the idea can be explained as follows.
Suppose we wish to embed a graph $H^+$ into a graph $G^+${,
		but} unfortunately $G^+$ does not meet the conditions of our {Blow-Up Lemma}.
We find a subgraph $G$ of $G^+$ which does meet the conditions of {the Blow-Up Lemma}, and we `pre-embed' some vertices of $H^+$ onto the vertices $V(G^+)\setminus V(G)$.
This leaves the induced subgraph $H$ of $H^+$ to embed into $G$.
The image restrictions then originate from these pre-embedded vertices: if $x\in X_i\subset V(H)$ (with $X_i$ as defined above) has neighbours $\{z_1,\dots,z_\ell\}$ in $V(H^+)\setminus V(H)$ which are pre-embedded to $\{u_1,\dots,u_\ell\}=J_x\subset V(G^+)\setminus V(G)$, then $J_x$ restricts the embedding of~$x$ to $I_x=\bigcap_{u \in J_x} N_{G^+}(u) \cap V_i$ (with $X_i$ as defined above).
In the following definition, we do not explicitly refer to the graphs~$H^+$ and~$G^+$,
but only to abstract restricting sets $J_x$.
In addition, to simplify notation, we define an	image restriction set~$I_x$ for each vertex~$x$ of~$H$.
For most vertices~$x$, however, this set will be the trivial set $I_x=X_i$  where $X_i$ is the part of
$\fX$ containing~$x$.

\begin{definition}[Image restrictions]\label{def:restrict}
	Let $R$ be a graph on $t$ vertices, and $(H,\fX)$ be an
	$R$-partition and $(G,\fV)$ a size-compatible $(\eps,d)$-regular
	$R$-partition, where $V(G)$ is contained in a superset $V$.

	Let
	$\fI=\{I_x\}_{x\in V(H)}$ be a collection of subsets of $V(G)$, called
	\emph{image restrictions}, and $\fJ=\{J_x\}_{x\in V(H)}$ be a collection of
	subsets of $V \setminus V(G)$, called \emph{restricting vertices}.
	For $\rho, \zeta >0$ and $\Delta, \Delta_J \geq $, we say that $\fI$ and $\fJ$ are a
	\emph{$(\rho,\zeta,\Delta,\Delta_J)$-restriction pair} if the
	following properties hold for each $1 \leq i \leq t$ and $x\in X_i$.
	\newcommand{\Iasdf}{I}
	\begin{enumerate}[(\Iasdf1)]
		\item\label{itm:restrict:Xis} The set $X_i^*\subset X_i$ of  \emph{image restricted}
		vertices in~$X_i$, that is, vertices $x \in X_i$ such that $I_x\neq V_i$, has size $|X_i^*|\leq\rho|X_i|.$

		\item\label{itm:restrict:sizeIx} If $x\in X_i^*$, then $I_x \subset V_i$ is of size at least $\zeta d^{|J_x|}|V_i|$.

		\item\label{itm:restrict:Jx} If $x\in X_i^*$, then $|J_x|+\deg_H(x)\le\Delta$ and
		if $x\not\in X_i^*$, then $J_x=\emptyset$.
 
		\item \label{itm:restrict:DeltaJ} Each vertex of $V$ appears in at most $\Delta_J$ of the sets of $\fJ$.\footnote{Note that this concept can be further simplified as in the work of Komlós, Sárközy and Szemerédi~\cite{KSS97}.}
	\end{enumerate}
\end{definition}

With these definitions at hand, we can now state the full version of the Blow-Up Lemma that we will use.

\begin{lemma}[{Blow-Up Lemma~\cite[Lemma 1.21]{ABH+16}}]\label{lem:blow-up}
	For all $\Delta,\DeltaRp,{\Delta_J},\kappa\ge 1$ and $\alpha,\zeta, d>0$
	there exist $\eps,\rho>0$ such that for all $t_1$ there is an $n_0$ such that
	for all $n\ge n_0$ the following holds.
	Let $R$ be a graph on $t\le t_1$ vertices and let $R'\subset R$ be a spanning subgraph with $\Delta(R')\leq \DeltaRp$.
	Let~$H$ and~$G$ be graphs on $n$ vertices with $\kappa$-balanced size-compatible vertex partitions $\fX=\{X_i\}_{i=1}^t$ and $\fV=\{V_i\}_{i=1}^t$, respectively.
	Let $\tfX=\{\tX_i\}_{i=1}^t$ be a family of subsets of $V(H)$, let $\fI=\{I_x\}_{x\in V(H)}$ be a family of image restrictions and $\fJ=\{J_x\}_{x\in V(H)}$  be a family of restricting vertices.
	Suppose that
	\begin{enumerate}[\upshape (B1)]
		\item $\Delta(H)\leq \Delta$, $(H,\fX)$ is an $R$-partition and $\tfX$ is an $(\alpha,R')$-buffer for $(H,\fX)$,
		\item \label{item:blowup-regsuperreg} $(G,\fV)$ is an $(\eps,d)$-regular $R$-partition, which is $(\eps,d)$-super-regular on~$R'$,
		\item $\fI$ and $\fJ$ form a $(\rho, \zeta, \Delta, \Delta_J)$-restriction pair.
	\end{enumerate}
	Then there is an embedding $\psi\colon V(H)\to V(G)$ such that $\psi(x) \in I_x$ for each $x \in V(H)$.
\end{lemma}

\section{Intermediate results} \label{section:intermediate}
This section contains a series of propositions that we require to prove \cref{thm:main-powham,thm:main-bandwidth}.

\subsection{Distributed matching}

The following lemma shows that, in dense $k$-graphs, we can find a matching which intersects each subgraph of a given family of subgraphs simultaneously.
Its proof is a standard application of concentration bounds and can be found in \cref{sec:appendix-probability}.

\begin{proposition}[Distributed matching]\label{prop:distributed-matching}
	Let  $1/n \ll \eta \ll \pi \ll \gamma, 1/k$.
	Let $\cH$ be a $k$-graph on $n$ vertices.
	Let $\fF$ be a collection of at most  $2^{\eta n}$ subgraphs of $\cH$, where each $\cF \in \fF$ has at least $\gamma n^k$ edges.
	Then there is a matching $\cM \subset \cH$ of size at most $\sqrt{\pi}n$ such that $\cM$ has at least $\pi n$ edges of every subgraph $\cF \in \fF$.
\end{proposition}

\subsection{Tight walks}
The following facts will help us to navigate along tight components.
Recall that a $k$-graph is tightly connected if there exists a closed tight walk which visits every edge.
In what follows, we show we can control some properties of these walks.

\begin{proposition} \label{proposition:orderedspanningwalk}
	Let $\cJ$ be a tightly connected $k$-graph.
	Let $(w_1, \dotsc, w_k)$ be an ordering of an edge in $\cJ$.
	Then there exists a closed walk $W$ which contains $(w_1, \dotsc, w_k)$ as a subwalk,
	and visits every edge of $\cJ$.
\end{proposition}

\begin{proof}
	Let $W$ be a closed walk which contains $(w_1, \dotsc, w_k)$ and visits the maximum number of edges of $\cJ$.
	If $W$ does not visit every edge of $\cJ$, there must exist edges $X, Y$ in $\cJ$ such that $|X \cap Y| = k-1$, and $W$ visits $X$ but not $Y$.
	Similarly to the proof of \cref{prop:tightly-connected-observation}, we can extend $W$ to obtain a walk $W'$ which also contains $(w_1, \dotsc, w_k)$ as a subwalk and visits every edge visited by $W$ and also $Y$, contradicting the maximality.
\end{proof}

Next, we bound the length of a walk visiting all the edges of a tightly connected graph.

\begin{proposition}[Closed tight walk visiting all edges]\label{prop:closed-walk-bounded-length}
	Let $\cJ$ be a tightly connected $k$-graph on $t \ge k$ vertices and $0 \leq  {q} \leq k-1$.
	Then there exists a closed tight walk $W$ which visits all edges of $\cJ$ and has length at most $k^2t^k+k\binom{t}{k} t^{k}$.
	Moreover, if $\cJ$ contains a closed walk of length coprime to $k$, then we can also assume $W$ has length congruent to $q \bmod k$.
\end{proposition}

\begin{proof}
	Suppose first that $\cJ$ contains a walk of length coprime to $k$; we describe the complementary case at the end.
	Let $C = (w_1,\dots, w_{\ell})$ be a walk in $\cJ$ of length coprime to $k$,
	and suppose $C$ has minimum length among all the closed walks with this property.
	We claim $\ell \leq k t^k$.
	Indeed, assume otherwise.
	Since there are at most $t^k$ sequences of $k$ distinct vertices of $V(\cJ)$,
	there must exist some sequence which has at least $k$ appearances in $C$.
	We can write $C = C_1 \dotsb C_k$ as concatenation of walks $C_1,\dots,C_k$ such that each $C_i$ begins with the same sequence of $k$ vertices.
	A pigeonhole argument shows the existence of $1 \leq i \leq j \leq k$ such that the walk $C_i \dotsb C_j$ has length divisible by $k$.\footnote{Indeed, let $s_i$ denote the length of the walk $C_1 \dotsb C_i$ for $1 \leq i \leq k$.
		If one of the numbers $s_i$ is divisible by $k$, we are done.
		Otherwise, there must be $1 \leq i < j \leq k$ such that $s_i \equiv s_j \bmod k$.
		But then the walk $C_{i+1} \dotsb C_j$ has length divisible by $k$, as desired.}
	By relabelling, we can assume that $(i,j) \neq (1,k)$.
	But then $C_1 \dotsb C_{i-1} C_{j+1} \dotsb C_k$ is a closed walk of length congruent to $\ell \bmod k$ (and thus has length coprime with $k$) and shorter than $C$; a contradiction.

	Let $(x_1, \dotsc, x_k)$ be a subwalk of $C$.
	By \cref{proposition:orderedspanningwalk}, there exists a walk $W$ which contains all edges of $\cJ$ and contains $(x_1, \dotsc, x_k)$ as a subwalk.
	Let $W$ be a walk of minimum length with those two properties.
	Let $e_1, \dotsc, e_r$ be an enumeration of the edges of $\cJ$ according to the first time they are visited on $W$, in particular we can assume $e_1 = \{x_1, \dotsc, x_k\}$.
	Let $W_i$ be the subwalk of $W$ between the first appearance of $e_i$ and the first appearance of $e_{i+1}$, in cyclic order, so $W = W_1 W_2 \dotsb W_r$.
	By minimality, each $W_i$ cannot contain two distinct subwalks of length $k$ which correspond to exactly the same ordered $k$-tuple of vertices, otherwise we could replace $W_i$ in $W$ by a shortened $W'_i$ and keep all of the desired properties.
	There are at most $t(t-1) \dotsb (t-k+1)$ ordered $k$-tuples of distinct vertices of $\cJ$, so each $W_i$ has length at most $t(t-1) \dotsb (t-k+1) + k-1 \leq t^k$.
	Since there are at most $\binom{t}{k}$ edges in $\cJ$, we obtain that $W$ has length at most $\binom{t}{k} t^{k}$.

	We obtain the desired walk as follows:
	concatenate at most $k$ copies of $W$ with itself, to obtain a walk $W'$ of length divisible by $k$ and at most $k \binom{t}{k} t^{k}$, which visits all edges of $\cJ$.
	Since the length $\ell$ of $C$ is coprime with $k$,
	there exists $0 \leq a \leq k-1$ such that $a \ell \equiv q \bmod k$.
	Concatenate $a$ copies of $C$ with itself, to obtain a walk $C'$ of length congruent to $q \bmod k$ and at most $k^2 t^k$.
	Then merge $C'$ and $W'$ into a single walk: this is possible since both share the subwalk $(x_1, \dotsc, x_k)$.
	The final walk has length congruent to $q \bmod k$ and at most $k^2 t^k+k\binom{t}{k} t^{k}$, as required.
	Finally, if $\cJ$ does not contain a walk of length coprime to $k$, the task becomes simpler: build $W$ as before, now starting from an arbitrary $(x_1, \dotsc, x_k)$.
\end{proof}

Finally, the next proposition allows us to find walks of controlled length between some prescribed ordered tuples.

\begin{proposition} \label{proposition:usingtheclique}
	Let $1/t \ll 1/k$.
	Let $\cJ$ be a tightly connected $k$-graph on $t$ vertices and let $K$ be a set of $k+1$ vertices of $\cJ$ spanning a clique.
	For any ordered edges $e_1 = (v_1, \dotsc, v_{k}) \in V(\cJ)^k$ and $e_2 = (c_1, \dotsc, c_k) \in K^k$, there exists a walk beginning with $e_1$ and ending with $e_2$ of length exactly $kt^{k}$.
\end{proposition}

\begin{proof}
	Since $\cJ$ is tightly connected, there exists a tight walk starting from $e_1$ and whose last $k$ vertices are those of $e_2$ (possibly in a different order).
	Let $W_1$ be such a walk of minimum length.
	The minimal length implies in particular that $W_1$ visits every tuple of $k-1$ vertices at most once, and hence its length is at most $t^{k}$.
	Let $f_2$ be the last $k$ ordered vertices of $W_1$.

	Now we show that there is a short walk from $f_2$ to $e_2$ using vertices from $K$ only.
	Suppose $K = \{c_1, \dotsc, c_{k+1}\}$ and $f_2 = (c_1, \dotsc, c_k)$.
	Then, given $1 \leq i < j \leq k$, the sequence
	\begin{equation*}
		c_1 \dotsb c_kc_1 \dotsb c_{i-1} c_{k+1} c_{i+1} \dotsb c_{j-1} c_i c_{j+1} \dotsb c_k c_1 \dotsb c_{i-1} c_{j} c_{i+1} \dotsb c_{j-1} c_i c_{j+1} \dotsb c_k
	\end{equation*}
	is a walk in $\cJ[K]$ of length $3k$, which starts with $f_2$ and ends with $f'_2$, consisting of the same ordered vertices of $f_2$ but now with $c_i$ and $c_j$ swapped.
	By repeating this, we can swap two more vertices, and iterate.
	If we need $s$ swaps to rearrange $f_2$ to get $e_2$, then there is a walk $W_2$ of length $k + s2k$ which starts with $f_2$ and ends with $e_2$.
	Since we can always reach $e_2$ with at most $k$ swaps, we can assume $W_2$ has length at most $k+2k^2$.
	Joining $W_1$ and $W_2$ together we get a walk $W_3$ of length at most $t^k+2k^2$ which starts with $e_1$ and ends with $e_2$.
	By concatenating $W_3$ with at most $k$ copies of the walk $c_{k+1} e_2$, we get a walk $W_4$ of length at most $t^k+2k^2+k(k+1)$ which starts with $e_1$, ends with $e_2$, and has length $0$ modulo $k$.
	Since $1/t \ll 1/k$, we have $2k^2+k(k+1) \leq (k-1)t^k$.
	We can concatenate $W_4$ with copies of $e_2$ to obtain a walk $W$ which starts with $e_1$, ends with $e_2$, and has length exactly $k t^k$, as desired.
\end{proof}

\subsection{Allocations}
In this section, we present some results that will help us to (re)allocate vertices in \rf{lem:lemma-for-C}.
Let $\cH$ be a $k$-graph.
We define the $2$-graph $\hat{\cH}$ on $V(\cH)$ by joining two vertices $x$ and $y$ if their link graphs have a common edge.
	{For $U \subset V(\cH)$, we denote the graph induced by $U$ in $\hat{\cH}$ by $\hat{\cH}[U]$.}
A simple induction shows that if there is a tight walk in $\cH$ of length congruent to $1\bmod k$ that starts with $x$ and ends with $y$, then $x$ and $y$ are in the same component in $\hat{\cH}$.
Recall that a vertex in a hypergraph is {isolated} if it is not contained in any edge.
\begin{proposition}\label{prop:reachability}
	Let $k \geq 2$ and $r \in \{1, k\}$,
	and
	let $\fU$ be an $(n,r)$-sized partition.
	Let $\cH$ be a $\fU$-partite $k$-graph, which has {no isolated vertices} and is tightly connected.
	If $r=1$, suppose further that $\cH$ contains a closed tight walk of length coprime to $k$.
	Then $\hat{\cH}[U]$ is connected for every $U \in \fU$.
\end{proposition}
\begin{proof}
	Consider vertices $x,y \in U$ for some fixed $U \in \fU$.
	By the previous observations, it suffices to find a tight walk in $\cH$ which begins with $x$, ends with $y$,
	and has length congruent to $1$ modulo $k$.
	Since $\cH$ has no isolated vertices, there are edges $e$ and $f$ with $x \in e$ and $y \in f$.

	Consider the case $r = k$ first.
	Since $\cH$ is tightly connected, there exists a closed tight walk $W$ visiting every edge of $\cH$, in particular $e$ and $f$.
	We deduce there exists a subwalk $W' \subseteq W$ which starts with $x$ and ends with $y$.
	In particular, when $r=k$, $\cH$ is $k$-partite, thus $W'$ must necessarily have length $1 \bmod k$, as required.

	For the case $r = 1$, we also suppose $\cH$ contains a closed tight walk of length coprime to $k$.
	Hence, by \cref{prop:closed-walk-bounded-length} there is a closed tight walk $W$ which visits every edge and has length $1$ modulo $k$.
	Without loss of generality, we can suppose $W$ starts with $x$.
	Let $W'$ be the subwalk of $W$ from its start $x$ until the first appearance of $y$,
	and let $\ell$ be its length.
	Let $0 \leq a < k$ be such that $a + \ell \equiv 1 \bmod k$.
	Construct $W''$ by concatenating $a$ copies of $W$ and one copy of $W'$.
	Then $W''$ is a tight walk in $\cH$ which begins with $x$, ends with $y$,
	and has length congruent to $1$ modulo $k$, as required.
\end{proof}

The following proposition shows that, under suitable conditions, for any bounded integer allocation $\vec b\colon V(\cH) \rightarrow \INTS$ over the vertices of a hypergraph there exists some appropriately bounded allocation $\vec w\colon E(\cH) \rightarrow \INTS$, where every vertex receives the exact amount $\vec b(u)$ from its incident edges.
The lemma and its proof have been taken from our earlier work with Garbe, Lo and Mycroft~\cite{GLL+21}, and can be found in \cref{sec:flow}.
Recall that  $\maxnorm{\vec v} = \max_{x \in X} |\vec v(x)|$ for a set $X$ and a vector $\vec v \in {\REALS}^X$.

\begin{proposition}\label{prop:flow}
	Let $\fU$ be an $(n,r)$-sized partition,
	and let $\cH$ be a $\fU$-partite $k$-graph.
	Suppose that $\hat{\cH}[U]$ is connected for each $U \in \fU$.
	Let $\vec b \in \INTS^{V(\cH)}$ such that $\sum_{v \in U} \vec b(v) =0$ for every $U \in \fU$ and $\maxnorm{\vec b} \leq s$.
	Then there exists $\vec w \in \INTS^{E(\cH)}$ with $\vec b (v) = \sum_{e \colon v \in e} \vec w(e)$ for all $v \in V(\cH)$ and $\maxnorm {\vec w}  \leq ks n^{2}$.
\end{proposition}

\subsection{Bounded degree cover}
The next lemma allows us to find an edge cover of bounded degree in a hypergraph with a perfect fractional matching.
This will be useful when finding the $R' \subset R$ in the setting of the Blow-Up Lemma.
The idea and proof comes again from our earlier work with Garbe, Lo and Mycroft~\cite{GLL+21} and can be found in \cref{sec:cover}.
For a hypergraph $\cJ$, we denote by $\Delta(\cJ)$ its maximum vertex degree; formally, the largest $m$ for which a vertex of $\cJ$ is in at least $m$ edges.

\begin{lemma}[Bounded degree cover]\label{lem:kcover}
	Let $\cJ$ be a $k$-graph on $n$ vertices which admits a perfect fractional matching.
	Then there is a spanning subgraph $\cJ' \subset \cJ$ with $\Delta(\cJ') \leq k^2 +1$.
\end{lemma}

\subsection{Frameworks and regularity}
In this section, we develop a few tools that help us to combine frameworks with regular partitions.
To begin, we show that the reduced graph inherits the property of being a Hamilton framework.
Consider families of disjoint sets $\fU$ and $\fV=\{V_i\}_{i=1}^t$, where $\fV$ is $\fU$-refining.
We say that a partition $\fW$ of $\{1,\dots,t\}$ is \emph{induced} by $(\fV,\fU)$,
if each part of $\fW$ contains precisely the indices of the clusters of $\fV$ which are contained in some part of $\fU$.

\begin{proposition}\label{prop:G-robust-proto-Hamiltonian=>R-robust-proto-Hamiltonian}
	Let $1/n \ll 1/t  \ll \mu,1/k$, and $r \in \{1, k \}$.
	Let $\fU$ be an $(n,r)$-sized partition,
	and let $(G,\cH)$ be a $(2\mu)$-robust $\fU$-partite $k$-uniform (aperiodic/zero-free) Hamilton framework.
	For a graph $R$ with vertex set $[t]$, let $(G,\fV)$ be a balanced {$\fU$-refining} $R$-partition {with $t$ clusters in each part of $\fU$.}
	Let $\cJ$ be the $k$-graph induced by $(\cH,\fV)$.
		{Let $\fW$ be the $(t,r)$-sized partition of $R$ induced by $(\fV,\fU)$.}
	Then $(R,\cJ)$ is a $\mu$-robust {$\fW$-partite} $k$-uniform (aperiodic/zero-free) Hamilton framework.
\end{proposition}

\begin{proof}
	We first show part~\ref{item:robust-deletions} of \rf{def:robustness}.
	Let $R' \subset R$ be an arbitrary {$\fW$-partite} $\mu$-approximation of $R$.
	Let $\fV'$ be obtained from $\fV$ by deleting the clusters whose indices are not in $V(R')$.
	We obtain $G' \subset G$ by deleting all vertices of clusters $V_i$ with $i \notin V(R')$ and all edges between $V_i$ and $V_j$ whenever $ij \notin E(R')$.
	Let $m$ denote the common cluster size of $\fV$.
	Since $n=mt$ and $1/n \ll 1/t \ll \mu$, it follows that $G'$ is a {$\fU$-partite} $(2\mu)$-approximation of $G$.
	Let $\cH'=\cH\cap K_k(G')$.
	Since $(G, H)$ is $(2\mu)$-robust, $(G', \cH')$ is a Hamilton framework,
	which is aperiodic or zero-free if $(G,\cH)$ was also aperiodic or zero-free.

	We have to show that  $(R',\cJ')$ is a  Hamilton framework for $\cJ' = \cJ \cap K_k(R')$.
	Observe that connectivity, aperiodicity and zero-freeness (as detailed in \cref{def:Hamilton-framework})  follow immediately from the respective property of $(G', \cH')$ and the definition of $R$-partitions.
	For the matchability, consider a perfect fractional matching $\vec w\colon  E(\cH') \to \REALS$.
	We then construct a perfect fractional matching $\vec w'\in \REALS^{E(\cJ')}$ of $\cJ'$  by setting $\vec w' (K)$ to be the sum of the weights $\vec w(e)/m$ over all edges $e$ that have all their vertices in the $k$ clusters of $K$.

	Finally, part~\ref{item:linked-edges} of \cref{def:robustness} follows by a simple averaging argument.
	Indeed, consider a vertex $i \in V(R')$, and let $v \in V_i$.
	By assumption, $v$ has at least $2\mu n^k$ linked edges in $\cH$.
	By the constant hierarchy, at least $(3\mu/2) n^k$ of these linked edges have their vertices in the clusters of $R'$.
	Now each such linked edge constitutes to one linked edge of $i$ in $\cJ'$.
	Moreover, every linked edge of $i$ in $\cJ'$ corresponds to at most $m^k$ such linked edges of $v$.
	Using again the constant hierarchy, we can rearrange this observation to find that $i$ has at least $\mu t^k$ linked edges in $\cJ'$.
\end{proof}

Reduced graphs of Hamilton frameworks also inherit a more nuanced matching property,
which plays an important role in the process of allocating vertices in \cref{lem:lemma-for-C,lem:lemma-for-H}.
Recall the definition of blow-ups from \cref{subsection:partitions}.
Informally, cluster-matchable graphs are such that their blow-ups by slightly imbalanced clusters admit large matchings.

\begin{definition}[Cluster-matchable]
	Let $\cJ$ be a $\fW$-partite graph where $\fW$ is a $(t,r)$-sized partition of $\cJ$.
	We say that $\cJ$ is \emph{$(\alpha,\beta,\fW)$-cluster-matchable} if the following holds.
	Let~$\fX=\{X_i\}_{i\in V(\cJ)}$ be any $(1+\alpha)$-balanced family of disjoint sets, called \emph{clusters}.
	Suppose that for every $W,W' \in \fW$, we have $|\bigcup_{i \in W} X_i| = |\bigcup_{i \in W'} X_i|$.
	Then the $(\cJ,\fX)$-blow-up $\cJ^\ast$ has a matching missing at most $\beta v(\cJ^\ast) + t^{k-1}$ vertices in each cluster of $\fX$.
\end{definition}

\begin{proposition} \label{prop:cluster-matchable}
	Let $0 < \mu \leq 1$, $k\geq 2$, and $r \in \{1,k\}$.
	Let $\fU$ be an $(n,r)$-sized partition such that $\mu^2n \ge 4t$.
	Let $G$ be a $\fU$-partite graph,
	and let $R$ be a graph with a $(t,r)$-sized partition $\fW$.
	Let $(G,\fV)$ be a $\fU$-refining balanced $R$-partition.
	Suppose that $(G,\cH)$ is a $\mu$-robust $\fU$-partite $k$-uniform Hamilton framework,
	and let $\cJ$ be the $k$-graph induced by $(\cH,\fV)$.
	Then $\cJ$ is $(\mu,{2/n},\fW)$-cluster-matchable.
\end{proposition}

\begin{proof}
	Let $\cJ^\ast$ be a $(\cJ,\fX)$-blow-up for some $(1+\mu)$-balanced family of disjoint sets $\fX=\{X_i\}_{i=1}^{{rt}}$ such that,
	for every $W \in \fW$, the number of elements in clusters of $\fX$ with index in $W$ is the same.
	We have to show that $\cJ^\ast$ contains a matching missing at most $2v(\cJ^\ast)/n + t^{k-1}$ vertices in each cluster of $\fX$.

	By assumption, there exists $m''$ such that $m'' \leq |X_i| \leq (1+\mu)m''$ for all $1 \leq i \leq {rt}$.
	This implies there is an $m'$ such that $m'/(1+\mu) \leq |X_i| \leq m'$ for each $1 \leq i \leq {rt}$.
	Let $m$ be the common cluster size of $\fV = \{V_i\}_{i=1}^{{rt}}$.
	Choose subsets $V_i' \subset V_i$ of size $\lfloor |X_i| m/m' \rfloor$.
	Let $\fV'=\{V_i'\}_{i=1}^{{rt}}$, $G' =G[\fV']$, and $\cH' = \cH \cap K_k(G')$.
	Recall that by assumption the number of elements in clusters of $\fX$ is the same over each part of $\fW$.
	Hence, after accounting for rounding errors, we have $|U \cap V(G')| = |U' \cap V(G')| \pm t$ for all $U,U'\in \fU$.
	To make this an equality, we delete up to $t$ vertices of every part of $\fU$, keeping the names for convenience.
	Consider a part $U  \in \fU$ and suppose that without loss of generality $U$ contains clusters $V_1,\dots,V_t$ of $\fV$.
	We have $$|U \cap V(G')| \geq \left(\sum_{i=1}^{t} \lfloor |X_i| m/m' \rfloor\right) -t \ge \left(\sum_{i=1}^{t} \lfloor m/(1+\mu) \rfloor\right) -t \geq (1-\mu){|U|},$$ where the last step uses $m = n/t \ge 4 \mu^{-2}$, $\mu \leq 1$ and $|U| = n$.
	It thus follows that $G'$ is a {$\fU$-partite} $\mu$-approximation of $G$.
	Since $(G, \cH)$ is a $\mu$-robust {$\fU$-partite} Hamilton framework, we conclude that $(G', \cH')$ is a Hamilton framework.
	Hence there is a perfect fractional matching $\vec w\colon E(\cH') \to \REALS$.

	We define $\vec w'\colon E({R}) \rightarrow \REALS$ by setting $\vec w'(K) = \lfloor \frac{m'}{m} \sum_{e \in \cH_K} \vec w(e) \rfloor$
	for every edge $K$ in $R$ 
	where $\cH_K \subseteq E(\cH')$ denotes the set of all ${k}$-edges $e\in E(\cH')$
	which have one vertex in each of the clusters of $K$.
	Observe that, for every $1 \leq i \leq {rt}$,
	\begin{equation} \label{equation:cluster-matchable-balance}
		\sum_{K\colon i \in K} \sum_{e \in \cH_k} \vec w(e) = |V'_i| = \left\lfloor |X_i| \frac{m}{m'} \right\rfloor,
	\end{equation}
	since $\vec w$ is a perfect fractional matching.
	In particular, by equation~\eqref{equation:cluster-matchable-balance} we get
	\[ \sum_{K\colon i \in K} \vec w'(K) = \sum_{K\colon i \in K} \left\lfloor \frac{m'}{m} \sum_{e \in \cH_K} \vec w(e) \right\rfloor \leq \frac{m'}{m} \sum_{K\colon i \in K} \sum_{e \in \cH_K} \vec w(e) = \frac{m'}{m} \left\lfloor |X_i| \frac{m}{m'} \right\rfloor \leq |X_i|\]
	for every $1 \leq i \leq {rt}$.
	By the definition of $\cJ^\ast$ and $\sum_{K\colon i \in K} \vec w'(K) \leq |X_i|$, we may greedily select a matching $M \subset \cJ^\ast$ which contains exactly $\vec w'(K)$ many $K$-edges for each $K \in E(R)$.
	Using equation~\eqref{equation:cluster-matchable-balance} again, it follows that $M$ covers at least
	\begin{align*}
		\sum_{K\colon i \in K} \vec w'(K)
		 & = \sum_{K\colon i \in K} \left\lfloor \frac{m'}{m} \sum_{e \in \cH_K} \vec w(e) \right\rfloor
		\geq \sum_{K\colon i \in K} \left( \frac{m'}{m} \sum_{e \in \cH_K} \vec w(e) - 1 \right)
		\geq \frac{m'}{m} \left\lfloor |X_i| \frac{m}{m'} \right\rfloor  - {t}^{k-1}             \\
		 & \geq \frac{m'}{m} \left( |X_i| \frac{m}{m'} - 1 \right) - {t}^{k-1}
		= |X_i| - \frac{m'}{m} - {t}^{k-1}
		\geq |X_i| - {\frac{2v(\cJ^\ast)}{n}} - {t}^{k-1}
	\end{align*}
	vertices of each cluster $X_i$,
	where in the last inequality we used that $m = n/t$ and $m' \leq (1 + \mu) v(\cJ^*)/(rt) \leq 2 v(\cJ^*)/t$ to get
	$m'/m \leq 2 v(\cJ^*)/n$ and $\mu \leq 1$.
\end{proof}

The next proposition states that a robust Hamilton framework $(G,\cH)$ admits a `reduced framework' $(R,\cJ)$, which approximately inherits its properties.
Alone, this is a consequence of the \fr{lem:regularity-classic} and \cref{prop:G-robust-proto-Hamiltonian=>R-robust-proto-Hamiltonian}.
But additionally, we will also ensure the existence of a small family $\fF$ of subgraphs of $\cJ$, which allows us to (re)integrate any small number of vertices of $G$ into the partition $\fV$.
We require this property to deal with these exceptional vertices in the proofs of the Lemmas for $G$ (\cref{lem:lemma-for-G-cycle,lem:lemma-for-G-bandwidth}).
We remark that a similar argument was used by Ebsen, Maesaka, Reiher, Schacht and Schülke~\cite{EMR+20}.

\begin{lemma}[Regular partition with entry points]\label{lem:regular-partition+absorber}
	Let $1/n \ll 1/t_1 \ll 1/t_0 \ll \eps  \ll d  \ll \eta  \ll \gamma \ll \mu, 1/k$ and $1\leq q\leq k$ and $r \in \{1,k\}$.
	Let $\fU$ be an $(n,r)$-sized partition,
	and let $(G,\cH)$ be a $\mu$-robust $\fU$-partite $k$-uniform Hamilton framework.

	Then there are $t_0 \leq t \leq t_1$ with $t \equiv q \bmod k$,
	a subgraph $G' \subset G$ with a balanced $\fU$-refining partition $\fV$ with $t$ clusters in each part of $\fU$,
	and a graph $R$ with a $(t,r)$-sized partition $\fW$ induced by $(\fV,\fU)$,
	which satisfy that
	\begin{enumerate}[\upshape (R1)]
		\item $G'$ is a {$\fU$-partite} $(\eps,4d)$-approximation of $G$ and
		\item $(G',\fV)$ is an $(\eps,d)$-regular $R$-partition.
	\end{enumerate}
	Let $\cH' = \cH \cap K_k(G')$, and let $\cJ$ be the $k$-graph induced by $(\cH',\fV)$.
	Then
	\begin{enumerate}[\upshape (R1)] \addtocounter{enumi}{2}
		\item \label{itm:regular-partition+absorber-R-framework} $(R,\cJ)$ is a $(\mu/4)$-robust {$\fW$-partite} $k$-uniform Hamilton framework and $(\mu/4,2/n,{\fW})$-cluster-matchable, and
		\item \label{itm:regular-partition+absorber-W} there is a collection ${\fF}$ of subgraphs of $\cJ$ such that
		      \begin{enumerate}[\upshape (i)]
			      \item ${\fF}$ has at most {$2^{\eta t}$} elements,
			      \item each subgraph in ${\fF}$ has at least $\gamma (rt)^k$ edges,
			      \item for every vertex $v \in V(G)$, there is an ${\cF} \in {\fF}$ such that $v$ has at least $\gamma (n/t)^k$ linked edges {in $\cH'$} in the clusters of $\fV$ corresponding to every $K \in E({\cF})$.
		      \end{enumerate}
	\end{enumerate}
	Finally, if  $(G,\cH)$ is $\mu$-robust aperiodic (resp. zero-free), then $(R,\cJ)$ is  $(\mu/4)$-robust aperiodic (resp. zero-free) as well.
\end{lemma}

\begin{proof}
	We introduce new constants $t^\ast \in \NATS$ and $\eps' >0$ such that
	\begin{align}\label{ref:constant-hierarchy-regular-partition+absorber}
		1/n \ll 1/t_1^\ast \ll 1/t_1 \ll 1/t_0 \ll  \eps \ll \eps' \ll  d \ll \eta \ll \gamma \ll \mu, 1/k. \tag{$\ll_G$}
	\end{align}
	In particular, we choose $t_1$ so that \rf{lem:regularity-classic} holds with parameters {$t_0,d,\eps,k$}.
	We further choose these constants such that the various inequalities along the proof are satisfied.
	Finally, let $1 \leq q \leq k$.
	We start with the case where $(G,\cH)$ is a $\mu$-robust $\fU$-partite $k$-uniform Hamilton framework.
	(The cases where $(G,\cH)$ is aperiodic or zero-free will be discussed in the end.)
	\medskip

	\noindent \emph{Step 1: Obtaining a regular partition.}
	We apply \cref{lem:regularity-classic} to  $G$ with input $\eps, d,t_0,k,q$ to obtain $t_0 \leq t \leq t_1$ with $t\equiv q \bmod k$, a subgraph $G' \subset G$ with a {$\fU$-refining} balanced partition $\fV=\{V_i\}_{i=1}^{{rt}}$ of $V(G')$ {with $t$ clusters in every part of $\fU$} and a graph $R$ such that
	\begin{enumerate}[\upshape (S1)]
		\item \label{itm:entry-points-reg-lem-approx} $G'$ is a {$\fU$-partite} $(\eps,d+\eps)$-approximation of $G$ and
		\item $(G',\fV)$ is a $(\eps,d)$-regular $R$-partition.
	\end{enumerate}
	Let $m$ be the common cluster size of $\fV$.
	Let $\cH' = \cH \cap K_k(G')$ and {$\fU'=\{U\cap V(G')\colon U \in \fU\}$ the restriction of $\fU$ to $V(G')$.}
	Note that~\ref{itm:entry-points-reg-lem-approx} and $\eps, d \ll \mu$ together
	imply that {$\fU'$-partite} $(\mu/2)$-approximations of $G'$ will be {$\fU$-partite} $\mu$-approximations of $G$.
	Since $(G,\cH)$ is an $\mu$-robust {$\fU$-partite} Hamilton framework,
	we deduce that $(G',\cH')$ must be a $(\mu/2)$-robust {$\fU'$-partite} Hamilton framework.
	\medskip

	\noindent \emph{Step 2: Locating entry points.}
	Let $\cJ$ be the $k$-graph induced by $(\cH',\fV)$.
	Note that since $(G', \fV)$ is an $R$-partition and $\cH' \subseteq K_k(G')$, we have $\cJ \subseteq K_k(R)$.
	The following claim states that every vertex of $\cH$ can be added to $\fV$ in many ways.
	\begin{claim}\label{cla:entry-points-course}
		There is a collection ${\fF}$ of subgraphs of $\cJ$ such that
		\begin{enumerate}[\upshape (i)]
			\item ${\fF}$ has at most $2^{t^k}$ elements,
			\item each subgraph in ${\fF}$ has at least $\gamma (rt)^k$ edges and
			\item for every vertex $v \in V(G)$, there is an ${\cF} \in {\fF}$ such that $v$ has at least $4\gamma m^k$ linked edges (in $\cH'$) in the clusters of $\fV$ corresponding to every $K \in E({\cF})$.
		\end{enumerate}
	\end{claim}
	\begin{proofclaim}
		Fix a vertex $v \in V(G)$.
		Since $(G,\cH)$ is a $\mu$-robust {$\fU$-partite} $k$-uniform Hamilton framework, $v$ has at least $\mu n^k$ linked edges in $\cH$.
		Note that removing all edges from $\cH$ which contain a given vertex (resp. a given pair of vertices) eliminates at most $n^{k-1}$ (resp. $n^{k-2}$) linked edges of $v$.
		Since $\cH' = \cH \cap K_k(G')$, $d, \eps \ll \mu$ and~\ref{itm:entry-points-reg-lem-approx}
		imply that at most $r(\eps n^k + (d+\eps)n^k) \leq \mu n^k/2$ linked edges {of $v$} are destroyed in passing from $\cH$ to $\cH'$.
		We deduce that at least $(\mu/2) n^k$ of these linked edges must be contained in~$\cH'$.
		We let ${\cF_v} \subseteq E(\cJ)$ be the set of edges $K$ such that $v$ has at least {$4 \gamma m^k$} linked edges in the clusters corresponding to $K$.
		Hence the number of linked edges is at most $m^k |{\cF_v}|  +  4\gamma m^k \left(e(\cJ) - |{\cF_v}|\right)$.
		Comparing this to the lower bound while taking into account that $1/t \ll \gamma \ll \mu,1/k$ reveals that $|{\cF_v}| \geq \gamma (rt)^k$.

		Let $\fF = \{\cF_v \colon v \in V(G)\}$.
		Since each ${\cF} \in {\fF}$ is a subset of edges of $\cJ$,
		and $\cJ$ has at most $t^k$ edges,
		we have $|{\fF}| \leq 2^{t^k}$, as required.
	\end{proofclaim}

	\medskip
	\noindent \emph{Step 3: Refining the partition.}
	Consider a collection ${\fF}$ as in \cref{cla:entry-points-course}.
	Note that the only difference to property~\ref{itm:regular-partition+absorber-W} in the lemma statement is the size of ${\fF}$ and the fact that part (iii) of the claim is stated in terms of $m$ as opposed to $n/t$.
		We deal with the latter in Step 4 and focus on the size of $\cF$ for now.
	We will adjust this size by finding a suitable refinement of $\fV$ and then considering the  {corresponding} blow-ups of $R$, $\cJ$ and ${\fF}$.

	Let $s$ be the least integer such that $s \geq t^{k-1}/\eta$ and $s \equiv 1 \bmod k$.
	Set $t^\ast = st$ and note that $t^\ast \leq t^\ast_1$.
	Moreover, $t^k \leq \eta t^\ast$ and $t^\ast \equiv t \equiv q \bmod k$.
	Now consider a random refinement $\fV^*$ of $\fV$ obtained as follows.
	For every $1 \leq i \leq {rt}$, we partition $V_i$ into $\{V_{i,j}\}_{j=1}^s$ by placing each vertex of $V_i$ uniformly at random in one of the parts $V_{i,1},\dots,V_{i,s}$.
	Let $m^\ast$ be the minimum cluster size of $\fV^*$.
	Let $R^\ast$ be the graph on ${rt^\ast}$ vertices induced by $(R,\fV^\ast)$.
	By \cref{prop:random-refinement}, with probability at least $2/3$, we have that
	\begin{enumerate}[{\upshape (R1$'$)}] \addtocounter{enumi}{2}
		\item \label{item:regularity-lemma-regular}  $(G',\fV^*)$ is a $(1+\eps')$-balanced $(\eps',d)$-regular $R^\ast$-partition.
	\end{enumerate}

	Let $\cJ^\ast$ be the $s$-blow-up of $\cJ$.
	Let ${\fF}^\ast = \{{\cF}^\ast \colon {\cF} \in {\fF}\}$  where ${\cF}^\ast \subset \cJ^\ast$ is the subgraph spanned by the edge set $$E({\cF}^\ast)=\{\{(i_1,j_1),\dots,(i_k,j_k)\} \colon \{i_1,\dots,i_k\} \in E({\cF}) \text{ and } 1 \leq j_1, \dots,j_k \leq s\}.$$
	
	\begin{claim}\label{cla:refinement-S}
		With probability at least $2/3$, we have:
		\begin{enumerate}[\upshape (i)]
			\item ${\fF}^\ast$ has at most $2^{\eta t^\ast}$ elements,
			\item each subgraph $\cF^\ast$ in ${\fF}^\ast$ has at least $\gamma (r t^\ast)^k$ edges and
			\item for every vertex $v \in V(G)$, there is an ${\cF}^\ast \in {\fF}^\ast$ such that $v$ has at least $2\gamma (m^\ast)^k$ linked edges ({in {$\cH'$}}) in the clusters of $\fV^\ast$ corresponding to every $K \in E({\cF}^\ast)$.
		\end{enumerate}
	\end{claim}

	\begin{proofclaim}
		We have $|{\fF}^\ast| = |{\fF}| \leq 2^{t^k} \leq 2^{\eta t^\ast}$ {deterministically}, by \cref{cla:entry-points-course} and the choice of $s$ and $t^\ast$.
		Analogously, for any $\cF^\ast \in \fF^\ast$, we have $|\cF^\ast| = |\cF|s^k \geq \gamma (rt^\ast)^k$, as required.

		Now consider a vertex $v \in V(G)$.
		By \cref{cla:entry-points-course}, there is an ${\cF} \in {\fF}$ such that $v$ has at least $4\gamma m^k$ linked edges in {$\cH'$} in the clusters of $\fV$ corresponding {to} every $K \in E({\cF})$.
		Given such $K$, let us write $K=\{i_1,\dots,i_k\}$ and fix $1 \leq j_1, \dots , j_k \leq s$.
		Let $X$ denote the number of linked edges in the clusters of $\fV^*$ corresponding to $\{(i_1,j_1),\dots,(i_k,j_k)\}$.
		So a linked edge of $v$ is counted by $X$ if and only if all of its $k$ vertices are placed into the right cluster in our random partition.
		It follows that $\expectation(X) \geq 4\gamma (m/s)^k$.
		Moreover, changing the outcome of any vertex placement affects $X$ by at most $m^{k-1}$.
		Hence we may apply McDiarmid's inequality (\cref{theorem:mcdiarmid}) with $m^{k-1}$ and $km$ playing the roles of $B$ and $m$ to find that $X \geq 2\gamma (m^\ast)^k$ with probability at least $1-\exp(-\Omega(n))$.
		Since $1/n \ll 1/t \ll 1/k$, taking a union bound over all vertices $v \in {V}$, subgraphs ${\cF} \in {\fF}$, edges $K \in E({\cF})$ and choices of $1 \leq i_1, \dots , i_k \leq s$, yields the desired result.
	\end{proofclaim}

	From now on, fix an outcome of $\fV^*$ (and thus $m^\ast$, $R^\ast$, $\cJ^\ast$ and ${\fF}^\ast$) satisfying \cref{item:regularity-lemma-regular} and the properties of \cref{cla:refinement-S}.

	\medskip
	\noindent \emph{Step 4: Balancing the partition.}
	{Let $\fW$ be the $(t^\ast,r)$-sized partition of $R^\ast$ induced by $(\fV^\ast,\fU')$.}
	Since $(G',\fV^*)$ is an $R^\ast$-partition,
	and $(G',\cH')$ is a $(\mu/2)$-robust {$\fU'$-partite} Hamilton framework,
	it follows that $(R^\ast, \cJ^\ast)$ is a $(\mu/4)$-robust {$\fW$-partite} Hamilton framework by \cref{prop:G-robust-proto-Hamiltonian=>R-robust-proto-Hamiltonian}.
	Moreover, $(R^\ast, \cJ^\ast)$ is also {$(\mu/2,2/n,{\fW})$}-cluster-matchable (and therefore $(\mu/4,2/n,{\fW})$-cluster-matchable) by \cref{prop:cluster-matchable}.

	Finally, since $\fV^*$ is $(1+\eps')$-balanced, we can delete up to $\eps' m^\ast$ vertices of each part of $\fV^*$ in order to ensure that all clusters have common size $m^\ast$.
	We also delete these vertices from $G'$ and $\cH'$.
	Note that the total number of vertices removed this way is bounded by $\eps'  m^\ast {r} t^\ast \leq \eps' n$.
	Since $\eps \ll \eps' \ll d$, $G'$ is still a {$\fU$-partite} $(2\eps',2d)$-approximation of $G$.
	As $\eps' \ll \gamma$, the last property of \cref{cla:refinement-S} still holds with ${(3/2)} \gamma ({n/t^\ast})^k$ in place of $2\gamma (m^\ast)^k$.
	Moreover, $(G',\fV^*)$ is still a $({100 \sqrt{\eps'}},d/2)$-regular $R^\ast$-partition by \rf{proposition:robust-regular}.
	Finally, consider $v \in V(G)$.
	Note that the removal of $\eps' m^\ast$ vertices in the clusters corresponding to $\{(i_1,j_1),\dots,(i_k,j_k)\}$ of $\fV^*$,
	could have destroyed, at most, $k \eps' m^{\ast k} \leq {(1/4) \gamma (n/t^\ast)^k}$ linked edges of $v$ living in those clusters.
	We thus deduce that for every vertex $v \in {V}$, there is an ${\cF} \in {\fF}^\ast$ such that $v$ has at least ${\gamma (n/t^\ast)^k}$ linked edges in the clusters corresponding to every edge $K \in E({\cF})$.
	Hence, we may finish with the parameters $100 \sqrt{\eps'}, t^\ast, t_1^\ast,  \fV^*, R^\ast, {\fF}^\ast, d/2$ playing the roles of $\eps, t,t_1 ,\fV, R, {\fF}, d$, respectively.
		{This finishes Step 4, and the proof, in the case where $(G,\cH)$ is not necessarily aperiodic or zero-free.}
	\medskip

	If $(G,\cH)$ is $\mu$-robust aperiodic or zero-free, we follow the same proof.
	It is easy to see in Step 1 that the property of being aperiodic or zero-free is inherited by $(G',\cH')$.
	In Step 4 the property is again inherited by $(R^\ast, \cJ^\ast)$, this time by \cref{prop:G-robust-proto-Hamiltonian=>R-robust-proto-Hamiltonian}.
	This finishes the proof in all cases.
\end{proof}

\section{Embedding powers of Hamilton cycles}\label{sec:cycle}
In this section, we prove \rf{thm:main-powham}.
In fact, we show a slightly stronger result in which the power of a Hamilton cycle we obtain also contains a predefined set of cliques.
This strengthening will be used in the proof of \rf{thm:main-bandwidth} later.

\begin{theorem}[Distributed Hamiltonicity]\label{thm:cycle-distributed}
	Let $1/n \ll \eta \ll \pi \ll \mu, 1/k$, and let $r \in \{1, k \}$.
	Let $\fU$ be an $(n,r)$-sized partition,
	and let $(G,\cH)$ be a $\mu$-robust $\fU$-partite $k$-uniform Hamilton framework.
	If $r = 1$, suppose in addition that $(G,\cH)$ is $\mu$-robust aperiodic.
	Let ${\fF}$ be a collection of at most $2^{\eta n}$ subgraphs of $\cH$, which each contain at least $\mu n^k$ $k$-edges.
	Then $G$ has as a subgraph the $(k-1)$st power of a Hamilton cycle $C$ such that $K_k(C)$ contains at least $\pi rn$ edges of each element of ${\fF}$.
\end{theorem}

The proof of \cref{thm:cycle-distributed} rests on the following two lemmas, whose proofs are given in \cref{sec:lemma-for-G-cycle,sec:lemma-for-C}, respectively.

\begin{lemma}[Lemma for $G$ -- cycles]\label{lem:lemma-for-G-cycle}
	Let $1/n \ll 1/t_1 \ll 1/t_0 \ll \eps \ll d \ll \mu, 1/k$, and let $r \in \{1,k\}$.
	Let $\fU$ be an $(n,r)$-sized partition,
	and let $(G,\cH)$ be a $\mu$-robust $\fU$-partite $k$-uniform Hamilton framework.
	If $r = 1$, then $(G, H)$ is in addition $\mu$-robust aperiodic.
	Then there are $t_0 \leq t \leq t_1$,
	a spanning subgraph $G' \subset G$,
	a $\fU$-refining $(1 + \eps)$-balanced partition $\fV$ of $G'$ with $t$ clusters inside each cluster of $\fU$,
	and a graph $R$ on vertex set $\fV$
	such that the following holds.
	Let $\cH' = \cH \cap K_k(G')$, and let $\cJ$ be the $k$-graph induced by $(\cH',\fV)$.
	There also exists a spanning subgraph $\cJ' \subset \cJ$ with $\Delta(\cJ') \leq k^2+1$, so that, for $R'=\partial_2 \cJ'$, we have
	\begin{enumerate}[\upshape (G1)]
		\item \label{itm:lemma-for-G-cycle-approx} $e(G') \geq (1-\mu/2)e(G)$,
		\item \label{itm:lemma-for-G-cycle-reg} $(G',\fV)$ is an $(\eps,d)$-regular $R$-partition, which is $(\eps,d)$-super-regular on~$R'$ and
		\item \label{itm:lemma-for-G-cycle-frame} $(R,\cJ)$ is a $(\mu/4)$-robust {$\fW$-partite} $k$-uniform  Hamilton framework and $(\mu/4,2/n,\fW)$-cluster-matchable 
		      where $\fW$ is the $(t,r)$-sized partition induced by $(\fV,\fU)$.
	\end{enumerate}
	Moreover, if $r=1$, then $(R,\cJ)$ is $(\mu/4)$-robust aperiodic as well.
\end{lemma}

\begin{lemma}[Lemma for $C$]\label{lem:lemma-for-C}
	Let $1/n \ll \pi \ll 1/t \ll \eps \ll \alpha \ll \mu, 1/k$, and let $r \in \{1,k\}$.
	
	Let $\fU$ be an $(n,r)$-sized partition and let $G$ be a $\fU$-partite graph.
	Let $\fV$ be a $\fU$-refining $(1+\eps)$-balanced partition with $t$ clusters inside each cluster of $\fU$,
	and let $R$ be a graph on $\fV$ such that $(G, \fV)$ is an $R$-partition.
	Let $\fW$ be the $(t,r)$-sized partition of $R$ induced by $(\fV,\fU)$.
	Let $\cJ \subset K_k(R)$.
	Suppose that $(R,\cJ)$ is a $\mu$-robust {$\fW$-partite} $k$-uniform  Hamilton framework and  $(\mu,2/n,{\fW})$-cluster-matchable.
	If $r = 1$, suppose in addition that $(R,\cJ)$ is aperiodic.
	Let $\cJ' \subset \cJ$ be a spanning subgraph with $\Delta({\cJ'}) \leq k^2+1$.
	
	Then there is a size-compatible (with $\fV$) vertex partition $\fX = \{X_i\}^{rt}_{i=1}$ of the vertex set of the $(k-1)$st power of a cycle $C$, a family of subsets $\tfX = \{ \tX_i \}_{i=1}^{rt}$ of $V(C)$,
	and families of $(k-1)$st powers of paths $\{ P_K \}_{K \in E(\cJ)}$ which are subgraphs of $C$ such that
	\begin{enumerate}[\upshape ({C}1)]
		\item \label{itm:allocate-vertices-partition} $(C,\fX)$ is an $R$-partition,
		\item \label{itm:allocate-vertices-buffer} $\tfX$ is an $(\alpha, R')$-buffer for $(C, \fX)$  where $R' = \partial_2 \cJ'$,
		\item \label{itm:allocate-vertices-paths} for each $K \in E(\cJ)$, ${P}_K$ has length $\pi n$ and ${P}_K \subset C[\bigcup_{i  \in K} X_i]$, and
		\item \label{itm:allocate-vertices-disjoint} the paths $\{ P_K \}_{K \in E(\cJ)}$ are pairwise vertex-disjoint.
	\end{enumerate}
\end{lemma}

Assuming the validity of these two lemmas, we now prove the main result of this section.

\begin{proof}[Proof of \cref{thm:cycle-distributed}]
	We introduce $\eta,\pi,\eps,\eps',d,\alpha>0$ and $t_1,t_0 \in \NATS$ to establish the following constant hierarchy
	\begin{align}\label{equ:cycle-distributed-constant-hierarchy}
		1/n \ll  \eta \ll \pi \ll 1/t_1 \ll 1/t_0 \ll  \eps \ll \eps' \ll d,\alpha  \ll  \mu \ll 1/k. \tag{$\ll$}
	\end{align}
	More precisely, given $k\geq 2$, we can assume that $\mu$ is small enough to satisfy the upcoming inequalities.
	Next, choose $d >0$ to satisfy \cref{lem:lemma-for-G-cycle}.
	For $k$ and $\mu/4$, choose $\alpha > 0$ to satisfy \cref{lem:lemma-for-C}.
	We select $\eps' \leq \min\{\mu/2, d(1-(3/4)^{1/(k-1)}) \}$ which is {in addition} sufficiently small so that \cref{prop:regularity-joint degree} is valid when the input is $k$ and $d$.
	Given $k, d, \eps'$, \cref{prop:regularity-joint degree} also outputs a suitable $\eps_0$ (in place of $\eps$).
	Let $\zeta = 1/2$.
	Then select $\eps_1, \rho > 0$, so that \cref{lem:blow-up} (Blow-Up Lemma) can be applied with {$\Delta = 3(k-1)$, $\Delta_J=k-1$}, $\DeltaRp={k^3}$, $\kappa = 2$,
	$\alpha$, $\zeta$ and $d$.
	Select also $\eps_2$ small enough to serve as input for \cref{lem:lemma-for-G-cycle} (given $\mu, k, d$) and $\eps_3$ to serve as an input for \cref{lem:lemma-for-C} (given $\mu/4$, $k$, $2\alpha$ in place of $\mu, k, \alpha$), and let $\eps = \min \{ \eps_0, \eps_1, \eps_2, \eps_3, \mu/4, 1/2 \}$.

	Choose $t_0$ that satisfies \cref{lem:lemma-for-G-cycle} with the above constants.
	Given this $t_0$, the hierarchy in \cref{lem:lemma-for-G-cycle} outputs a $t_1$.
	Handing $t_1$ to \cref{lem:lemma-for-C} (in place of $t$) outputs an appropriate {$\pi_1$} (in place of $\pi$).
	We give $k$ and $\mu/(4 t_1^k)$ as an input (in place of $k$ and $\gamma$) to \rf{prop:distributed-matching} and obtain $\pi_2$ as an output (in place of $\pi$).
	We let $\pi = \min\{(\pi_1/3k)^2, \pi_2, (\rho/(2kt))^2\}$.
	With this choice of $\pi$, the hierarchy in \cref{prop:distributed-matching} outputs $\eta$.
	Finally, we choose $n$ large enough for \cref{lem:blow-up,lem:lemma-for-G-cycle,lem:lemma-for-C} and \cref{prop:distributed-matching} are satisfied (together with the above choices).

	Now suppose that $G$, $\fU$ and $\fF$ meet the conditions of \cref{thm:cycle-distributed}.
	Namely, $(G,\cH)$ is a $\mu$-robust $\fU$-partite $k$-uniform Hamilton framework  {where $\fU$ is an $(n,r)$-sized partition of $G$.}
	In addition, if $r = 1$ then $(G,\cH)$ is $\mu$-robust aperiodic.
	Also ${\fF}$ is a collection of at most $2^{\eta n}$ subgraphs of $\cH$, which each contain at least $\mu n^k$ edges.
	We have to show that $G$ has a tight Hamilton cycle $C$, which contains at least {$\pi n$} edges of each element of ${\fF}$.

	\medskip
	\noindent \emph{Step 1: Applying the \nameref{lem:lemma-for-G-cycle}.}
	By \cref{lem:lemma-for-G-cycle}, there are $t_0 \leq t \leq t_1$,
	a spanning subgraph $G' \subset G$,
	a $\fU$-refining $(1 + \eps)$-balanced partition $\fV=\{V_i\}_{i=1}^{{rt}}$ of $G'$ with $t$ clusters in each part of $\fU$,
	a graph $R$ on $\fV$ such that the following holds.
	Let $\cH' = \cH \cap K_k(G')$, and let $\cJ$ be the $k$-graph induced by $(\cH',\fV)$.
	There also exists a spanning subgraph $\cJ' \subset \cJ$ with $\Delta(\cJ') \leq k^2+1$, so that, for $R'= \partial_2 \cJ'$, we have
	\begin{enumerate}[\upshape (G1)]
		\item $e(G') \geq (1-\mu/2)e(G)$,
		\item $(G',\fV)$ is an $(\eps,d)$-regular $R$-partition, which is $(\eps,d)$-super-regular on~$R'$, and
		\item $(R,\cJ)$ is a $(\mu/4)$-robust {$\fW$-partite} $k$-uniform Hamilton framework and  {$(\mu/4,2/n,\fW)$-cluster-matchable} 
		      where $\fW$ is the $(t,r)$-sized partition induced by $(\fV,\fU)$.
	\end{enumerate}
	Moreover, if $(G,\cH)$ is $\mu$-robust aperiodic, then $(R,\cJ)$ is $(\mu/4)$-robust aperiodic as well.
	Note that $\Delta(R') \leq (k-1)(k^2+1) \leq k^3$.

	\medskip
	\noindent \emph{Step 2: Applying the \nameref{lem:lemma-for-C}.}
	Next, we apply \cref{lem:lemma-for-C} (with $\pi_1$ {and $2\alpha$}  in place of $\pi$ {and $\alpha$}) to obtain a size-compatible (with $\fV$) vertex partition $\fX=\{X_i\}_{i=1}^{{rt}}$ of the vertex set of the $(k-1)$st power of a cycle $C$, a family $\tfX = \{ \tX_i \}_{i=1}^{rt}$, and a family of $(k-1)$st powers of paths $\{ P_K \}_{K \in E(\cJ)}$ which are subgraphs of $C$ such that
	\begin{enumerate}[\upshape ({C}1)]
		\item $(C,\fX)$ is an $R$-partition,
		\item \label{itm:buffer-from-lemmas-for-C} $\tfX$ is a $(2 \alpha, R')$-buffer for $(C, \fX)$,
		\item \label{itm:paths-from-lemmas-for-C} for each $K \in E(\cJ)$, ${P}_K$ has length $\pi_1 n$ and ${P}_K \subset C[\bigcup_{i  \in K} X_i]$, and
		\item  the paths $\{ P_K \}_{K \in E(\cJ)}$ are pairwise vertex-disjoint.
	\end{enumerate}

	\medskip
	\noindent \emph{Step 3: Fixing representatives of $\fF$.}
	Next, we find a matching $M \subset \cH'$ of bounded size, which contains many edges of every subgraph of $\fF$, as follows.
	Fix a subgraph $\cF \in \fF$, and let $\cF_1 = \cF \cap \cH'$.
	Since $\cF$ has at least $\mu n^k$ edges and
	$e(G') \geq (1-\mu/2)e(G)$, it follows that $\cF_1$ retains at least $(\mu/2)n^k$ edges.
	(Here we used that each edge of $G$ is in at most $n^{k-2}$ edges of $\cH$.)
	Moreover, by the pigeonhole principle  and since $\cJ$ has at most $t^k$ edges,  there is an edge $K_{\cF}$ in ${\cJ}$ such that there are at least $(\mu/2) (n/t)^k$ edges $\cF_2 \subset \cF_1$ which are partite to the clusters of $K_{\cF}$.

	Let us call an edge $e$ of $\cF_2$ \emph{typical},
	if every $k-1$ vertices of $e$ share at least $(d-\eps')^{k-1} n/t$ common neighbours in $G'$ in the cluster of $K_{\cF}$ that they miss.
	By \rf{prop:regularity-joint degree}, all but at most $\eps' (n/t)^{k}$ edges of $\cF_2$ are typical.\footnote{
			Formally, \cref{prop:regularity-joint degree} requires a balanced partition.
			Since our partition is $(1+\eps')$-balanced, we may adjust this by deleting a few vertices, and hide the quantitative differences in our choice of constants.
		}
	We write $\cF_3 \subset \cF_2$ for the subgraph of typical edges of $\cF_2$.
	From the choice of $\eps'$ and $t \leq t_1$ we note that $|\cF_3| \ge |\cF_2| - \eps' (n/t)^{k} \ge (\mu/4)(n/t)^k \ge (\mu/4) (n/t_1)^k$.
	Let $\fF' = \{ \cF_3 \colon \cF \in \fF \}$.
	It follows by \rf{prop:distributed-matching} applied with $\eta,\pi r$ and $\mu/(4t_1^k)$ playing the role of $\gamma$ that there is a matching ${M} \subset \cH'$ of size at most $\sqrt{\pi}n$ such that ${M}$ has at least $\pi r n$ edges of every subgraph $\cF_3 \in \fF'$.

	For each $e \in M$, let $K$ be the edge of {$\cJ$} whose clusters of $\fV$ contain~$e$,
	and select a $(k-1)$st power of path $P_e \subseteq P_K \subseteq C$
	of length $3k$.
	This can be done in such a way that the family $\{ P_e \}_{e \in M}$ consists of pairwise vertex-disjoint powers of paths.
	Indeed, by the choice of $\pi$, we have $\pi_1 n \ge 3 k \pi^{1/2} n$;
	thus property~\ref{itm:paths-from-lemmas-for-C} implies that for each $K \in E({\cJ})$, the associated power of path $P_K$ has length at least $3k \pi^{1/2}n \geq 3 k e(M)$,
	where the latter inequality follows from $e(M) \leq \sqrt{\pi} n$.

	\medskip
	\noindent \emph{Step 4: Applying the Blow-Up Lemma.}
	We finish the proof with an application of the Blow-Up Lemma (\cref{lem:blow-up}).
	Recall that $\{\tilde{X}_i\}_{i=1}^{rt}$ is a $(2\alpha, R')$-buffer for $C$.
	To finish, we define a subgraph $C' \subset C$, a family of image restrictions $\fI=\{I_x\}_{x\in V(C)}$ and a family of restricting vertices $\fJ=\{J_x\}_{x\in V(C)}$.
	The idea is to `crop out' the middle part of each power of a path $P_e$,
	and ensure that the gap will be mapped on the edge $e \in E (M)$ via an image restriction.

	More precisely, for every edge $e=\{v_1,\dots,v_k\}$ in $M$ whose corresponding power of a path $P_e \subseteq C$ is given by $P_e=(x_1,\dots,x_k,y_1,\dots,y_k,x_{k+1},\dots,x_{{2k}})$, let $Y_e = \{y_1, \dotsc, y_k\} \subseteq V(P_e)$.
	We define $C' \subseteq C$ by $C' = C - \bigcup_{e \in {E(M)}} Y_e$.
	Given this, the restriction pairs for $x_1,\dots,x_k$ are defined by setting $J_{x_j} = e \sm \{v_j\}$ and $I_{x_j} = \bigcap_{v \in J_{x_j}} N_{G'}(v) \cap V_i$ for $1 \leq j \leq k$ where $i$ is the index of the cluster of $\fX$ that contains $x_j$.
	Note that since $e$ is typical, we have $|I_{x_j}| \geq (d-\eps')^{k-1} n/t \geq \zeta d^{k-1} |V_i|$, where in the inequality we used that $\zeta = 1/2$, that $\fV$ is $(1+\eps)$-balanced, $\eps \leq 1/2$ and the choice of~$\eps'$.
	The restriction pairs for $x_{k+1},\dots,x_{2k}$ are defined analogously by setting $J_{x_{k+j}} = e \sm \{v_j\}$ and $I_{x_{k+j}} = \bigcap_{v \in J_{x_{k+j}}} N_{G'}(v) \cap V_{i}$ for $1 \leq j \leq k$ where $i$ is the index of the cluster of $\fX$ that contains $x_{k+j}$.
	For the vertices $x \in V(C')$ not covered by any of the paths $P_e$ with $e \in M$, we simply set $J_x=\es$ and  $I_x=V_{i}$  where $i$ is the index of the cluster of $\fX$ that contains $x$.
	
	Let $\fX'$ be obtained from $\fX$ via restriction to $V(C')$.
	So $\{\tilde{X}_i\}_{i=1}^{rt}$ is still a $(\alpha, R')$-buffer for $(C',\fX')$ by choice of $\pi$ and $\alpha$.
	We claim that $\fI$ and $\fJ$ form a $(\rho, \zeta, \Delta, \Delta_J)$-restriction pair for $(C',\fX')$ as formulated in \cref{def:restrict}.
	Indeed, part~\ref{itm:restrict:Xis} is satisfied because the total number of image restricted vertices is $2|v(M)| \leq 2k \sqrt{\pi} n$, and 
	thus certainly there are at most {$2 k \sqrt{\pi} n  \leq \rho |X_i'|$} image restricted vertices in each $X_i'$ (the inequality follows from the choice of $\pi$).
	Part~\ref{itm:restrict:sizeIx} has already been verified above using the typicality of the edges of $M$.
	For part~\ref{itm:restrict:Jx}, it suffices to note that $|J_x|+\deg_{C'}(x) \leq 3(k-1) = \Delta$.
	Finally, part~\ref{itm:restrict:DeltaJ} is also satisfied since each vertex of $V(C')$ that is covered by $M$ appears in $k-1$ of the sets of $\fJ$, and vertices not covered by $M$ appear in none of them.
	Hence $\fI$ and $\fJ$ present a $(\rho, \zeta, \Delta, \Delta_J)$-restriction pair as claimed.

	Finally, we apply \cref{lem:blow-up} (Blow-Up Lemma) with $G=G' - V(M)$ and $H=C'$ together with the above chosen parameters (note that $\kappa \ge 1+ \eps$) to obtain an embedding $\psi\colon V(C')\to V(G)$ such that $\psi(x) \in I_x$ for each $x \in V(\cH)$.
	It follows that $G$ contains the $(k-1)$st power of a Hamilton cycle $C''$, where $K_k(C'')$ contains all edges of $M$.
	By the construction of $M$, we deduce $K_k(C'')$ contains at least $\pi r n$ edges of each element of $\fF$, as required.
\end{proof}

\subsection{Proof of the \texorpdfstring{\nameref{lem:lemma-for-G-cycle}}{Lg}}\label{sec:lemma-for-G-cycle}
In the following, we show \cref{lem:lemma-for-G-cycle}.

\begin{proof}[Proof of~\cref{lem:lemma-for-G-cycle}]
	We establish the following constant hierarchy for the remainder of the proof{.}
	\begin{align}\label{ref:constant-hierarchy}
		1/n \ll 1/t_1 \ll 1/t_0 \ll  \eps \ll \eps' \ll \eps''\ll d'' \ll d' \ll d \ll \eta  \ll \pi \ll \gamma \ll \mu, 1/k. \tag{$\ll_G$}
	\end{align}
	In particular, we choose $\gamma,\pi,\eta,d,\eps,t_0,t_1,n$ so that we can apply \cref{prop:distributed-matching} (with $rt$ and $\gamma$ playing the roles of $n$ and $\gamma$),  \cref{prop:super-regularizing-R'} (with $t_0 \leq t \leq t_1$ and $\Delta_{R'}=k^3$) and \cref{lem:regular-partition+absorber}.
	We also choose these constants such that the various inequalities along the proof are satisfied.
	Now let $\fU$ be an $(n,r)$-sized partition, and let $(G,\cH)$ be a $\mu$-robust $\fU$-partite $k$-uniform Hamilton framework.
	If $r = 1$, we assume that $(G, H)$ is $\mu$-robust aperiodic.
	\medskip

	\noindent \emph{Step 1: Establishing a regular partition.}
	We begin by applying \rf{lem:regular-partition+absorber} to $(G,\cH)$ and $\fU$.
	Hence, there are $t_0 \leq t \leq t_1$, a subgraph $G' \subset G$ with a balanced $\fU$-refining partition $\fV=\{V_i\}_{i=1}^{rt}$ with $t$ clusters  in each part of $\fU$ each of size $m$,
	a graph $R$ on $\fV$, such that
	\begin{enumerate}[(R1)]
		\item \label{itm:lemma-for-G-cycle-G'} $G'$ is an $(\eps,4d)$-approximation of $G$ and
		\item $(G',\fV)$ is a $(\eps,d)$-regular $R$-partition.
	\end{enumerate}
	Moreover, if we let $\cH' = \cH \cap K_k(G')$,
	we let $\cJ$ be the $k$-graph induced by $(\cH',\fV)$,
	and we let $\fW$ be the $(t,r)$-sized partition induced by $(\fV,\fU)$, then we also have that
	\begin{enumerate}[(R1),resume]
		\item \label{itm:lemma-for-G-cycle-RJ} $(R,\cJ)$ is a $(\mu/4)$-robust $\fW$-partite Hamilton framework,
		\item $(R,\cJ)$ is $(\mu/4,2/n,\fW)$-cluster-matchable, and
		\item if $r = 1$, $(R,\cJ)$ is a $(\mu/4)$-robust aperiodic $k$-uniform  Hamilton framework.
	\end{enumerate}
	Moreover, there is a collection ${\fF}$ of subgraphs of $\cJ$ such that
	\begin{enumerate}[(R1),resume]
		\item \label{itm:lemma-for-G-cycle-W}
		      \begin{enumerate}[(i)]
			      \item ${\fF}$ has at most {$2^{\eta t}$} elements,
			      \item each subgraph in $\fF$ has at least $\gamma (rt)^k$ edges,
			      \item for every vertex $v \in V({G})$, there is an $\cF \in {\fF}$ such that $v$ has at least $\gamma (n/t)^k$ linked edges {in $\cH'$} in the clusters of $\fV$ corresponding to every $K \in E({\cF})$.
		      \end{enumerate}
	\end{enumerate}

	We then apply \rf{prop:distributed-matching} to $\cJ$ and $\fF$.
	By the constant hierarchy~\eqref{ref:constant-hierarchy}, we obtain
	\begin{enumerate}[\upshape (R1$'$)]\addtocounter{enumi}{3}
		\item \label{itm:proof-lfg-cycle-absorbing-matching} a matching $\cM\subset \cJ$ of size at most $\sqrt{\pi}rt \leq {(\mu/(8k))} t$ such that $\cM$ has at least $\pi rt$ edges of every subgraph ${\cF} \in {\fF}$.
	\end{enumerate}

	\medskip
	\noindent \emph{Step 2: Finding $R'$.}
	We start by selecting a subhypergraph $\cQ \subset \cJ-V(\cM)$ of bounded maximum degree.
	Recall that $(R,\cJ)$ is a $(\mu/4)$-robust {$\fW$-partite} $k$-uniform  Hamilton framework by \cref{itm:regular-partition+absorber-R-framework}.
	Since $\cM$ has size at most $\mu t/(8k)$, we deduce $|V(\cM)| \leq \mu t / 8$.
	Therefore, we get that $(R-V(\cM), \cJ - V(\cM))$ is a $(\mu/8)$-robust {$\fW'$-partite} Hamilton framework, {where $\fW'$ is obtained by deleting $V(\cM)$ from the parts of $\fW$.}
	In particular, the $k$-graph $\cJ - V(\cM)$ must have a perfect fractional matching.
	By \rf{lem:kcover}, $\cJ-V(\cM)$ contains a spanning subhypergraph $\cQ \subset \cJ-V(\cM)$ such that each vertex of $\cQ$ is in at most $k^2+1$ edges.

	Let $\cJ' = \cM \cup \cQ$ and $R' = \partial_2 \cJ'$.
	Note that $\Delta(R') \leq (k^2+1)(k-1) \leq k^3$.

	\medskip
	\noindent \emph{Step 3: Super-regularising $R'$.}
	Next, we find a family of subsets in the partition $\fV$, which has suitable super-regular properties.
	We apply \rf{prop:super-regularizing-R'} with $\eps$, $d$ and $R'$, $G'$, $k^3$ playing the roles of $R'$, $G$, $\Delta_{R'}$ respectively.
	This gives a balanced family $\fV'=\{V_i'\}_{i=1}^{{rt}}$ of subsets $V_i' \subset V_i$ of size $m' =\lceil (1 - {\sqrt{\eps}})m \rceil$ such that $(G'[\cV'],\fV')$ is a $({2\eps,d})$-regular $R$-partition, which is $({2\eps,d})$-super-regular on~$R'$.
	Let $V_0 = V(G) \sm \bigcup_{i=1}^{{rt}}   V'_i$ denote the set of `exceptional vertices'.
	Note that $|V_0| \leq \eps r n + rt \lceil {\sqrt{\eps}} m \rceil \leq 3 {\sqrt{\eps}} {rn}$, where the first inequality follows from property~\ref{itm:lemma-for-G-cycle-G'} and the choice of $\fV'$.

	In passing from $\fV$ to $\fV'$, no more than $\lceil {\sqrt{\eps}} m \rceil \leq 2 {\sqrt{\eps}} m$ vertices were removed in each cluster,
	and each of these removals can affect at most $m^{k-1}$ of the linked edges of a vertex which are located in an edge $K \in E(\cM)$.
	These observations, together with properties~\ref{itm:lemma-for-G-cycle-W}, \ref{itm:proof-lfg-cycle-absorbing-matching} and $\eps \ll \gamma$, imply that
	\begin{enumerate}[\upshape (R1$'$)]\addtocounter{enumi}{4}
		\item \label{itm:proof-lemmaforGcycle-secondneighboursfinal}
		      for every $v \in V(G)$,
		      there are at least $\pi r t$ edges $K \in E(\cM)$,
		      such that $v$ has at least $(\gamma/2) (m')^k$ linked edges in $\cH$ in the clusters of $\fV'$ corresponding to $K$.
	\end{enumerate}

	\medskip
	\noindent \emph{Step 4: The return of the exceptional vertices.}
	To finish, we have to add the exceptional vertices $V_0$ back to the partition $\fV'$.
	We will use the linked edges for this, for which we introduce some terminology.
	Recall that a linked edge of $v$ corresponds to an edge $e \in \cH$ such that there exists $f \in {E(\cH)}$ with $v \in f$ and $|e \cap f| = k-1$.
	Say the unique vertex in $e \setminus f$ is the \emph{representative} of $v$ in $e$.
	Fix $v \in V_0$, and some $K \in E(\cM)$ such that $v$ has at least $(\gamma/2) (m')^k$ linked edges which are located in the $k$ clusters of $\fV'$ corresponding to $K$.
	Since $v$ has at least $(\gamma/2) (m')^k$ linked edges in the clusters of $\fV'$ corresponding to $K$,
	by averaging there exists $i_v \in V(K)$ such that $v$ has at least $(\gamma/2k) (m')^k$ linked edges whose representative lies in the cluster $V'_{i_v}$.
	Say such $i_v$ is a \emph{representative index} of $v$.
	In particular, note that this implies $v$ has at least $(\gamma/2k) m' $ neighbours in each of the clusters corresponding to the vertices of $K$ which are not $V'_{i_v}$.

	The previous argument and property~\ref{itm:proof-lemmaforGcycle-secondneighboursfinal} implies that for each $v \in V_0$ there are at least $\pi r t$ representative indices.
	Hence, we can assign, for each vertex $v \in V_0$, a representative index $i_v$ in such a way that no index is assigned to more than $|V_0|/(\pi r t) \leq (3{\sqrt{\eps}} {rn}) / (\pi r t) \leq \eps'' {n/t}$ vertices of $V_0$, where the inequality follows from the constant hierarchy~\eqref{ref:constant-hierarchy}.

	For each $1 \leq i \leq {rt}$, let $V'''_i$ be obtained from $V''_i$ by adding the vertices $v$ from $V_0$ such that $i = i_v$, and let $\fV''' = \{ V'''_i \colon 1 \leq i \leq {rt}\}$.
	Note that $\fV'''$ partitions $V(G)$.
	For each $v \in V_0$, write $K_v$ for the edge $K \in E(\cM)$ which contains the representative index $i_v$.
	Let $G''$ be the graph obtained from $G$ by deleting, for every vertex $v \in V_0$, all edges emanating from $v$ which are not in the clusters corresponding to $K_v \setminus \{i_v\}$.
	As argued before,
	$v$ has at least $(\gamma/2k) m'$ neighbours in each of the clusters corresponding to the vertices of $K_v \setminus \{i_v\}$.
	It follows by the constant hierarchy~\eqref{ref:constant-hierarchy}, the choice of $R'$ and \cref{proposition:robust-regular} that $(G'',\fV'')$ is a $(1+\eps'')$-balanced $(\eps'',d'')$-regular $R$-partition, which is $(\eps'',d'')$-super-regular on~$R'$.
	(The super-regularity comes from the fact that each vertex $v \in V_0$ has an edge $K_v$ in $\cM \subset \cJ'$ and $R' = \partial_2 \cJ'$.)
	Moreover, $e(G'') \geq (1-\mu/2)e(G)$ since $|V_0| \leq {3\sqrt{\eps} r n}$, property~\ref{itm:lemma-for-G-cycle-G'}, $\eps,d \ll \mu,1/k$ and by construction of $G'[\cV']$.
	Hence, we may finish the proof with $G''$, $\fV''$, $\eps''$, $d''$ playing the roles of $G'$, $\fV$, $\eps$, $d$.
\end{proof}

\subsection{Proof of the \texorpdfstring{\nameref{lem:lemma-for-C}}{Lg}.}\label{sec:lemma-for-C}
For a $k$-graph $\cJ$, we define the \emph{incidence matrix} $A_{\cJ} \in \NATS^{V(\cJ)\times E(\cJ)}$ by setting for all $v \in V(\cJ)$ and $e \in E(\cJ)$,
\begin{align*}
	A(v,e) =
	\begin{cases}
		1 & \text{if $v \in e$}\\
		0 & \text{otherwise}.
	\end{cases}
\end{align*}

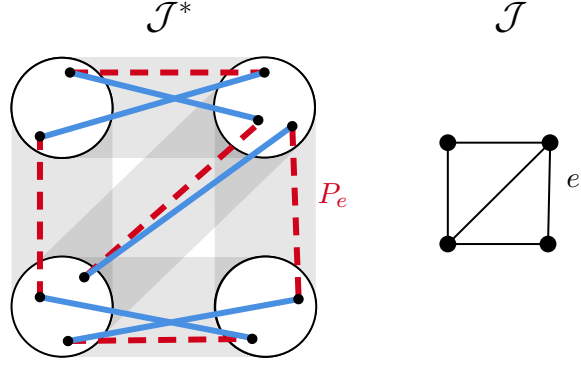
\begin{figure}

	\tikzset{every picture/.style={line width=0.75pt}} 

	\tikzset{every picture/.style={line width=0.75pt}} 
	
	\begin{tikzpicture}[x=0.75pt,y=0.75pt,yscale=-1,xscale=1]
		
		\draw  [fill={rgb, 255:red, 255; green, 255; blue, 255 }  ,fill opacity=1 ] (51.6,95.2) .. controls (51.6,81.28) and (62.88,70) .. (76.8,70) .. controls (90.72,70) and (102,81.28) .. (102,95.2) .. controls (102,109.12) and (90.72,120.4) .. (76.8,120.4) .. controls (62.88,120.4) and (51.6,109.12) .. (51.6,95.2) -- cycle ;
		\draw   (153,95) .. controls (153,81.19) and (164.19,70) .. (178,70) .. controls (191.81,70) and (203,81.19) .. (203,95) .. controls (203,108.81) and (191.81,120) .. (178,120) .. controls (164.19,120) and (153,108.81) .. (153,95) -- cycle ;
		\draw  [fill={rgb, 255:red, 255; green, 255; blue, 255 }  ,fill opacity=1 ] (52.03,195.18) .. controls (52.03,181.38) and (63.23,170.18) .. (77.03,170.18) .. controls (90.84,170.18) and (102.03,181.38) .. (102.03,195.18) .. controls (102.03,208.99) and (90.84,220.18) .. (77.03,220.18) .. controls (63.23,220.18) and (52.03,208.99) .. (52.03,195.18) -- cycle ;
		\draw  [fill={rgb, 255:red, 255; green, 255; blue, 255 }  ,fill opacity=1 ] (153.02,195.19) .. controls (153.02,181.27) and (164.3,169.98) .. (178.23,169.98) .. controls (192.15,169.98) and (203.44,181.27) .. (203.44,195.19) .. controls (203.44,209.12) and (192.15,220.4) .. (178.23,220.4) .. controls (164.3,220.4) and (153.02,209.12) .. (153.02,195.19) -- cycle ;
		\draw  [draw opacity=0][fill={rgb, 255:red, 0; green, 0; blue, 0 }  ,fill opacity=0.1 ] (78,70) -- (178,70) -- (178,120.4) -- (78,120.4) -- cycle ;
		\draw  [draw opacity=0][fill={rgb, 255:red, 0; green, 0; blue, 0 }  ,fill opacity=0.1 ] (77.03,170.4) -- (178.44,170.4) -- (178.44,220.4) -- (77.03,220.4) -- cycle ;
		\draw  [draw opacity=0][fill={rgb, 255:red, 0; green, 0; blue, 0 }  ,fill opacity=0.1 ] (153.03,195) -- (153,95) -- (203.4,94.98) -- (203.44,194.98) -- cycle ;
		\draw  [draw opacity=0][fill={rgb, 255:red, 0; green, 0; blue, 0 }  ,fill opacity=0.1 ] (51.63,195.2) -- (51.6,95.2) -- (102,95.18) -- (102.03,195.18) -- cycle ;
		\draw  [draw opacity=0][fill={rgb, 255:red, 0; green, 0; blue, 0 }  ,fill opacity=0.1 ] (59.11,177.8) -- (160.96,78.52) -- (195.51,113.96) -- (93.66,213.24) -- cycle ;
		\draw  [fill={rgb, 255:red, 255; green, 255; blue, 255 }  ,fill opacity=1 ] (51.6,95.2) .. controls (51.6,81.28) and (62.88,70) .. (76.8,70) .. controls (90.72,70) and (102,81.28) .. (102,95.2) .. controls (102,109.12) and (90.72,120.4) .. (76.8,120.4) .. controls (62.88,120.4) and (51.6,109.12) .. (51.6,95.2) -- cycle ;
		\draw  [fill={rgb, 255:red, 255; green, 255; blue, 255 }  ,fill opacity=1 ] (152.6,95) .. controls (152.6,81.08) and (163.88,69.8) .. (177.8,69.8) .. controls (191.72,69.8) and (203,81.08) .. (203,95) .. controls (203,108.92) and (191.72,120.2) .. (177.8,120.2) .. controls (163.88,120.2) and (152.6,108.92) .. (152.6,95) -- cycle ;
		\draw  [fill={rgb, 255:red, 255; green, 255; blue, 255 }  ,fill opacity=1 ] (51.63,195.18) .. controls (51.63,181.27) and (62.91,169.98) .. (76.83,169.98) .. controls (90.75,169.98) and (102.03,181.27) .. (102.03,195.18) .. controls (102.03,209.1) and (90.75,220.39) .. (76.83,220.39) .. controls (62.91,220.39) and (51.63,209.1) .. (51.63,195.18) -- cycle ;
		\draw  [fill={rgb, 255:red, 255; green, 255; blue, 255 }  ,fill opacity=1 ] (153.24,195.2) .. controls (153.24,181.28) and (164.52,170) .. (178.44,170) .. controls (192.36,170) and (203.64,181.28) .. (203.64,195.2) .. controls (203.64,209.12) and (192.36,220.4) .. (178.44,220.4) .. controls (164.52,220.4) and (153.24,209.12) .. (153.24,195.2) -- cycle ;
		\draw [color={rgb, 255:red, 208; green, 2; blue, 27 }  ,draw opacity=1 ][line width=2.25]  [dash pattern={on 6.75pt off 4.5pt}]  (80.8,77.4) -- (177.8,77.4) ;
		\draw [color={rgb, 255:red, 208; green, 2; blue, 27 }  ,draw opacity=1 ][line width=2.25]  [dash pattern={on 6.75pt off 4.5pt}]  (79.8,212.51) -- (173.8,211.4) ;
		\draw [color={rgb, 255:red, 208; green, 2; blue, 27 }  ,draw opacity=1 ][line width=2.25]  [dash pattern={on 6.75pt off 4.5pt}]  (194.8,189.3) -- (191.69,104.4) ;
		\draw [color={rgb, 255:red, 208; green, 2; blue, 27 }  ,draw opacity=1 ][line width=2.25]  [dash pattern={on 6.75pt off 4.5pt}]  (65.8,107.4) -- (65.8,190.4) ;
		\draw [color={rgb, 255:red, 208; green, 2; blue, 27 }  ,draw opacity=1 ][line width=2.25]  [dash pattern={on 6.75pt off 4.5pt}]  (87.9,180.4) -- (176.8,101.4) ;
		\draw [color={rgb, 255:red, 74; green, 144; blue, 226 }  ,draw opacity=1 ][line width=2.25]    (80.8,77.4) -- (176.8,101.4) ;
		\draw [color={rgb, 255:red, 74; green, 144; blue, 226 }  ,draw opacity=1 ][line width=2.25]    (87.9,180.4) -- (191.69,104.4) ;
		\draw [color={rgb, 255:red, 74; green, 144; blue, 226 }  ,draw opacity=1 ][line width=2.25]    (79.8,212.51) -- (194.8,191.4) ;
		\draw [color={rgb, 255:red, 74; green, 144; blue, 226 }  ,draw opacity=1 ][line width=2.25]    (65.8,190.4) -- (173.8,211.4) ;
		\draw [color={rgb, 255:red, 74; green, 144; blue, 226 }  ,draw opacity=1 ][line width=2.25]    (67.9,109.51) -- (177.8,77.4) ;
		\draw  [fill={rgb, 255:red, 0; green, 0; blue, 0 }  ,fill opacity=1 ] (78.69,77.4) .. controls (78.69,76.24) and (79.64,75.3) .. (80.8,75.3) .. controls (81.96,75.3) and (82.9,76.24) .. (82.9,77.4) .. controls (82.9,78.56) and (81.96,79.51) .. (80.8,79.51) .. controls (79.64,79.51) and (78.69,78.56) .. (78.69,77.4) -- cycle ;
		\draw  [fill={rgb, 255:red, 0; green, 0; blue, 0 }  ,fill opacity=1 ] (175.69,77.4) .. controls (175.69,76.24) and (176.64,75.3) .. (177.8,75.3) .. controls (178.96,75.3) and (179.9,76.24) .. (179.9,77.4) .. controls (179.9,78.56) and (178.96,79.51) .. (177.8,79.51) .. controls (176.64,79.51) and (175.69,78.56) .. (175.69,77.4) -- cycle ;
		\draw  [fill={rgb, 255:red, 0; green, 0; blue, 0 }  ,fill opacity=1 ] (192.69,191.4) .. controls (192.69,190.24) and (193.64,189.3) .. (194.8,189.3) .. controls (195.96,189.3) and (196.9,190.24) .. (196.9,191.4) .. controls (196.9,192.56) and (195.96,193.51) .. (194.8,193.51) .. controls (193.64,193.51) and (192.69,192.56) .. (192.69,191.4) -- cycle ;
		\draw  [fill={rgb, 255:red, 0; green, 0; blue, 0 }  ,fill opacity=1 ] (77.69,212.51) .. controls (77.69,211.35) and (78.64,210.4) .. (79.8,210.4) .. controls (80.96,210.4) and (81.9,211.35) .. (81.9,212.51) .. controls (81.9,213.67) and (80.96,214.61) .. (79.8,214.61) .. controls (78.64,214.61) and (77.69,213.67) .. (77.69,212.51) -- cycle ;
		\draw  [fill={rgb, 255:red, 0; green, 0; blue, 0 }  ,fill opacity=1 ] (189.59,104.4) .. controls (189.59,103.24) and (190.53,102.3) .. (191.69,102.3) .. controls (192.86,102.3) and (193.8,103.24) .. (193.8,104.4) .. controls (193.8,105.56) and (192.86,106.51) .. (191.69,106.51) .. controls (190.53,106.51) and (189.59,105.56) .. (189.59,104.4) -- cycle ;
		\draw  [fill={rgb, 255:red, 0; green, 0; blue, 0 }  ,fill opacity=1 ] (169.59,211.4) .. controls (169.59,210.24) and (170.53,209.3) .. (171.69,209.3) .. controls (172.86,209.3) and (173.8,210.24) .. (173.8,211.4) .. controls (173.8,212.56) and (172.86,213.51) .. (171.69,213.51) .. controls (170.53,213.51) and (169.59,212.56) .. (169.59,211.4) -- cycle ;
		\draw  [fill={rgb, 255:red, 0; green, 0; blue, 0 }  ,fill opacity=1 ] (85.8,180.4) .. controls (85.8,179.24) and (86.74,178.3) .. (87.9,178.3) .. controls (89.06,178.3) and (90,179.24) .. (90,180.4) .. controls (90,181.56) and (89.06,182.51) .. (87.9,182.51) .. controls (86.74,182.51) and (85.8,181.56) .. (85.8,180.4) -- cycle ;
		\draw  [fill={rgb, 255:red, 0; green, 0; blue, 0 }  ,fill opacity=1 ] (172.59,101.4) .. controls (172.59,100.24) and (173.53,99.3) .. (174.69,99.3) .. controls (175.86,99.3) and (176.8,100.24) .. (176.8,101.4) .. controls (176.8,102.56) and (175.86,103.51) .. (174.69,103.51) .. controls (173.53,103.51) and (172.59,102.56) .. (172.59,101.4) -- cycle ;
		\draw  [fill={rgb, 255:red, 0; green, 0; blue, 0 }  ,fill opacity=1 ] (63.69,109.51) .. controls (63.69,108.35) and (64.64,107.4) .. (65.8,107.4) .. controls (66.96,107.4) and (67.9,108.35) .. (67.9,109.51) .. controls (67.9,110.67) and (66.96,111.61) .. (65.8,111.61) .. controls (64.64,111.61) and (63.69,110.67) .. (63.69,109.51) -- cycle ;
		\draw  [fill={rgb, 255:red, 0; green, 0; blue, 0 }  ,fill opacity=1 ] (63.69,190.4) .. controls (63.69,189.24) and (64.64,188.3) .. (65.8,188.3) .. controls (66.96,188.3) and (67.9,189.24) .. (67.9,190.4) .. controls (67.9,191.56) and (66.96,192.51) .. (65.8,192.51) .. controls (64.64,192.51) and (63.69,191.56) .. (63.69,190.4) -- cycle ;
		\draw  [fill={rgb, 255:red, 0; green, 0; blue, 0 }  ,fill opacity=1 ] (265.6,112.7) .. controls (265.6,110.66) and (267.25,109) .. (269.3,109) .. controls (271.34,109) and (273,110.66) .. (273,112.7) .. controls (273,114.75) and (271.34,116.4) .. (269.3,116.4) .. controls (267.25,116.4) and (265.6,114.75) .. (265.6,112.7) -- cycle ;
		\draw  [fill={rgb, 255:red, 0; green, 0; blue, 0 }  ,fill opacity=1 ] (316.6,112.7) .. controls (316.6,110.66) and (318.25,109) .. (320.3,109) .. controls (322.34,109) and (324,110.66) .. (324,112.7) .. controls (324,114.75) and (322.34,116.4) .. (320.3,116.4) .. controls (318.25,116.4) and (316.6,114.75) .. (316.6,112.7) -- cycle ;
		\draw  [fill={rgb, 255:red, 0; green, 0; blue, 0 }  ,fill opacity=1 ] (265.6,163.7) .. controls (265.6,161.66) and (267.25,160) .. (269.3,160) .. controls (271.34,160) and (273,161.66) .. (273,163.7) .. controls (273,165.75) and (271.34,167.4) .. (269.3,167.4) .. controls (267.25,167.4) and (265.6,165.75) .. (265.6,163.7) -- cycle ;
		\draw  [fill={rgb, 255:red, 0; green, 0; blue, 0 }  ,fill opacity=1 ] (315.6,163.7) .. controls (315.6,161.66) and (317.25,160) .. (319.3,160) .. controls (321.34,160) and (323,161.66) .. (323,163.7) .. controls (323,165.75) and (321.34,167.4) .. (319.3,167.4) .. controls (317.25,167.4) and (315.6,165.75) .. (315.6,163.7) -- cycle ;
		\draw    (320.3,112.7) -- (269.3,163.7) ;
		\draw    (269.3,163.7) -- (269.3,112.7) ;
		\draw    (320.3,112.7) -- (269.3,112.7) ;
		\draw    (319.3,163.7) -- (269.3,163.7) ;
		\draw    (320.3,116.4) -- (319.3,163.7) ;
		
		\draw (291,40.4) node [anchor=north west][inner sep=0.75pt]  [font=\Large]  {$\cJ$};
		\draw (117,40.4) node [anchor=north west][inner sep=0.75pt]  [font=\Large]  {$\cJ^{\ast }$};
		\draw (203,131.4) node [anchor=north west][inner sep=0.75pt]  [color={rgb, 255:red, 208; green, 2; blue, 27 }  ,opacity=1 ]  {$P_{e}$};
		\draw (327,127.4) node [anchor=north west][inner sep=0.75pt]    {$e$};

	\end{tikzpicture}
	
	\caption{
		The short cycle $C_1$ (dashed and straight lines) and the paths $P_e$ (dashed) for $k=2$.
		Once obtained, the remainder of the argument consists in carefully extending the paths $P_e$.
	}
	\label{fig:lemma-for-c}
\end{figure}

\begin{proof}[Proof of \cref{lem:lemma-for-C}]
	We establish the following hierarchy.
	\begin{align}\label{ref:constant-hierarchy-lemma-for-C} \tag{$\ll_C$}
		1/n \ll   \rho \ll \pi  \ll 1/t  \ll \eps \ll \alpha   \ll \mu, 1/k.
	\end{align}
	Now, consider $\fU$, $G$, $\fV$, $R$, $\cJ$, $\cJ'$ as in the statement.
	It will be convenient to formulate the proof in terms of tight cycles instead of powers of cycles.
	Denote by $\cJ^\ast$ the $(\cJ,\fV)$-blow-up $k$-graph, where the $i$th vertex of $\cJ$ corresponds to a set of size $|V_i|$.
	Let $\fX$ be the corresponding, size-compatible with $\fV$, vertex partition of $V(\cJ^*)$.
	Recall that the $(k-1)$st power of a Hamilton cycle can be regarded as a tight cycle in the corresponding $k$-clique hypergraph.
	Hence the statement of the lemma follows if $\cJ^\ast$ has a tight Hamilton cycle, which satisfies (the hypergraph analogues of) properties~\ref{itm:allocate-vertices-partition} to~\ref{itm:allocate-vertices-disjoint}.

	\medskip
	\noindent \emph{Step 1: Sketching a cycle.}
	To begin, we set up a (comparatively) short tight cycle $C_1$ in $\cJ^\ast$, which `visits' every edge of $\cJ$ and leaves the number of uncovered vertices in $\cJ^\ast$ divisible by~$k$.
	The latter is trivial when $r=k$, because every tight cycle will satisfy that property;
	for $r=1$, we ensure this condition using the aperiodicity of $(R,\cJ)$.
	For an illustration, see \cref{fig:lemma-for-c}.
	
	\begin{claim}
		There is a tight cycle $C_1\subset \cJ^\ast$ which contains a tight path $P_e$ of length $k$ as subpath for every edge $e \in E(\cJ)$, whose vertices are in the $k$ clusters corresponding to $e$.\footnote{{To be clear, at this point $P_e$ is just a single edge whose vertices follow the ordering of $C_1$. In Step (3), we replace some of the paths $P_e$ with a longer paths.}}
		Moreover, $C_1$ has length at most $(rt)^{2k}$ and $|V(\cJ^\ast) \sm V(C_1)|$ is divisible by $k$.
	\end{claim}
	\begin{proof}
		Recall that $(R,\cJ)$ is a $k$-uniform  Hamilton framework.
		In particular, $\cJ$ is tightly connected.

		We first assume that $r=k$.
		By \rf{prop:closed-walk-bounded-length}, there exists a closed tight walk $W$ in $\cJ$ of length at most $k^2(rt)^k+k\binom{rt}{k} (rt)^{k} \leq (rt)^{2k}$.
		Since $\cJ$ is $k$-partite, $W$ must have length divisible by $k$.
		Since a	closed tight walk is a homomorphism of a tight cycle and $1/n\ll 1/t$,
		it follows immediately that we can lift the walk $W$ in $\cJ$ to a tight cycle $C_1$ in $\cJ^\ast$ of the same length.
		By construction, $C_1$ includes the required paths $P_e$.

		Now, consider the case $r = 1$.
		By assumption, we also have that $(R,\cJ)$ is an aperiodic Hamilton framework,
		and thus we get that $\cJ$ contains a closed tight walk of length coprime to $k$.
		Now, \cref{prop:closed-walk-bounded-length} yields a closed tight walk in $\cJ$, which visits every edge of $\cJ$, has length at most $k^2t^k+k\binom{t}{k} t^{k} \leq t^{2k}$, and in addition we can assume the length of the walk is congruent to $n$ modulo $k$.
		We then turn the walk in $\cJ$ into a tight cycle in $\cJ^*$ as before.
	\end{proof}
	Let us fix a cycle $C_1$ as in the claim and define $\vec w_1 \in \NATS^{E(\cJ)}$ to be the vector that counts, for every $e \in E(\cJ)$, the number of edges of $C_1$ that are contained in the clusters of $\fX$ corresponding to $e$.

	\medskip
	\noindent \emph{Step 2: Allocations.}
	It will be convenient to formulate the following allocations in terms of matrices.
	Let $A=A_{\cJ} \in \NATS^{V(\cJ)\times E(\cJ)}$ be the incidence matrix of $\cJ$.
	Let $\vec b \in \NATS^{V(\cJ)}$ count the vertices of each cluster of $\fX$, that is $\vec b(i) = |X_i|$.
	Note that $\fV$ is $(1+\eps)$-balanced, thus the cluster sizes are between $m$ and $(1+\eps)m$ for some integer $m$.
	Our goal is now to find an allocation $\vec w \in \NATS^{E(\cJ)}$ such that
	\begin{enumerate}[{\upshape (i)}]
		\item \label{item:allocation-first} $A  (\vec w_1 + \vec w)= \vec b$,
		\item $\vec w (e) \geq {4 \pi n} $ for every edge $e \in E(\cJ)$,
		\item \label{item:allocation-last} $\vec w (e) \geq {3 \alpha m}$ for every edge $e \in E(\cJ')$.
	\end{enumerate}

	{Note that part (i) states that the allocation $\vec w_1 + \vec w$ attains precisely the cluster sizes of $\fX$.
	Part (ii) tells us that every edge of $\cJ$ is allocated some edges of $\cJ^\ast$.
	This is necessary for finding the paths required for part~\ref{itm:allocate-vertices-paths}.
	Finally, part (iii) ensures that edges in $\cJ'$ are allocated many edges.
	We use this to obtain the buffer vertices of part~\ref{itm:allocate-vertices-buffer}.}

	We will find $\vec w$ as the sum of four vectors $\vec w_2, \dotsc, \vec w_5$.
	First, we note that $\maxnorm{A \vec w_1} \leq (rt)^{2k} \leq \alpha m$.
	{This holds, because $m$ can be crudely bounded from below by $n/(2t)$ and is hence much larger than $(rt)^{2k}/\alpha$ by choice of the constants.}
	We begin by choosing $\vec w_2 \in \NATS^{E(\cJ)}$ and $\vec w_3 \in \NATS^{E(\cJ)}$ as follows.
	We let $\vec w_2(e) = \lceil 3 \alpha m \rceil$ for every edge $e \in E(\cJ')$ and $0$ otherwise,
	and we let $\vec w_3(e) = \lceil 5 \pi n \rceil$ for every edge $e \in E(\cJ)$.
	Note that $\maxnorm{A \vec w_2} \leq (k^2+1) \lceil 3 \alpha m \rceil \leq 4 k^2 \alpha m$.
	{This is because $\Delta(\cJ') \leq k^2 + 1$,}
	and the last inequality follows from the hierarchy~\eqref{ref:constant-hierarchy-lemma-for-C}.
	On the other hand, every vertex is in at most $t^{k-1}$ edges of $\cJ$,
	thus $\maxnorm{A \vec w_3} \leq \lceil 5 \pi n \rceil {t^{k-1}} \leq \alpha m$,
	where again we used the hierarchy~\eqref{ref:constant-hierarchy-lemma-for-C}.

	Next, we allocate most of the remaining vertices.
	\begin{claim}
		There is $\vec w_4 \in \NATS^{E(\cJ)}$ such that  $ 0\leq \vec b -A  \left( \vec w_1 + \vec w_2 + \vec w_3 + \vec w_4 \right)\leq {\rho n}$.
	\end{claim}
	\begin{proofclaim}
		Let $\vec w' = \vec w_1 + \vec w_2 + \vec w_3 $ and note that $\maxnorm{ A \vec w '} \leq \maxnorm{ A \vec w_1} + \maxnorm{ A \vec w_2} + \maxnorm{ A \vec w_3} \leq (4k^2 + 2) \alpha m$.
		Let $\fX'$ be obtained from $\fX$, by deleting, for every $K \in E(\cJ)$, an amount of $\vec w' (K)$ vertices of every cluster corresponding to $K$.
		Note that, for every part $W \in \fW$, the number of elements in clusters of $\fX'$ with index in $W$ is the same.
		For $r=1$, this is trivial.
		For $r=k$, it follows because $\cJ$ is $\fW$-partite and because the property {(of having the same number of clusters in each part)} was true for~$\fV$, which is size-compatible with $\fX$.

		Since {$(4 k^2 + 2) \alpha \leq \mu/2$},
		we deduce $\fX'$ is $(1+\mu/2)$-balanced.
		Since ${(R,\cJ)}$ is $(\mu/2,2/n,{\fW})$-cluster-matchable, the $(\cJ,\fX')$-blow-up contains a matching $\cM$ that covers all but at most $2r + t^{k-1} \leq \rho n$ vertices in each cluster of $\fX'$, the inequality following from \eqref{ref:constant-hierarchy-lemma-for-C}.
		We can then define $\vec w_4 \in \NATS^{E(\cJ)}$ by counting, for each $e \in E(\cJ)$, the number of $K$-edges of $\cM$ for each $K \in E(\cJ)$.
		(Meaning those edges $e$, which have one vertex in each of the clusters of $K$.)
	\end{proofclaim}

	Let $\vec w_4$ be as in the claim.
	Finally, we complete the allocation by finding $\vec w_5$.

	\begin{claim}
		There is $\vec w_5 \in \INTS^{E(\cJ)}$ with $\maxnorm{\vec w_5} \leq \pi n$ such that $A  (\vec w_1 + \vec w_2 + \vec w_3 + \vec w_4 + \vec w_5)= \vec b$.
	\end{claim}

	\begin{proofclaim}
		Set $\vec c = \vec b- A  (\vec w_1 + \vec w_2 + \vec w_3 + \vec w_4)$.
		By the choice of $\vec w_4$,
		we have $\Lonenorm{\vec c} \leq rt \maxnorm{\vec c} \leq rt \rho n$.
		Further, $\Lonenorm{\vec c} \equiv 0 \bmod k$ by the choice of $\vec w_1$.
		Fix an arbitrary edge $e_0 \in E(\cJ)$ and define $\vec w_5' \in \NATS^{E(\cJ)}$ by setting $\vec w_5'(e_0)= \Lonenorm{\vec c} /k$ and $0$ everywhere else.
		Then let $\vec c' = \vec c - A \vec w_5'$.
		{Note that $\maxnorm{\vec c'} \leq \maxnorm{\vec c} + \Lonenorm{\vec c} \leq \rho n + rt \rho n \leq 2 rt \rho n$.}
		Note also that $\sum_{i \in {W}} \vec c'(v) = 0$ for each $W \in \fW$.
		Recall that $(R,\cJ)$ is a $\fW$-partite $k$-uniform  Hamilton framework.
		It follows that $\cJ$ has no isolated vertices (thanks to having a perfect fractional matching) and is tightly connected.
		Moreover, if $r=1$, then our assumption implies that $\cJ$ also contains a closed tight walk of length coprime to $k$.
		These are precisely the conditions required by \cref{prop:reachability} (applied here with $\fW$ and $\cJ$ in place of $\fU$ and $\cH$) to deduce that the $2$-graph $(\hat{\cJ})[W]$ is connected for every $W \in \fW$.
		Hence, by \cref{prop:flow} (applied here with $\fW$, $\cJ$, $\vec c'$, ${2}t \rho n$ in place of $\fU$, $\cH$, $\vec b$ and $s$), there exists $\vec w''_5 \in \INTS^{E(\cJ)}$ with $A \vec w''_5 = \vec c'$ and {$\maxnorm {\vec w'_5 + \vec w_5''}  \leq 2 k (rt)^3 \rho n + rt \rho n$.}
		By the constant hierarchy~\eqref{ref:constant-hierarchy-lemma-for-C}, we have {${2 k (rt)^3 \rho n} + rt \rho n \leq \pi n$} and the claim follows by setting $\vec w_5 = \vec w'_5 + \vec w''_5$.
	\end{proofclaim}

	By the choice of $\vec w_3$ and $\maxnorm{\vec w_5} \leq \pi n$,
	we have that $\vec w_3 + \vec w_5 \geq 5 \pi n - \pi n \geq 4 \pi n \geq 0$.
	All together, it follows that $\vec w = \vec w_2 + \vec w_3 + \vec w_4 + \vec w_5$
	is such that $\vec w \in \NATS^{E(\cJ)}$ has properties~\ref{item:allocation-first} to~\ref{item:allocation-last}, as desired.

	\medskip
	\noindent \emph{Step 3: Defining $C$.}
	Given the vector $\vec w$, we can easily setup the desired tight Hamilton cycle $\cC$ in ${\cJ^\ast}$.
	We simply replace each $P_e$ with a tight path of length $k(\vec w (e) + 1)$ as illustrated in \cref{fig:replace-paths}, which is possible since ${\cJ^\ast}$ is a $\fV$-blow-up of $\cJ$.
	By the choice of $\vec w$, for every $e \in E(\cJ)$ we have $\vec w(e) \geq 4 \pi n$, and
	for every $e \in E(\cJ')$ we have $\vec w(e) \geq 3 \alpha m$.
	This allows us to find vertex-disjoint subpaths $\{P_e\}_{e \in E(\cJ)}$ of the required lengths {(of at least $\pi n$ and at least $\alpha m$ for edges $e$ in $\cJ'$)}
	in the subgraph of ${\cJ^\ast}$ induced by the clusters of every edge $e \in E(\cJ)$.
	
	\begin{figure}

		\tikzset{every picture/.style={line width=0.75pt}} 

		\tikzset{every picture/.style={line width=0.75pt}} 
		
		\begin{tikzpicture}[x=0.75pt,y=0.75pt,yscale=-1,xscale=1]
			
			\draw   (173,115) .. controls (173,101.19) and (184.19,90) .. (198,90) .. controls (211.81,90) and (223,101.19) .. (223,115) .. controls (223,128.81) and (211.81,140) .. (198,140) .. controls (184.19,140) and (173,128.81) .. (173,115) -- cycle ;
			\draw  [fill={rgb, 255:red, 255; green, 255; blue, 255 }  ,fill opacity=1 ] (173.02,215.19) .. controls (173.02,201.27) and (184.3,189.98) .. (198.23,189.98) .. controls (212.15,189.98) and (223.44,201.27) .. (223.44,215.19) .. controls (223.44,229.12) and (212.15,240.4) .. (198.23,240.4) .. controls (184.3,240.4) and (173.02,229.12) .. (173.02,215.19) -- cycle ;
			\draw  [draw opacity=0][fill={rgb, 255:red, 0; green, 0; blue, 0 }  ,fill opacity=0.1 ] (173.03,215) -- (173,115) -- (223.4,114.98) -- (223.44,214.98) -- cycle ;
			\draw  [fill={rgb, 255:red, 255; green, 255; blue, 255 }  ,fill opacity=1 ] (172.6,115) .. controls (172.6,101.08) and (183.88,89.8) .. (197.8,89.8) .. controls (211.72,89.8) and (223,101.08) .. (223,115) .. controls (223,128.92) and (211.72,140.2) .. (197.8,140.2) .. controls (183.88,140.2) and (172.6,128.92) .. (172.6,115) -- cycle ;
			\draw  [fill={rgb, 255:red, 255; green, 255; blue, 255 }  ,fill opacity=1 ] (173.24,215.2) .. controls (173.24,201.28) and (184.52,190) .. (198.44,190) .. controls (212.36,190) and (223.64,201.28) .. (223.64,215.2) .. controls (223.64,229.12) and (212.36,240.4) .. (198.44,240.4) .. controls (184.52,240.4) and (173.24,229.12) .. (173.24,215.2) -- cycle ;
			\draw [color={rgb, 255:red, 208; green, 2; blue, 27 }  ,draw opacity=1 ][line width=2.25]  [dash pattern={on 6.75pt off 4.5pt}]  (214.8,209.3) -- (211.69,124.4) ;
			\draw  [fill={rgb, 255:red, 0; green, 0; blue, 0 }  ,fill opacity=1 ] (212.69,211.4) .. controls (212.69,210.24) and (213.64,209.3) .. (214.8,209.3) .. controls (215.96,209.3) and (216.9,210.24) .. (216.9,211.4) .. controls (216.9,212.56) and (215.96,213.51) .. (214.8,213.51) .. controls (213.64,213.51) and (212.69,212.56) .. (212.69,211.4) -- cycle ;
			\draw  [fill={rgb, 255:red, 0; green, 0; blue, 0 }  ,fill opacity=1 ] (209.59,124.4) .. controls (209.59,123.24) and (210.53,122.3) .. (211.69,122.3) .. controls (212.86,122.3) and (213.8,123.24) .. (213.8,124.4) .. controls (213.8,125.56) and (212.86,126.51) .. (211.69,126.51) .. controls (210.53,126.51) and (209.59,125.56) .. (209.59,124.4) -- cycle ;
			\draw  [fill={rgb, 255:red, 0; green, 0; blue, 0 }  ,fill opacity=1 ] (185.59,126.4) .. controls (185.59,125.24) and (186.53,124.3) .. (187.69,124.3) .. controls (188.86,124.3) and (189.8,125.24) .. (189.8,126.4) .. controls (189.8,127.56) and (188.86,128.51) .. (187.69,128.51) .. controls (186.53,128.51) and (185.59,127.56) .. (185.59,126.4) -- cycle ;
			\draw  [fill={rgb, 255:red, 0; green, 0; blue, 0 }  ,fill opacity=1 ] (184.59,208.4) .. controls (184.59,207.24) and (185.53,206.3) .. (186.69,206.3) .. controls (187.86,206.3) and (188.8,207.24) .. (188.8,208.4) .. controls (188.8,209.56) and (187.86,210.51) .. (186.69,210.51) .. controls (185.53,210.51) and (184.59,209.56) .. (184.59,208.4) -- cycle ;
			\draw   (300,115) .. controls (300,101.19) and (311.19,90) .. (325,90) .. controls (338.81,90) and (350,101.19) .. (350,115) .. controls (350,128.81) and (338.81,140) .. (325,140) .. controls (311.19,140) and (300,128.81) .. (300,115) -- cycle ;
			\draw  [fill={rgb, 255:red, 255; green, 255; blue, 255 }  ,fill opacity=1 ] (300.02,215.19) .. controls (300.02,201.27) and (311.3,189.98) .. (325.23,189.98) .. controls (339.15,189.98) and (350.44,201.27) .. (350.44,215.19) .. controls (350.44,229.12) and (339.15,240.4) .. (325.23,240.4) .. controls (311.3,240.4) and (300.02,229.12) .. (300.02,215.19) -- cycle ;
			\draw  [draw opacity=0][fill={rgb, 255:red, 0; green, 0; blue, 0 }  ,fill opacity=0.1 ] (300.03,215) -- (300,115) -- (350.4,114.98) -- (350.44,214.98) -- cycle ;
			\draw  [fill={rgb, 255:red, 255; green, 255; blue, 255 }  ,fill opacity=1 ] (299.6,115) .. controls (299.6,101.08) and (310.88,89.8) .. (324.8,89.8) .. controls (338.72,89.8) and (350,101.08) .. (350,115) .. controls (350,128.92) and (338.72,140.2) .. (324.8,140.2) .. controls (310.88,140.2) and (299.6,128.92) .. (299.6,115) -- cycle ;
			\draw  [fill={rgb, 255:red, 255; green, 255; blue, 255 }  ,fill opacity=1 ] (300.24,215.2) .. controls (300.24,201.28) and (311.52,190) .. (325.44,190) .. controls (339.36,190) and (350.64,201.28) .. (350.64,215.2) .. controls (350.64,229.12) and (339.36,240.4) .. (325.44,240.4) .. controls (311.52,240.4) and (300.24,229.12) .. (300.24,215.2) -- cycle ;
			\draw  [fill={rgb, 255:red, 0; green, 0; blue, 0 }  ,fill opacity=1 ] (339.69,211.4) .. controls (339.69,210.24) and (340.64,209.3) .. (341.8,209.3) .. controls (342.96,209.3) and (343.9,210.24) .. (343.9,211.4) .. controls (343.9,212.56) and (342.96,213.51) .. (341.8,213.51) .. controls (340.64,213.51) and (339.69,212.56) .. (339.69,211.4) -- cycle ;
			\draw  [fill={rgb, 255:red, 0; green, 0; blue, 0 }  ,fill opacity=1 ] (336.59,124.4) .. controls (336.59,123.24) and (337.53,122.3) .. (338.69,122.3) .. controls (339.86,122.3) and (340.8,123.24) .. (340.8,124.4) .. controls (340.8,125.56) and (339.86,126.51) .. (338.69,126.51) .. controls (337.53,126.51) and (336.59,125.56) .. (336.59,124.4) -- cycle ;
			\draw  [fill={rgb, 255:red, 0; green, 0; blue, 0 }  ,fill opacity=1 ] (312.59,126.4) .. controls (312.59,125.24) and (313.53,124.3) .. (314.69,124.3) .. controls (315.86,124.3) and (316.8,125.24) .. (316.8,126.4) .. controls (316.8,127.56) and (315.86,128.51) .. (314.69,128.51) .. controls (313.53,128.51) and (312.59,127.56) .. (312.59,126.4) -- cycle ;
			\draw  [fill={rgb, 255:red, 0; green, 0; blue, 0 }  ,fill opacity=1 ] (311.59,208.4) .. controls (311.59,207.24) and (312.53,206.3) .. (313.69,206.3) .. controls (314.86,206.3) and (315.8,207.24) .. (315.8,208.4) .. controls (315.8,209.56) and (314.86,210.51) .. (313.69,210.51) .. controls (312.53,210.51) and (311.59,209.56) .. (311.59,208.4) -- cycle ;
			\draw [color={rgb, 255:red, 245; green, 166; blue, 35 }  ,draw opacity=1 ][line width=2.25]  [dash pattern={on 2.53pt off 3.02pt}]  (313.69,208.4) -- (314.69,128.51) ;
			\draw [color={rgb, 255:red, 245; green, 166; blue, 35 }  ,draw opacity=1 ][line width=2.25]  [dash pattern={on 2.53pt off 3.02pt}]  (313.69,208.4) -- (338.69,124.4) ;
			\draw [color={rgb, 255:red, 245; green, 166; blue, 35 }  ,draw opacity=1 ][line width=2.25]  [dash pattern={on 2.53pt off 3.02pt}]  (341.8,211.4) -- (314.69,128.51) ;
			\draw   (256,153.85) -- (274.48,153.85) -- (274.48,150) -- (286.8,157.7) -- (274.48,165.4) -- (274.48,161.55) -- (256,161.55) -- cycle ;
			
			\draw (223,151.4) node [anchor=north west][inner sep=0.75pt]  [color={rgb, 255:red, 208; green, 2; blue, 27 }  ,opacity=1 ]  {$P_{e}$};
			\draw (357,155.4) node [anchor=north west][inner sep=0.75pt]  [color={rgb, 255:red, 245; green, 166; blue, 35 }  ,opacity=1 ]  {$P_{e}$};

		\end{tikzpicture}
		
		\caption{
			The path $P_e$ is replaced by a longer one for $k=2$ and $\vec w (e) = 1$.
			This is possible since the $k$-graph between the clusters corresponding to $e$ is complete partite.
		}
		\label{fig:replace-paths}
	\end{figure}
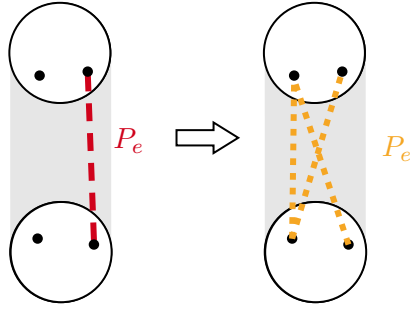

	We consider now the $(k-1)$st power of a cycle $C$ given by $C = \partial_2 \cC$,
	and the $(k-1)$st powers of a path $\partial_2 P_e \subseteq C$.
	We now define the required $(\alpha, R')$-buffer in $C$  where, recall, $R' = \partial_2 \cJ'$.
	Let $1 \leq i \leq rt$ and $V_i \in \fV = V(\cJ)$ be some cluster.
	Since $\cJ'$ is a spanning subgraph of $\cJ$, we can fix an edge $K_i \in E(\cJ')$ such that $i \in V(K_i)$.
	By construction, we have ensured that $\vec w(K_i) \geq 3 \alpha m$.
	This implies that $C$ contains a subpath $P_{K_i}$ of length at least $3 \alpha k m$ completely contained in the clusters of $K_i$.
	We have $|V(P_{K_i}) \cap X_i| \geq 3 \alpha m \geq 2 \alpha |X_i|$, where the latter inequality holds because $\fX$ is $(1+\eps)$-balanced.
	Define $\tilde{X}_i$ from $V({P}_{K_i}) \cap X_i$ by removing the first and last $2k-2$ vertices according to the order of ${P}_{K_i}$.
	We claim that $\{\tilde{X}_i\}_{i=1}^{rt}$ is an $(\alpha, R')$-buffer for $C$.
	Indeed, $\tilde{X}_i \subseteq X_i$, and by construction, $|\tilde{X}_i| \geq 2 \alpha |X_i| - 2k + 2 \geq \alpha |X_i|$.
	Now let $x \in \tilde{X}_i$ and $xy, yz \in E(C)$ such that $y \in X_j$ and $z \in X_k$.
	Since we removed the first and last $2k - 2$ vertices from $V({P}_{K_i}) \cap X_i$, we actually have that $y, z \in V({P}_{K_i})$.
	This implies that $j, k \in K_i$.
	Since $K_i \in E(\cJ')$ and $R' = \partial_2 \cJ'$, we have $ij, jk \in R'$, as required.
\end{proof}

\section{Embedding graphs of sublinear bandwidth}\label{sec:bandwidth}
In this section we give the proof of \cref{thm:main-bandwidth}.
As in the proof of \cref{thm:main-powham}, the proof has two central lemmas which prepare the graphs for the embedding; a `Lemma for $G$' which deals with the host graph,	and a `Lemma for $\cH$' which deals with the guest graph.

We begin by stating the corresponding `Lemma for $G$'.
Note that the differences between the next lemma and the previous \rf{lem:lemma-for-G-cycle} is that $\cJ'$ is a tight Hamilton cycle here (instead of an arbitrary spanning bounded-degree subgraph as before),
that we insist that $t \equiv 1 \bmod k$,
and that we allow for zero-free frameworks.

\begin{lemma}[Lemma for $G$ -- bandwidth] \label{lem:lemma-for-G-bandwidth}
	Let $1/n \ll 1/t_1 \ll 1/t_0 \ll \eps \ll d \ll \mu, 1/k$, and let $r \in \{1, k\}$.
	Let $\fU$ be an $(n,r)$-sized partition, and let $(G,\cH)$ be a $\mu$-robust $\fU$-partite $k$-uniform  Hamilton framework.

	Then there are $t_0 \leq t \leq t_1$,
	a spanning subgraph $G' \subset G$,
	a $\fU$-refining $(1 + \eps)$-balanced partition $\fV$ of $G'$ with $t$ clusters inside each cluster of $\fU$,
	and a graph $R$ on $tr$ vertices such that the following holds.
	Let $\cH' = \cH \cap K_k(G')$, and let $\cJ$ be the $k$-graph induced by $(\cH',\fV)$.
	There also exists a tight Hamilton cycle $\cJ' \subset \cJ$ so that, for $R'= \partial_2 \cJ'$, we have:

	\begin{enumerate}[\upshape (G1)]
		\item \label{itm:lemma-for-G-bandwidth-approx} $e(G') \geq (1-\mu/2)e(G)$,
		\item \label{itm:lemma-for-G-bandwidth-reg} $(G',\fV)$ is an $(\eps,d)$-regular $R$-partition, which is $(\eps,d)$-super-regular on~$R'$ and
		\item \label{itm:lemma-for-G-bandwidth-frame} $(R,\cJ)$ is a $(\mu/4)$-robust $\fW$-partite $k$-uniform  Hamilton framework, $(\mu/4,2/n,\fW)$-cluster-matchable 
		      where $\fW$ is the $(t,r)$-sized partition induced by $(\fV,\fU)$ and
     		\item \label{itm:lemma-for-G-bandwidth-modulo} $t \equiv 1 \bmod k$.
	\end{enumerate}
	Moreover, if $r = 1$ and $(G,\cH)$ is $\mu$-robust aperiodic (resp. zero-free), then $(R,\cJ)$ is $(\mu/4)$-robust (resp. zero-free) as well.
\end{lemma}

The following is our `Lemma for $\cH$'.
 
\begin{lemma}[Lemma for $H$] \label{lem:lemma-for-H}
	Let $0 < 1/n \ll \beta, 1/z \ll \xi \ll \pi \ll \eps, 1/t, 1/k$ and $\eps \ll \alpha \ll \mu, 1/k, 1/\Delta$, and $t \equiv 1 \bmod k$, and let $r \in \{ 1, k\}$.

	Let $\fU$ be an $(n,r)$-sized partition, and let $G$ be a $\fU$-partite graph.
	Let $\fV$ be a $\fU$-refining $(1+\varepsilon)$-balanced partition with $t$ clusters inside each cluster of $\fU$,
	and let $R$ be a graph on $tr$ vertices such that $(G, \fV)$ is an $R$-partition.
	Let $\fW$ be the $(t,r)$-sized partition of $R$ induced by $(\fV, \fU)$.
	Let $\cJ \subset K_k(R)$.
	Suppose that $(R,\cJ)$ is a $\mu$-robust $\fW$-partite $k$-uniform  Hamilton framework which is $(\mu, 2/n, \fW)$-cluster-matchable.
	Let $\cJ' \subseteq \cJ$ be a tight Hamilton cycle on vertex set $1, 2, \dotsc, tr$.

	Let $H$ be a graph on $n$ vertices with $\Delta(H)\leq \Delta$ which admits an ordering $\sigma$ of its vertices with bandwidth at most $\beta n$.
	Suppose one of the following is true:
	\begin{enumerate}[\upshape (i)]
		\item \label{itm:LH-r=1} $r = 1$, $H$ is $k$-colourable and $(R,\cJ)$ is a $\mu$-robust aperiodic $k$-uniform Hamilton framework,
		\item \label{itm:LH-r=1-zero-free} $r = 1$, $H$ admits a $(z, \beta)$-zero-free $(k+1)$-colouring under $\sigma$, and $(R,\cJ)$ is a $\mu$-robust zero-free $k$-uniform Hamilton framework,
		\item \label{itm:LH-r=k} $r = k$, and $H$ admits an equitable $k$-colouring.
	\end{enumerate}

	Then there is a vertex partition $\fX=\{X_i\}_{i=1}^{rt}$ of $V(H)$,
	and families of subsets $\tfX = \{ \tX_i \}_{i=1}^{rt}$,
	$\fZ = \{ Z_i \}_{i=1}^{rt}$,
	such that
	\begin{enumerate}[\upshape (H1)]
		\item \label{itm:allocate-vertices-buffer-lemmaforH}
		      $(H,\fX)$ is an $R$-partition,
		\item $\tfX$ is an $(\alpha,R')$-buffer for $(H,\fX)$  where $R' = \partial_2 \cJ'$,
		\item \label{itm:lemmaH-forward} for each $1 \leq i \leq rt$, $Z_i \subseteq X_i \setminus \tilde{X}_i$,
		      $|Z_i| \geq \pi |X_i|$,
		      and for each $v \in Z_i$, $N_H(v) \subseteq X_{i-1} \cup \dotsb \cup X_{i-k+1}$, with indices understood modulo $t$,
		\item \label{item:lemmaH-sizes} $|X_i| = |V_i| \pm \xi n$ for each $i \in V(R)$ and
		\item \label{item:lemmaH-partition-size} $\sum_{i \in W} |X_i| = n$ for each $W \in \fW$.
	\end{enumerate}
\end{lemma}

Assuming the validity of the previous two lemmas,
we now give the proof of \cref{thm:main-bandwidth}.
Its proof can be sketched as follows.
Focusing on case~\ref{item:main-bandwidth-aperiodic} of \cref{thm:main-bandwidth}, the plan is to first apply \cref{lem:lemma-for-G-bandwidth,lem:lemma-for-H}.
This results in a setup which comes close to the requirements for embedding $H$ into $G$ via \cref{lem:blow-up} (Blow-Up Lemma), which would allow us to find the required embedding and conclude.
Unfortunately, the sizes of the clusters $\tfX$ (returned by \cref{lem:lemma-for-H}) are not quite compatible with those of $\fV$ (returned by \cref{lem:lemma-for-G-bandwidth}).
There is little we can do about $\tfX$ to change this.
We are therefore challenged to somehow manipulate the sizes of any two given clusters $V_i,V_j \in \fV$ such that $|V_i|$ goes to down by one and $|V_j|$ goes up by one, while all other clusters of $\fV$ retain the size and, crucially, the super-regular properties of $R'$ (returned by \cref{lem:lemma-for-G-bandwidth}) are not affected.
Note that such a `shift' is always possible when $V_i$ and $V_j$ are at distance $k$ in the $k$-uniform cycle $\cJ$ that corresponds to $R'$.
They key fact is now that the cycle $\cJ$ has length coprime to $k$.\footnote{Note that the existence of such a $\cJ$ is ultimately a consequence of the aperiodicity of the Hamilton framework.} 
It follows that one can combine these single `shifts' along the ordering of $\cJ$ to eventually reach any desired cluster.
Exploiting this aperiodic behaviour, we may adjust the cluster sizes of $\fV$ accordingly and finish the proof by applying \cref{lem:blow-up}.
Now come the details.

\begin{proof}[Proof of \cref{thm:main-bandwidth}]
	We establish the following constant hierarchy
	\begin{align}\label{equ:bandwidth-constant-hierarchy}
		1/n \ll \beta, 1/z \ll \xi \ll \pi \ll 1/t_1 \ll 1/t_0 \ll \eps,\rho \ll  d,\alpha  \ll  \mu, 1/k. \tag{$\ll$}
	\end{align}
	We describe the choice of constants with more detail now.
	Given $k,\Delta \geq 1$ and $\mu >0$, choose $d >0$ that satisfies \cref{lem:lemma-for-G-bandwidth}.
	Let $\zeta = 1/2$.
	Then select $\eps_0,\rho > 0$, so that \cref{lem:blow-up} (Blow-Up Lemma) can be applied with $\DeltaRp=2(k-1)$, $\kappa = 2$, $\alpha$, $\zeta$, $d/2$ as well as $\max \{\Delta+1,2(k-1) \}$ playing the role of $\Delta$, and $\Delta$ playing the role of $\Delta_J$.
	Next, pick  $0 < \beta, \xi,\pi, \eps \leq \eps_0/3$ and $t_0,t_1,z \in \NATS$ such that the conditions of \cref{lem:lemma-for-G-bandwidth,lem:lemma-for-H} (where we apply the latter with $\mu/4$ instead of $\mu$) are satisfied with these choices.
	We obtain $n_0$ from \cref{lem:blow-up} with input $3\eps,\rho$ and $t_1$.
	Finally, choose an integer $n \ge n_0$, such that the conditions of \cref{lem:lemma-for-G-bandwidth,lem:lemma-for-H} are satisfied (together with the above choices).

	The proofs of \cref{item:main-bandwidth-aperiodic,item:main-bandwidth-partite,item:main-bandwidth-zerofree} of \cref{thm:main-bandwidth} are very similar and follow the same overall structure.
	We will prove \cref{item:main-bandwidth-aperiodic} with full details,
	and outline the necessary changes to prove \cref{item:main-bandwidth-partite,item:main-bandwidth-zerofree} later.

	To prove \cref{item:main-bandwidth-aperiodic}, assume that $(G,\cH)$ is a $\mu$-robust aperiodic $k$-uniform Hamilton framework.
	So $r = 1$, and $\fU$ is a trivial partition.
	Let $H$ be an arbitrary $k$-chromatic graph on $n$ vertices with $\Delta(H) \leq \Delta$ and bandwidth at most $\beta n$.
	To prove \cref{item:main-bandwidth-aperiodic}, we have to show that $H \subset G$.

	\medskip
	\noindent \emph{Step 1: Applying the Lemma for $G$.}
	By \cref{lem:lemma-for-G-bandwidth}, there are $t_0 \leq t \leq t_1$,
	a spanning subgraph $G' \subset G$,
	a $(1 + \eps)$-balanced partition $\fV=\{V_i\}_{i=1}^t$ of $V(G')$,
	and a graph $R$ on $t$ vertices such that following is true.
	Let $\cH' = \cH \cap K_k(G')$, and let $\cJ$ be the $k$-graph induced by $(\cH',\fV)$.
	Then there exists a tight Hamilton cycle $\cJ' \subset \cJ$ so that, for $R'= \partial_2 \cJ'$,
	\begin{enumerate}[\upshape (G1)]
		\item \label{item:lemmaGbandwidth-proofbandwidth-t} $t \equiv 1 \bmod k$,
		\item $e(G') \geq (1-\mu/2)e(G)$,
		\item $(G',\fV)$ is an $(\eps,d)$-regular $R$-partition, which is $(\eps,d)$-super-regular on~$R'$, and
		\item \label{item:lemmaGbandwidth-proofbandwidth-aperiodic-first}
		      $(R,\cJ)$ is a $(\mu/4)$-robust $\fW$-partite $k$-uniform Hamilton framework which is $(\mu/4,2/n, \fW)$-cluster-matchable where $\fW$ is the $(t,1)$-sized partition induced by $(\fV, \fU)$.
	\end{enumerate}
	Since $r = 1$ and $(G,\cH)$ is $\mu$-robust aperiodic, we are also guaranteed (by the last part of \cref{lem:lemma-for-G-bandwidth}) that $(R,\cJ)$ is $(\mu/4)$-robust aperiodic.
	Since $\fU$ is the trivial partition on $V(G')$,
	then $\fW$ also corresponds to a trivial partition of the clusters of $\fV$.
	From this we deduce
	\begin{enumerate}[\upshape (G4$'$)]
		\item \label{item:lemmaGbandwidth-proofbandwidth-aperiodic-final}
		      $(R,\cJ)$ is a $(\mu/4)$-robust aperiodic $k$-uniform Hamilton framework which is $(\mu/4,2/n,\fW)$-cluster-matchable.
	\end{enumerate}

	Without loss of generality (by relabelling clusters, if necessary) we assume the tight Hamilton cycle $\cJ'$ has cyclic order following the indices $1, 2, \dotsc, t$.

	\medskip
	\noindent \emph{Step 2: Applying the Lemma for $H$.}
	Next, we apply \cref{lem:lemma-for-H} (with $\mu/4$ in place of $\mu$)
	to obtain a 
	vertex partition $\fX=\{X_i\}_{i=1}^t$ of $V(\cH)$
	and two families of subsets, $\tfX= \{ \tilde{X}_i\}_{i=1}^t$ and $\fZ = \{Z_i\}_{i=1}^t$ of $V(\cH)$,
	such that
	\begin{enumerate}[\upshape (H1)]
		\item \label{item:lemmaHbandwidth-proofbandwidth-Rpartition} $(H,\fX)$ is an $R$-partition,
		\item \label{item:lemmaHbandwidth-proofbandwidth-buffer} $\tfX$ is an $(\alpha,R')$-buffer for $(H,\fX)$,
		\item \label{item:lemmaHbandwidth-proofbandwidth-forward} for each $1 \leq i \leq t$, $Z_i \subseteq X_i \setminus \tilde{X}_i$ and $|Z_i| \geq \pi |X_i|$,
		      and for each $v \in Z_i$, $N_H(v) \subseteq X_{i-1} \cup \dotsb \cup X_{i-k+1}$, with indices understood modulo $t$,
		      and
		\item \label{item:lemmaHbandwidth-proofbandwidth-sizes} $|X_i| = |V_i| \pm \xi n$ for each $i \in V(R)$.
	\end{enumerate}

	\medskip
	\noindent \emph{Step 3: Adjusting the partition of $G$.}
	Consider integers $m_1,\dots,  m_t$ such that  $|X_i| = |V_i| + m_i$ for $1 \leq i \leq t$.
	Note that \cref{item:lemmaHbandwidth-proofbandwidth-sizes} implies that $-\xi n \leq m_i \leq \xi n$,
	and we also have $\sum_{i=1}^t m_i = 0$.
	In particular, we have $\lfloor \sqrt{\xi} n\rfloor \geq \sum_{i=1}^t  |m_{i}|$ since $\xi \ll 1/t$.
	Let us now select $U_{i} \subset V_{i}$ for every $1 \leq i \leq t$, as follows.
	Let $1 \leq j \leq t$ such that $i \equiv jk \bmod t$, which exists and is unique since $t$ is coprime to $k$ by \cref{item:lemmaGbandwidth-proofbandwidth-t} and $k < t$.
	Then choose $U_{jk} \subseteq V_{jk}$ such that $|U_{jk}|= \lfloor \sqrt{\xi} n\rfloor - \sum_{\ell=1}^j  m_{\ell k}$ for each $1 \leq j \leq t$, where the indices are evaluated modulo $t$.
	Define $\fV'=\{V_i'\}_{i=1}^t$ by setting $V_{jk}' =  V_{jk} \sm U_{jk}$ for every $1 \leq j \leq t$.
	We remark that $(G',\fV')$ is a $(3\eps,d/2)$-regular $R$-partition, which is $(3\eps,d/2)$-super-regular on~$R'$ by \rf{proposition:robust-regular}.
	Moreover, we have for all $1 \leq j \leq t$,
	\begin{align*}
		|V_{jk}' \cup U_{(j-1)k}| & = |V_{jk} \sm U_{jk}| + |U_{(j-1)k}|
		\\&=|V_{jk}| - \left(\lfloor \sqrt{\xi} n\rfloor - \sum_{\ell=1}^j  m_{\ell k}\right) + \left(\lfloor \sqrt{\xi} n\rfloor - \sum_{\ell=1}^{j-1}  m_{\ell k} \right)
		\\&= |V_{jk}| + m_{jk} =|X_{jk}|.
	\end{align*}

	\medskip
	\noindent \emph{Step 4: Applying the Blow-Up Lemma.}
	To finish, we define a subgraph $H' \subset H$,
	a family of image restrictions $\fI=\{I_x\}_{x\in V(\cH)}$,
	and a family of restricting vertices $\fJ=\{J_x\}_{x\in V(\cH)}$.
	To do this, consider the sets $Z_i \subseteq X_i$ given by \cref{item:lemmaHbandwidth-proofbandwidth-forward}.
	For each $1 \leq i \leq t$, will select a set $Y_i \subset Z_i$ of size $|U_{i-k}|$ such that every pair of distinct vertices in $Y_i$ have distance at least three in $\cH$, which is equivalent to say that they are not adjacent and also their neighbourhoods do not intersect.
	Such a set $Y_i$ can be found greedily.
	Indeed, by \cref{item:lemmaHbandwidth-proofbandwidth-forward} we have $|Z_i| \geq \pi |X_i|$.
	For each $y \in Z_i$ there less than $\Delta^2$ other vertices $z \in Z_i$ whose neighbourhoods are not disjoint with $y$.
	We can then pick available vertices greedily, each new choice will forbid less than $\Delta^2$ new vertices.
	This implies we can select a set inside $Z_i$ of size at least $\pi |X_i|/\Delta^2$, and vertices in this set are pairwise at distance at least three in $G$.
	But such a set has size at least $|U_{i-k}|$, since $|U_{i-k}| \leq \sqrt{\xi} n$ and $\xi \ll 1/t, \pi, 1/\Delta$ per the choice of the constant hierarchy.
	Let $Y=\bigcup_i Y_i$ and $H'=H-Y$.

	Next, define $U = \bigcup_{i=1}^t U_i$ and choose an arbitrary bijection $\psi'\colon Y \to U$ that maps $Y_i$ to {$U_{i-k}$} for all $1 \leq i \leq t$.
	This exists, because $|Y_i| = |U_{i-k}|$ for all $1 \leq i \leq t$.
	Now we define the necessary restriction pairs for all vertices in $V(H')$ to apply the Blow-Up Lemma to embed $H'$ in $G' - U$.
	For vertices $x \in V(H) \setminus (\bigcup_{y\in Y} N_H(y) \cup Y)$, we simply set $J_x=\es$ and $I_x=V'_{j}$ (where $x \in X_{j}$).
	It is left to consider the vertices in $\bigcup_{y\in Y} N_H(y)$.
	Given some $y \in Y_i$ and $x \in N_H(y)$, we will define $J_{x}$ and $I_{x}$ as follows.
	Since $Y_i \subseteq Z_i$, by \cref{item:lemmaHbandwidth-proofbandwidth-forward} we have $x \in X_{i-1} \cup \dotsb \cup X_{i-k+1}$.
	Let $i-k+1 \leq j \leq i-1$ be such that $x \in X_j$.
	Now let $u = \psi'(y) \in U_{i-k}$.
	We define $J_x = \{u\}$ and $I_x = N_{G'}(u) \cap V'_j$.
	We check that this choice ensures $|I_x|$ is sufficiently large: since $\cJ'$ is a tight Hamiltonian cycle,
	the edge $ij$ exists in $R' = \partial_2 \cJ'$.
	Then, since $(G', \fV')$ is $(3 \eps, d/2)$-regular on $R'$,
	we have that $|I_x| = |N_{G'}(u) \cap V'_j| \geq (d/2 - 3 \eps)|V'_j| \geq \zeta d |V'_j|$, by the choice of $\zeta$ at the beginning of the proof.
	Next, notice that each vertex in $Y_i$ can restrict the image of at most $\Delta$ vertices of $V'_{i-k+1}, \dotsc, V'_{i-1}$.
	Therefore, the number of image-restricted vertices in each $V'_i$ is at most $\Delta \sum_{i+1 \leq j \leq i+k-1} |Y_j| \leq \Delta k \sqrt{\xi} n \leq \rho |V'_i|$.
	It follows that $\fI$ and $\fJ$ form a $(\rho, \zeta, \Delta+1, \Delta)$-restriction pair.
	
	Finally, we apply \cref{lem:blow-up} (Blow-Up Lemma) with $G = G' - U$, $\fV=\fV'$, $H=H'$, $\Delta = \max\{\Delta,2(k-1)\}$, $\DeltaRp=2(k-1)$, $\Delta_J = \Delta$ and $1+\eps \leq 2 = \kappa$ to obtain an embedding $\psi\colon V(H')\to V(G)$ such that $\psi(x) \in I_x$ for each $x \in V(H)$.
	It follows that $G$ contains a copy of $H$.
	Since $H$ was arbitrary, this shows that $G$ is $(\beta, \Delta, k)$-Hamiltonian.
	This concludes the proof of part~\ref{item:main-bandwidth-aperiodic} of \cref{thm:main-bandwidth}.

	\medskip
	\noindent \emph{Step 5: The zero-free case.}
	We show part~\ref{item:main-bandwidth-zerofree} of \cref{thm:main-bandwidth}.
	For this, our assumption is that $r = 1$ and $(G,\cH)$ is a $k$-uniform $\mu$-robust zero-free $k$-uniform Hamilton framework.
	Consider an arbitrary graph $H$ on $n$ vertices with $\Delta(H) \leq \Delta$ which admits an ordering with bandwidth $\beta n$ and a $(z, \beta)$-zero-free $(k + 1)$-colouring.
	We need to show $H \subseteq G$.

	The proof follows the same steps as before.
	Step 1 is virtually the same, but now since $(G,\cH)$ is zero-free and not just aperiodic, the application of \cref{lem:lemma-for-G-bandwidth} yields that, instead of \cref{item:lemmaGbandwidth-proofbandwidth-aperiodic-final}, we get that $(R,\cJ)$ is a $(\mu/4)$-robust zero-free $k$-uniform Hamilton framework and $(\mu/4,2/n,\fW)$-cluster-matchable.
	In Step 2, this extra property allows us to apply \cref{lem:lemma-for-H} with $H$ as an input.
	The remainder of the proof is exactly the same.

	\medskip
	\noindent \emph{Step 6: The partite case.}
	To show part~\ref{item:main-bandwidth-partite} of \cref{thm:main-bandwidth},
	again we essentially follow the same steps.
	Now the assumption is that $r = k$ and $(G,\cH)$ is a $k$-uniform $\mu$-robust $\fU$-partite Hamilton framework.
	To show that $G$ is $\fU$-partite $(\beta, \Delta, k)$-Hamiltonian,
	we select an arbitrary graph $H$ on $nk$ vertices with an equitable $k$-colouring, $\Delta(H) \leq \Delta$ and bandwidth at most $\beta k n$; and we need to show $H \subseteq G$.

	In this situation, the application of \cref{lem:lemma-for-G-bandwidth} in Step 1 yields that $(R,\cJ)$ is a $(\mu/4)$-robust $\fW$-partite Hamilton framework and $(\mu/4, 2/n, \fW)$-cluster matchable 
	where $\fW$ is the $(t,r)$-sized partition induced by $(\fV, \fU)$.
	This allows us to apply Lemma for $H$ in Step 2 with $H$ as an input, and obtain a vertex partition $\fX$ of $V(H)$ and families $\tfX$, $\fZ$ satisfying \cref{item:lemmaHbandwidth-proofbandwidth-Rpartition,item:lemmaHbandwidth-proofbandwidth-buffer,item:lemmaHbandwidth-proofbandwidth-forward,item:lemmaHbandwidth-proofbandwidth-sizes,item:lemmaH-partition-size}.

	In Step 3, we proceed similarly but now adjusting the allocation separately in each of the $k$ clusters of $\fU$.
	With more detail, recall that the Hamilton cycle $\cJ' \subseteq \cJ$ is given by $1, \dotsc, tr$ on $V(R)$.
	Define $m_i = |X_i| - |V_i|$ for all $1 \leq i \leq tr$.
	Again, we have $- \xi n \leq m_i \leq \xi n$.
	Note that, for any $W \in \fW$, there exists $1 \leq q \leq k$,
	such that the clusters of $\fV$ contained in $\fW$ correspond exactly to $\{ V_{q+ik} \}_{i=0}^{t-1}$.
	Now fix any $W \in \fW$, we have $\sum_{i \in W} m_i = 0$.
	This allows us to define $U_i, V'_i$ as before, again by considering steps of $k$ clusters,
	and the rest of the proof follows as before.
\end{proof}

\subsection{Proof of the \texorpdfstring{\nameref{lem:lemma-for-G-bandwidth}}{Lg}}
The proof of~\cref{lem:lemma-for-G-bandwidth} is almost identical to the one of \cref{lem:lemma-for-G-cycle}, so we will only sketch the main differences.

\begin{proof}[Proof of \cref{lem:lemma-for-G-bandwidth} (sketch)]
	We begin the proof as the one of \cref{lem:lemma-for-G-cycle} by choosing the constants.
	The only difference is that we choose $\gamma,\pi,\eta,t$ so that \rf{thm:cycle-distributed} can be applied with $\gamma$ in place of $\mu$, which will replace \rf{prop:distributed-matching}.

	We then proceed as before with Step 1 and find a balanced regular $R$-partition $(G',\fV)$ together with a collection $\fF$ of subgraphs of $\cJ$ satisfying properties~\ref{itm:lemma-for-G-cycle-G'} to~\ref{itm:lemma-for-G-cycle-W}.
	In addition, whenever we apply \cref{lem:regular-partition+absorber} we insist that $t \equiv 1 \bmod k$, as required here.
	We also note that if $(G,\cH)$ is $\mu$-robust zero-free (and not just aperiodic) then we can replace \ref{itm:lemma-for-G-cycle-RJ} by
	\begin{enumerate}[\upshape (R3$'$)]
		\item $(R,J)$ is a $(\mu/4)$-robust zero-free Hamilton framework and is also $(\mu/4,2/n,\fW)$-cluster-matchable.
	\end{enumerate}

	Then, in Step 2, we apply \rf{thm:cycle-distributed} (with $\gamma$ playing the roles of $\mu$) to $\cJ$.
	It follows that $R$ has as a subgraph the $(k-1)$st power of a Hamilton cycle $R'$ such that
	\begin{enumerate}[\upshape (R1$'$)]\addtocounter{enumi}{3}
		\item  \label{itm:lemmaG-J}	 $\cJ' := K_k(R')$ contains at least $\pi r t$ edges of each element of $\fF$.
	\end{enumerate}
	Step 3 then follows exactly in the same way as before.
	In Step 4, the only difference is that we can use property \ref{itm:lemmaG-J} to sort in the vertices of $V_0$ without considering a matching $\cM$.\footnote{The fact that the edges of $\cJ'$  do not form a matching does not matter for this argument. The only reason to choose $\cM$ as a matching was to control its maximum degree.}
\end{proof}

\section{Proof of the \texorpdfstring{\nameref{lem:lemma-for-H}}{Lg}} \label{sec:lemmaforH}
In this section, we show \cref{lem:lemma-for-H}.
We first sketch the argument and then proceed to give the necessary auxiliary results in \cref{sec:markov-chains,sec:randomised-allcocation,sec:imbalance-correcting} followed by the proof itself in \cref{sec:proof-lemma-for-H-proof-itself}.

\subsection{Proof overview}\label{sec:lemma-for-H-proof-sketch}
We proceed in a series of steps to prove \cref{lem:lemma-for-H}.
Assume we are given a $(k+1)$-colouring of $H$ with a low-bandwidth ordering, and a zero-free framework~$(R, \cJ)$.
The task is to allocate the vertices of $H$ onto vertices of $R$ (which represent clusters).
Then the set of vertices allocated to any given vertex of $R$ will yield a cluster in the desired partition $\cX$ of $V(H)$.

First we will partition $V(H)$ into runs of consecutive vertices $H_1, \dotsc, H_s$,
where in each $H_i$ few vertices at the beginning are coloured using the colour zero,
and the remaining vertices are $k$-coloured.
For each $i$ we will set $H^-_i, H^+_i \subseteq H_i$ as short initial and final intervals of $H_i$, where in particular $H^-_i$ contains all zero-coloured vertices of $H_i$.
The required allocation for $H$ will consist in two main parts: in an initial pass we will allocate the bulk of the vertices, which lie in $\bigcup_i H_i - (H^-_i \cup H^+_i)$.
They will be allocated using a semi-randomised algorithm, which will be described with more detail shortly.
In the second part, we will complete the allocation by including the remaining vertices in $\bigcup_i H_i^- \cup H_i^+$.
This last part will be done by allocating these vertices `by hand' into $R$, making sure the runs are connected in an appropriate way and we allocate the $(k+1)$-colourable part of $H_i$ in a $(k+1)$-clique of $R$ (which exists if $(R, \cJ)$ is a zero-free framework).

Now we describe the first part of the embedding sketched before.
To describe the allocation assume the graph we need to allocate is $H$,
a $k$-colourable graph on $n$ vertices (if $H$ itself is $k$-colourable then this is exactly what we do, otherwise we use this method considering $\bigcup_i H_i - (H^-_i \cup H^+_i)$, which is $k$-colourable).
There are three steps in our embedding.
Initially, we will show $H$ can be allocated into the blow-up of a long $(k-1)$st power of a path $P$.
This is done deterministically, by simply scanning the vertices using the bandwidth ordering and allocating the vertices into the corresponding coloured vertices of $P$.
However, since the $k$-colouring of $H$ can have colour classes of wildly different sizes, this does not quite give the balanced allocation we want (as to satisfy condition~\ref{item:lemmaH-sizes}).

In a second step, we correct this by allocating the vertices of a $(k-1)$st power of a path $P$ in a $(k-1)$st power of a cycle $C$ by simulating a random walk of $P$ over the vertices of $C$; which in turn gives an allocation of the vertices of $H$ (which were previously allocated to $P$) into $C$ (\cref{lemma:lemma-for-H-balanced}).
We analyse this random allocation using a Markov chain, and under mild conditions we will be able to show this gives a very balanced allocation of $H$ into $C$, with clusters of almost the same size, say up to error $\xi n / |C|$ each.
At this point we could try to use such an allocation using $C = R'$  where $R'$ is the given $(k-1)$st power of a Hamilton cycle in $R$.
However, this is still not quite what we need.
This is because in the statement of \cref{lem:lemma-for-H} (specifically in part~\ref{item:lemmaH-sizes}) we require to find an allocation of $H$ into $t$ clusters whose sizes resemble closely the sizes of the $t$ clusters of $\fV$.
But $\fV$ is a $(1 + \eps)$-balanced allocation instead of a perfectly balanced allocation.
Since $\xi \ll \eps$, our allocation up to this point does not necessarily approximate the cluster sizes of $\fV$ in a good way.

In a final step, we overcome this last obstacle by fixing some large $\ell$ (depending on $\eps, t$ only), finding a homomorphism of a power of a cycle $C'$ on $\ell$ vertices into $R$ such that each vertex $i \in V(R)$ is mapped to a number of times proportional to $|V_i|/(rn)$ (\cref{lemma:lemma-for-C-correcting}).
Thus, finally our desired allocation is found by embedding first $H$ into $C'$, and then $C'$ into $R$.
Since the first embedding is almost balanced, and the second one accounts for the varying sizes of the clusters of $R$, the final allocation will satisfy the required size conditions of the partition $\fV$.

\subsection{Splitting into intervals}\label{sec:splitting-into-intervals}
The first step is to allocate our guest graph $H$ into a (modified) power of a path, which allows us to distinguish between the $k$- and $(k+1)$-colourable cases.
If the given $k$-colouring is also equitable, this will also be captured by the allocation.

Given $k, t$ integers, let $P_{k,t}$ be the $(k-1)$st power of a path on $t$ vertices,
we will consider $V(P_{k,t}) = \{ p_1, \dotsc, p_t \}$ in the natural ordering.
Let $P_{k,t}^+$ be the graph obtained from $P_{k,t}$ by adding an extra \emph{universal} vertex $p^+$,
which is joined to all vertices of $V(P_{k,t})$.
Thus $V(P^+_{k,t}) = \{ p^+ \} \cup V(P_{k,t})$.

Given $a, b > 0$, if $\phi\colon V(G) \rightarrow V(H)$ satisfies $|\phi^{-1}(v)| \leq a$ for all $v \in V(H)$, we say $\phi$ is \emph{at most $a$-to-one}.
If it satisfies $a - b \leq |\phi^{-1}(v)| \leq a + b$, we say $\phi$ is \emph{$(a \pm b)$-to-one}.
By $k\NATS$, we will mean the set of all multiples of $k$.

\newcommand{\blq}{{q}}

\begin{lemma}[Allocating $H$ to a power of a path] \label{lemma:Htopowpath}
	Let	$1/n \ll \beta, 1/k$ and $\blq = \lceil \beta^{-1} \rceil + k - 1$.
	Let $H$ be a graph on $n$ vertices with an ordering of bandwidth at most $\beta n$,
	and suppose $H$ has a $(z, \beta)$-zero-free $(k+1)$-colouring $\chi$ with respect to the given ordering.
	For all $0 \leq j \leq k$, let $H_j = \chi^{-1}(\{j\})$ denote the $j$-coloured vertices of $H$.
	Let $p_1, \dotsc, p_{\blq}$ be the vertex set of $P = P_{k, \blq}$,
	and for each $1 \leq j \leq k$ let $P_j = \{ p_i \colon i \in j + k \mathbb{N} \}$.
	Then there exists a homomorphism $\phi\colon V(H) \rightarrow V(P^{+}_{k,\blq})$ such that
	\begin{enumerate}[{\upshape (P1)}]
		\item \label{item:Hpowpath-P1} $\phi^{-1}(\{ p^+ \}) = H_0$,
		\item for each $1 \leq j \leq k$, $\phi^{-1}(P_j) = H_j$,
		\item \label{item:Hpowpath-P3} $\phi$ restricted to $V(H) \setminus H_0$ is an at most $(k \lceil \beta n \rceil)$-to-one homomorphism to $P_{k,\blq}$, and
		\item \label{item:Hpowpath-P4} there exists $s \ge 0$ and indices $1 \leq i_1 \leq i_2 \leq \dotsb \leq i_s \leq \blq$ such that
		      \begin{enumerate}[{\upshape (a)}]
			      \item $i_{j+1} - i_j \ge z$ for all $1 \leq j < s$,
			      \item for every $v \in V(H)$ adjacent to a vertex in $H_0$, there exists $1 \leq j \leq s$ such that $\phi(v) \in \{ p_{i_j-1}, \dotsc, p_{i_j+k} \}$.
		      \end{enumerate}
	      Further, if $H_0 = \emptyset$, then $s = 0$.
	\end{enumerate}
\end{lemma}

\begin{proof}
	Initially, set $\phi(x) = p^+$ for all $x \in H_0$.
	It thus remains only to allocate the vertices of $H_{-0} = V(H) \setminus H_0$ along $P_{k,\blq}$ to complete the definition of $\phi$.

	Consider the given ordering $v_1, \dotsc, v_n$ of $V(H)$, which satisfies the corresponding bandwidth and zero-free colouring properties.
	Partition $V(H)$ along its consecutive $\beta$-blocks.
	More precisely, we divide $V(H)$ into $\blq-k+1$ intervals $X_1, \dotsc, X_{\blq-k+1}$ of consecutive vertices as follows:
	initially, greedily select intervals of consecutive vertices of size exactly $\lceil \beta n \rceil$, until this is no longer possible.
	So, for instance, $X_1 = \{v_1, \dotsc, v_{\lceil \beta n \rceil}\}$, $X_2 = \{ v_{\lceil \beta n \rceil + 1}, \dotsc, v_{2 \lceil \beta n \rceil} \}$, and so on.
	Then, choose a single interval of size less than $\lceil \beta n \rceil$ covering the remainder of the vertices,
	and finalise by declaring all of the next intervals to be empty.

	For each $1 \leq i \leq \blq-k+1$ and $0 \leq j \leq k$, let $X_{i}^j = H_j \cap X_i$,
	that is, the vertices in $X_i$ which were $j$-coloured under $\chi$.
	Note that $H_{-0} = \bigcup_{1 \leq i \leq \blq-k+1, 1 \leq j \leq k} X^i_j$.
	Certainly, $|X_{i}^j| \leq |X_i| \leq \lceil \beta n \rceil$ for all $i, j$.
	Because of the colouring and the bandwidth condition, a vertex in $X_i^j$ can only have neighbours in $X_{i'}^{j'}$ if $|i - i'| \leq 1$ and $j \neq j'$.

	The idea is to allocate the vertices of $H_{-0}$ in $P$ in a way such that the following two properties are satisfied: \begin{enumerate}[{\upshape (i)}]
		\item $p_1$ only receives $1$-coloured vertices, $p_2$ only receives $2$-coloured vertices, and so on, and this pattern is repeated cyclically; and
		\item the vertices of $X_i \setminus H_0$ are allocated only to $\{ p_i, p_{i+1}, \dotsc, p_{i+k-1} \} \subseteq V(P)$, for each $1 \leq i \leq \blq-k+1$.
	\end{enumerate}
	Formally, this idea is accomplished by the function $\phi\colon H_{-0} \rightarrow P$ which sends all of $X^j_i$ to $p_{k\lfloor(i-j+k-1)/k \rfloor+j}$, for all $1 \leq i \leq \blq-k+1$ and $1 \leq j \leq k$.
	The previous discussion ensures $\phi$ is a graph homomorphism when $H$ is restricted to $H_{-0}$.
	By construction, for each $v \in V(P)$ we have $|\phi^{-1}(v)| \leq k \lceil \beta n \rceil$.
	So $\phi$ satisfies \ref{item:Hpowpath-P1} to \ref{item:Hpowpath-P3}.
	
	It remains to check that $\phi$ satisfies \ref{item:Hpowpath-P4}.
	If $H_0 = \emptyset$ there is nothing to check, to we assume otherwise.
	To verify this, we examine the position of the zero-coloured vertices in the $\beta$-blocks.
	Let $Z = \{ i_1, \dotsc, i_s \} \subseteq \{ 1, \dotsc, \blq-k+1 \}$ be the set of indices $i$ such that $X_i^0 \neq \emptyset$.
	Since $\chi$ is a $(z, \beta)$-zero-free colouring, we obtain $i_{j+1} - i_j \ge z$ for all $1 \leq j < s$.
	Now, let $x \in H_0$ and $1 \leq j \leq s$ such that $x \in X_{i_j}$.
	If $v \in V(H)$ is adjacent to $x$, then $v \in (X_{i_j-1} \cup X_{i_j} \cup X_{i_j+1}) \setminus H_0$.
	By definition of $\phi$, it holds that $\phi(X_i \setminus H_0) \subseteq \{ p_{i}, \dotsc, p_{i+k-1} \}$ for all $1 \leq i \leq \blq - k + 1$.
	We deduce that $\phi(v) \in \phi(X_{i_j-1} \setminus H_0) \cup \phi(X_{i_j} \setminus H_0) \cup \phi(X_{i_j+1} \setminus H_0) \subseteq \{ p_{i_j - 1}, \dotsc, p_{i_j+k} \} $, as required.
\end{proof}

\subsection{Markov chains and concentration}\label{sec:markov-chains}
We will work with finite Markov chains, which are random processes $\{ X_i \}_{i \ge 1}$ taking values in a state space $\cX$.
The trajectories of the random process are determined by a transition matrix $P = (P_{i,j})_{i \in \cX, j \in \cX}$, whose columns and rows are indexed by $\cX$, and each entry is in $[0,1]$.
Here, the entry $P_{i,j}$ will correspond to the probability of evolving from state $i \in \cX$ to state $j \in \cX$ in a single step.
Thus the row of $P$ corresponding to $i \in \cX$ will contain the numbers $\{P_{i,j} \colon j \in \cX\}$, which must satisfy $\sum_{j \in \cX} P_{i,j} = 1$.
A Markov chain is \emph{irreducible} if for all $i, j \in \cX$ there exists $t > 0$ such that $P^{t}_{ij} > 0$ 
where $P^t$ is the $t$th power of the matrix $P$ and $P^t_{ij}$ is the entry in the $i$th row and $j$th column of $P^t$.
The \emph{period} $d_i$ of a state $i \in \cX$ is the greatest common divisor of the numbers in $\{ t > 0 \colon P^{t}_{ii} > 0 \}$, and a Markov chain is \emph{aperiodic} if $d_i = 1$ for all $i \in \cX$.
A probability distribution $\pi$ on $\cX$ (understood as a row vector) is \emph{stationary} for the Markov chain if the equality
$\sum_{i \in \cX} \pi_i P_{ij} = \pi_j$ holds for every $j \in \cX$, or equivalently, if $\pi = \pi P$.
A basic fact from Markov chain theory states that finite and irreducible Markov chains admit a unique stationary distribution.

We recall that, for finite Markov chains, irreducibility can also be expressed in a more combinatorial way in terms of strong connectivity of directed graphs.
Let $D = D(P, \cX)$ be the directed graph on $\cX$ where $a \rightarrow b$ whenever $P_{ab} > 0$.
Then the Markov chain $\{ X_i \}_{i \ge 1}$ is irreducible if and only if $D$ is \emph{strongly connected}, meaning that  there exists a directed walk between any two arbitrary vertices of $D$.
Similarly, an irreducible and finite Markov chain will be aperiodic if $D$ contains two directed closed walks whose lengths are coprime.

The following result is a concentration inequality for Markov-dependent random variables.
The first inequality of this kind was proven by Gillman~\cite{Gillman1998}, although many variations exist (see, e.g., Section 3.5 in~\cite{DubhashiPanconesi2009}).
We will use a recent Hoeffding-type inequality for Markov chains, as proven by Fan, Jiang and Sun~\cite{FJS18} and independently by Rao~\cite{Rao2019}.

We use the following notation: given a function $f\colon \cX \rightarrow \REALS$ and a probability distribution $\pi$ over $\cX$, we write $\expectation_\pi[f] = \sum_{x \in \cX} \pi(x) f(x)$.

\begin{theorem}[{\cite[Theorem 2.1]{FJS18}}] \label{theorem:fanjiangsun}
	Let $\{ X_i \}_{i \ge 1}$ be a finite, irreducible and aperiodic Markov chain with state space $\cX$ and transition matrix $P$.
	There exists $C > 0$, depending on $P$ only, such that the following holds.

	Suppose $X_1$ is distributed according to the stationary distribution $\pi$ of the chain.
	Let $\{ f_i \}_{i \ge 1}$ be functions $f_i \colon \cX \rightarrow [a_i, b_i]$.
	Then, for every $\eps > 0$ and $s \in \mathbb{N}$,
	\begin{align}
		\probability \left[ \left| \sum_{i=1}^s f_i(X_i) - \sum_{i=1}^s \expectation_\pi[f_i] \right| > \eps \right] \leq 2 \exp \left( - \frac{C \eps^2}{2 \sum_{i=1}^s (b_i - a_i)^2 / 4 }  \right). \label{equation:markovhoeffding}
	\end{align}
\end{theorem}

We remark that the original statement of~\cite[Theorem 2.1]{FJS18} applies to the larger class of (possibly infinite) Markov chains with positive absolute spectral gap. 
We state it only for the particular case we need, which is that of finite, irreducible and aperiodic Markov chains,
it is well-known that these properties imply the Markov chain has a non-trivial spectral gap.

\subsection{Randomised allocation}\label{sec:randomised-allcocation}
In this subsection we will construct an allocation of a guest graph $H$ into a power of a cycle, focusing on the case where $H$ is $k$-colourable.
We account for the aperiodic case (where $r = 1$) and the partite case (where $r = k$) here as well.
In the following, we denote by $C_{k,t}$ the $k$th power of a Hamilton cycle on $t$ vertices.

\begin{lemma}[Randomised allocation for $H$] \label{lemma:lemma-for-H-balanced}
	Let $1/n \ll \beta, 1/q \ll \xi, 1/t, 1/k$ and such that $t$ and $k$ are coprime, and $r \in \{1, k\}$.
	Let $P_{k,q} = (p_1,\dots,p_q)$, $H$ be a graph on $nr$ vertices and $\phi_1 \colon V(H) \rightarrow V(P_{k,q})$ a homomorphism which is at most $(k \lceil \beta r n \rceil)$-to-one.
	Suppose in addition that if $r=k$, then $|\phi^{-1}_1( P_j )| = n$  where $P_j = \{ p_i \colon i \in j + k \mathbb{N} \}$.
	
	Then there exists a homomorphism $\phi_2 \colon V(P_{k,q}) \rightarrow V(C_{k, rt})$ such that the composition 
	$\phi = \phi_2 \circ \phi_1$ is a homomorphism from $H$ to $C_{k,rt}$ which is $((1 \pm \xi)\frac{n}{t})$-to-one.
\end{lemma}

As sketched before, we will prove \cref{lemma:lemma-for-H-balanced} by
embedding $P_{k,q}$ into $C_{k, rt}$ appropriately and combining this with the given $\phi_1$.
We will describe a process which finds a homomorphism from $P_{k,q}$ into $C_{k,rt}$ using a randomised algorithm.
To explain the idea, suppose that we are in the middle of the embedding process, and suppose that vertices $p_1,\dots,p_s$ in $P_{k,q}$ have been embedded, and in particular $p_{s-k+2}, \dotsc, p_s$ were mapped on the vertices $c_{i-k+2}, \dotsc, c_i$ of $C_{k,rt}$.
It is the turn to define the image of the next vertex $p_{s+1}$.
If we were simulating a simple random walk in the cycle, we would then like to map $p_{s+1}$ uniformly on $c_{i+1}$ and $c_{i-1}$, where the index computations are modulo $rt$.
When $k=2$, this is fine.
But for $k>2$, this does not work, because $P_{k,q}$ and $C_{k, rt}$ are the $k$th powers of a path and cycle respectively, and so, we need that each $k$ consecutive vertices of $P_{k,q}$ are mapped to some $k$ vertices which are consecutive in $C_{k,rt}$.
In short, we need to avoid the $k-1$ trailing vertices of $P_{k,q}$ when we embed $p_{s+1}$ in order to obtain a (partial) homomorphism.

We achieve this by `looping around' the last $k-1$ vertices.
This means we embed $p_{s+1}$ instead in a vertex chosen uniformly at random between $c_{i+1}$ and $c_{i-(k-1)}$.
In the latter case, we afterwards map $p_{s+2}, \dots , p_{s+k-2}$ deterministically to $c_{i-(k-2)}, \dots , c_{i-1}$,
before doing the next random choice, and so on.
We formalise this construction as follows.

The above outlined algorithm can be modelled with a Markov chain $X = \{ X_i \}_{i \ge 1}$.
The state space of $X$ is $\cX_{k,rt} := \{1,\dots,rt\} \times \{1,\dots,k-1\}$.
The transition probabilities $P_{(a,b), (a',b')}$ for all $(a,b), (a',b') \in \{1,\dots,rt\} \times \{1,\dots,k-1\}$ are given by
\begin{align*}
	P_{(a, 1), (a+1, 1)}   & = 1/2, &                                       \\
	P_{(a, 1), (a, 2)}     & = 1/2, &                                       \\
	P_{(a, b), (a, b+1)}   & = 1    & \text{if $j \in \{2, \dotsc, k-2\}$,} \\
	P_{(a, k-1), (a-1, 1)} & = 1,   &
\end{align*}
and $0$ otherwise
(sums in the indices are always understood modulo $rt$ or $k-1$ so they make sense).
To recover the intuition from the previous discussion, the reader should picture that the random walk will visit the $i$th vertex of the cycle at step $s$ whenever $X_s = (i,1)$ or whenever $b > 1$ and $X_s = (i + k + 1 - b, b)$.
The second type of visits accounts for the described `loops' which are necessary to maintain the fact that the current partial embedding is a partial homomorphism.

For each $1 \leq j \leq k$, let $\cX_{k,rt}^{(j)} \subseteq \cX_{k, rt}$ consist of the tuples $(a,b)$ such that $a+b-1 \equiv j \bmod k$.
The following lemma encapsulates the Markov-theoretic properties we need.

\begin{lemma}
	Suppose $t$ and $k$ are coprime.
	Then 
	\begin{enumerate}[\upshape (M1)]
		\item \label{item:markov-irreducible} $X$ is irreducible,
		\item \label{item:markov-stationary} the (unique) stationary distribution $\pi$ of $X$ is given, for all $a$, by
		      $\pi_{(a,1)} = 2/(krt)$, and $\pi_{(a,b)} = 1/(krt)$ for $b \neq 1$,
		\item \label{item:markov-aperiodic} if $r = 1$, $X$ is aperiodic, and
		\item \label{item:markov-steps} if $r = k$, for any $1 \leq j \leq k$, the chain $X^{(j)} = \{ X_{j + ik} \}_{i \ge 0}$ attains only values on $\cX_{k,rt}^{(j')}$ for some $1 \leq j' \leq k$, is irreducible and aperiodic, and its stationary distribution $\pi^{(j)}$ is given by $\pi^{(j)}_{(a,1)} = 2/kt$, and $\pi^{(j)}_{(a,b)} = 1/kt$ for $b \neq 1$.
	\end{enumerate}
\end{lemma}

\begin{proof}
	We begin by showing that $\cX$ is irreducible.
	As discussed before, it suffices to show that the associated digraph $D = D(P,\cX_{k,rt})$ is strongly connected.
	Note that the directed cycle $C := (1,1)(2,1) \dotsb (rt,1)$ covers all tuples of the form $(a,1)$, and for any $a \in \{1, \dotsc, rt\}$ we have $C_a := (a,1) (a,2) \dotsb (a, k-1)(a-1, 1)$ is a directed cycle which covers all tuples of the type $(a, b)$, for fixed $a$.
	Then it is easy to construct, for any pair of given $(a,b), (a',b')$ a directed walk in $D$ from one to the other; first by traversing $C_a$ to arrive at $(a-1,1)$, then following $C$ to get to $(a', 1)$, and finally using $C_{a'}$ to get to $(a', b')$.

	The proof of~\cref{item:markov-stationary} reduces to checking the matrix equation $\pi = \pi P$ holds.
	Specifically, given $(a, b) \in \cX_{k,rt}$, we need to check that
	\[ \sum_{(a', b') \in \cX_{k,rt}} \pi_{(a', b')} P_{(a', b'), (a,b)} = \pi_{(a, b)}. \]
	If $b = 1$, then $P_{(a', b'), (a,1)} \neq 0$ only if $(a',b') \in \{ (a-1, 1), (a+1, k-1) \}$, so we have
	\begin{align*}
		\sum_{(a', b') \in \cX_{k,rt}} \pi_{(a', b')} P_{(a', b'), (a,b)}
		& = \pi_{(a-1, 1)} P_{(a-1, 1), (a,1)} + \pi_{(a+1, k-1)} P_{(a+1, k-1), (a,1)} \\
		& = \frac{2}{krt} \times \frac{1}{2} + \frac{1}{krt} \times 1 = \frac{2}{krt} = \pi_{(a, 1)},
	\end{align*}
	as desired.
	If $b \neq 1$, then $P_{(a', b'), (a,1)} \neq 0$ only if $(a',b') = (a, b-1)$, and the equality can be checked in a similar way, which we omit.
	We conclude~\cref{item:markov-stationary} holds.

	To see that the chain is aperiodic when $r = 1$, note that $C$ is a directed cycle of length $rt = t$ in $D$,
	and $C_1$ is a directed cycle of length $k$ in $D$.
	Since $k$ and $t$ are coprime, this readily implies the chain is aperiodic.

	Finally, to see \cref{item:markov-steps}, note that when $r = k$ the chain is periodic with period $k$.
	When $r = k$, the sets $\{  \cX_{k,kt}^{(j)} \}_{j=1}^k$ partition the state space.
	Also, for every $i$, note that if $X_i \in (a,b)$ and $X_{i+1} \in (a',b')$, it must hold that $a'+b' = a+b+1 \bmod k$.
	Then we deduce that for any $1 \leq j \leq k$, if $X_i \in \cX_{k,rt}^{(j)}$, then $X_{i+k} \in \cX_{k,rt}^{(j)}$.
	This shows the chain $X^{(j)}$ is irreducible and aperiodic.
	Since the transition matrix of $X^{(j)}$ is given by (a restriction of) $P^k$,
	the stationary distribution is given by the restriction of $\pi$ to $\cX_{k,rt}^{(j)}$ and normalising, which yields the result.
\end{proof}

The algorithm will simulate a Markov chain $\{ X_i \}_{i \ge 1}$ using the transition probabilities which were described before.

\begin{algorithm} \label{algorithm:randomwalk}
	Let $p_1, \dotsc, p_q$ be the ordered vertices of $P_{k,q}$.
	Let $c_1, \dotsc, c_{rt}$ be the cyclically ordered vertices of $C_{k,rt}$.
	Initially, at step $s = 1$ we will choose $X_1 \in \cX_{k,rt}$ according to the stationary distribution $\pi$ of the chain.
	If at step $s \ge 1$ we are given $X_s = (a,b)$ and $b = 1$, then embed $p_s \mapsto c_{a}$,
	otherwise embed $p_s \mapsto c_{a-k+b-1}$ (indices understood modulo $rt$).
\end{algorithm}

\begin{lemma} \label{lemma:algorithmcorrect}
	\cref{algorithm:randomwalk} yields a valid embedding of $P_{k,q}$ into $C_{k,rt}$.
\end{lemma}

\begin{proof}
	For each $s \ge 1$, let $X_s = (a_s, b_s)$ be the states generated by the Markov chain used in~\cref{algorithm:randomwalk}.
	Note that the only admissible transitions of the Markov chain force that once we have entered a state of type $(a, 2)$, we must use $k-1$ steps to arrive at $(a-1,1)$, after only which the option to use $(a-1,2)$ or $(a,1)$ is possible.
	This observation entails that a sequence $(a_{s+1}, b_{s+1}) \dotsb (a_{s+k}, b_{s+k})$ of $k$ consecutive states of the Markov chain must necessarily have the form $X_{a,b} Y_{a-1, a'} Z_{a',b'}$  where $X_{a,b}$ is either empty or has the form $(a, b) (a, b+1) \dotsb (a, k-1)$ for some $a$ and $2 \leq b \leq k-1$, $Y_{a-1, a'} = (a-1, 1) (a, 1) (a+1, 1) \dotsb (a', 1)$, for some $a'$ (possibly $a' = a-1$), and $Z_{a', b'}$ is either empty or has the form $(a', 2) (a', 3) \dotsb (a', b')$ for some $2 \leq b \leq k-1$.

	Thus, given $a, b, a', b'$ which define $X_{a,b} Y_{a-1, a'} Z_{a',b'}$, \cref{algorithm:randomwalk} will embed $k$ consecutive vertices of $P_{k,q}$ in the vertices
	\begin{equation}
		c_{a-k+b-1} c_{a-k+b} \dotsb c_{a-2} c_{a-1} c_{a} \dotsb c_{a'} c_{a'-k+1} c_{a' - k + 2} \dotsb c_{a' - k - 1 + b'} \label{equation:algsequence}
	\end{equation}
	 of $C_{k, t}$, with the indices always understood modulo $t$.
	Since there are $k$ steps, we must have $k = |X_{a,b}| + |Y_{a-1, a'}| + |Z_{a', b'}| = (k-b)+(a'-a+2)+(b'-1) = k-b+a'-a+b'+1$ and thus $a+b = a'+b'+1$.
	Together with \eqref{equation:algsequence}, his implies that the vertices used in the embedding correspond to precisely $k$ consecutive vertices in $C_{k,rt}$, namely those from $c_{a'-k+1}$ to $c_{a'}$.
	Since this is true for any run of $k$ consecutive steps of the algorithm, this implies that the embedding is valid.
\end{proof}

Now we are ready to prove \cref{lemma:lemma-for-H-balanced}.

\begin{proof}[Proof of \cref{lemma:lemma-for-H-balanced}]
	Let $p_1, \dotsc, p_q$ be the ordered vertices of $P_{k,q}$, and let
	$\phi_1\colon V(H) \rightarrow V(P_{k,q})$ be a homomorphism as in the statement.
	In particular, for all $1 \leq i \leq q$, if we let $H_i = \phi^{-1}_1(\{ p_i \})$ then
	\[ |H_i| \leq k \lceil \beta r n \rceil =\colon M. \]
	Now let $\phi_2\colon V(P_{k,q}) \rightarrow V(C_{k,rt})$ be the random function given by \cref{algorithm:randomwalk}.
	By \cref{lemma:algorithmcorrect}, $\phi_2$ is a graph homomorphism.
	Then the composition $\phi = \phi_2 \circ \phi_1\colon V(H) \rightarrow V(C_{k,rt})$ defines a valid allocation from $H$ into $C_{k,rt}$.
	We now check that it satisfies the required properties with non-zero probability.

	Recall that, for $X_s = (a,b)$, \cref{algorithm:randomwalk} maps $p_s \mapsto c_{a}$ when $b=1$ and $p_s \mapsto c_{a-k+b-1}$ when $b>1$.
	So we have $\phi_2(p_s) = c_i$ for $1 \leq i \leq rt$ precisely when $X_s = (i,1)$ or $X_s = (i+k+1-b,b)$ for $2 \leq b \leq k-1$.
	Motivated by this, we define
	\[S_i = \{ (i,1) \} \cup \{ (i+k+1-b,b) \colon 2 \leq  b \leq k-1 \} \subseteq \cX_{k,rt}.\]
	So $\phi_2(p_s) = c_i$ if and only if $X_s \in S_i$.
	Now define functions $f_1, \dotsc, f_q \colon \cX_{k,rt} \rightarrow [0, M]$ such that, for all $1 \leq s \leq q$,
	\[ f_s( (a,b) ) = \begin{cases}
			|H_s| & \text{if $(a,b) \in S_i$}, \\
			0     & \text{otherwise.}
		\end{cases} \]
	Since $H_s$ will be embedded (via $\phi$) into $c_i$ if and only if $X_s \in S_i$, we have
	\begin{equation}
		|\phi^{-1}(c_i)| = \sum_{s = 1}^q f_s(X_s). \label{equation:fisthesum}
	\end{equation}
	As the Markov chain $\{ X_i \}_{i \ge 1}$ in Algorithm~\ref{algorithm:randomwalk} was started from its stationary distribution $\pi$, it follows that $\probability[X_s = (a,b)] = \pi_{(a,b)}$ for all $1 \leq s \leq q$ and $(a,b) \in \cX_{k,rt}$.
	Together with~\cref{item:markov-stationary} we get $\probability[X_s \in S_i] = 1/(rt)$ for all $1 \leq s \leq q$.
	Therefore $\expectation_\pi[f_s] = |H_s|/(rt)$ and thus
	\begin{equation}
		\sum_{s = 1}^r \expectation_\pi[f_s] = \sum_{s=1}^r \frac{|H_s|}{rt} = \frac{n}{t}. \label{equation:ntistheexpectation}
	\end{equation}
	We have $\sum_{s=1}^q |H_s| = rn$ and $|H_s| \leq M = k \lceil \beta r n \rceil$ holds for each $1 \leq s \leq q$.
	By convexity, $\sum_{s=1}^q |H_s|^2$ is maximised when as many $|H_s|$'s as possible take the value $M$, which can happen certainly for at most $\lceil rn/M \rceil$ possible values of $|H_s|$.
	In consequence,
	\begin{equation}
		\sum_{s=1}^q |H_s|^2 \leq \left\lceil \frac{rn}{M} \right\rceil M \leq  2 k \beta r^2 n^2, \label{equation:constraints}
	\end{equation}
	where the extra `$2$' accounts for the removal of the ceilings.

	Now, suppose we are in the case $r = 1$.
	Then the chain $\{X_i\}_{i \ge 1}$ is finite, irreducible and aperiodic by~\cref{item:markov-irreducible} and~\cref{item:markov-aperiodic}.
	Applying
	\cref{theorem:fanjiangsun} gives $C = C(k,t)$ depending on $k, t$ only.
	We apply inequality \eqref{equation:markovhoeffding} with $\xi n/t, |H_s|$ playing the roles of $\eps, b_i-a_i$ respectively, to get
	\begin{align*}
		\probability \left[ \left| |\phi^{-1}(c_i)| - \frac{n}{t}  \right| > \frac{\xi n}{t} \right]
		 & \overset{\eqref{equation:fisthesum}, \eqref{equation:ntistheexpectation}}{=} \probability \left[ \left| \sum_{s=1}^q f_s(X_s) - \sum_{s=1}^q \expectation_\pi[f_s] \right| > \frac{\xi n}{t} \right] \\
		 & \overset{\eqref{equation:markovhoeffding}}{\leq} 2 \exp \left( - \frac{C \xi^2 n^2}{2 t^2 \sum_{s=1}^q |H_s|^2 / 4 }  \right)                                                                        \\
		 & \overset{\eqref{equation:constraints}}{\leq} 2 \exp \left( - \frac{C \xi^2}{t^2 k \beta}  \right)
		< \frac{1}{t},
	\end{align*}
	where we used $\beta \ll \xi, 1/t, 1/k$ in the last step.
	Using a union bound over all the $t$ possible choices of $c_i$,
	we deduce that with positive probability $|\phi^{-1}(c_i)| = (1 \pm \xi)n/t$ for all $c_i \in V(C_{k,t})$, as desired.

	Now suppose $r = k$.
	Fix $j \in \{ 1, \dotsc, k \}$.
	By assumption, we have that $|\phi^{-1}(P_j)| = n$.
	Let $V(C_{k, kt}) = \{c_1, \dotsc, c_{kt}\}$.
	Note that the power of cycle $C_{k,kt}$ is $k$-partite with a unique (up to relabelling) partition, whose $k$ parts are $R_j := \{ c_{j+ik} \colon 0 \leq i < t \}$, for $1 \leq j \leq k$.
	The main point here is that (without loss of generality, after relabelling vertices of the cycle if necessary) the vertices in $P_j$ must always be mapped, via $\phi_2$, to vertices of $C_j$.
	This follows since, by \Cref{lemma:algorithmcorrect}, the (random) allocation $\phi_2$ is a valid homomorphism from $P_{k,q}$ to $C_{k, kt}$, which in particular must respect the $k$-partite structure.
	In terms of the Markov chain $\{X_s\}_{s \geq 0}$, we have that $\phi_2(p_s) \in R_j$ if and only if $X_s \in S_i$ with $i \in \{ j+ik \colon i \geq 0 \}$.
	Unraveling the definition of $S_i$, we deduce that there exists $0 \leq j' < k$ such that $\phi_2(p_s) \in R_j$ if and only if $s \in \{ j'+ik \colon i \geq 0 \}$.
	We thus have that the mapping of vertices into $R_j$ is governed by the chain $\{X_{j+ik}\}_{i \ge 0}$.
	Since the chain $\{X_{j+ik}\}_{i \ge 0}$ is irreducible and aperiodic by~\cref{item:markov-steps},
	we can apply \eqref{equation:markovhoeffding} using this new chain, and conclude in a similar way to the $r = 1$ case.
\end{proof}

\subsection{A cycle to correct imbalances}\label{sec:imbalance-correcting}
Here we find an allocation of a tight cycle into a $k$-graph $\cJ$ such that $(R,\cJ)$ forms a robust Hamilton framework,
with the additional restriction that the number of vertices allocated to each cluster is proportionally very close to the sizes prescribed by a given $(1+\eps)$-balanced partition.
Its proof will be an application of~\rf{lem:lemma-for-C}.

\begin{lemma}[Imbalance-correcting Lemma for $C$] \label{lemma:lemma-for-C-correcting}
	Let 	$1/n \ll \xi \ll \pi \ll 1/t, \eps \ll \alpha \ll \mu, 1/k$ and let $r \in \{1, k\}$.
	Let $\fU$ be an $(n,r)$-sized partition, and
	let $G$ be a $\fU$-partite graph.
	Let $\fV=\{V_i\}_{i=1}^{rt}$ be a $\fU$-refining $(1+\eps)$-balanced partition of $V(G)$,
	with $t$ clusters inside each cluster of $\fU$, such that $G$ is $\fV$-partite.
	Let $R$ be a graph on $tr$ vertices such that $(G, \cV)$ is an $R$-partition.
	Let $\fW$ be the $(t,r)$-sized partition of $R$ induced by $(\fV, \fU)$.
	Let $\cJ \subseteq K_k(R)$.
	Suppose $(R,\cJ)$ is a $(\mu/2)$-robust $\fW$-partite $k$-uniform Hamilton framework and $(\mu/2, 2/n, \fW)$-cluster-matchable.
	If $r = 1$, suppose in addition that $(R,\cJ)$ is $(\mu/2)$-robust aperiodic.
	Let $V(R) = \{1, \dotsc, t\}$ and suppose $\cJ' = (1, 2, \dotsc, t)$ is a tight Hamilton cycle in $\cJ$.

	Then there is $\ell \leq 2 r t^2 \xi^{-1}$ coprime to $k$ and a homomorphism $\phi\colon V(C_{k, r \ell}) \rightarrow V(R)$ such that
	\begin{enumerate}[\upshape (J1)]
		\item \label{item:imbalanced-correcting} for every $i \in V(R)$,
		      \[ \frac{|\phi^{-1}(i)|}{r \ell} = \frac{|V_i|}{r n} \pm \xi. \]
		\item \label{item:imbalanced-buffer} for every $i \in V(R)$ there exists $\tilde{X}_i \subseteq \phi^{-1}(i) \subseteq V(C_{k, r\ell})$ such that $|\tilde{X}_i| \ge \alpha |\phi^{-1}(i)|$ and for each $x \in \tilde{X}_i$ and $xy, yz \in E(C_{k, r\ell})$, with $\phi(y) = j$ and $\phi(z) = l$, we have $il, jl \in E(R')$  where $R' = \partial_2 \cJ'$, and
		\item \label{item:imbalanced-forward} for every $i \in V(R)$ there exists $Z_i \subseteq \phi^{-1}(i) \sm \tX_i $ such that $|Z_i| \ge \pi |\phi^{-1}(i)|$ and for each $x \in Z_i$, $y \in N_{C_{k, r \ell}}(x)$, we have
		      $\phi(y) \in \{ i-k+1, i-k+2, \dotsc, i-1 \}$.
	\end{enumerate}
\end{lemma}

\begin{proof}
	Choose the largest $\ell$ such that $\ell \leq 2 r t^2 \xi^{-1}$ and $\ell$ is coprime with $k$.
	By choice of $\xi$, we can assume that the hierarchy $1/\ell \ll 1/t, \eps, \alpha, \mu, 1/k$ holds.
	In particular, we can assume $\ell \ge r t^2 \xi^{-1}$.
	Next, we choose $\ell_1, \dotsc, \ell_{rt} \ge 0$ such that
	\begin{enumerate}[\upshape (L1)]
		\item If $r = 1$, then $\sum_{i=1}^{t} \ell_i = \ell$,
		\item \label{item:cycleimbalance-r-sum} if $r = k$, then for each $1 \leq j \leq k$, $\sum_{i=0}^{t-1} \ell_{j + ik} = \ell$, and
		\item the quantity $\sum_{i=1}^{rt} \left| \frac{\ell_i}{r\ell} - \frac{|V_i|}{rn} \right|$ is minimised among all the possible choices.
	\end{enumerate}
	We claim that, under this choice, it is not possible that $\frac{\ell_i}{\ell} - \frac{|V_i|}{n} > t/\ell$ for some $i$.
	We argue by contradiction, assuming first the case $r = 1$, and assuming $\frac{\ell_i}{\ell} - \frac{|V_i|}{n} > t/\ell$ holds.
	Since $\sum_{j=1}^t \ell_j = \ell$ and $\sum_{j=1}^t |V_j| = n$, then $\sum_{j=1}^{t} \left( \frac{\ell_j}{\ell} - \frac{|V_j|}{n} \right) = 0$.
	Therefore, $0 > \sum_{j=1}^{t} \left( \frac{\ell_i}{r\ell} - \frac{|V_j|}{rn} \right) = t/\ell + \sum_{j \neq i} \left( \frac{\ell_j}{\ell} - \frac{|V_j|}{n} \right)$, and thus
	$-t/\ell > \sum_{j \neq i} \left( \frac{\ell_j}{\ell} - \frac{|V_j|}{n} \right)$.
	By averaging, there must exist $j \neq i$ such that $-1/\ell > \frac{\ell_j}{\ell} - \frac{|V_j|}{n}$.
	Adding one to $\ell_j$ and substracting one from $\ell_i$ yields a new valid sequence where the objective sum is strictly minimised, a contradiction.
	The case $r = k$ follows in the same fashion, by using \ref{item:cycleimbalance-r-sum}.
	
	An analogous argument shows that $\frac{\ell_i}{\ell} - \frac{|V_i|}{n} \ge -t/\ell$ for all $i$ as well,
	thus we have
	\begin{align}
		\sum_{i=1}^{rt} \left| \frac{\ell_i}{\ell} - \frac{|V_i|}{n} \right| \leq \frac{r t^2}{\ell} \leq \xi.
		\label{equation:lemma-for-C-correcting-balance}
	\end{align}
	Let $G'$ be the graph obtained from $R$ by blowing-up each $i \in V(R)$ by $\ell_i$ vertices,
	let $\fV'$ be the natural associated partition of $V(G')$ where $|V'_i| = \ell_i$ for all $i$;
	and let $\fU'$ be the $(\ell, r)$-partition of $V(G')$ where each cluster $U' \in \fU'$ corresponds to the blown-up vertices of $R$ that all lie in the same cluster of the partition $\fW$ of $V(R)$.

	We check the hypothesis of \cref{lem:lemma-for-C} hold with $G', \fV', \fU', R, r, t, \ell, 2 \eps, \mu/2, 2 \alpha, \pi, \cJ'$ playing the roles of $G, \fV, \fU, R, r, t, n, \eps, \mu, \alpha, \pi, \cJ'$, respectively.
	By construction, $\fV'$ is an $\fU'$-refining partition.
	Since $\fV = \{V_1, \dotsc, V_{rt}\}$ is a $(1+\eps)$-balanced partition, together with \cref{equation:lemma-for-C-correcting-balance} we deduce that $\fV'$ is a $(1 + 2 \eps)$-balanced partition.
	By choice, $(G', \fV')$ is an $R$-partition.
	If $\fW'$ is the $(t,r)$-sized partition of $R$ induced by $(\fV', \fU')$, note that $\fW' = \fW$.
	By assumption, $(R,\cJ)$ is a $(\mu/2)$-robust $\fW$-partite $k$-uniform Hamilton framework which is $(\mu/2, 2/n, \fW)$-cluster-matchable;
	and if $r = 1$ we also assume $(R,\cJ)$ is $(\mu/2)$-robust aperiodic.
	In particular, since $\ell \leq n$, $(R,\cJ)$ is $(\mu/2, 2/\ell, \fW)$-cluster-matchable.
	By assumption, $\cJ' \subseteq \cJ$ is a tight Hamilton cycle, which satisfies $\Delta(\cJ') \leq k \leq k^2+1$.

	The application of \cref{lem:lemma-for-C} yields a size-compatible (with $\fV'$) vertex partition $\fX = \{X_i\}_{i=1}^{rt}$ of the vertex set of the $(k-1)$st power of a Hamilton cycle $C=C_{k, r \ell}$,
	a family of subsets $\tfX = \{\tX_i\}_{i=1}^{rt}$,
	and a family of $(k-1)$st powers of paths $\{ P_K \}_{K \in E(\cJ)}$ which are subgraphs of $C$, such that
	\begin{enumerate}[\upshape ({C}'1)]
		\item \label{itm:allocate-vertices-partition'} $(C,\fX)$ is an $R$-partition,
		\item \label{itm:allocate-vertices-buffer'} $\tfX$ is a $(2\alpha, R')$-buffer for $(C, \fX)$  where $R' = \partial_2 \cJ'$,
		\item \label{itm:allocate-vertices-paths'} for each $K \in E(\cJ)$, ${P}_K$ has length $\pi \ell$ and ${P}_K \subset C[\bigcup_{i  \in K} X_i]$, and
		\item \label{itm:allocate-vertices-disjoint'} the paths $\{ P_K \}_{K \in E(\cJ)}$ are pairwise vertex-disjoint.
	\end{enumerate}
	
	Note that the partition $\fX$ defines a natural homomorphism $\phi\colon V(C) \rightarrow V(R)$, and we can assume (after relabelling) that $\phi^{-1}(i) = X_i$.
	Now \cref{item:imbalanced-correcting} follows from \cref{equation:lemma-for-C-correcting-balance}.
	It remains to obtain \cref{item:imbalanced-buffer} and \cref{item:imbalanced-forward}.
	
	We focus on \cref{item:imbalanced-forward} first, by defining the sets $Z_i$.
	Consider an arbitrary $i\in V(R)$ and note that $K=\{ i-k+1, i-k+2, \dotsc, i \}$ is an edge in $\cJ' \subset \cJ$.
	Hence by \ref{itm:allocate-vertices-paths'}--\ref{itm:allocate-vertices-disjoint'}, there is a $(k-1)$st power of a path ${P}_K \subset C[\bigcup_{i  \in K} X_i]$ of length $\pi \ell \geq 2 k \pi \ell_i = 2 k \pi |\phi^{-1}(i)|$ (in the inequality we used $\ell \geq 2k \ell_i$, which follows from \eqref{equation:lemma-for-C-correcting-balance}).
	We therefore may choose, for each $i$, a set $Z_i \subseteq \phi^{-1}(i) \cap P_K \subseteq V(C_{k,r \ell})$ of size $\pi |\phi^{-1}(i)|$, and in such a way all $Z_i$ are vertex-disjoint.
	For each $x \in Z_i$ and a neighbour $y$ of $x$ in the cycle $C$, we have that $y \in P_K$ and also $\phi(y) \neq i$ (since $\phi(x) = i$, and $x, y$ are neighbours), thus $\phi(y) \in K \setminus \{ i \} = \{ i - k + 1, i - k + 2, \dotsc, i-1 \}$, as required.
	
	Now we show \cref{item:imbalanced-buffer}.
	From the existence of $\tfX$ and \ref{itm:allocate-vertices-buffer'}, we have that the sets $\{ \tX_i \}_{i=1}^{rt}$ form a $(2 \alpha, R')$-buffer.
	We are almost done, except for the fact that \cref{item:imbalanced-forward} requires that each set $Z_i$ is disjoint from $\tX_i$.
	This is achieved by replacing each $\tX_i$ with $\tX'_i := \tX_i \sm Z_i$.
	It can be easily checked that this turns the $(2\alpha, R')$-buffer $\tfX$ into an $(\alpha, R')$-buffer $\{ \tX'_i \}_{i=1}^{rt}$, and thus we are done.
\end{proof}

\subsection{Proof of the \texorpdfstring{\nameref{lem:lemma-for-H}}{Lg}}\label{sec:proof-lemma-for-H-proof-itself}
Finally, we show the main result of \cref{sec:lemmaforH}.
\begin{proof}[Proof of \cref{lem:lemma-for-H}]
	We carry out the proofs for the cases \ref{itm:LH-r=1}~to~\ref{itm:LH-r=k}  simultaneously.
	Initially, assume $H$ admits an ordering with bandwidth at most $\beta n$ and a $(z, \beta)$-zero-free 
	$(k+1)$-colouring $\chi \colon V(H) \rightarrow \{0, 1, \dotsc, k\}$.
	If $r=k$, we will furthermore assume that $\chi$ is actually an equitable $k$-colouring.
	We will also assume, for now, that $(R,\cJ)$ is only a Hamilton framework (not zero-free).

	Then, the proof has four steps.
	At the end of Step 3 we will be done with the embedding if $H$ has no zero-coloured vertices under $\chi$, and this will cover cases \ref{itm:LH-r=1}~and~\ref{itm:LH-r=k} of \cref{lem:lemma-for-H}.
	In Step 4, we will introduce the extra hypothesis of zero-freeness to $(R,\cJ)$ to continue, and this will finish the proof in all cases. \medskip

	\noindent
	\emph{Step 1: Finding an imbalance-correcting cycle in $R$.}
	Recall that $\cJ' \subseteq \cJ$ is a tight Hamilton cycle whose cyclic order we assume to be $1, 2, \dotsc, rt$.
	Let $R' = \partial_2 \cJ'$.

	First, apply \cref{lemma:lemma-for-C-correcting} with $\xi^2$, $3\pi$ and $3 \alpha$ in place of $\xi$, $\pi$ and $\alpha$, respectively, to obtain $\ell \leq 2 r  t^2 \xi^{-2}$ coprime to $k$ and a homomorphism $\phi_3\colon V(C_{k, r \ell}) \rightarrow V(R)$ satisfying
	\begin{enumerate}[({J}1)]
		\item \label{item:lemmaforH-balancephi3} for every $i \in V(R)$,
		      \begin{equation}
			      \frac{|\phi^{-1}_3(i)|}{r\ell} = \frac{|V_i|}{rn} \pm \xi^2,
			      \label{equation:lemmaforH-k-phi3}
		      \end{equation}
		\item \label{item:lemmaforH-bufferphi3} for every $i \in V(R)$ there exists $\tilde{Y}_i \subseteq \phi^{-1}_3(i) \subseteq V(C_{k, \ell})$ such that $|\tilde{Y}_i| \ge 3 \alpha |\phi^{-1}_3(i)|$ and for each $x \in \tilde{Y}_i$ and $xy, yz \in E(C_{k, \ell})$, with $\phi_3(y) = j$ and $\phi_3(z) = l$, we have $ij, jl \in E(R')$, 
		\item \label{item:lemmaforH-forward3} for every $i \in V(R)$, there exists $\tilde{Z}_i \subseteq \phi^{-1}_3(i) \setminus \tilde{Y}_i$ such that $|\tilde{Z}_i| \ge 3\pi |\phi^{-1}_3(i)|$
		      and for each $z \in \tilde{Z}_i$, $y \in N_{C_{k, \ell}}(z)$ we have $\phi_3(y) \in \{ i - k + 1, i - k + 2, \dotsc, i - 1 \}$.
	\end{enumerate}\medskip

	\noindent
	\emph{Step 2: Preparing $H$.}
	For each $0 \leq i \leq k$, let $H_i \subseteq H$ be the set of $i$-coloured vertices of $H$ under $\chi$.
	Let $H_{-0} = V(H) \setminus H_0$.
	Note that $H_0$ could be empty.
	Let $q = \lceil \beta^{-1} \rceil + k - 1$.
	Recall that $P_{k,q}$ is the $(k-1)$st power of a path with vertices $p_1, \dotsc, p_q$,
	and $P^+_{k,q}$ is obtained from $P_{k,q}$ by adding an universal vertex $p^+$.
	For each $1 \leq j \leq k$, let $P_j = \{ p_i \colon i \in j+k\mathbb{N} \}$.
	
	By \cref{lemma:Htopowpath}, applied here with $|V(H)| = rn$ in place of $n$, there exists a homomorphism $\phi_1\colon V(H) \rightarrow V(P^+_{k,q})$ such that
	\begin{enumerate}[(P1)]
		\item $\phi_1^{-1}(\{ p^+ \}) = H_0$,
		\item \label{item:zeroblock-colourclass} for each $1 \leq j \leq k$, $\phi^{-1}_1(P_j) = H_j$,
		\item \label{item:zeroblock-bounded} $\phi_1$ restricted to $H_{-0}$ is an at most $(k \lceil \beta r n \rceil)$-to-one homomorphism to $P_{k,q}$, and
		\item \label{item:zeroblock} there exists $s \ge 0$ and indices $1 \leq i_1 \leq i_2 \leq \dotsb \leq i_s \leq q$ such that
		      \begin{enumerate}[{\upshape (i)}]
			      \item \label{item:zeroblock-wellapart} $i_{j+1} - i_j \ge z$ for all $1 \leq j < s$,
			      \item \label{item:zeroblock-zeroadjacent} for every $v \in V(H)$ adjacent to a vertex in $H_0$, there exists $1 \leq j \leq s$ such that $\phi_1(v) \in \{ p_{i_j-1}, \dotsc, p_{i_j+k} \}$.
		      \end{enumerate}
	      	Further, if $H_0 = \emptyset$, then $s = 0$.
	\end{enumerate}

	Now, note that if we are in the $r=k$ case, then $\chi$ is assumed to be an equitable $k$-colouring.
	This implies that there are no zero-coloured vertices, and each colour class has $n$ vertices.
	Thus from \ref{item:zeroblock-colourclass}, we get:

	\begin{enumerate}[(P1),resume]
		\item  \label{itm:lem-H-equitable} If $r=k$, then $\phi_1$ maps $n$ vertices to $P_j$, for each $1 \leq j \leq k$.
	\end{enumerate}
	Also, note that \cref{item:zeroblock}\ref{item:zeroblock-wellapart} implies in particular that $s \leq qz^{-1}+1$.
	From the definition of $(z, \beta)$-zero-free $(k+1)$-colourings we can deduce that
	\begin{enumerate}[(P1),resume]
		\item \label{item:zeroblock-zerofew} $|H_0| \leq s \lceil \beta n \rceil \leq (qz^{-1}+1)\lceil \beta n \rceil \leq 2 z^{-1}n \leq \xi^2 n$,
	\end{enumerate}
	where we used $1/z \ll \xi$ in the last inequality.
	
	Let $\omega = kt^{k}$.
	To allow an easier description of the embedding later, we will extend $P^+_{k,q}$ by adding dummy vertices at the beginning and the end of the power of path.
	More precisely, we consider new vertices $p_{0}, p_{-1}, \dotsc, p_{-\omega + k}$ and $p_{q + 1}, \dotsc, p_{q+\omega}$, and obtain a new graph $P^\omega_{k,q}$.
	We do this in such a way that $V(P^\omega_{k,q}) \setminus \{ p^+ \}$ is a $(k-1)$st power of a path on $q + 2 \omega -k + 1$ vertices according to the ordering $p_{- \omega + k} p_{- \omega + k + 1} \dotsb p_{-1} p_{0} p_1 \dotsb p_{q-1} p_q p_{q+1} \dotsb p_{q+ \omega}$.
	Since $P^\omega_{k, q}$ contains $P^+_{k,q}$, we can trivially extend the embedding $\phi_1 \colon V(H) \rightarrow V(P^+_{k,q})$ to an embedding $\phi^\omega_1 \colon V(H) \rightarrow V(P^\omega_{k,q})$, where the new vertices do not receive the image of any vertex of $H$.

	Now we partition $H$ in vertex-disjoint subgraphs which we will allocate separately.
	For each $a,b \in \mathbb{Z}$ with $- \omega + k - 1 \leq a \leq b \leq q + \omega + 1$, we let $H_{[a,b]} = (\phi_1^\omega)^{-1}( \{ p_i \colon a \leq i \leq b \} ) \subseteq H_1$.
	In words, $H_{[a,b]}$ consists of the preimages of $b-a+1$ consecutive vertices of $P^\omega_{k,q}$, those from $p_a$ to $p_b$.
	
	For each $1 \leq j \leq s$ we define $a_j, b_j, c_j, d_j \in \mathbb{Z}$ as
	\begin{align*}
		b_j & = i_j + k - 2,   &
		a_j & = b_j + 1 - \omega \\
		c_j & = i_j + 1,              &
		d_j & = c_j - 1 + \omega,
	\end{align*}
	These numbers and tuples were selected so that the intervals $[a_j, b_j]$ and $[c_j, d_j]$ have $\omega$ elements each;
	and the last $k$ elements of $[a_j, b_j]$ (resp. the first $k$ elements of $[c_j, d_j]$) correspond to the first (resp. last) $k$ elements of {$\{ i_j - 1, \dotsc, i_j+  k \}$.}
	Since $1/z \ll 1/t, 1/k$, we have $z > 2 \omega - k + 4$.
	Together with \cref{item:zeroblock}\ref{item:zeroblock-wellapart}, we deduce that for each $1 \leq j \leq s$, the intervals $[a_j, d_j]$ are pairwise disjoint.
	Also, note that since $1 \leq i_j \leq q$ holds for every $1 \leq j \leq s$, the intervals $[a_j, b_j]$ and $[c_j, d_j]$ are completely contained in $[- \omega + k, q + \omega]$.
	This implies that $\{ p_i \colon i \in [a_j, b_j]  \}$ and $\{ p_i \colon i \in [c_j, d_j] \}$ are well-defined sets of vertices in $P^\omega_{k,q}$.

	For each $1 \leq j \leq s$, let
	\begin{align}
		\text{$F_j = H_{[a_j+k, d_{j}-k]}$ and $H' = H_{-0} \setminus \bigcup_{1 \leq j \leq s} F_j$.}
	\end{align}
	Note that $H', F_1, \dotsc, F_s$ and $H_0$ partition the vertex set of $H$,
	and $H_0$ contains all zero-coloured vertices of $H$.
	\medskip

	\noindent
	\emph{Step 3: Embedding the non-zero coloured vertices.}
	We will embed $H'$ first using \rf{lemma:lemma-for-H-balanced} and so we proceed to check the necessary hypothesis.
	Let $r n'$ be the number of vertices in $H'$.
	By construction, for each $1 \leq j \leq s$, the set $F_j$ corresponds to the preimage of (at most) $2 \omega + 2k$ vertices of $P^\omega_{k,q}$.
	By \cref{item:zeroblock-bounded}, each vertex of $P_{k,q}$ is the image of at most $k \lceil \beta r n \rceil$ vertices of $H$ via $\phi_1$, and this also holds for $\phi^\omega_1$.
	Thus we have, for each $1 \leq j \leq s$, that $|F_j| \leq (2 \omega + 2k)k\lceil \beta r n \rceil$.
	Using this, we can write the following inequalities (with explanations to follow):
	\begin{align*}
		\sum_{j=1}^s |F_j|
		& \leq s (2 \omega + 2k)k\lceil \beta r n \rceil
		\leq 2 s (2 \omega + 2k)k^2 \beta n \\
	 	& \leq 4 q z^{-1} (2 \omega + 2k)k^2 \beta n 
	 	\leq 8 z^{-1} (2 \omega + 2k)k^2 n \leq \xi^2 n.
	\end{align*}
	Here, in the second inequality we used $r \leq k$ and used $\lceil \beta k n \rceil \leq 2 \beta k n$;
	in the third inequality we used $s \leq qz^{-1} + 1 \leq 2 q z^{-1}$, which, as noted before, follows from \cref{item:zeroblock}\ref{item:zeroblock-wellapart}.
	In the fourth inequality we recalled that $q = \lceil \beta^{-1} \rceil + k - 1$ and $\beta \ll 1/k$ to deduce $q \beta \leq 2$.
	Finally, in the last inequality we used $1/z \ll \xi, 1/k, 1/t$.
	
	Together with \cref{item:zeroblock-zerofew}, we deduce
	\begin{equation}
		n' = \frac{|H'|}{r} = n - |H_0| - \sum_{j=1}^s |F_j|
		\ge (1 - \xi^2)n - \sum_{j=1}^s |F_j| \ge (1 - 2 \xi^2)n,
		\label{equation:lemmaforH-finalbound}
	\end{equation}
	Since $\ell \leq 2 r t^2 \xi^{-2}$, $q = \lceil \beta^{-1} \rceil + k - 1$ and $\beta \ll \xi, 1/t$,
	we can assume the hierarchy $1/n' \ll \beta, 1/q \ll \xi, 1/\ell, 1/k$ is satisfied.
	Also, $\ell$ and $k$ are coprime.
	Let $\phi_1'$ be the restriction of $\phi^\omega_1$ to $H'$.
	Note that $\phi_1'$ is a homomorphism from $H'$ to $P^\omega_{k,q}$ which is at most $(k \lceil \beta r n \rceil)$-to-one.
	Hence, we deduce that $\phi_1'$ is at most $(k \lceil 2 \beta r n' \rceil)$-to-one.
	Finally, we recall that if $r=k$ then there are no zero-coloured vertices, so $H_0 = \emptyset$ and thus $s = 0$ by \ref{item:zeroblock}.
	In particular, if $r=k$, then we have $H' = H$ and thus $\phi'_1 = \phi_1$ and $n' = n$.
	Therefore, by \cref{itm:lem-H-equitable}, it follows that if $r=k$, then $|(\phi'_1)^{-1}( P_j )| = n'$ holds for each $1 \leq j \leq k$.
	
	Thus we can apply \cref{lemma:lemma-for-H-balanced} with $n', 2 \beta, \xi^{2}, H', \phi_1', \ell, q + 2 \omega + k + 3$ in place of $n, \beta, \xi, H, \phi_1, t, q$ respectively, to obtain a homomorphism $\phi'_2 \colon V(P^\omega_{\smash{k,q}}) \rightarrow V(C_{k,r \ell})$ such that the composition $\phi_2 = \phi'_2 \circ \phi'_1$ is an homomorphism from $H'$ to $C_{k, r \ell}$ which is $( (1 \pm \xi^2) \frac{n'}{\ell} )$-to-one.
	Since $n' = (1 \pm 2\xi^2)n$, it follows that $\phi_2$ is $( (1 \pm 3 \xi^2) \frac{n}{\ell} )$-to-one.

	Then $\Phi = \phi_3 \circ \phi_2$ is a homomorphism from $H'$ to $R$.
	This naturally defines a partition $\fX'$ of $V(H')$ by $X'_i = \Phi^{-1}(i)$ for each $1 \leq i \leq rt$,
	and so we have that
	\begin{enumerate}[(H1)]
		\item  $(H', \fX')$ is an $R$-partition.
	\end{enumerate}

	For all $1 \leq i \leq rt$, let $\tX'_i=\phi_2^{-1}(\tilde{Y}_i)$ with $\tfX' = \{ \tX'_i \}_{i=1}^{rt}$,
	and let $Z_i=\phi_2^{-1}(\tilde{Z}_i)$ with $\fZ = \{\tilde{Z}_i\}_{i=1}^{rt}$  where the sets $\tilde{Y}_i, \tilde{Z}_i$ are from \cref{item:lemmaforH-bufferphi3,item:lemmaforH-forward3}.
	From \cref{item:lemmaforH-balancephi3,item:lemmaforH-bufferphi3,item:lemmaforH-forward3}, the fact that $\phi_2$ is  $( (1 \pm 3 \xi^2) \frac{n}{\ell} )$-to-one and $\xi \ll \pi,\alpha$, we deduce that
	\begin{enumerate}[(H1), resume]
		\item \label{item:lemmaforH-Phi-2} $\tfX'$ is a $(2 \alpha , R')$-buffer for $(H', \fX')$,
		\item \label{item:lemmaforH-Phi-3} for each $1 \leq i \leq rt$, $Z_i \subseteq X'_i \setminus \tX'_i$ is such that $|Z_i| \geq 2 \pi |X'_i|$,
		      and for each $z \in Z_i$, $N_H(z) \subseteq X_{i-1} \cup \dotsb \cup X_{i-k+1}$, and
		\item \label{item:lemmaforH-Phi-1} for each $i \in V(R)$, $|X'_i| = |V_i| \pm 4 \xi^2 n$.
	\end{enumerate}

	At this point of the proof we stop to conclude if the input graph $H$ was $k$-colourable, which covers cases \ref{itm:LH-r=1} and~\ref{itm:LH-r=k} of \cref{lem:lemma-for-H}.
	Indeed, if $H$ is $k$-colourable then $H_0 = \emptyset$ and $s = 0$.
	Therefore, $H' = H$.
	Hence, $\Phi$ is a homomorphism from $H$ to $R$ and $\fX'$ is a partition of $V(H)$ which satisfies the required properties.
	\medskip

	\noindent
	\emph{Step 4: Embedding the zero-coloured vertices.}
	From now on, we can assume $H' \neq H$ and that $(R,\cJ)$ is actually a $(\mu/2)$-robust zero-free Hamilton framework.
	In particular (since $H_0 \neq 0$), we can, and will, assume that in case \ref{itm:LH-r=1-zero-free} of Lemma~\ref{lem:lemma-for-H}.
	This implies that $r = 1$,
	and also that there exists a copy of a $(k+1)$-clique $K$ in $\cJ$.
	Let $\{g_1, \dotsc, g_{k+1} \}$ be its vertices.
	Recall that we need to extend the homomorphism $\Phi$ to allocate the remaining vertices, which are in $H_0$ and $F_1, \dotsc, F_s$.
	
	Fix $1 \leq j \leq s$.
	The homomorphism $\Phi  = \phi_3 \circ \phi_2\colon V(H') \to V(R)$ has already allocated $H'$, including the vertices in $H_{[a_j, a_j+ k-1]}$ and $H_{[d_j - k + 1, d_j]}$ in $R$.
	Recall that, using the definition introduced in Step 2, $H_{[a_j, a_j + k - 1]}$ corresponds to $k$ vertex-disjoint (possibly empty) subsets of vertices $V_{a_j}, \dotsc, V_{a_j + k - 1}$ in $H'$.
	To be precise, for each $0 \leq i < k$, the subset $V_{a_j + i}$ is the preimage via $\phi^\omega_1$ of the vertex $p_{a_j + i} \in V(P^\omega_{k,q})$.
	Since, the vertices in $H_{[a_j, a_j+ k-1]}$ correspond to the preimage of $k$ consecutive vertices of $P^\omega_{k,q}$ via $\phi^\omega_1$, and by definition of $\phi_2$, it follows that the sets $V_{a_j}, \dotsc, V_{a_j + k - 1}$ are embedded via $\phi_2$ to $k$ consecutive vertices of $C_{k, \ell}$.
	Then, the sets $V_{a_j}, \dotsc, V_{a_j + k - 1}$ are then embedded via the homomorphism $\phi_3$ to $k$ vertices in $R$ forming a $k$-clique.
	The conclusion is that there exist vertices $v_{a_j}, v_{a_j + 1}, \dotsc, v_{a_k + k-1} \in V(R)$, which form an edge of $R$, and such that $\Phi(V_{a_j + i}) = v_{a_j + i}$ holds for each $0 \leq i < k$.
	Let $e_1 = (v_{a_j}, \dotsc, v_{a_j + k-1}) \in V(R)^{k}$ be the corresponding ordered edge of $R$.
	Similarly, we can find an ordered edge $e_4 = (v_{d_j - k + 1}, \dotsc, v_{d_j}) \in V(R)^k$ such that, if $V_{d_j - k + 1}, \dotsc, V_{d_j} \subseteq V(H')$ are the $k$ disjoint subsets corresponding to the preimages of $k$ consecutive vertices of $P^\omega_{k,q}$ via $\phi^\omega_1$, and $\Phi(V_{d_j - k + 1 + i}) = v_{d_j - k + 1 + i}$ holds for all $0 \leq i < k$.

	Let $e_2 = (g_{k-1}, g_k, g_1, \dotsc, g_{k-2})$ and $e_3 = (g_1, \dotsc, g_k)$, both $e_2, e_3$ are ordered edges in $R$ (since they are in the $(k+1)$-clique $K$).
	By \cref{proposition:usingtheclique}, $R$ contains a walk $W_1 = ( w_1 ,\dots , w_{\omega})$  from $e_{1}$ to $e_2$, and a walk $W_2=(w'_1, \dots, w'_{\omega})$  from $e_3$ to $e_{4}$  where $\omega=kt^{k}$.
	We allocate $H_{[a_j, b_j]}$ in $R$ by assigning $(\phi_1^\omega)^{-1}(p_{a_j + i - 1})$
	to $w_i \in V(R)$, for all $1 \leq i \leq \omega$,
	and we allocate $H_{[c_j, d_j]}$ in $R$ by assigning $(\phi_1^\omega)^{-1}(p_{c_j + i - 1})$
	to $w'_i \in V(R)$, for all $1 \leq i \leq \omega$.
	Our choice of $a_j, c_j, e_{1}, e_2, e_3, e_4$ and $W_1, W_2$ ensures that this allocation is well-defined.
	Doing this with all $F_j$, this extends the embedding of $H'$ to an embedding of all of $H - H_0$.
	It only remains to allocate the vertices in $H_0$, we do that as follows.

	By~\cref{item:zeroblock}\ref{item:zeroblock-zeroadjacent}, the vertices in $H_0$ are joined, in $H$, only to vertices which are located in $(\phi^\omega_1)^{-1}( \{ p_{i_j - 1}, \dotsc, p_{i_j + k}  \} )$ for $1 \leq j \leq s$.
	Recall that, by choice, the interval $\{ i_j - 1, \dotsc, i_j + k \}$ is in the union of the last $k$ elements of $[a_j, b_j]$ and the first $k$ elements of $[c_j, d_j]$.
	By the allocation of $F_j$, these vertices are allocated to vertices which are all in $e_2$ and $e_3$, which are located inside $\{g_1, \dotsc, g_k\}$.
	Thus we get a valid allocation of $H$ to $R$ by setting $\Phi(H_0) = g_{k+1}$, since $\{g_1, \dotsc, g_{k+1}\}$ forms a $(k+1)$-clique and thus $g_{k+1}$ is a neighbour (in $R$) of all $\{g_1, \dotsc, g_k\}$.
	
	Again, the final allocation naturally defines a vertex-partition $\fX$ of $H$ by setting $\Phi^{-1}(i)$ for each $1 \leq i \leq t$.
	Finally, by inequality~\eqref{equation:lemmaforH-finalbound} we deduce that the above steps change the number of vertices allocated to any given cluster by at most $2 \xi^2 n$ (from the number of vertices allocated when only the embedding of $H'$ was defined).
	In combination with \cref{item:lemmaforH-Phi-2,item:lemmaforH-Phi-3,item:lemmaforH-Phi-1}, it is straightforward to see that this final allocation satisfies all the required properties of \cref{lem:lemma-for-H}.
\end{proof}

\section{Conclusion}\label{sec:conclusion}
In this paper, we investigated a set of general properties that guarantee the existence of (powers of) Hamilton cycles as well as containment of large graphs of sublinear bandwidth.
As an application we recovered several classic results and also proved new results.
Besides getting a clearer picture of what Hamiltonicity in dense graphs is about, we hope that our work will pave the way for future results in this line of research.
We finish with some reflections and present a few problems that appear to be within reach.

\subsection{Chvátal-type degree conditions}
It would be very interesting to characterise the degree sequences that witness the existence of the $(k-1)$st power of a Hamilton cycle.
For $k=2$, this was done by Chvátal~\cite{Chv72}, who showed that a degree sequence $d_1 \leq \dots \leq d_n$ that satisfies $d_i \geq i+1$ or $d_{n-i} \geq n - i$ for every $i < n/2$ guarantees a Hamilton cycle.
Moreover, for every sequence $a_1 \leq \dots \leq a_n$ such that $a_h \leq h$ and $a_{n-h} \leq n-h-1$ for some $h < n/2$ there exists a non-Hamiltonian graph whose degree sequence $d_1 \leq \dots \leq d_n$ satisfies $d_i \ge a_i$ for all $1 \leq i\leq n$.
Up to an error term, these conditions were extended to ensure the existence of bipartite spanning graphs of sublinear bandwidth by Knox and Treglown~\cite{KT13} (as a corollary of \cref{thm:knoxtreglown-robustexpanders}).

For $k \geq 3$ much less is known.
Balogh, Kostochka and Treglown~\cite[{Question 19}]{BKT11} asked whether a degree sequence $d_1 \leq \dotsb \leq d_n$, satisfying
$	d_i \geq  \frac{k-2}{k} n  + i    $  or $ d_{n-i(k-1)+1} \geq   n  - i$ for all $i \leq n/k$,
would guarantee the existence of a $k$-clique factor.
They also provided constructions showing that the value in the second part of the condition cannot be lowered.
It is conceivable that graphs which follow this degree sequence, together with an additional additive error term, allow for a Bandwidth Theorem.
The following conjecture would be a first step in this direction.

\begin{conjecture}
	For any $\mu > 0$ and $n_0 \in \mathbb{N}$ such that the following holds for every $n \geq n_0$.
	Let $G$ be a graph with degree sequence  $d_1 \leq \dots \leq d_n$.
	Suppose that
	$d_i \ge \frac{k-2}{k} n  + i +\mu n$ or $d_{n-i(k-1)+1} \ge  n  - i +\mu n$
	for every $i \leq n/k$.
	Then $(G,H)$ is a zero-free $k$-uniform Hamilton framework  where $H =K_k(G)$.
\end{conjecture}

\subsection{Multipartite graphs}
\cref{theorem:multipartite-keevashmycroft} provides the minimum degree condition for tight Hamilton cycles in multipartite graphs, where the minimum degree is defined relative to each of the clusters.
Do other theorems have partite versions?
There are Pósa-type~\cite{MM63} and Chvátal-type~\cite[Corollary 20]{Let19} degree conditions which ensure Hamilton cycles in balanced bipartite graphs.
If the degree conditions on those theorems is augmented by an $o(n)$ term and the host graphs taken to be large enough, then our results on frameworks apply (namely, we can apply \cref{thm:main-bandwidth} with $r=k=2$), and the output is strengthened from Hamiltonicity to the containment of bipartite graphs of sublinear bandwidth.
It is certainly conceivable that these results can be extended to powers of cycles in multipartite graphs.
More generally, {it would be interesting to determine} under which general conditions the multi-partite and ordinary form of such problems are equivalent.

In a different direction than \cref{theorem:multipartite-keevashmycroft}, Lo and the second author~\cite[Corollary 1.5]{LS2019} showed a result for $k$-clique factors in balanced $k$-partite graphs where the overall minimum degree is controlled (instead of the minimum degree of every vertex to every other cluster).
\begin{theorem}
	Let $1/n \ll 1/k, \mu$ with $k \geq 2$.
	Let $G$ be a balanced $k$-partite graph on $kn$ vertices with $\delta(G) \geq (k - 3/2 + \mu )n$.
	Then $G$ has a $k$-clique factor.
\end{theorem}
The result is best possible apart from the $\mu n$ term.
It seems sensible to expect this result could be combined with \cref{thm:main-bandwidth} to prove such graphs also contain powers of Hamilton cycles, or more generally, are partite $(\beta, \Delta, k)$-Hamiltonian.
However, we were unable to decide if $K_k(G)$ is tightly connected for such graphs.

\subsection{Robust expansion}
Robust Hamilton frameworks provide a generalisation of robust expansion in the following sense: a graph which is a robust expander and has large minimum degree contains Hamilton cycles and certain spanning bipartite graphs (as ensured by \cref{thm:knoxtreglown-robustexpanders}), where as in graphs with $k$-uniform robust Hamilton frameworks we can ensure the existence of powers of cycles and certain $k$-colourable graphs. 
The concept of robust expanders has proven to be useful to tackle problems about Hamilton cycles that go beyond their simple existence, we survey a few of them next.

Firstly, robust expanders admit a natural extension to digraphs~\cite{KOT2010} and have been used to find directed Hamilton cycles.
Secondly, robust expanders have been applied to problems of algorithmic nature, such as designing efficient parallel algorithms to find Hamilton cycles in dense graphs~\cite{CKKO12}, or problems concerning the domination ratio of the Asymmetric TSP problem~\cite[Corollary 1.5]{KO13}.
Thirdly, in some situations dense graphs that are \emph{not} robust expanders can be vertex-partitioned into a bounded number of parts, each of which is robust.
This can be applied to find Hamilton cycles under additional assumptions such as high vertex-connectivity~\cite{KLOS15} or vertex-transitivity~\cite{CHM14}.
Lastly, robust expanders have led to the resolution of long-standing problems about decompositions of graphs and digraphs into Hamilton cycles (see~\cite{KO14} for a comprehensive survey).
It would be interesting to find a generalisation of Hamilton frameworks that allows one to decompose into spanning graphs with chromatic number $k \geq 3$.
We note that some changes in the definition of frameworks are necessary, since the requirements for embedding and decomposition are known to be different in this situation.
For instance, the relative minimum degree threshold for finding a single triangle factor is $2/3$, while the threshold for finding a decomposition into triangle factors is at least $3/4$.
In the setting of minimum degree conditions, a first step in this direction was achieved by Condon, Kim, Kühn and Osthus~\cite{CKKO19} by approximately decomposing into spanning graphs of sublinear bandwidth. 

On a final note, it would be quite interesting to understand whether a characterisation for fractional matchings like the one in \rf{thm:tutte} exists for more complicated structures such as clique factors.

\subsection{Mitigating conflicts}
In many applications, the host graph is edge-coloured and we have to find a \emph{rainbow} copy of the guest graph, meaning that all its edges have distinct colours.
Apart from being a natural question on its own, this encodes many combinatorial problems such as the Ryser--Brualdi--Stein Conjecture~\cite{BR91} on transversals in latin squares and Ringel's conjecture~\cite{MPS20} about partitioning a graph into edge-disjoint copies of a tree.

Coulson and Perarnau~\cite{CP17} proved that any locally-bounded edge-colouring (i.e. an edge-colouring where every vertex sees at most $\mu n$ edges of any given colour, for some small $\mu > 0$) of a subgraph $G \subset K_{n,n}$ with $\delta(G) \geq n/2$ has a rainbow perfect matching.
One can therefore wonder, whether this is an instance of a more general phenomenon, namely that any Dirac-type problem asymptotically allows for a rainbow version?
For minimum degree conditions, this was confirmed by Glock and Joos~\cite{GJ20} who proved a Bandwidth Theorem in the rainbow setting.
We believe that our results can be generalised in a similar way.
\begin{conjecture}\label{con:rainbow-bandwidth}
	For $k,\Delta \in \NATS$, $\mu > 0$, and $r=1$ or $r=k$ there are $\nu,\beta, z$ and $n_0 \in \NATS $ with the following property.
	Let $G$ be as in the assumptions of \cref{thm:main-bandwidth}.
	Suppose that $G$ has in addition an edge-colouring such that every colour appears on at most $\nu n$ edges.
	Then $G$ contains rainbow copies of the graphs detailed in \cref{thm:main-bandwidth}~\ref{item:main-bandwidth-partite}--\ref{item:main-bandwidth-zerofree}.
\end{conjecture}

\subsection{Algorithmic aspects}
For fixed $k$, one can search for a $k$-uniform Hamilton framework in polynomial time (see \cref{sec:complexity}).
On the other hand, searching for a Hamilton cycle or any of its powers is $\text{NP}$-complete in general.
It is not clear whether one can efficiently search for (or even decide about the existence of) a $\mu$-robust $k$-uniform Hamilton framework.
However, it seems reasonable that our proof method could yield an efficient algorithm which finds a $(k-1)$st power of a cycle or a given $k$-colourable sublinear-bandwidth spanning graph, or otherwise returns that the input graph is not a $\mu$-robust $k$-uniform Hamilton framework.
This could be done using the algorithmic versions of the Regularity Lemma and the Blow-Up Lemma (see~\cite{KSS98c}), subject to the derandomisation of the parts where we rely on probabilistic arguments.

\subsection{Hypergraphs}
Finally, it is natural to ask whether these types of results generalise to hypergraphs.
On the practical side, first steps in this direction were undertaken by Schülke~\cite{Sch19} for tight cycles under codegree sequence conditions in $3$-uniform hypergraphs and by Bowtell and Hyde~\cite{BH20} for perfect matchings under minimum vertex degree sequence conditions in $3$-uniform hypergraphs.
On the theoretical side, the existence of Hamilton frameworks minimum degree conditions for tight Hamilton cycles has been proposed in our earlier work~\cite{LS20}.
We plan to revisit this line of research in the near future.

\subsection*{Acknowledgments}
We thank Andrew Treglown, Katherine Staden and Louis DeBiasio for helpful discussions and feedback.
Thanks to Mengjiao Rao for a careful read and corrections, and to Felipe Subiabre for his help with finding items of the bibliography.
Finally, we express our gratitude to the anonymous referees that gave us an extraordinary large number of insightful remarks, which greatly improved the presentation and correctness of the paper.

\bibliographystyle{amsplain}
\bibliography{bibliography.bib}

\appendix

\section{Constructions} \label{sec:constructions}
In this section, we discuss a few constructions that in one way or the other suggest that our results and definitions are close to certain obstacles.

\subsection{Pósa-type conditions}\label{sec:posa-lower-bound}

In this subsection we show that the degree sequence condition of \rf{thm:bandwidth-posa} is almost tight.

\newcommand{\bkt}{B}

\begin{proposition} \label{proposition:extremalexample}
	Let $k \ge 3$ and let $n$ be sufficiently large and divisible by $k$.
	Then there exists a graph $G$ on $n$ vertices whose degree sequence $d_1 \leq \dots \leq d_n$ satisfies $d_i \ge \frac{k-2}{k}n  + i + \frac{n^{1/2}}{12 k}$,
	for $i \leq n/k$,
	and does not contain the $(k-1)$st power of a Hamilton cycle.
\end{proposition}

The proof of \cref{proposition:extremalexample} relies on the following construction of Balogh, Kostochka and Treglown~\cite{BKT11} which covers the case of $k = 3$.

\begin{construction} \label{construction:BKT}
	Let $n$ be divisible by $3$, and let $c_n = \lfloor n^{1/2} / 12 \rfloor$ and $k_n = \lceil n^{1/2} \rceil$.
	Let $\bkt_n$ be the graph on $n$ vertices consisting of three vertex classes $V_1 = \{v\}$, $V_2$ and $V_3$  where $|V_2| = n/3 + c_n + 1$ and $|V_3| = 2n/3 - c_n - 2$, and has the following edges:
	\begin{enumerate}[(i)]
		\item all edges from $v$ to $V_2$,
		\item all edges between $V_2$ and $V_3$, and all possible edges inside $V_3$,
		\item $k_n$ vertex-disjoint stars in $V_2$, each of size $\lfloor |V_2| / k_n \rfloor$ or $\lceil |V_2| / k_n \rceil$, which cover all of~$V_2$.
	\end{enumerate}
\end{construction}

\begin{proposition}[{\cite[Proposition 22]{BKT11}}]\label{prop:BKT}
	For $n$ sufficiently large, let $\bkt = \bkt_n$ and denote its degree sequence by $d_1 \leq \dots \leq d_{n}$.
	Then
	\begin{enumerate}[\upshape (1)]
		\item \label{item:BKTdegree} $d_i \ge n/3 + c_n + i$ for all $1 \leq i  \leq n/3$,
		\item \label{item:BKTobstacle} $\bkt [N(v)] = V_2$ does not contain any path of length $3$.
	\end{enumerate}
\end{proposition}

A necessary condition for a graph $G$ to have the square of a Hamilton cycle is that, for every $x \in V(G)$, $G[N(x)]$ contains a path of length $3$.
Thus \cref{prop:BKT}\ref{item:BKTobstacle} shows that $\bkt_n$ does not contain the square of a Hamilton cycle.
We use the graphs $\bkt_n$ to show \cref{proposition:extremalexample}.

\begin{proof}[Proof of \cref{proposition:extremalexample}]
	We can assume $k \ge 4$.
	Let $N = kn$ be divisible by $k$.
	Consider a complete $(k-2)$-partite graph with $k-3$ clusters of size exactly $n$ and one large cluster $U$ of size $3n$.
	Inside the large cluster $U$, we place a copy of the graph $\bkt_{3n}$.
	Let $\bkt$ be the resulting graph on $N$ vertices.

	Note that every vertex outside of $U$ has degree exactly $(k-1)n = (k-1)N/k$, and each vertex in $U$ has exactly $(k-3)n = (k-3)N/k$ neighbours outside $U$.
	From \cref{prop:BKT}\ref{item:BKTdegree} it is straightforward to check that if $d_1 \leq \dots \leq d_N$ is the degree sequence of $\bkt$, then $d_i \ge (k-2)N/k + c_{3n} + i \ge (k-2)N/k + N^{1/2} / (12 k) + i$ for all $i  \leq N/k$, as required.

	It remains to show that $\bkt$ does not contain the $(k-1)$st power of a Hamilton cycle.
	To this end, let $\{V_1, V_2, V_3\}$ be the partition of $U$ corresponding to the partition of $\bkt_{3n}$  where $V_1=\{v\}$.
	Observe that it suffices to show that $\bkt[N(v)]$ does not contain the $(k-2)$nd power of a path on $2k-2$ vertices.
	Let us assume that $P = (u_1, \dots ,u_{2k-2})$ is such a path and show how this results in a contradiction.

	Since $\bkt \setminus U$ is $(k-3)$-partite and $v$ has no neighbours in $V_3$, every $k$-clique containing $v$ must intersect $V_2$ in at least two vertices.
	Further, since $G[V_2]$ does not have triangles, every $k$-clique containing $v$ must intersect $V_2$ in precisely two vertices.
	Thus, every $(k-1)$-clique in $\bkt[N(v)]$ must intersect $V_2$ in exactly two vertices.
	Moreover, $k-1$ consecutive vertices in $P$ span a $(k-1)$-clique, and thus each consecutive $k-1$ vertices in $P$ must contain exactly two vertices in $V_2$.
	Let $a, b \in \{1, \dotsc, k-1\}$ be such that $a < b$ and $u_a, u_b \in V_2$.
	This implies that $u_{a+k-1}, u_{b+k-1} \in V_2$, and further both $u_b, u_{a+k-1}$ and $u_{a+k-1}, u_{b+k-1}$ are neighbours in $\bkt$ (since each of these pairs belong to a common $(k-1)$-clique of consecutive vertices of $P$).
	Thus $(u_a, u_b, u_{a+k-1}, u_{b+k-1})$ is a path of length $3$ in $V_2$.
	Since $\bkt[U]$ is a copy of $\bkt_{3n}$, this contradicts \cref{prop:BKT}\ref{item:BKTobstacle}.
\end{proof}

\subsection{Ore- and Pósa-type conditions}\label{sec:posa-ore-different}

It is an easy exercise to show that every graph satisfying the conditions of Ore's theorem also satisfies the conditions of Pósa's theorem.
For $k \geq 3$, however, a graph satisfying the conditions of \rf{thm:bandwidth-ore} does not necesarily satisfy the conditions of \rf{thm:bandwidth-posa}, so these results do not imply each other.
This is shown by the following construction.

\begin{proposition}
	For $k \geq 3$ and $1/n \ll \mu \ll 1/k$ there is graph $G$ on $n$ vertices such that
	\begin{enumerate}[{\upshape (i)}]
		\item \label{item:posaoredifferente-oreyes} for each $xy \notin E(G)$, $\deg(x)+\deg(y) \geq 2 \frac{k-1}{k}n + \mu n$, but
		\item \label{item:posaoredifferente-posano} there exists $i \leq n/k$ such that $d_i < \frac{k-2}{k}n + i$.
	\end{enumerate}
\end{proposition}

\begin{proof}
	Let $\beta = 5 \mu$ and $n$ sufficiently large.
	We also assume $\beta, \mu, n$ are chosen so that all numbers denoting cardinalities  in what follows are integers.
	Partition a set of $n$ vertices into three sets $X_0, X_1, X_2$ of sizes $|X_0| = n/(2k)$, $|X_1| = 3n/(2k) + \beta n$ and $|X_2| = n - |X_0| - |X_1|$.
	Form $G$ by adding every edge completely contained in $X_0$, $X_1$ or $X_2$ and
	add every edge between $X_2$ and its complement.
	Moreover, between $X_0$ and $X_1$ add a bipartite graph such that every vertex in $X_0$ is joined to exactly $d_0 := \beta n$ vertices in $X_1$,
	and every vertex in $X_1$ is joined to $d_1 := \beta n |X_0|/|X_1|$ vertices in $X_0$.
	Since $\beta \ll 1/k$ we can assume $d_1 \geq \beta n/4$.

	Every vertex in $X_0$ has degree $|X_0|-1 + |X_2| + d_0 = (k-2)n/k + n/(2k) - 1$.
	Every vertex in $X_1$ has degree $|X_1|-1 + |X_2| + d_1 \geq (k-1)n/k + n/(2k) + \beta n / 4 - 1$.
	Every vertex in $X_2$ has degree $n-1$.
	Thus the degree sequence consists of the vertices of $X_0, X_1, X_2$, in that order.
	Note that $|X_0| \leq n/k$ and the $|X_0|$th vertex in the degree sequence has degree $(k-2)n/k + n/(2k) - 1 < (k-2)n/k + |X_0|$, so \cref{item:posaoredifferente-posano} holds.
	On the other hand, the only non-edges consist of a pair of vertices $xy$ with $x \in X_0, y \in X_1$.
	For these pairs, we have
	\begin{align*}
		\deg(x)+\deg(y)
		 & = |X_0|-1 + |X_2| + d_0 + |X_1|-1 + |X_2| + d_1                              \\
		 & \geq 2 \frac{k-1}{k}n - 2 + \frac{\beta n}{4} \geq 2 \frac{k-1}{k}n + \mu n,
	\end{align*}
	so \cref{item:posaoredifferente-oreyes} holds, as required.
\end{proof}

\section{Remarks on frameworks}

In this appendix we give more details about the remarks on Hamilton frameworks made in Section~\ref{section:remarksonframeworks}.

\subsection{Robustly having a Hamilton framework does not imply Hamiltonicity}\label{sec:fragile-frameworks}
In the following we construct a graph $G$ in which every approximation $G' \subset G$ admits some (spanning) $3$-uniform Hamilton framework $(G',H')$, but which does not contain a square of a Hamilton cycle.
This explains why we require the Hamilton framework itself to be robust, as opposed to robustly containing a Hamilton framework.

Let $q$ be some large prime, and let $n=4q$.
Let $P_1$ and $P_2$ be cubes (third powers) of two paths on vertex sets $v_1,\dots,v_{n}$ and $v_{n+1},\dots,v_{2n}$  where the path ordering follows the labels.
Let $C$ be a cube of a cycle on $v_1,\dots,v_{2n}$ and vertex sequence $(v_{\sigma(1)},\dots, v_{\sigma(2n)})$  where $\sigma(i) \equiv 5i \bmod 2n$.
(Note that $\sigma$ is a bijection, since $5$ is coprime to $2n$.)
Let $F$ be a graph obtained from $P_1 \cup P_2 \cup C$ by adding a vertex $v_0$ and connecting it to $v_1,v_2,v_{n+1},v_{n+2}$, {and adding the edge $v_1 v_{n+1}$.}
Let $G$ be obtained from $F$ by blowing up the vertices $v_1,\dots,v_{2n}$ by a set $V_i$ of size $m$ each, where $m$ is sufficiently large.

Let $\cH = K_3(G)$.
We claim that $\cH$ has two $3$-uniform tight components,
corresponding to the triangles spanned by $P_1,P_2,v_0$ and $C$, respectively.
Indeed, since $G$ is a blow-up of $F$, it suffices to analyse $F$.
The triangles corresponding to $P_1$ can easily seen to belong to the same tight component in $K_3(F)$, and the same is true if we consider $P_2$ or $C$ instead.
Since the edges $v_0 v_1,\, v_0 v_2,\, v_0 v_{n+1},\, v_0 v_{n+2},\, v_1 v_{n+1}$ were added to $F$, we see that the triangles spanned by $P_1$ and $P_2$ can be joined by tight walks via $v_0$, and thus belong to the same tight component.
On the other hand, the choice of $n = 4q$ ensures that no edge of $C$ is shared with $P_1$, $P_2$ or $v_1 v_{n+1}$, so indeed we have two distinct tight components.
However, $G$ does not contain the square of a Hamilton cycle, since the vertex $v_{0}$ is not on a square of a cycle.

To conclude, we show that  $(G,\cH)$ has the property that any $\mu$-approximation $G'$ of $G$ contains a $3$-uniform zero-free Hamilton framework, provided that $1/m \ll \mu \ll 1$.
Indeed, let ${\cH}' = \cH \cap K_3({G}')$, and denote by $V_i'$ the remainder of the cluster $V_i$ in $G'$.
Now, if $v_0 \in V(G')$, then there is a spanning tight component $\cH_1 \subset \cH'$ with edges corresponding to the triangles of $P_1,P_2,v_0$.
On the other hand, if $v_0 \notin V(G')$, then there is a spanning tight component $\cH_2  \subset \cH'$  whose edges correspond to the (remaining) triangles of $C$.
Since $P_1$ and $C$ are cubes of paths and cycles respectively, both $\cH_1$ and $\cH_2$ contain a tetrahedron.
Moreover, it is also not hard to see that $\cH_1$ and $\cH_2$ each have a perfect fractional matching.
Indeed, in the case of $\cH_1$ we begin by matching $v_0$ arbitrarily on a triangle with ends in $V_1'$, $V_2'$ and delete those endpoints from these clusters.
Afterwards, we may find perfect fractional matchings in each $\cH'[V_{4i+1}'\cup \dots \cup V_{4i+4}']$ for $0 \leq i \leq q-1$.
Indeed, this follows by applying \cref{lemma:multipartite-matching} to the $4$-partite graph $\cH[V_{4i+1}\cup \dots \cup V_{4i+4}]$ (with $r = 4$, $k=3$, $\mu = 1/k$, $\mu' = \mu$).
This implies that $\cH'[V_{4i+1}'\cup \dots \cup V_{4i+4}']$ has perfect fractional matchings for each $0 \leq i \leq q-1$.

Repeating the same argument for $P_2$ gives a perfect fractional matching of $\cH_1$.
Similarly, we see that $\cH_2$ has a perfect fractional matching, by first partitioning each $V_i'$ into almost balanced clusters $V_i^1,\dots,V_i^4$ and then finding perfect fractional matchings in each $\cH'[V_{\sigma(i+1)}^1\cup V_{\sigma(i+2)}^2 \cup V_{\sigma(i+3)}^3 \cup V_{\sigma(i+4)}^4]$ for $0 \leq i \leq 2n-1$.
Hence $G'$ does indeed contain a $3$-uniform zero-free Hamilton framework.

\subsection{Searching for Hamilton frameworks} \label{sec:complexity}
In the following, we argue that one can search for a $k$-uniform Hamilton framework in a given  graph $G$ on $n$ vertices in time $O(n^{7k})$.
We did not make a serious effort to optimise the number in the exponent.

We begin by identifying the tight components of $K_k(G)$.
Let $D$ be the directed graph whose vertices are the ordered $k$-cliques of $G$ and a directed edge $K_1K_2$ whenever the last $k-1$ vertices of $K_1$ coincide with the first $k-1$ vertices of $K_2$.
Note that $|V(D)| = O(n^k)$ and $|E(D)| = O(n^{k+1})$.
Recall that two vertices in a directed graph are in the same \emph{strongly connected component} if they are on a common closed directed walk.
Observe that each tight component of $K_k(G)$ corresponds to at most $k!$ strongly connected components in $G$ (corresponding to all possible reorderings of an edge).
This allows us to find all tight components of $G$ by identifying the strongly connected components of $D$, using e.g. Kosajaru's algorithm (see, e.g.~\cite[Section 22.5]{CLRS2009}), which runs in time $O(|V(D)|+|E(D)|) = O(n^{k+1})$.

Now consider a tight component $C \subset K_k(G)$.
Since $K_k(G)$ has at most $n^k$ tight components, it suffices to show that one can check in time $O(n^{6k})$ whether $C$ admits a perfect fractional matching and contains a closed tight walk of length coprime to $k$.
Finding a perfect fractional matching can be phrased as solving an instance of linear programming with $n^k$ variables (one for each edge) and $n$ restrictions (one for each vertex).
This can be done in time $O(n^{6k})$ using interior-point methods~\cite{Sch03}.
Turning to aperiodicity, let $D'$ be the directed graph associated with $C$ in the same way as above.
Observe that a tight closed walk of length $\ell$ in $C$ corresponds to a  closed directed walk of length $\ell$ in $D'$ and vice versa.
Let $d$ be the greatest common divisor of all directed cycle lengths of $D'$.
Note that $d$ is also a common divisor of the lengths of all closed directed walks in $D'$.
Using a depth-first search in $D'$, one can determine $d$ in $O(|V(D')|+|E(D')|) = O(n^{k+1})$ steps~\cite{JS99}.
Hence, it is decidable in time $O(n^{k+1})$ whether $C$ contains a closed tight walk of length coprime to $k$.
Moreover, the same algorithm also finds such a walk when $d$ is coprime to $k$.

\subsection{Linked edges}\label{sec:linked-edges}
This section is dedicated to the proof of \cref{prop:linked-edges-for-free}.
We shall use the following lemma, which is a particular case of the Graph Removal Lemma~\cite{EFR86}.

\begin{lemma} \label{lemma:graphremoval}
	For every $\eps >0$ and $k \in \NATS$ there exists $\delta >0$ and $n_0 \in \NATS$ such that for all graphs $G$ on $n\geq n_0$ vertices the following holds:
	if one needs to delete more than $\eps n^2$ edges to eliminate all $k$-cliques of $G$, then $G$ contains at least $\delta n^k$ many $k$-cliques.
\end{lemma}

\begin{proof}[Proof of \cref{prop:linked-edges-for-free}]
	Clearly, we only need to show that there exists $\mu' \leq \mu$ such that every vertex has at least $\mu' n^k$ linked edges in $K_k(G)$.
	We assume $n$ is sufficiently large. 
	Let $v \in V(G)$ be an arbitrary vertex.
	Let $G_v \subseteq G -  v$ be the subgraph formed by all the edges which intersect $N_G(v)$.
	Note that each edge in $K_k(G_v)$ intersects $N_G(v)$ in at least $k-1$ vertices and does not contain $v$,
	and thus corresponds to a linked edge of $v$ in $K_k(G)$.
	It is therefore enough to show that $G_v$ contains at least $\mu' n^k$ many $k$-cliques.
	
	We want to show that, after the removal of any $(\mu^2/8) n^2$ edges from $G_v$, at least one $k$-clique remains in $G_v$.
	Indeed, suppose we remove $(\mu^2/8) n^2$ edges $Y$ from $G_v$.
	A double-counting argument reveals that at most $(\mu/2) n$ vertices had more than $(\mu/2) n$ edges removed from them, let $X$ be those vertices.
	Note that, in particular since $v \notin V(G_v)$, $v \notin X$.
	In $G$, remove the edges corresponding to $Y$ and also remove the vertices corresponding to $X$,
	to obtain a subgraph $G' \subseteq G$.
	Note that $v \in V(G')$.
	Clearly, $G'$ is a $\mu$-approximation of $G$.
	Since $(G, K_k(G))$ is a $\mu$-proto-robust $k$-uniform Hamilton framework,
	we deduce $K_k(G')$ is a Hamilton framework.
	Since $K_k(G')$ is a Hamilton framework, it contains a perfect fractional matching, which we call $\vec{w}$.
	In particular, since $v$ is covered by $\vec{w}$, there is at least one edge $e_1$ in $K_k(G')$ containing $v$.
	Note that $v(G') = v(G) - |X| \geq n - \mu n / 2 \ge 2k$.
	If every edge of $K_k(G')$ with non-zero weight in $\mathbf{w}$ intersected $v$, then we would have $2 \leq v(G')/k = \sum_{e \in K_k(G')} \vec{w}(e) = \sum_{e\colon v \in e} \vec{w}(e) = 1$, a contradiction.
	We deduce that there must also exist an edge $e_2 \in K_k(G')$ not containing $v$.
	Since $K_k(G')$ is tightly connected, there must exist a tight walk between $e_1, e_2$,
	and the first edge of that walk not containing $v$ is a linked edge of $v$ in $K_k(G')$.
	By construction, this corresponds to a $k$-clique in $G_v - Y$, as required.
	
	Assuming $n$ is sufficiently large, what we have shown implies that $G_v$ satisfies the assumptions of \cref{lemma:graphremoval}.
	Thus there exists $\mu' > 0$ such that the number of $k$-cliques in $G_v$ is at least $\mu' n^k$, as desired.
\end{proof}

\section{Technical tools}

In this appendix we prove lemmas concerning probability or regularity which were required in Section~\ref{sec:regularity} and~\ref{section:intermediate}.
Their involved arguments are by now quite standard.

\subsection{Probability} \label{sec:appendix-probability}
In this section, we show \cref{prop:distributed-matching} by means of a concentration argument.
We recall the following classic inequalities for random variables.

\begin{lemma}[Chernoff's inequality {\cite[Theorem 2.1]{JLR01}}]\label{lem:che}
	Let $0<\alpha<3/2$ and $X \sim \emph{\text{Bin}}(n,p)$ be a binomial random variable.
	Then $\Pr\left(|X - np | > \alpha n p \right) < 2e^{-\alpha^2np/3}$.
\end{lemma}

\begin{lemma}[McDiarmid's inequality~\cite{McDiarmid1989}] \label{theorem:mcdiarmid}
	Suppose $X_1, \dotsc, X_m$ are independent Bernoulli random variables and $b_1, \dotsc, b_m \in [0, B]$.
	Suppose $X$ is a real-valued random variable determined by $X_1, \dotsc, X_m$ such that changing the outcome of $X_i$ changes $X$ by at most $b_i$ for all $1 \leq i \leq m$.
	Then, for all $\lambda > 0$, we have
	\[ \Pr \left(|X - \expectation[X] | > \lambda \right) \leq 2 \exp \left(  -\frac{2 \lambda^2}{B \sum_{i=1}^m b_i} \right).  \]
\end{lemma}

\begin{proof}[Proof of \cref{prop:distributed-matching}]
	Let $\zeta = 4\pi/\gamma$ and $p= \zeta n^{-k+1}$.
	Consider a random selection $\cX \subset V(\cH)^k$ obtained by choosing each $k$-tuple of distinct vertices independently with probability $p$.
	Note that $2\zeta \leq \sqrt{\pi}$, as $\pi \ll \gamma$.
	So $\expectation(|\cX|) \leq p n^k \leq \zeta n$, together with Markov's inequality, gives $\Pr(|\cX| > \sqrt{\pi} n) \leq \Pr(|\cX| > 2 \zeta n) \leq 1/2$.

	Next, let us call two distinct $k$-tuples from $V(\cH)^k$ \emph{overlapping} if there is a vertex occurring in both.
	Note that there are at most $k^2 n^{2k-1}$ pairs of overlapping $k$-tuples.
	Denote by $\cA$ the random variable counting the number of such pairs, both of whose components are in $\cX$.
	We have $\expectation(\cA) \leq k^2 n^{2k-1} p^2 = (k\zeta)^2 n$.
	Moreover, $\pi \geq  4(k \zeta)^2$, since $\pi \ll \gamma \ll 1/k$.
	So another application of Markov's inequality yields $\Pr(\cA > \pi n) \leq \Pr(\cA > 4(k \zeta )^2 n) \leq 1/4$.

	Finally, let $\cB_F$ denote the number of tuples in $\cX$ that correspond to a (directed) edge of some $\cF \in \fF$.
	By assumption, we can bound $\expectation(\cB_\cF) \geq \gamma n^k p = 4\pi n$.
	Moreover, since $\cB_\cF$ is binomially distributed, Chernoff's inequality (\cref{lem:che}) applied with $\alpha =1/2$ implies
	\begin{align*}
		\Pr(\cB_\cF \leq  2 \pi n) \leq 2\exp(-(1/2)^2 4 \pi n/3) = 2\exp(- \pi n/3).
	\end{align*}
	Since $1/n \ll \eta \ll \pi$, we have $2^{\eta n} \cdot 2 \exp(- \pi n/3) < 1/4$.
	It follows by the union bound that the events ``$|\cX| \leq \sqrt{\pi} n$'', ``$\cA \leq \pi n$'', and ``$\cB_\cF \geq 2 \pi n$ for all $\cF \in \fF$'' happen simultaneously with positive probability.
	Let us fix such a choice of $\cX$.
	To obtain the desired matching $M$, we simply take the edges counted by $\cB_\cF$ excluding the at most $\pi n$ overlapping ones.
\end{proof}

\subsection{Regularity}\label{sec:appendix-regularity}
In this section, we show two well-known facts about regularity, \cref{prop:random-refinement,prop:regularity-joint degree}.
For the proof of \cref{prop:random-refinement}, we will use the following handy criterion for regularity.

\begin{proposition}[Codegree criterion for regularity {\cite[Propositions 2.4 and 2.5]{DLR96}}] \label{prop:codegree-regularity}
	Let $1/m \ll \eps, d$.
	Let $G$ be a bipartite graph of density $d$ with colour classes $A$ and $B$ of size at least $m$ each.
	Then the following holds:
	\begin{enumerate}[\upshape (i)]
		\item if the pair $(A,B)$ is $\eps$-regular, then there are at least $(1-5\eps)|A|^2/2$ pairs of vertices $u,v \in A$ that satisfy $\deg(u),\deg(v) \geq ({d}-\eps) |B|$ and $|N(u) \cap N(v)| < (d+\eps)^2 |B|$, and

		\item if at least $(1-5\eps)|A|^2/2$ pairs of vertices $u,v \in A$ satisfy $\deg(u),\deg(v) \geq (d-\eps) |B|$ and $|N(u) \cap N(v)| < (d+\eps)^2 |B|$, then the pair $(A,B)$ is $(16\eps)^{1/5}$-regular.
	\end{enumerate}
\end{proposition}

\begin{proof}[Proof of \cref{prop:random-refinement}]
	Denote the minimum cluster size of $\fV$ by $m$.
	Since $(G, \fV)$ is a $(1+\eps)$-balanced $R$-partition,
	we deduce that every cluster of $\fV$ has size at most $(1+\eps)m$.
	By Chernoff's inequality (\cref{lem:che}), it follows that,
	for every $1 \leq i \leq t$ and $1\leq j \leq q$,
	\begin{align}
		\label{equ:random-refinement-cluster-size}
		(1-\eps)|V_i|/q  \leq |V_{i,j}| \leq (1+\eps)|V_i|/q
	\end{align}
	fails with probability at most $\exp({-\Omega(mq^{-1})})$.

	Denote by $d_{i,i'}$ the density of the pair $(V_i,V_{i'})$ for every $1 \leq i,i' \leq t$ and note that $d_{i,i'} \geq d$.
	For all $1 \leq j, j' \leq q$,
	an edge is included in $e_G(V_{i,j}, V_{i',j'})$ with probability $q^{-2}$,
	thus $\expectation[e_G(V_{i,j}, V_{i',j'})] = d_{i,i'} |V_i| |V_{i'}| q^{-2}$.
	Changing the random outcome of a single vertex in $V_{i} \cup V_{i'}$ changes $e_G(V_{i,j}, V_{i',j'})$ by at most $\max \{ |V_i|, |V_{i'}| \} \leq 2m$.
	Assuming \eqref{equ:random-refinement-cluster-size} holds for all $i, j$,
	an application of McDiarmid's inequality~(\cref{theorem:mcdiarmid}) gives that, for every $1 \leq i, i' \leq t$ and $1\leq j, j' \leq q$,
	\begin{align} 		\label{equ:random-refinement-dense}
		|d(V_{i,j},V_{i',j'})  - d_{i,i'}| \leq 2\eps
	\end{align}
	will fail with probability at most $\exp(-\Omega(m q^{-4}))$.

	Fix $1 \leq i,i' \leq t$ and $1 \leq j \leq q$.
	By \cref{prop:codegree-regularity} there are at least {$(1-5\eps)|V_i|^2/2$} pairs of vertices $u,v \in V_i$ that satisfy $\deg(u{, V_{i'}}),\deg(v{, V_{i'}}) \geq ({d_{i,i'}}-\eps) {|V_{i'}|}$ and $|N(u) \cap N(v) \cap V_{i'}| < (d_{i,i'}+\eps)^2 {|V'_i|}$.
	Let $X_{i,i'}$ be the set of pairs in $V_i$ as before, let $X_{i,i',j} \subseteq X_i$ be those pairs which are contained in $V_{i,j}$.
	Note that $\expectation[|X_{i,i',j}|] = |X_{i,i'}|q^{-2}$.
	Changing the random outcome of a single vertex in $V_i$ changes $X_{i,i',j}$ by at most $|V_{i}| \leq 2m$.
	Assuming item~\eqref{equ:random-refinement-cluster-size} holds for $i, j$,
	an application of McDiarmid's inequality, shows that
	\begin{equation}
		|X_{i,i',j}| \geq (1 - 10 \eps)|V_{i,j}|^2/2
	\end{equation}
	fails with probability at most $\exp(-\Omega(mq^{-4}))$.
	Further applications of McDiarmid's inequality show that, for all fixed $\{u,v\} \in X_{i,i'}$,
	and all $1 \leq j' \leq q$,
	the properties
	\begin{align}
		|N(u) \cap V_{i',j'}|           & \geq ({d_{i,i'}}-2\eps) {|V_{i',j'}|},                                         \\
		|N(v) \cap V_{i',j'}|           & \geq ({d_{i,i'}}-2\eps) {|V_{i',j'}|}, \text{ and}                             \\
		|N(u) \cap N(v) \cap V_{i',j'}| & < (d_{i,i'}+2\eps)^2 {|V_{i',j'}|} \label{equation:random-refinement-codegree}
	\end{align}
	fail with probability at most $\exp(-\Omega(mq^{-2}))$ each.

	Using the union bound,
	we conclude that the probability of failure of items~\eqref{equ:random-refinement-cluster-size} to~\eqref{equation:random-refinement-codegree} for some $1 \leq i, i' \leq t$, some $1 \leq j, j' \leq q$ or some $\{u,v\} \in X_{i,i'}$ is at most
	\begin{align*}
		tq\exp(-\Omega(mq^{-1})) + (t^2q^2 + t^2) \exp(-\Omega(-mq^{-4})) + t^2q(2m)^2 \exp(-\Omega(mq^{-2})) \leq 1/3,
	\end{align*}
	where in the last bound we used $m \ge n/2t$ and $1/n \ll 1/q, 1/t$.

	It remains to show that a partition $\fV^\ast = \{ V_{i,j} \}_{1 \leq i \leq t, 1 \leq j \leq 1}$ satisfying items~\eqref{equ:random-refinement-cluster-size} to~\eqref{equation:random-refinement-codegree} is such that $(G, \fV^\ast)$ is a $(1+\eps')$-balanced $(\eps', d)$-regular $R^\ast$-partition.
	It is clear that it is an $R^\ast$-partition.
	To see that $\fV^\ast$ is $(1+\eps')$-balanced, let $m'$ be the minimal size of a cluster in $\fV$ and let $V_{i,j}$ be any other cluster.
	By item~\eqref{equ:random-refinement-cluster-size}, we have $|V_{i,j}| \leq (1+\eps)|V_i|q^{-1} \leq (1+\eps)^2 m q^{-1} \leq (1+\eps)^2 (1 - \eps)^{-1} m' \leq (1+\eps')m'$,
	where we used that $\fV$ was $(1+\eps)$-balanced and $\eps \ll \eps'$.
	Finally, the claim that the partition is $(\eps', d)$-regular follows readily from items~\eqref{equ:random-refinement-dense} to~\eqref{equation:random-refinement-codegree} together with \cref{prop:codegree-regularity} and $\eps \ll \eps'$.
\end{proof}

\begin{proof}[Proof of \cref{prop:regularity-joint degree}]
	Using the definition of $(\eps,d)$-regularity, it is not hard to see that all but at most $(k-1)\eps m$ vertices of $V_1$ have at least $(d-\eps)m$ neighbours $V_j' \subset V_j$ for $2 \leq j \leq k$.
	Note that the pair $(V_i,V_j')$ is still $(2\eps,d)$-regular as $1/m \ll \eps \ll d$.
	Thus, all but at most $(k-2)2\eps m$ vertices of $V_2$ have at least $(d-2\eps)^2 m$ neighbours $V_j'' \subset V_j'$ for $3 \leq j \leq k$.
	Continuing this way, we see that all but at most  {$\eps' m^{k} / k!$} sets $S=\{v_1,\dots,v_k\}$ with $v_i \in V_i$ have the property that
	${|\bigcap_{1 \leq i \leq k-1} N(v_i)\cap V_k|} \geq (d-\eps')^{k-1} m$.
	By symmetry, the same holds by considering the clusters in any of its $k!$ possible orderings.
	Hence the result follows.
\end{proof}

\section{Structural tools}

In this section, we give proofs of a few structural tools used in Section~\ref{section:intermediate} that we have borrowed from our joint work with Garbe, Lo and Mycroft~\cite{GLL+21}.

\subsection{Allocation} \label{sec:flow}

In this section, we show \cref{prop:flow}.

\begin{proof}[Proof of \cref{prop:flow}]
	We will consider \emph{integer flows} $f$ defined on the $2$-graph $\hat{\cH}$, which are defined as follows.
	Consider the set $D$ which contains the directed edges $(u,v)$ and $(v,u)$ for every unordered edge $uv \in E(\hat{\cH})$.
	An integer flow $f$ is an assignment of integers to the directed edges in $D$, in such a way that $f(x,x')=-f(x',x)$ holds for every edge $xx' \in E(\hat{\cH})$.
	
	Consider a part $U \in \fU$.
	Since $\hat{\cH}{[U]}$ is connected and $\sum_{v \in {U}} \vec{b}(v) = 0$, we can choose an integer flow $f$ such that each vertex $x \in {U}$ receives flow ${\sum_{xx' \in \hat{\cH}{[U]}} f(x',x)} = \vec{b}(x)$.
	Indeed, we can construct such a flow in an iterative fashion, starting from an empty flow.
	If we are not done with our current flow $f$, then there must exist distinct vertices $x$ and $y$ in $U$ such that $0 \leq {\sum_{xx' \in E(\hat{\cH}{[U]})} f(x',x)} < \vec{b}(x)$
	and $\vec{b}(y) < {\sum_{yy' \in E(\hat{\cH}{[U]})} f(y',y)} \leq 0$.
	We can send a unit flow from $y$ to $x$ along a path joining those vertices,
	by changing the flow in the edges of the path by one.
	Clearly, this iteration takes at most $\sum_{v \in {U}} |\vec{b}(v)| \leq sn$ steps.
	Each step changes the flow of an edge at most by one,
	thus we have that in the final flow, each $xx' \in E(\hat{\cH}{[U]})$ satisfies $|f(x,x')| \leq  sn$.
	Repeating this for the others parts of $\fU$ yields the desired flow.
	
	We now define $\vec{w}\in \mathbb{Z}^{E(\cH)}$ as follows.
	Initially, we set $\vec{w} = 0$.
	Consider each $xx' \in E(\hat{\cH})$ in turns.
	Suppose there is a positive flow $c$ from~$x$ to~$x'$, that is $f(x,x') = c \ge 0$.
	Pick $Q_{xx'} \in L_\cH(x) \cap L_\cH(x')$.
	We decrease the edge weight~$\vec{w}(\{x\}\cup Q_{xx'})$ by~$c$ and increase the edge weight~$\vec{w}(\{x'\}\cup Q_{xx'}) $ by~$c$.
	
	The resulting $\vec{w}$ satisfies $\vec b (v) = \sum_{e \colon v \in e} \vec w(e)$ for all $v \in V(\cH)$ by our construction.
	Note that for each edge $e \in E(\cH)$ there are at most $k n$ edges $xx' \in E(\hat{\cH})$ such that $e = \{x\}\cup  Q_{xx'}$ (indeed, there are $k$ choices for $x \in e$ and at most $n$ choices for $x' \notin e$).
	Since each $xx'$ has flow of size at most $sn$,
	we get $\maxnorm{\vec{w}} \le ksn^2$.
\end{proof}

\subsection{Fractional matchings and bounded degree covers}\label{sec:cover}
Here we prove~\rf{lem:kcover}.
To do so, we require a few definitions.
Let $k \in \mathbb{N}$.
We say that a hypergraph~$\cG$ is a \emph{$[k]$-graph} if $|e| \in [k]$ for every edge $e \in E(\cG)$.
We use the following non-standard form of restriction for $[k]$-graphs.
Let $\cG$ be a $[k]$-graph and let $X \subseteq V(\cG)$.
Then $\cG \cap X$ is the $[k]$-graph with vertex set $X$ and edge set
$E( \cG \cap X) := \{ e \cap X \colon e \in E(\cG)\}$.
For convenience, we write in the following $e \in \cG$ instead of $e \in E(\cG)$ and $\REALS^\cG$ instead of $\REALS^{E(\cG)}$ if the context is clear.

The following proposition follows immediately from these definitions.

\begin{proposition} \label{proposition:restrictedmatching}
	Let $\cG$ be a $[k]$-graph which admits a perfect fractional matching~$\vec{w}$ and let $X \subseteq V(\cG)$.
	Then $\vec{w}' \in \mathbb{R}^{\cG \cap X}$ defined by $\vec{w}' (e') = \sum_{{e \in \cG \colon e \cap X = e'}} \vec{w}(e)$ for each edge $e' \in \cG \cap X$ yields a perfect fractional matching $\vec{w}'$ in $\cG \cap X$.
\end{proposition}

\begin{proposition} \label{proposition:edgeoutside}
	Let $\cG$ be a $[k]$-graph on $n$ vertices which admits a perfect fractional matching~$\vec{w}$.
	Then for any set $X \subseteq V(\cG)$ with $|X| < n/k$, there is an edge of~$\cG$ which does not intersect~$X$.
\end{proposition}

\begin{proof}
	Since $\sum_{e \in \cG \colon v \in e} \vec{w}(e) = 1$ for every~$v \in V(\cG)$, we have $\sum_{e \in \cG \colon e \cap X \neq \emptyset} \vec{w}(e) \leq |X|$ and $\sum_{e \in \cG} |e| \vec w(e) = n$.
	It follows from the latter that $\sum_{e \in \cG\colon} \vec{w}(e) \geq n / k > |X|$, so there must be some edge $e$ of~$\cG$ (in particular, one with $\vec{w}(e) > 0$) which does not intersect~$X$.
\end{proof}

\begin{proposition} \label{proposition:fewweights}
	Let $\cG$ be a $[k]$-graph on $n$ vertices which admits a perfect fractional matching~$\vec{w}$.
	Then there exists a perfect fractional matching~$\vec{w}'$ in~$\cG$ in which at most $n$ edges have non-zero weight.
\end{proposition}

\begin{proof}
	Let $P = \{ \vec{1}_e \colon e \in \cG \} \subseteq \{0,1\}^{V(\cG)}$ be the collection of all indicator vectors of the edges of $\cG$.
	Since $\cG$ admits a perfect fractional matching, the all-ones vector $\vec{1} \in \mathbb{R}^{V(\cG)}$ lies in the positive cone generated by the vectors of~$P$, i.e. it can be written as a positive combination of those vectors.
	By the conic version of Carath\'eodory's theorem (see, e.g.,~\cite[Section 1.2, Proposition 1.2.1]{Bertsekas2009}), $\vec{1}$ can in fact be written as a positive combination of at most $n$ vectors in~$P$, say $\vec{1} = \sum_{i \in [n] }\lambda_i \vec{1}_{e_i}$ for some constants $\lambda_i \ge 0 $ and some edges $e_1, \dotsc, e_n \in \cG$.
	It follows immediately that $\lambda_i \leq 1$ for all $i \in [n]$.
	The proof is complete by setting $\vec{w}'(e_i) = \lambda_i$ for each $i \in [n]$, and $\vec{w}'(e) = 0$ for every other edge $e$.
\end{proof}

\begin{proof}[Proof of \cref{lem:kcover}]
	Fix $k$.
	We say that a \emph{$\Delta$-cover} of a hypergraph is a subgraph with minimum vertex-degree at least $1$ and maximum vertex-degree at most $\Delta$.
	We will show the stronger result that every $[k]$-graph with a perfect fractional matching has a $(k^2 + 1)$-cover.
	Suppose for a contradiction that this does not hold, and let $\cG$ be a counterexample of minimal order.
	So $\cG$ is a $[k]$-graph on $n$ vertices which admits a perfect fractional matching~$\vec{w}$ but does not admit a $(k^2+1)$-cover.
	By \cref{proposition:fewweights}, we may assume that at most $n$ edges of~$\cG$ have non-zero weight in~$\vec{w}$.
	So we may assume that $n > k^2+1$.
	Let $\cH$ be the spanning subgraph of~$\cG$ consisting of all such edges, namely $\cH = \{e \in \cG \colon \vec{w}(e) \ne 0 \}$.
	Define $Y \subseteq V(\cH)$ to consist of all vertices of~$\cH$ with degree at least $k^2 + 1$ in~$\cH$.
	Then $(k^2 + 1) |Y| \leq k |\cH| \leq k n$, so $|Y| < n/k$.
	By \cref{proposition:edgeoutside}, it follows that there is a non-empty edge~$e^* \in \cH$ which does not intersect~$Y$; choose such an edge and set $X := V(\cH) \sm e^*$.
	By \cref{proposition:restrictedmatching}, there exists a perfect fractional matching~$\vec{w}'$ in $\cH \cap X$, and since $\cG$ was a minimal counterexample, it follows that $\cH \cap X$ admits a $(k^2+1)$-cover~$\cC'$.
	We now form a $(k^2+1)$-cover~$\cC$ in~$\cH$ as follows.
	For each edge~$e \in \cC'$ choose an edge~$f(e) \in \cH$ with $e \subseteq f(e)$ and $f(e) \cap X = e$ (such an edge $f(e)$ must exist by definition of $\cH \cap X$), and set $\cC := \{f(e) \colon e \in \cC'\} \cup \{e^*\}$.
	Then each vertex of~$X$ is contained in the same number of edges in $\cC$ as in $\cC'$ (that is, between $1$ and $k^2+1$).
	Furthermore, each vertex of~$e^*$ is contained in \emph{at least} one edge of~$\cC$ (namely $e^*$) and \emph{at most} $k^2$ edges of~$\cC$ (since $e^*$ does not intersect~$Y$, so each vertex of $e^*$ has degree at most $k^2$ in $\cH$, and $\cC \subseteq \cH$).
	We conclude that $\cC$ is a $(k^2+1)$-cover of~$\cH$ and therefore of~$\cG$, as required.
\end{proof}

\end{document}